\documentclass[oneside,english,12pt]{amsart}
\usepackage[T1]{fontenc}
\usepackage[latin9]{inputenc}
\usepackage{amstext}
\usepackage{amsthm}
\usepackage{amssymb}
\usepackage{stackrel}
\usepackage[hidelinks]{hyperref}
\usepackage{xcolor}

\usepackage{fullpage}

\makeatletter
\numberwithin{equation}{section}
\numberwithin{figure}{section}
\theoremstyle{plain}
\newtheorem{thm}{\protect\theoremname}[section]
\theoremstyle{remark}
\newtheorem*{rem*}{\protect\remarkname}
\newtheorem{rem}[thm]{\protect\remarkname}
\theoremstyle{plain}
\newtheorem{prop}[thm]{\protect\propositionname}
\theoremstyle{definition}
\newtheorem{defn}[thm]{\protect\definitionname}
\theoremstyle{plain}
\newtheorem{lem}[thm]{\protect\lemmaname}
\newtheorem{cor}[thm]{\protect\corname}

\makeatother

\usepackage{babel}
\providecommand{\definitionname}{Definition}
\providecommand{\lemmaname}{Lemma}
\providecommand{\propositionname}{Proposition}
\providecommand{\remarkname}{Remark}
\providecommand{\theoremname}{Theorem}
\providecommand{\corname}{Corollary}
\providecommand{\exname}{Example}

\begin{document}
\global\long\def\E{\mathbb{E}}
\global\long\def\Unif{\mathrm{Unif}}
\global\long\def\c{\mathfrak{c}}
\global\long\def\Q{\mathbf{\mathbb{Q}}}%
\global\long\def\R{\mathbf{\mathbb{R}}}%
\global\long\def\C{\mathbf{\mathbb{C}}}%
\global\long\def\Z{\mathbf{\mathbb{Z}}}%
\global\long\def\D{\mathbf{\mathbb{D}}}%
\global\long\def\N{\mathbf{\mathbb{N}}}%
\global\long\def\proj{\mathbb{proj}}%
\global\long\def\T{\mathbb{T}}%
\global\long\def\Im{\mathrm{Im}}%
\global\long\def\Re{\mathrm{Re}}%
\global\long\def\eval{\mathrm{eval}}%
\global\long\def\irr{\mathrm{irr}}%
\global\long\def\err{\mathrm{err}}%
\global\long\def\id{\mathrm{id}}%
\global\long\def\H{\mathcal{H}}%
\global\long\def\M{\mathbb{M}}%
\global\long\def\P{\mathcal{P}}%
\global\long\def\L{\mathcal{L}}%
\global\long\def\F{\mathcal{\mathcal{F}}}%
\global\long\def\s{\sigma}%
\global\long\def\Rc{\mathcal{R}}%
\global\long\def\W{\tilde{W}}%
\global\long\def\avg{\mathrm{avg}}%
\global\long\def\dist{\mathrm{dist}}%

\global\long\def\G{\mathcal{G}}%
\global\long\def\d{\partial}%
 
\global\long\def\jp#1{\langle#1\rangle}%
\global\long\def\ul#1{\underline{#1}}%
\global\long\def\norm#1{\|#1\|}%
\global\long\def\mc#1{\mathcal{\mathcal{#1}}}%
\global\long\def\diam{\mathrm{diam}}%
\global\long\def\rank{\mathrm{rank}}%
\global\long\def\vol{\mathrm{Vol}}%
\global\long\def\Alt{\mathrm{Alt}}%
\global\long\def\Seq{\mathrm{Seq}}%
\global\long\def\Span{\mathrm{Span}}%
\global\long\def\Char{\mathrm{Char}}%

\global\long\def\Right{\Rightarrow}%
\global\long\def\Left{\Leftarrow}%
\global\long\def\les{\lesssim}%
\global\long\def\hook{\hookrightarrow}%

\global\long\def\err{\mathrm{err}}%
\global\long\def\rad{\mathrm{rad}}%

\global\long\def\lcm{\mathrm{lcm}}%
\global\long\def\env{\mathrm{Env}}%
\global\long\def\re{\mathrm{re}}%
\global\long\def\im{\mathrm{im}}%

\global\long\def\d{\partial}%
 
\global\long\def\jp#1{\langle#1\rangle}%
\global\long\def\norm#1{\|#1\|}%
\global\long\def\ol#1{\overline{#1}}%
\global\long\def\wt#1{\widehat{#1}}%
\global\long\def\tilde#1{\widetilde{#1}}%

\global\long\def\br#1{(#1)}%
\global\long\def\Bb#1{\Big(#1\Big)}%
\global\long\def\bb#1{\big(#1\big)}%
\global\long\def\lr#1{\left(#1\right)}%

\global\long\def\ve{\varepsilon}%
\global\long\def\la{\lambda}%
\global\long\def\al{\alpha}%
\global\long\def\be{\beta}%
\global\long\def\ga{\gamma}%
\global\long\def\La{\Lambda}%
\global\long\def\De{\Delta}%
\global\long\def\na{\nabla}%
\global\long\def\ep{\epsilon}%
\global\long\def\fl{\flat}%
\global\long\def\sh{\sharp}%
\global\long\def\calN{\mathcal{N}}%
\global\long\def\supp{\mathrm{supp}}%
\global\long\def\UAP{\mathrm{UAP}}%

\global\long\def\Q{\mathbf{\mathbb{Q}}}%
\global\long\def\R{\mathbf{\mathbb{R}}}%
\global\long\def\C{\mathbf{\mathbb{C}}}%
\global\long\def\Z{\mathbf{\mathbb{Z}}}%
\global\long\def\N{\mathbf{\mathbb{N}}}%
\global\long\def\T{\mathbb{T}}%
\global\long\def\Im{\mathrm{Im}}%
\global\long\def\Re{\mathrm{Re}}%
\global\long\def\H{\mathcal{H}}%
\global\long\def\P{\mathbb{P}}%
\global\long\def\F{\mathcal{F}}%
\global\long\def\G{\mathcal{G}}%
\global\long\def\s{\sigma}%

\global\long\def\d{\partial}%
 
\global\long\def\jp#1{\langle#1\rangle}%
\global\long\def\norm#1{\|#1\|}%
\global\long\def\mc#1{\mathcal{\mathcal{#1}}}%

\global\long\def\les{\lesssim}%
\global\long\def\hook{\hookrightarrow}%
\global\long\def\weak{\rightharpoonup}%
\global\long\def\supp{\mathrm{supp}}%
\global\long\def\loc{\mathrm{loc}}%

\global\long\def\err{\mathrm{err}}%
\global\long\def\poly{\mathrm{poly}}%
 
\global\long\def\jp#1{\langle#1\rangle}%
\global\long\def\norm#1{\|#1\|}%
\global\long\def\tilde#1{\widetilde{#1}}%

\global\long\def\la{\lambda}%
\global\long\def\al{\alpha}%
\global\long\def\be{\beta}%
\global\long\def\ga{\gamma}%
\global\long\def\La{\Lambda}%
\global\long\def\De{\Delta}%
\global\long\def\na{\nabla}%

\global\long\def\rc{\mathrm{rc}}%

\title[GWP of the cubic NLS on $\T^{2}$]{Global well-posedness of the cubic nonlinear Schr\"odinger equation on $\T^{2}$}

\author{Sebastian Herr}
\address{Fakultat f\"ur Mathematik, Universit\"at Bielefeld, Postfach 10 01 31, 33501 Bielefeld, Germany}
\email{herr@math.uni-bielefeld.de}
\author{Beomjong Kwak}
\address{Department of Mathematical Sciences, KAIST, 291 Daehak-ro, Yuseong-gu, Daejeon, Korea}
\email{beomjong@kaist.ac.kr}

\begin{abstract}
    We prove global well-posedness for the cubic nonlinear Schr\"odinger equation for periodic initial data in the mass-critical dimension $d=2$ for initial data of arbitrary size in the defocusing case and data below the ground state threshold in the focusing case. 
   The result is based on a new inverse Strichartz inequality, which is proved by using incidence geometry and additive combinatorics, in particular, the inverse theorems for Gowers uniformity norms by Green-Tao-Ziegler. This allows to transfer the analogous results of Dodson for the non-periodic mass-critical NLS to the periodic setting.
   In addition, we construct an approximate periodic solution that implies the sharpness of the results.
\end{abstract}

\maketitle
\section{Introduction and main results}\label{sec:intro}
 Consider the square torus $\T^2=\R^2/(2\pi\Z)^2$. In this paper we prove the global well-posedness of the Cauchy problem of the nonlinear Schr{\"o}dinger equation (NLS) on $\T^{2}$
\begin{equation}
\begin{split}
i\partial_t u+\De u&=\pm|u|^{2}u,\\
u(0)&=u_{0}\in H^{s}(\T^{2}),
\end{split}\tag{{NLS}}\label{eq:NLS}
\end{equation}
with \emph{defocusing} nonlinearity $\mc N(u)=\left|u\right|^{2}u$ or \emph{focusing} nonlinearity $\mc N(u)=-\left|u\right|^{2}u$.

On $\R^2$ the cubic NLS is scale invariant in $L^2(\R^2)$, therefore it is referred to as the mass-critical problem. In the last decade, complete global well-posedness and scattering theory in critical space $L^2(\R^2)$ have been developed in the seminal work of Dodson \cite{dodson-foc,dodson-def}.

In the case of NLS focusing on $\R^{2}$, there exists a ground state solution $Q$, that is, the unique positive radially symmetric Schwartz solution to $\De Q-Q=-Q^3$. See 
\cite{berestycki-lions} concerning the existence, \cite{Kwong89} for a proof of uniqueness, and \cite{Wei} for the relation to the Gagliardo-Nirenberg inequality. This gives rise to the solution $u(t,x)=e^{it}Q(x)$ of NLS and by pseudo-conformal invariance this yields an explicit finite time blow-up solution. Therefore, in the focus case, $\|Q\|_{L^2(\R^2)}$ is the natural threshold for the size of the initial data for global existence (also called the ground state threshold).

The corresponding theory in the periodic setting for the defocusing and focusing NLS Cauchy problem has been open since the work of Bourgain \cite{bourgain1993fourier}, where the first subcritical local well-posedness and global existence in the energy space $H^1(\T^2)$ (below the ground state threshold in the focusing case) were established. Very recently, small data global well-posedness in the full subcritical range has been proven \cite{herr2024strichartz}. In the present paper, we consider the large data problem and establish the analogue of Dodson's results in the periodic setting.
Note that solutions do not converge to a free solution in the periodic problem, i.e., there is no scattering \cite[Appendix]{ckstt}.

\begin{thm}[GWP for defocusing NLS]
\label{thm:GWP defocus}Let $s>0$. The defocusing \eqref{eq:NLS} is globally well-posed for initial data $u_0 \in H^s(\T^2)$.
Moreover, we have the following quantitative bound:
Let $M>0$ and $T>0$. For $u_0\in H^s(\T^2)$ such that $\norm{u_0}_{L^2(\T^2)}\le M$, the solution $u$ to the defocusing \eqref{eq:NLS} with the initial data $u_0$ satisfies
\begin{equation}\label{eq:GWP defocus L4 bound}
\norm{u}_{L^4_{t,x}([0,T)\times\T^2)}\les_{s,M,T}\left(\log\norm{u_0}_{H^s}\right)^{1/4}.
\end{equation}
\end{thm}
Here and in the sequel, we are using the notation $\log(x)=1+\log^+(x)$, $x>0$.
The optimality of the regularity assumption will be discussed in the following.

In \cite{bgt-03} it was shown that in the focusing cubic NLS there exist finite-time blow-up solutions in $H^1(\T^2)$ with initial data $u_0\in H^1(\T^2)$ such that $\norm{u_0}_{L^2(\T^2)}=\norm Q_{L^2(\R^2)}$. 
 We have the following sharp result in this case:

\begin{thm}[GWP for focusing NLS]
\label{thm:GWP focus}
 Let $s>0$. The focusing \eqref{eq:NLS} is globally well-posed for initial data in $H^s(\T^2)$ such that $\norm{u_{0}}_{L^{2}(\T^{2})}<\norm Q_{L^{2}(\R^{2})}$.
Moreover, we have the following quantitative bound:
Let $0<M<\norm Q_{L^2(\R^2)}$ and $T>0$. For $u_0\in H^s(\T^2)$ such that $\norm{u_0}_{L^2(\T^2)}\le M$, the solution $u$ to the focusing \eqref{eq:NLS} with the initial data $u_0$ satisfies
\begin{equation}\label{eq:GWP focus L4 bound}
\norm{u}_{L^4_{t,x}([0,T)\times\T^2)}
\les_{s,M,T}\left(\log\norm{u_0}_{H^s}\right)^{1/4}.
\end{equation}
\end{thm}

For energy-critical nonlinear Schr\"odinger equations the sharp global well-posedness theory  \cite{colliander2008global,ryckman2007global,visan2007defocusing,visan2012global,kenig2006global} has been transferred from $\R^d$ to the periodic setting, we refer to \cite{ionescu2012energy,YUE2021754,Kwak2024global}. Theorems \ref{thm:GWP defocus} and \ref{thm:GWP focus} are the first such results for a mass-critical nonlinear Schr\"odinger equation.

The key ingredient in the proofs is a new inverse Strichartz estimate. Before stating this, we recall the sharp $L^4$-Strichartz estimate.
\begin{prop}[Theorem 1.2 in \cite{herr2024strichartz}]\label{prop:L4 Stri}
There exists $c>0$ such that for all bounded sets $S\subset\Z^2$ and all $\phi\in L^2(\T^2)$, we have
\begin{equation}\label{eq:L4 Stri}
\norm{e^{it\De}P_S\phi}_{L^4_{t,x}([0,\frac{1}{\log\#S}]\times\T^2)}\le c\norm{\phi}_{L^2(\T^2)}.
\end{equation}
\end{prop}
The interval size $\frac{1}{\log\#S}$ in \eqref{eq:L4 Stri} turns out to be sharp \cite{herr2024strichartz,kishimoto2014remark}. As a consequence \cite[Theorem 1.4]{herr2024strichartz}, for $s>0$ the cubic NLS is globally well-posed for initial data in $H^s(\T^2)$  which is small in $L^2(\T^2)$, which is the optimal semi-linear (perturbative) result (see Corollary \ref{cor:unif-cont}). Also, we point out that the $L^4$-Strichartz estimate \eqref{eq:L4 Stri} on the unit interval $[0,1]$ requires a $\log(\# S)^{\frac14}$ loss \cite{herr2024strichartz}.

In this work, we show the inverse result of Proposition \ref{prop:L4 Stri}.
\begin{thm}\label{thm:inverse L4 Stri}
Let $\epsilon>0$. There exists $\delta>0$ satisfying the following:

For every $\phi\in L^2(\T^2)$ with bounded Fourier support $S=\supp(\widehat\phi)\subset\Z^2$ such that
\[
\norm{e^{it\De}\phi}_{L^4_{t,x}([0,\frac{\delta}{\log\#S}]\times\T^2)}\ge\epsilon\norm{\phi}_{L^2(\T^2)},
\]
there exist $N\in2^\N$, $t_0\in[0,2\pi)$, $x_0\in\T^2$, and $\xi_0\in\Z^2$ such that
\begin{equation}
\left|\jp{\psi,e^{it_0\De}\phi}_{L^2(\T^2)}\right|\gtrsim_\epsilon\norm{\phi}_{L^2(\T^2)}
\end{equation}
for the \emph{profile}
\[
\psi(x)=N^{-1}e^{i\xi_0\cdot x}P_N\delta(x-x_0).
\]
\end{thm}
Here, $P_N\delta$ denotes the Littlewood-Paley cutoff of the Dirac distribution $\delta(\cdot)$ on $\T^2$.
The size of the interval $[0,\frac{\delta}{\log\#S}]$ in Theorem \ref{thm:inverse L4 Stri} is sharp, as one can directly check by the nonexample $\phi=N^{-1}\F^{-1}(\chi_{N\Z^2\cap[-N^2,N^2]^2})$. The parameters $N,t_0,x_0,\xi_0$ correspond to the scaling, time and space translations, and the Galilean symmetry of the Schr\"{o}dinger operator, all of which are crucial.

The proof of Theorem \ref{thm:inverse L4 Stri} is one of the main contributions of this paper. We develop a new method based on additive combinatorics 
to prove this PDE result. We use the theory of sum sets
and  multiprogressions to reduce the problem to square Fourier supports. Then, we use deep results, mainly from \cite{GreenTao08(U3),GreenTao12Quantitative,GreenTaoZiegler12(Ud)}, on inverse theorems for Gowers norms.
Roughly speaking, in additive combinatorics, Gowers norms are used to detect quasipolynomial behaviour. We develop methods that transfer the largeness of the $L^4$-Strichartz norm to largeness of Gowers norms. Then, the above results are used to extract quadratic phase functions and deduce Theorem \ref{thm:inverse L4 Stri}.

Given Theorem \ref{thm:inverse L4 Stri}, we proceed with the proofs of Theorems \ref{thm:GWP defocus} and \ref{thm:GWP focus} as follows: Assuming the contrary, we can use Theorem \ref{thm:inverse L4 Stri} and a new transference argument (see also the recent work \cite{Kwak2024global} of the second author for an energy-critical case) to obtain a contradiction to Dodson's results \cite{dodson-def,dodson-foc}.

In addition, we construct a family of solutions which is interesting in its own right.
\begin{thm}\label{thm:L4 counterexample lem}
For $N\in2^\N$, denote by $\phi^N$ the function
\[
\phi^N:=\F^{-1}(\chi_{N^{10}\Z^2}\cdot e^{-|\xi/N^{11}|^2}).
\]
Let $T>0$ and $\la>0$ be a small number. Let $u^N\in C_{\loc}^\infty(\R \times \T^2)$ be the solution to \eqref{eq:NLS} with the initial datum $u^N_0:=\la N^{-1}\phi^N$ provided by \cite[Theorem 1.4]{herr2024strichartz}. 
We have 
\begin{align*}
&\|u_0^N\|_{L^2}\sim \lambda, \quad  \log \|u_0^N\|_{H^1}\lesssim \log N, \text{ and }\\
&\limsup_{N\rightarrow\infty}\norm{u^N(t)-e^{\mp 3it\la^2\ln N}e^{it\De}u^N_0}_{C^0L^2\cap L^4_{t,x}([0,\frac{T}{\log N})\times\T^2)}=0,
\end{align*}
where $\pm$ corresponds to the sign of the nonlinearity of \eqref{eq:NLS}.
\end{thm}

We remark that the phase correction factor $3$ is  a result of a subtle computation involving the density of coprime integers, see Remark \ref{rem:dens} for more details.

As a first consequence of Theorem \ref{thm:L4 counterexample lem} we obtain
\begin{cor}\label{cor:unif-cont}
The flow map $u_0\mapsto u$ of \eqref{eq:NLS} fails to be locally uniformly continuous in $L^2(\T^2)$.
\end{cor}
We refer to Proposition \ref{prop:unif counterexample} for more details. Note that in \cite[Cor.~1.3]{kishimoto2014remark} it was proven that the flow map fails to have bounded derivatives of order $3$ at the origin in $L^2(\T^2)$. 

As a second consequence of Theorem \ref{thm:L4 counterexample lem}, we obtain
\begin{cor}\label{cor:sharp-bound}
Fix $T\geq 1$ and $\la>0$ small. There exist solutions $u^N \in C_{\loc}^\infty(\R \times \T^2)$ such that $\norm{u^N(0)}_{L^2(\T^2)}\sim \la$, $\log \|u^N(0)\|_{H^1}\lesssim \log N$, and 
\[
\limsup_{N\rightarrow\infty}\norm{u^N}_{L^4_{t,x}([0,\frac{T}{\log N})\times\T^2)}\gtrsim_\la T^{1/4}.
\]
\end{cor}
This implies the sharpness of the bounds \eqref{eq:GWP defocus L4 bound} and \eqref{eq:GWP focus L4 bound} for solutions to the nonlinear equation: First, there cannot be a quantitative $L^4$-bound for initial data at the $L^2$ regularity. Second, the bounds are sharp at the time scale $T\sim (\log \|u_0\|_{H^s})^{-1}$ for large $\|u_0\|_{H^s}$, see \eqref{eq:prec-bound} for a more precise statement of the bound.

\subsection*{Organisation of the paper}
In Section \ref{sec:pre}, notation and preliminary results on incidence geometry and additive combinatorics are introduced.
In Section \ref{app-sec:deg}, inverse Gowers theorems and equidistribution theory of nilsequences are recalled and a degree-lowering theorem is proved.
In Section \ref{sec:norms-rec}, norms and inverse theorems concerning rectangular resonances are discussed.
Section \ref{sec:prof} contains an extinction lemma, introduces periodic extensions and frames, and provides an inverse theorem for a bounded sum of profiles. Section \ref{sec:proof-inv} contains the proof of Theorem \ref{thm:inverse L4 Stri}. 
In Section \ref{sec:GWP}, the large data global well-posedness results for
\eqref{eq:NLS}, i.e.\ Theorems \ref{thm:GWP defocus} and \ref{thm:GWP focus}, are proved.
In Section \ref{sec:proof-approx}, the proofs of Theorem \ref{thm:L4 counterexample lem} and the Corollaries \ref{cor:unif-cont} and \ref{cor:sharp-bound} are provided.

\subsection*{Logical structure}
The logical structure of the proofs of Theorem \ref{thm:GWP defocus} and Theorem \ref{thm:GWP focus} is similar. Both rely on an indirect argument, which is based on a concentration argument (see Section \ref{sec:GWP}) and Theorem \ref{thm:inverse L4 Stri} (more precisely, Prop.\ \ref{prop:profile'}), which finally allows us to transfer Dodson's results \cite{dodson-def, dodson-foc} to the periodic setting. These steps essentially rely on harmonic analysis and PDE techniques, with the notable exception of Theorem \ref{thm:inverse L4 Stri}.
A significant part of this paper is devoted to the proof of Theorem \ref{thm:inverse L4 Stri}, see Sections \ref{sec:pre}, \ref{app-sec:deg}, \ref{sec:norms-rec}, and \ref{sec:proof-inv}. Here, we use methods and results from incidence geometry and additive combinatorics, such as Gowers uniformity norms and degree lowering theorems.
Finally, Theorem \ref{thm:L4 counterexample lem} and its corollaries are based on an explicit construction, see Section \ref{sec:proof-approx}.

\section{Preliminaries}\label{sec:pre}
For positive reals $A,B>0$, we denote $A\les B$ if $A\le CB$ for some absolute constant $C$. We denote by $A\sim B$ the comparability, i.e., $A\les B\les A$. We denote $A\ll B$ if $A\le \epsilon B$ holds for some sufficiently small constant $\epsilon>0$. We denote by subscripts (e.g., $\les_a$) to denote the parameters $C$ or $\epsilon$ depend on.

For $N\in\N$, $[N]$ denotes the integer set $[N]=\{-N,\ldots,N\}$. More generally, for a positive real number $r>0$, $[r]$ denotes $[\lfloor r\rfloor]$.

Given a set $E$, we denote by $\chi_E$ the sharp cutoff at $E$. For a proposition $P$, we denote by $1_P$ the indicator function
\[
1_P:=\begin{cases}
1&,\,$P$\text{ is true}
\\0&,\,\text{otherwise}
\end{cases}.
\]
We use the expectation notation: for a finite set $S\neq\emptyset$ and $f:S\rightarrow\C$,
\[
\E_{x\in S}f(x):=\frac{1}{\#S}\sum_{x\in S}f(x).
\]
For $d\in\N$, $S\subset\Z^d$, and $f:S\rightarrow\C$, we use the convention $f(x)=0$ for $x\in\Z^d\setminus S$.

As subsets of $\C$, $\T$ and $\D$ denote the sets $\T:=\left\{ z\in\C:\left|z\right|=1\right\} $ and $\D:=\{z\in\C:|z|\le 1\}$.
As a quotient of $\R$, $\T$ denotes $\R/2\pi\Z$.

A parallelogram is a quadruple $Q=(\xi_1,\xi_2,\xi_3,\xi_4)$ in $\R^m,m\in\N$ such that $\xi_1+\xi_3=\xi_2+\xi_4$. For the case $m=2$, we denote by $\mc Q$ the set of all parallelograms with vertices in $\Z^2\subset\R^2$. For $Q=(\xi_{1},\xi_{2},\xi_{3},\xi_{4})\in\mc Q$, we denote $\tau_{Q}:=2(\xi_{1}-\xi_{2})\cdot(\xi_{1}-\xi_{4})$.
For $\tau\in\Z$, $\mc Q^{\tau}$
denotes the set of $Q\in\mc Q$ such that $\tau_{Q}=\tau$. $\mc Q(S)$ and
$\mc Q^{\tau}(S)$ denote the subsets of $\mc Q$ and $\mc Q^{\tau}$, respectively, whose
vertices lie in $S$. For $M\in\N$, we denote $\mc Q^{\le M}=\cup_{\tau=0}^M\mc Q^\tau$.

For $Q=(\xi_1,\xi_2,\xi_3,\xi_4)\in\mc Q$ and $f:\Z^2\rightarrow\C$, we use the convention $f(Q)=f(\xi_1)\ol{f(\xi_2)}f(\xi_3)\ol{f(\xi_4)}$.

For two vectors $u,v\in\R^2\setminus\{0\}$, we denote by $\angle(u,v)\in(-\pi,\pi]$ the counterclockwise angle between $u,v$.

For $\xi=(a,b)\in\R^2$, $\xi^\perp$ denotes the counter-clockwise rotation $\xi^\perp:=(-b,a)$.

For an integer point $\xi=(a,b)\in\Z^2\setminus\{0\}$, we denote $\gcd(\xi):=\gcd(a,b)$.

We denote $\Z^2_\irr=\{\xi\in\Z^2\setminus\{0\}:\gcd(\xi)=1\}$.

For $f:\T^d\rightarrow\C$, $d\in\N$, we denote by either $\F(f)=\F_{\T^d}(f)$ or $\widehat{f}$ the Fourier series of $f$.

We use Littlewood-Paley projection operators. Denote the set of dyadic numbers by $2^\N=\{1,2,\ldots\}$. Let $\psi:\R\rightarrow[0,\infty)$ be a smooth even bump function such that $\psi\mid_{[-1,1]}=1$ and $\supp(\psi)\subset[-\frac{11}{10},\frac{11}{10}]$. For $N\in2^\N$, we denote by $\psi^N$ the function $\psi^N(\xi)=\psi(\xi/N)$ and $P_{\le N}$ the Fourier multiplier induced by $\psi^N(\xi_1)\cdots\psi^N(\xi_d)$. We denote $P_N:=P_{\le N}-P_{\le N/2}$ with the convention $P_{\le 2^{-k}}:=0$ for $k>0$. For simplicity, we use abridged notations $u_N:=P_Nu$ and $u_{\le N}:=P_{\le N}u$ for a function $u:\T^d\rightarrow\C$.

Analogous to $\T^d$, we use Littlewood-Paley operators on $\R^d$. We use the same notation, except that we allow $N\in2^\Z$.

For combinatorial discussions, we also use the sharp Fourier cutoffs on $\T^2$. For $S\subset\Z^2$ and $f\in L^2(\T^2)$, we denote $P_Sf:=\F^{-1}(\chi_S\widehat f)$.

We provide preliminary facts for additive combinatorics (and incidence geometry).
\subsection*{Szemer\'edi-Trotter Theorem}
\begin{prop}[\cite{Szemeredi-trotter}]
Let $m,n\in\N$. Let $S\subset\R^{2}$ be a set of $n$ points. Let $\mc L$ be a set of $m$ lines on $\R^2$. We have
\begin{equation}\label{eq:SzTr fundamental}
\#\{(p,\ell)\in S\times\mc L:p\in \ell\}\les m^{2/3}n^{2/3}+m+n.
\end{equation}
\end{prop}
\begin{prop}[{\cite[Corollary 8.5]{tao2006additive}}]
Let $n$ and $k\ge 2$ be integers. Let $S\subset\R^{2}$ be a set of $n$ points. The number $m$ of lines on $\R^{2}$ passing
through at least $k$ points of $S$ is bounded by
\begin{equation}
m\les\frac{n^{2}}{k^{3}}+\frac{n}{k}.\label{eq:SzTr}
\end{equation}
\end{prop}
\begin{lem}\label{lem:SzTr'}
Let $n\in\N$ and $k\ge 2$ be integers. Let $S\subset\R^{2}$ be a set of $n$ points. We have
\begin{equation}\label{eq:SzTr'}
\#\{(p,\ell):\ell\text{ is a line through $p\in S$ and }\#(\ell\cap S)\ge k\}\les \frac{n^2}{k^2}+n.
\end{equation}
\end{lem}
\begin{proof}
Let $\mc L$ be the set of lines $\ell$ such that $\#(\ell\cap S)\ge k$. By \eqref{eq:SzTr}, we have $\#\mc L\les \frac{n^2}{k^3}+\frac{n}{k}$. Plugging to \eqref{eq:SzTr fundamental}, the size of \eqref{eq:SzTr'} is bounded by
\[
\Bb{\frac{n^2}{k^3}+\frac{n}{k}}^{2/3}n^{2/3}+\frac{n^2}{k^3}+\frac{n}{k}+n\les\frac{n^2}{k^2}+\frac{n^{4/3}}{k^{2/3}}+n\les\frac{n^2}{k^2}+n,
\]
finishing the proof.
\end{proof}
We recall results on counting rectangles and parallelograms from \cite{pach1992repeated,herr2024strichartz}, shown using the Szemer\'{e}di-Trotter Theorem.
\begin{prop}\label{prop:count log}
Let $S\subset\Z^{2}$ be a finite set. We have
\begin{equation}\label{eq:count log}
\#\mc Q^{0}(S)\les\log\#S\cdot(\#S)^{2}
\end{equation}
and, for $M\ge\log \#S$,
\begin{equation}\label{eq:para-count}
    \# \mc Q^{\leq M}(S) \les M (\#S)^2.
\end{equation}
\end{prop}
\begin{proof}
\eqref{eq:count log} is a consequence of \cite{pach1992repeated}, which showed that the number of right triangles whose vertices are in $S$ are $O(\log\#S\cdot(\#S)^2)$. 
For the proof of \eqref{eq:para-count}, using the Fej\'er kernel $F_{2M}$, denoting $f=\chi_S$, we have
\[
\#\mc Q^{\le M}(S)\les\sum_{\tau=-M}^{M}\frac{2M-|\tau|}{2M}\sum_{Q\in\mc Q^\tau}f(Q)\les\int_{[-\pi,\pi]\times\T^2}F_{2M}(t)|e^{it\De}\F^{-1}f|^4dxdt.
\]
Since $F_{2M}(t)\les \min\{M,\frac{1}{Mt^2}\}$, by \eqref{eq:L4 Stri} and a decomposition of $[-\pi,\pi]$ into subintervals of length $1/\log\#S$, this is bounded by \begin{equation}\label{eq:Fejer bound}
\int_{[-\pi,\pi]\times\T^2}F_{2M}(t)|e^{it\De}\F^{-1}f|^4dxdt\les M\norm{\F^{-1}f}_{L^2(\T^2)}^4\les M\norm{f}_{\ell^2(\Z^2)}^4,
\end{equation}
from which \eqref{eq:para-count} is immediate. (Plugging $M\sim\log\#S$ also yields \eqref{eq:count log}.)
\end{proof}
\subsection*{Roth's Theorem}

The following is Roth's Theorem on arithmetic progressions of length $3$ (which was further generalized to arbitrary length $k\ge 3$ in \cite{Szemeredi-sets}):
\begin{prop}[\cite{Roth}]
\label{prop:Sz+}For $\epsilon>0$, there
exists $N_{\epsilon}\in\N$ satisfying the following:

For $N\ge N_{\epsilon}$ and a set $S\subset\left\{ 1,...,N\right\} $ such that $\#S\ge\epsilon N$, there exists an arithmetic progression of length $3$ contained in $S$.
\end{prop}
\subsection*{A decomposition lemma}

We introduce a Stone-Weierstra{\ss} technique that develops inverse inequalities to profile decompositions.
\begin{lem}\label{lem:multiplicative decomp}
For $\epsilon,\delta>0$, there exist $k,J=O_{\delta,\epsilon}(1)$ satisfying the following:

Let $S\neq\emptyset$ be a finite set and $\norm{\cdot}$ be a norm on $f:S\rightarrow\C$ such that $\norm{f}\le\norm{f}_{\ell^\infty(S)}$. Let $\mc G$ be a set of functions $g:S\rightarrow\D$ closed under the complex conjugation. Denote
\[
\mc G^k:=\{g_1\cdots g_k:g_1,\ldots,g_k\in\mc G\}.
\]
Assume that for $f:S\rightarrow\D$ such that $\norm{f}\ge\epsilon$, there exists $g\in\mc G$ such that
\[
|\jp{f,g}_{\ell^2(S)}|\ge\delta\#S.
\]
Then, for any $f:S\rightarrow\D$, there exist $c_1,\ldots,c_J\in\D$ and $g_1,\ldots,g_J\in\mc G^k$ such that
\[
h:=\sum_{j\le J} c_jg_j
\]
satisfies
\begin{equation}\label{eq:h l infty<1}
\norm{h}_{\ell^\infty(S)}\le 1,
\end{equation}
\begin{equation}\label{eq:f-h norm}
\norm{f-h}<\epsilon,
\end{equation}
and
\begin{equation}\label{eq:f-h,h ortho}
\left|\jp{f-h,h}_{\ell^2(S)}\right|<\epsilon\#S.
\end{equation}
\end{lem}
\begin{proof}
Let $h_0=0$, which satisfies \eqref{eq:f-h,h ortho} and \eqref{eq:h l infty<1}. Let $\psi_0:\C\rightarrow\D$ be the function
\[
\psi_0(z):=\begin{cases}
z&\qquad,|z|\le1
\\ z/|z|&\qquad,|z|>1.
\end{cases}
\]
By the Stone-Weierstra{\ss} Theorem, there exists a sequence $\{\psi_m\}$ of polynomials of $z,\bar z$ such that $\norm{\psi_m-\psi_0}_{C^0(m\D)}\rightarrow0$. Up to a slight resizing, $\norm{\psi_m}_{C^0(m\D)}\le 1$ can further be assumed.

For $k\in\N$, we define $h_{k+1}:S\rightarrow\D$ recursively in terms of $h_k:S\rightarrow\D$ as follows:

We stop if $h=h_k$ satisfies both \eqref{eq:f-h norm} and \eqref{eq:f-h,h ortho}. In case that \eqref{eq:f-h norm} fails, there exists $g\in\mc G$ such that
\begin{equation}\label{eq:f-h,g big}
|\jp{f-h_k,g}_{\ell^2(S)}|\ge\delta\#S.
\end{equation}
Then by \eqref{eq:f-h,g big} and $\norm{f-h_k}_{\ell^\infty}\le 1+1=2$, we have
\begin{equation}\label{eq:updater g not small}
\norm{g}_{\ell^2(S)}\ge\frac{\delta}2\sqrt{\#S}.
\end{equation}
Set $h_k'$ as the orthogonal projection
\[
h_k'=h_k-\frac{\jp{h_k-f,g}_{\ell^2(S)}}{\norm{g}_{\ell^2(S)}^2}g,
\]
which satisfies by \eqref{eq:updater g not small} that
\begin{equation}\label{eq:f-h_k' is smaller}
\norm{f-h_k'}_{\ell^2(S)}^2\le\norm{f-h_k}_{\ell^2(S)}^2-\frac{\delta^2}{4}\#S.
\end{equation}
Since $|f(x)|\le1$ for each $x\in S$, $|f-\psi_0(h_k')|\le|f-h_k'|$ holds and thus by \eqref{eq:f-h_k' is smaller},
\[
\norm{f-\psi_0(h_k')}_{\ell^2(S)}^2\le\norm{f-h_k}_{\ell^2(S)}^2-\frac{\delta^2}{4}\#S
\]
holds. Then, since $\norm{\psi_m-\psi_0}_{C^0(m\D)}\rightarrow0$, there exists $m=O_{\delta,k}(1)$ such that $h_{k+1}:=\psi_m(h_k'):S\rightarrow\D$
satisfies
\begin{equation}
\norm{f-h_{k+1}}_{\ell^2(S)}^2\le\norm{f-h_k}_{\ell^2(S)}^2-\frac{\delta^2}{8}\#S.\label{eq:f-h_k+1 is smaller}
\end{equation}
For the case that \eqref{eq:f-h norm} holds but \eqref{eq:f-h,h ortho} fails, we set $h_{k+1}$ identically, except that we replace $g$ by $h_k$.

Each step of the process defines $h_{k+1}$ as a polynomial of $h_k$ and $g\in\mc G$ having $O_{k,\delta}(1)$-bounded degree and coefficients, hence $h_k$ is a linear combination of members in $\G^{O_{k,\delta}(1)}$ with $O_{k,\delta}(1)$-bounded coefficients. Up to subpartitions (e.g., $2g=g+g$), we assume all coefficients are in $\D$.
When the process stops, all of \eqref{eq:h l infty<1}, \eqref{eq:f-h norm}, and \eqref{eq:f-h,h ortho} hold true. By \eqref{eq:f-h_k+1 is smaller}, the process above can iterate at most $O_\delta(1)$ times. Thus, the proof is complete.
\end{proof}
\subsection*{Progressions}
Let $V$ be an $\R$-vector space and let $A,B\subset V$ be any subsets.
We denote by $A\pm B$ the sumset
\[
A\pm B:=\left\{ a\pm b:a\in A,b\in B\right\} .
\]
For $k\in\N$, $k\cdot A$ denotes the set
\[
k\cdot A:=\stackrel[k]{}{\underbrace{A+\cdots+A}}.
\]
For $a\in\R$ and $b\in V$, we denote
\[
aA+b:=\left\{ av+b:v\in A\right\} .
\]
Next, we introduce the notion of multiprogression. This can be read as a version of proper generalized arithmetic progression, e.g., in \cite{tao2006additive}.
\begin{defn}[multiprogression]
A \emph{multiprogression} $(P,\Omega)$ into $\R^d,d\in\N$ of \emph{rank} $r\in\N$
is a couple of a linear map $P:\R^r\rightarrow\R^d$ and a set $\Omega\subset\Z^r$ of the form
\begin{equation}\label{eq:P Omega}
\Omega=[N_1]\times\cdots\times[N_r],
\end{equation}
where $N_1,\ldots,N_r\in\N$, such that $P$ is injective on $\Omega$.
We denote
\[
\ul\Omega:=[-N_1,N_1]\times\cdots[-N_r,N_r]\subset\R^r.
\]
If $N_1,\ldots,N_r\gg1$ holds, $\Omega$ and $(P,\Omega)$ are said to be \emph{thick}.

For $k\in\N$, a multiprogression $(P,\Omega)$ into $\R^d,d\in\N$ is said to be $k$\textit{-injective} if $P$ is injective on $k\cdot\Omega$. In particular, every multiprogression is $1$-injective.

For $\rho>0$, we denote
\[
\rho\cdot\Omega:=\rho\ul\Omega\cap\Z^r,
\]
which extends the notation $k\cdot\Omega$ for $k\in\N$.

If $P$ is allowed to be affine, $(P,\Omega)$ is said to be \emph{affine}.
\end{defn}
In additive combinatorics, there are several instances of the principle that smallness of the sumset $S+S$ implies comparability of $S$ to small-rank multiprogressions. In the following, we list some of them.
The following is a consequence of \cite{bilu1999structure}:
\begin{prop}
\label{prop:Freiman 2^n}Let $\s>0$ be a positive real number. Let $G$ be either $\R^d$ or $\Z^d$, $d\in\N$. Let $S\subset G$ be a finite set such that
\[
\#(S+S)\le\s\cdot\#S.
\]
Then, there exists a multiprogression $(P,\Omega)$ into $G$ of rank $r\in\N$ and points
$x_{1},...,x_{n}\in G,n\les_{\s,d}1$ such that
\[
r\le\left\lfloor \log_{2}\s\right\rfloor ,
\]
\begin{equation}\label{eq:S covered by P}
S\subset P(\Omega)+\left\{ x_{1},...,x_{n}\right\},
\end{equation}
and
\begin{equation}\label{eq:S sim P}
\#S\sim_{\s,d}\#P(\Omega).
\end{equation}
\end{prop}
\begin{proof}
If $\#S\sim 1$, we may choose $\{x_1,\ldots,x_n\}:=S$ and there is nothing to prove. Assuming $\#S\gg 1$, we bring the result of \cite[Theorem 1.2-1.3]{bilu1999structure} (with the choice $s=1$); there exists a multiprogression $(P_0,\Omega_0)$ with $\Omega_0=\prod_{j\le r_0}[N_j]$,
$N_{1}\ge\cdots\ge N_{r_0}\ge1$ such that
\begin{equation}\label{eq:r<s-1}
r_0\le\s-1,
\end{equation}
\begin{equation}\label{eq:!>logn}
N_{j}\les_{\s}1\text{ for }j>\left\lfloor \log_{2}\s\right\rfloor,
\end{equation}
\begin{equation}\label{eq:Im>S}
P_0(\Omega_0)\supset S,
\end{equation}
\begin{equation}\label{eq:P0 sim S}
\#P_0(\Omega_0)\les_{\s}\#S.
\end{equation}
Let $r:=\min\{r_0,\lfloor\log_2\s\rfloor\}$. Let $P:\R^r\rightarrow\R^d$ be $P(\cdot):=P_0(\cdot,0)$ and $\Omega:=P(\Omega_0)$.
Let $\{x_1,\ldots,x_n\}:=P_0(\{0\}\times\prod_{j=r+1}^{r_0}[N_j])$. By \eqref{eq:r<s-1} and \eqref{eq:Im>S}, we have \eqref{eq:S covered by P}. By \eqref{eq:!>logn} and \eqref{eq:P0 sim S}, we also have \eqref{eq:S sim P}, finishing the proof.
\end{proof}
Next, we briefly summarize basic relations between concepts concerning sumsets (namely the additive energy, doubling constant, and progression); see \cite{tao2006additive} for further discussions.
\begin{prop}\label{prop:Balog}Let $d\in \N$ be an integer and $\epsilon>0$. Let $G$ be either $\R^d$ or $\Z^d$. For a finite set $S\subset G$, the following are equivalent:
\begin{enumerate}
\item\label{enu:diag-Q(S)d}
There exists $A\subset G$ such that $\# A\les_{\epsilon,d}\#S$ and
\[
\#\{(x_1,x_3)\in S^2 : \frac{x_1+x_3}{2}\in A\}\sim_{\epsilon,d}(\#S)^2.
\]
\item\label{enu:Q(S)d}
We have
\[
\#\{(x_1,x_2,x_3,x_4)\in S^4 : x_1+x_3=x_2+x_4\}\sim_{\epsilon,d}(\#S)^{3}.
\]
\item\label{enu:A+A<S}
There exists a set $A\subset S$ such that $\# A\sim_{\epsilon,d} \# S$ and $\#(A+A)\sim_{\epsilon,d} \# S$.
\item\label{enu:~P}
There exists an affine multiprogression $(P,\Omega)$ into $G$ of rank $r=O_{\epsilon,d}(1)$
such that
\[
\#(S\cap P(\Omega))\sim_{\epsilon,d}\#P(\Omega)\sim_{\epsilon,d}\#S.
\]
\end{enumerate}
\end{prop}
\begin{proof}
Assuming \eqref{enu:diag-Q(S)d}, by Cauchy-Schwarz, we have
\begin{align*}
&\#\{(x_1,x_2,x_3,x_4)\in S^4 : x_1+x_3=x_2+x_4\}
\\  
\ge&\sum_{a\in A}\#\{(x_1,x_3)\in S^2:\frac{x_1+x_3}{2}=a\}^2
\\ 
\ge&\frac{1}{\# A}\cdot\#\{(x_1,x_3)\in S^2:\frac{x_1+x_3}{2}\in A\}^2\gtrsim_{\epsilon,d}(\# S)^3,
\end{align*}
which is just \eqref{enu:Q(S)d}.

In \cite{balog1994statistical}, it was shown that \eqref{enu:Q(S)d} implies \eqref{enu:A+A<S}.
By Proposition \ref{prop:Freiman 2^n}, \eqref{enu:A+A<S} implies \eqref{enu:~P}.

Assume \eqref{enu:~P}. We choose $A=\{P(x/2):x\in2\cdot\Omega\}$. For $(x_1,x_3)\in (S\cap P(\Omega))^2$, we have $\frac{x_1+x_3}{2}\in A$. This concludes \eqref{enu:diag-Q(S)d}. 
\end{proof}
\begin{prop}\label{prop:rank reduction}
Let $r,k,d\in\N$. Let $(P,\Omega)$ be a multiprogression into $\R^d$ of rank $r$ that is not $k$-injective. Then, there exists a multiprogression $(\tilde{P},\tilde\Omega)$ of rank $\tilde{r}\le r-1$ such that
\[
\#\tilde\Omega\sim_{r,k}\#\Omega
\]
and
\[
P(k\cdot\Omega)\subset\tilde{P}(\tilde\Omega)\subset P(O_{r,k}(1)\cdot\Omega)
\]
\end{prop}
\begin{proof}
This is a version of \cite[Theorem 3.40]{tao2006additive}, applied to $(P,k\cdot\Omega)$.
\end{proof}
\begin{prop}[\cite{Ruzsa-Freiman}]\label{prop:Ruzsa}
Let $G$ be an additive group. Let $A,B\subset G$ be finite nonempty sets. There exist $x_1,\ldots,x_J\in G$, $J\le\frac{\#(A+B)}{\#A}$ such that
\[
B\subset(A-A)+\{x_1,\ldots,x_J\}.
\]
\end{prop}
\begin{lem}\label{lem:comparability criterion}
Let $d,r_1,r_2\in\N$. Let $(P_1,\Omega_1)$ and $(P_2,\Omega_2)$ be multiprogressions into $\R^d$ of ranks $r_1$ and $r_2$, respectively.
Assume $\#\Omega_1\sim\#\Omega_2$. The following are equivalent:
\begin{enumerate}
    \item\label{enu:P1<P2}$P_1(\Omega_1)$ is covered by $O_{d,r_1,r_2}(1)$ translates of $P_2(\Omega_2)$.
    \item\label{enu:P2<P1}$P_2(\Omega_2)$ is covered by $O_{d,r_1,r_2}(1)$ translates of $P_1(\Omega_1)$.
    \item\label{enu:P1cP2}$\max_{\xi\in\R^d}\#\left(P_1(\Omega_1)\cap (P_2(\Omega_2)+\xi)\right)\sim_{d,r_1,r_2}\#\Omega_1\sim_{d,r_1,r_2}\#\Omega_2$ holds.
\end{enumerate}
\end{lem}
\begin{proof}
Either \eqref{enu:P1<P2} or \eqref{enu:P2<P1} immediately implies \eqref{enu:P1cP2} by the pigeonhole principle. Assume \eqref{enu:P1cP2}. Let $\xi$ be a point saturating \eqref{enu:P1cP2}. By Proposition \ref{prop:Ruzsa} setting $A=P_1(\Omega_1)\cap (P_2(\Omega_2)+\xi)$ and $B=P_2(\Omega_2)$, there exist $x_1,\ldots,x_{J_0}\in\R^d$, $J_0\le\#(B+B)/\#A=O_{d,r_1,r_2}(1)$, such that
\[
P_2(\Omega_2)\subset (A-A)+\{x_1,\ldots,x_{J_0}\}\subset P_1(2\cdot \Omega_1)+\{x_1,\ldots,x_{J_0}\}.
\]
$P_1(2\cdot\Omega_1)$ can be covered by $2^{r_1}$ translates of $P_1(\Omega_1)$, thus \eqref{enu:P2<P1} holds. Similarly, \eqref{enu:P1<P2} holds, finishing the proof.
\end{proof}
\begin{defn}
Let $d\in\N$. A multiprogression $(P,\Omega)$ into $\R^d$ and a finite set $A\subset\R^d$ are said to be \emph{comparable} (denoted by $(P,\Omega)\sim A$) if
\[
\max_{\xi\in\R^d}\#\left((P(\Omega)+\xi)\cap A\right)\sim\#P(\Omega)\sim\# A.
\]
Two multiprogressions $(P_1,\Omega_1),(P_2,\Omega_2)$ into $\R^d,d\in\N$ are said to be \emph{comparable} if $(P_1,\Omega_1)\sim P_2(\Omega_2)$, or equivalently, all criteria of Lemma \ref{lem:comparability criterion} are satisfied.
\end{defn}
As a result, $(P_1,\Omega_1)\sim(P_2,\Omega_2)\sim(P_3,\Omega_3)$ implies $(P_1,\Omega_1)\sim(P_3,\Omega_3)$. This enables us to formally regard the symbol $\sim$ as an equivalence relation between multiprogressions of bounded ranks, at least when we compare a bounded number of multiprogressions. One can also check that $(P_1,\Omega_1)\sim(P_2,\Omega_2)\sim A$ implies $(P_1,\Omega_1)\sim A$.

\subsection*{\label{subsec:Gowers norm}Gowers norms}
We introduce the Gowers norm and recall relevant facts.
\begin{defn}
Let $G$ be an additive group and $f:G\rightarrow\C$ be a function. For $\eta\in G$, we define the function
\[
\Alt_\eta f(x):=f(x)\overline{f(x+\eta)}.
\]
\end{defn}
We note that
\[
\Alt_{\eta_1}\Alt_{\eta_2}f=\Alt_{\eta_2}\Alt_{\eta_1}f=:\Alt_{\eta_1,\eta_2}f
\]
holds for any $\eta_1,\eta_2\in G$ and $f:G\rightarrow\C$.

The Gowers uniformity norms (Gowers norms) are defined as follows:
\begin{defn}[Gowers norms on a group]\label{def:Gowers norm}
Let $G$ be a finite additive group. For $f:G\rightarrow\C$ and $k\in\N$, we inductively define
\[
\norm f_{U^1(G)}:=\left|\E_{x\in G}f(x)\right|
\]
and
\begin{equation}\label{eq:Uk def inductive}
\norm{f}_{U^{k+1}(G)}=\Bb{\E_{\eta\in G}\norm{\Alt_\eta f}_{U^k}^{2^k}}^{1/2^{k+1}}.
\end{equation}
Equivalently, $\norm f_{U^{k}},k\ge 1$ can be written more explicitly as
\begin{equation}\label{eq:Uk def explicit}
\norm{f}_{U^{k}(G)}=\Bb{\E_{x,\eta_1,\ldots,\eta_{k}\in G}\Alt_{\eta_1,\ldots,\eta_k}f(x)}^{1/2^{k}}.
\end{equation}
\end{defn}
It is known that $U^{k+1}(G)$ is a norm for $k\ge 1$ and any finite additive group $G$; see, e.g., \cite[Section 11.1]{tao2006additive}. Also, for $f:G\rightarrow\C$ and $k\in\N$, one has
\begin{equation}\label{eq:Ud monotone}
\norm{f}_{U^{k+1}(G)}\ge\norm{f}_{U^k(G)}.
\end{equation}
(See, e.g., {\cite[(11.7)]{tao2006additive}}.)
Although we defined Gowers norms over finite additive groups, functions on boxes $[N]^d$ can be dealt in the following manner (analogous, e.g., to \cite{GreenTaoZiegler11(U4)}):
\begin{defn}[Gowers norm on a box]\label{defn:Gowers on box}
Let $N,k,d\in\N$. Consider a function $f:[N]^d\rightarrow\C$. Choose any $N'\ge 2^{k+2}N$. Let $f':\Z_{N'}^d\rightarrow\C$ be defined as $f'(x)=f(x)$ for $x\in[N]^d$ and $f'(x)=0$ otherwise. We define the Gowers $U^{k+1}$-norm for $f$ as
\[
\norm{f}_{U^{k+1}}:=\norm{f'}_{U^{k+1}(\Z_{N'}^d)}/\norm{\chi_{[N]^d}}_{U^{k+1}(\Z_{N'}^d)},
\]
which does not depend on $N'\ge 2^{k+2}N$ and so is well-defined.
For $\tilde N\in \N$, we denote
\[
\norm{f}_{U^{k+1}([\tilde N]^d)}:=\norm{\tilde f}_{U^{k+1}},
\]
where $\tilde f:[\tilde N]^d\rightarrow \C$ is the restriction of $f$ if $\tilde N\le N$ and is $\chi_{[N]^d}f$ if $\tilde N>N$.
\end{defn}
In Definition \ref{defn:Gowers on box}, $N'\ge 2^{k+2}N$ is assumed to avoid products between different copies of $[N]^d$ in any calculation of $\Alt$ in \eqref{eq:Uk def explicit}.
Attaching the domain $[N]^d$ to the Gowers norm for the case $N=\tilde N$ is merely a matter choice; we usually omit the domain for that case.
\begin{prop}\label{prop:von Neumann}
Let $k\in\N$ and $a_1,\ldots,a_k\in\Z\setminus\{0\}$ be distinct integers. Let $N,d\in\N$ and $f,g_1,\ldots,g_k:[N]^d\rightarrow\D$. We have
\[
\left|\sum_{\eta,x\in\Z^d}f(x)\prod_{j\le k} g_j(a_j\eta+x)\right|\les_{a_1,\ldots,a_k,d}N^{2d}\norm{f}_{U^k}.
\]
\end{prop}
\begin{proof}
This can be shown by repeating the proof in \cite[Lemma 11.4]{tao2006additive}.
\end{proof}
A particularly useful instance of Proposition \ref{prop:von Neumann} is the following: Let $b_1,\ldots,b_m\in\Z\setminus\{0\}$ be such that
\[
\s_1b_1+\cdots+\s_m b_m:(\s_1,\ldots,\s_m)\in\{0,1\}^m\setminus\{0\}
\]
are all distinct and nonzero; then enumerating above as $a_1,\ldots,a_k$, $k=2^m-1$, Proposition \ref{prop:von Neumann} implies
\begin{cor}\label{cor:von Neumann Alt}
Let $m\in\N$ and $b_1,
\ldots, b_m$ as above. Then, for $N,d\in\N$ and $f:[N]^d\rightarrow\D$, we have
\[
\left|\sum_{\eta,x\in\Z^d}\Alt_{b_1\eta,\ldots,b_m\eta}f(x)\right|\les_{b_1,\ldots,b_m,d}N^{2d}\norm{f}_{U^{2^m-1}}.
\]
\end{cor}
\subsection*{Lattice-convex sets}
\begin{defn}\label{defn:lattice-convex set}
Let $d\in\N$ and $\La\subset\R^d$ be a lattice. A finite set $\Omega\subset\La$ is a \emph{lattice-convex set} if there exists a convex set $\ul\Omega\subset\R^d$ such that
\[
\Omega=\ul\Omega\cap\La.
\]
$\Omega$ and $\ul\Omega$ are \emph{thick} if $\Omega$ is not contained in $O(1)$ affine translates of a proper subspace of $\R^d$. For an affine subspace $\mc P\subset\R^d$, $\Omega$ and $\ul\Omega$ are said to be \emph{relatively thick} if $\Omega\cap\mc P$ is thick within $\La\cap\mc P\neq\emptyset$.
For $m\in\N$, we denote
\[
\frac{1}{m}\cdot\Omega:=\frac{1}{m}\ul\Omega\cap\La=\frac{1}{m}(\ul\Omega\cap m\La).
\]
$\Omega$ is \emph{symmetric} if $\Omega=-\Omega$.
\end{defn}
\begin{prop}[\cite{John-Extremum}]\label{prop:John}
Let $d\in\N$ and $\ul\Omega\subset\R^d$ be a convex set. Then, there exists a linear transform $T:\R^d\rightarrow\R^d$ and $x_0\in\ul\Omega$ such that
\[
B_d\subset T(\ul\Omega-x_0)\subset d\cdot B_d,
\]
where $B_d$ is the unit ball in $\R^d$. If $\ul\Omega$ is symmetric, we can further assume $x_0=0$.
\end{prop}
\begin{prop}[{\cite[Lemma 10.3]{GreenTao08(U3)}}]\label{prop:discrete John}
Let $d\in\N$ and $\Omega=\ul\Omega\cap\La\subset\R^d$ be a symmetric lattice-convex set. Assume $\La$ is of full rank. Then, there exists a $d$-tuple $(w_1,\ldots,w_d)\in\La^d$ generating $\La$ and $N_1,\ldots,N_d\in\N$ such that
\[
\frac{1}{d^{2d}}\cdot\Omega\subset[N_1]w_1+\cdots+[N_d]w_d\subset \Omega\subset [d^{2d}N_1]w_1+\cdots+[d^{2d}N_d]w_d.
\]
\end{prop}
\begin{lem}\label{lem:convex Mahler}
Let $d\in\N$. There exists $C=O_d(1)$ satisfying the following:

Let $\La\subset\R^d$ be a full-rank lattice and $\Omega=\ul\Omega\cap\La$ be a thick lattice-convex set. Then, there exist a $d$-tuple $(w_1,\ldots,w_d)\in\La^d$ generating $\La$, $N_1,\ldots,N_d\in\N$, and $x_0\in\La$ such that
\begin{equation}\label{eq:convex Mahler}
[N_1]w_1+\cdots+[N_d]w_d\subset\Omega-x_0\subset[CN_1]w_1+\cdots+[CN_d]w_d.
\end{equation}
Moreover, $N_1,\ldots,N_d\gg1$ holds.
\end{lem}
\begin{proof}
By Proposition \ref{prop:John}, up to a linear transform, there exists $x_0\in\R^d$ such that
\[
B_d\subset\ul\Omega-x_0\subset d\cdot B_d.
\]
Here, since $\Omega$ is thick, $\dist(x_0,\La)\ll1$ holds. Thus, we can perturb $x_0$ so that $x_0\in\La$ and
\[
\frac{1}{2}B_d\subset\ul\Omega-x_0\subset 2d\cdot B_d.
\]
Applying Proposition \ref{prop:discrete John} to $B_d$, there exist linearly independent $(w_1,\ldots,w_d)\in\La^d$ and $N_1,\ldots,N_d\in\N$ such that
\[
[N_1]w_1+\cdots+[N_d]w_d\subset B_d\cap\La\subset [d^{2d}N_1]w_1+\cdots+[d^{2d}N_d]w_d.
\]
Thus, choosing $C=4d\cdot d^{2d}$ and resizing $N_j$ by $1/2$ yields \eqref{eq:convex Mahler}. By the second inclusion in \eqref{eq:convex Mahler} and the thickness of $\Omega-x_0$, we have $N_1,\ldots,N_d\gg1$, which finishes the proof.
\end{proof}
An immediate consequence of Lemma \ref{lem:convex Mahler} is that thickness of a convex set $\ul\Omega\subset\R^d$ is a translation-invariant property (because any translate of $[N_1]w_1+\cdots+[N_d]w_d$, $N_1,\ldots,N_d\gg1$ is thick). Similarly, thickness is invariant over $O(1)$-scalings.

We recall a version of a Weyl-type property.
\begin{prop}\label{prop:Weyl 1D}
Let $\epsilon>0$.
Let $N\in\N$ and $a,b\in\R$. Assume that
\begin{equation}\label{eq:Weyl 1D assump}
\left|\E_{n\in[N]}e^{i(an^2+bn)}\right|\ge\epsilon.
\end{equation}
Then, we have
\begin{equation}\label{eq:Weyl 1D a}
\dist(a,\frac{2\pi}{m}\Z)\les_\epsilon\frac{1}{N^2}
\end{equation}
and
\begin{equation}\label{eq:Weyl 1D b}
\dist(b,\frac{2\pi}{m}\Z)\les_\epsilon\frac{1}{N},
\end{equation}
where $m=O_\epsilon(1)$ is an integer.
\end{prop}
\begin{proof}This is a quantitative version of Weyl's equidistribution theorem; for an explicit proof, see, e.g., \cite[Proposition 4.3]{GreenTao12Quantitative} testing the equidistribution with $e^{2\pi ix}$.
\end{proof}
\begin{lem}\label{lem:Weyl bound implies Lipschitz}
Let $\epsilon>0$, $d\in\N$, and $m\in\N$. For $\delta\ll_{\epsilon,d,m}1$, the following holds: Let $\theta=(\theta_1,\ldots,\theta_d)\in\R^d$, $\eta\in\R$, and $N_1,\ldots,N_d\in\N$. Assume that
\begin{equation}\label{eq:Weyl bound implies Lipschitz}
\#\{n\in[N_1]\times\cdots\times[N_d]:\dist(\theta\cdot n+\eta,\frac{2\pi}{m}\Z)\le\delta\}\ge\epsilon N_1\cdots N_d.
\end{equation}
Then, for $j=1,\ldots,d$, there exists an integer $m'=O_{\epsilon,d,m}(1)$ such that
\begin{equation}\label{eq:Weyl...}
\dist(\theta_j,\frac{2\pi}{m'}\Z)\les_{\epsilon,d,m}\frac{\delta}{N_j}.
\end{equation}
\end{lem}
\begin{proof}
First, we show the case $d=1$. If $N_1\les_{\epsilon,m}1$, \eqref{eq:Weyl...} holds with $m'=1$; we assume $N_1\gg_{\epsilon,m}1$. Then, by the pigeonhole principle, there exist $n_1,n_2\in[N_1]$ such that $h=n_1-n_2\in[2/\epsilon]$ and $\dist(\theta n_j+\eta,\frac{2\pi}{m}\Z)\le\delta$ for $j=1,2$. Then, $\dist(h\theta,\frac{2\pi}{m}\Z)\le 2\delta$ holds.
If there exists $h'\in h\Z\cap[\epsilon N_1]$ such that $2\delta<\dist(h'\theta,\frac{2\pi}{m}\Z)\le 4\delta$, then for every $n_0\in[N_1]$, at most one member of $n_0,n_0+h',\ldots,n_0+\lfloor\frac{1}{m\delta}\rfloor h'$ satisfies $\dist(\theta n+\eta,\frac{2\pi}{m}\Z)\le\delta$. For $\delta\ll\epsilon/m$, this contradicts \eqref{eq:Weyl bound implies Lipschitz}.
Thus, for $h'\in h\Z\cap[\epsilon N_1]$, we have $\dist(\theta h',\frac{2\pi}{m}\Z)\le2\delta$. This implies $\dist(\theta h,\frac{2\pi}{m}\Z)\les_{\epsilon,m}\frac{\delta}{N_1}$; \eqref{eq:Weyl...} holds with $m'=hm$, concluding the case $d=1$.

Now consider arbitrary $d\in\N$. Denote $\theta=(\theta_{\le d-1},\theta_d)\in\R^{d-1}\times\R$. By the pigeonhole principle, there exists $n_{\le d-1}\in[N_1]\times\cdots\times[N_{d-1}]$ such that $\#\{n_d\in[N_d]:\dist(\theta_{\le d-1}\cdot n_{\le d-1}+\theta_d n_d+\eta,\frac{2\pi}{m}\Z)\le\delta\}\gtrsim_{\epsilon,d} N_d$. Thus, $\dist(\theta_d,\frac{2\pi}{m'}\Z)\les_{\epsilon,m,d}\frac{\delta}{N_d}$ holds for some $m'=O_{\epsilon,m,d}(1)$. This concludes \eqref{eq:Weyl...} for $j=d$; proceeding similarly for $j=1,\ldots,d-1$ finishes the proof.
\end{proof}
\begin{lem}\label{lem:convex Weyl}
Let $d\in\N$ and $\epsilon>0$. There exists $m=O_{d,\epsilon}(1)$ satisfying the following:

Let $w_1,\ldots,w_d\in\R^d$ be linearly independent. Let $L,Q:\R^d\rightarrow\R$ be linear and quadratic forms, respectively. For $N_1,\ldots,N_d\gg_{\epsilon,d}1$ and any lattice-convex set $\Omega\subset[N_1]w_1+\cdots+[N_d]w_d$ such that
\begin{equation}\label{eq:convex Weyl assump}
\left|\sum_{x\in\Omega}e^{i(Q+L)(x)}\right|\ge\epsilon N_1\cdots N_d,
\end{equation}
for each $j,k=1,\ldots,d$, we have
\begin{equation}\label{eq:convex Weyl Q}
\dist(D_{w_j}D_{w_k}Q,\frac{2\pi}{m}\Z)\les_{\epsilon,d}\frac{1}{N_jN_k}.
\end{equation}
and
\begin{equation}\label{eq:convex weyl L}
\dist(D_{w_j}L,\frac{2\pi}{m}\Z)\les_{\epsilon,d}\frac{1}{N_j}.
\end{equation}
Here, $D_w$ denotes the directional derivative for the direction $w\in\R^d$.
\end{lem}
\begin{proof}
Let us assume $j=1$ for simplicity. Let $E$ be the set
\[
E:=\left\{(n_2,\ldots,n_d)\in[N_2]\times\cdots\times[N_d]:\left|\sum_{n_1\in[N_1]:x=n_1w_1+\cdots+n_dw_d\in\Omega}e^{i(Q+L)(x)}\right|\gtrsim_{\epsilon,d} N_1\right\}.
\]
We keep denoting $x=n_1w_1+\cdots+n_dw_d$. By \eqref{eq:convex Weyl assump}, we have
\[
\#E\gtrsim_{\epsilon,d} N_2\cdots N_d.
\]
For each $(n_2,\ldots,n_d)\in E$, by Proposition \ref{prop:Weyl 1D}, there exists $m=O(1)$ such that
\begin{equation}\label{eq:Weyl DDQ}
\dist\left(D_{w_1}D_{w_1}Q(x),\frac{2\pi}{m}\Z\right)\les_{\epsilon,d}\frac{1}{N_1^2}
\end{equation}
and
\begin{equation}\label{eq:Weyl DQ+L}
\dist\left(D_{w_1} Q(x)+D_{w_1}L,\frac{2\pi}{m}\Z\right)\les_{\epsilon,d}\frac{1}{N_1}.
\end{equation}
\eqref{eq:convex Weyl Q} for $(j,k)=(1,1)$ is immediate from \eqref{eq:Weyl DDQ}. For $k\neq 1$, since $N_1,N_k\gg_{\epsilon,d}1$, applying Lemma \ref{lem:Weyl bound implies Lipschitz} to \eqref{eq:Weyl DQ+L} yields \eqref{eq:convex Weyl Q} for $(1,k)$.

Now that we showed \eqref{eq:convex Weyl Q}, for $x\in\Omega$, we have $\dist(D_{w_1}Q(x),\frac{2\pi}{m}\Z)\les_{\epsilon,d}\frac{1}{N_1}$. Plugging this into \eqref{eq:Weyl DQ+L} yields \eqref{eq:convex weyl L}, finishing the proof.
\end{proof}
\begin{defn}
Let $d\in\N$. For a symmetric convex set $C\subset\R^d$, $d\in\N$, and $u\in\R^d$, we denote $\norm{u}_C:=\inf\{\rho>0:u\in\rho C\}$. 
\end{defn}
\begin{cor}\label{cor:convex Weyl}
Let $d\in\N$ and $\epsilon>0$. There exists $m=O_{d,\epsilon}(1)$ satisfying the following:

Let $\La\subset\R^d$ be a lattice of full rank. Let $\Omega=\ul\Omega\cap\La\subset\R^d$ be a thick lattice-convex set. Let $\Omega'\subset\Omega$ be a lattice-convex set. Let $L,Q:\R^d\rightarrow\R$ be linear and quadratic forms, respectively. If
\[
\left|\sum_{x\in\Omega'}e^{i(Q+L)(x)}\right|\ge\epsilon\#\Omega,
\]
then there exists an integer $m=O_{\epsilon,d}(1)$ such that for $x\in\La$,
\begin{equation}\label{eq:convex weyl cor Q}
\dist(Q(x),\frac{2\pi}{m}\Z)\les_{\epsilon,d}\norm{x}_{\ul\Omega-\ul\Omega}^2.
\end{equation}
Furthermore, if $Q=0$ holds, there exists an integer $m=O_{\epsilon,d}(1)$ such that for $x\in\La$,
\begin{equation}\label{eq:convex weyl cor L}
\dist(L(x),\frac{2\pi}{m}\Z)\les_{\epsilon,d}\norm{x}_{\ul\Omega-\ul\Omega}.
\end{equation}
\end{cor}
\begin{proof}
Since $\Omega$ is thick, by Lemma \ref{lem:convex Mahler}, there exist $(w_1,\ldots,w_d)\in\La^d$ generating $\La$, $N_1,\ldots,N_d\gg_{\epsilon,d}1$, and $x_0\in\La$ such that $N_1\cdots N_d\les_d\#\Omega$ and
\[
\Omega-x_0\subset[N_1]w_1+\cdots+[N_d]w_d.
\]
Up to a translation, we may assume $x_0=0$. Now applying Lemma \ref{lem:convex Weyl} finishes the proof.
\end{proof}
\subsection*{Bohr sets and the inverse Gowers $U^3$ theorem}
\begin{defn}[Locally polynomial modulation]
Let $G$ be an additive group. A function $\phi:G\rightarrow\T\cup\{0\}$ is said to be \emph{a locally polynomial modulation of degree at most $k\ge1$} if $\Alt_{\eta_1,\ldots,\eta_{k+1}}\phi(x)\in\{0,1\}$ holds for every $x,\eta_1,\ldots,\eta_{k+1}\in G$.

In particular, locally polynomial modulations of degrees at most $k=1,2$ are called locally linear and quadratic modulations, respectively.

$\phi$ is \emph{supported on} $S\subset G$ if $\supp(\phi)=S$.
\end{defn}
\begin{lem}\label{lem:lift-box}
Let $k,d\in\N$. For every locally polynomial modulation $\phi:\Z^d\rightarrow\T\cup\{0\}$ of degree at most $k$ supported on $k\cdot\{0,1\}^d$, there exists a polynomial $F:\R^d\rightarrow\R$ of degree at most $k$ such that $\phi(x)=e^{iF(x)}$ holds for $x\in k\cdot\{0,1\}^d$.

Moreover, for any two such $F$ and $F'$, $F-F':\Z^d\rightarrow 2\pi\Z$ holds.
\end{lem}
\begin{proof}
Let $\De:=\{(x_1,\ldots,x_d)\in\N^d:x_1+\cdots+x_d\le k\}$. For $x=(x_1,\ldots,x_d)\in k\cdot\{0,1\}^d\setminus\De$, $\phi(x)$ is determined by $\Alt_{-e_1,\ldots,-e_d}\phi(x)=1$ with each $-e_j$ iterated at most $x_j$ times. Hence, $\phi\mid_\De$ uniquely determines $\phi$.

Similarly, a polynomial $F$ of degree $\deg F\le k$ is uniquely determined by $F\mid_{\De}$ (primarily on $\N^d$, and hence on $\R^d$ since $F$ is a polynomial). Since $\#\De$ equals the dimension of the vector space of polynomials $F:\R^d\rightarrow\R$ of degree $\deg F\le k$, such extension $F\mid_{\De}\rightarrow F$ is well-defined as an isomorphism.

Now consider $F$ extending $\log\phi\mid_\De:\De\rightarrow(-\pi,\pi]$. Then, $e^{iF}\mid_{k\cdot\{0,1\}^d}$ is a locally polynomial modulation and equals to $\phi$ on $\De$, hence $\phi(x)=e^{iF(x)}$ holds for $x\in k\cdot\{0,1\}^d$.

It only remains to check $F-F':\Z^d\rightarrow2\pi\Z$ for any two $F,F'$. Let $\phi=e^{i(F-F')}$ be defined on $\N^d$. Then, we have $\phi\mid_\De=1$ and for $x\in\N^d\setminus\De$, $\phi(x)$ is determined by $\Alt_{-e_1,\ldots,-e_r}\phi(x)=1$ with each $-e_j$ iterated at most $x_j$ times. Thus, $\phi\equiv 1$ holds and this implies $F-F':\N^d\rightarrow 2\pi\Z$. Since $F$ and $F'$ are polynomials, this easily extends to $\Z^d$.
\end{proof}
The following lemma generalizes Lemma \ref{lem:lift-box} to thick lattice-convex sets.
\begin{lem}\label{lem:lift}
Let $d,k\in\N$. Let $\Omega\subset\Z^d$ be a thick lattice-convex set. Let $\phi:\Z^d\rightarrow\T\cup\{0\}$ be a locally polynomial modulation of degree $k$ supported on $\Omega$. Then, $\phi\mid_\Omega$ is a restriction of $e^{iF}$ for some polynomial $F:\R^d\rightarrow\R$  of degree $k$. Moreover, for any two such $F,F'$, $F-F':\Z^d\rightarrow2\pi\Z$ holds.
\end{lem}
\begin{proof}
By Lemma \ref{lem:convex Mahler}, up to an affine transform, we may assume $k\cdot\{0,1\}^d\subset\Omega$. By Lemma \ref{lem:lift-box}, there exists a polynomial $F$ of degree at most $k$ such that $\phi\mid_{k\cdot\{0,1\}^d}=e^{iF}\mid_{k\cdot\{0,1\}^d}$. The uniqueness part $F-F':\Z^d\rightarrow2\pi\Z$ is also immediate from Lemma \ref{lem:lift-box}.

Next, we show that $\phi\mid_{k\cdot\{0,1\}^d}$ uniquely determines $\phi$. Once we show this, since $\phi\mid_{k\cdot\{0,1\}^d}$ already extends to a polynomial modulation $e^{iF}$, $e^{iF}=\phi$ holds globally on $\Omega$. Let $\Omega^*\supset k\cdot\{0,1\}^d$ be a lattice-convex set of the minimal cardinality such that two locally polynomial modulation $\phi\neq\phi'$ supported on $\Omega^*$ exist. Choose a corner $\xi$ of the convex hull of $\Omega^*$ not lying in $k\cdot\{0,1\}^d$. Choose $\xi'\in k\cdot\{0,1\}^d$ such that $\xi-\xi'\in(k+1)\Z^d$. Then, $\phi(\xi)$ is uniquely determined by the values on $\{\frac{j}{k+1}\xi'+(1-\frac{j}{k+1})\xi:1\le j\le k+1\}\subset\Omega^*\setminus\{\xi\}$. Since $\Omega^*\setminus\{\xi\}$ is also a lattice-convex set, this contradicts the minimality of $\Omega^*$. This concludes the proof.
\end{proof}
\begin{defn}[Bohr set]
Let $d,N\in\N$ and $\rho\in(0,\frac{1}{2})$. For a finite set $S=\{\theta_1,\ldots,\theta_r\}\subset(\R/\Z)^d$, we denote
\[
B^d(S,\rho,N):=\{n\in[N]^d:\max_j\dist(\theta_j\cdot n,\Z)\le\rho\}.
\]
$B^d(S,\rho,N)$ is called a \emph{Bohr set} of \emph{rank} $r$ and \emph{radius} $\rho$.
\end{defn}
We note the following consequence of \cite{GreenTao08(U3)}:
\begin{prop}[Inverse Gowers $U^3$ theorem]\label{prop:inverse U^3}
Let $d,N\in\N$ and $\delta>0$. Let $f:[N]^d\rightarrow\D$ be a function such that
\[
\norm{f}_{U^3}\ge\delta.
\]
Then, there exist a Bohr set $\mc B=B^d(S,\rho,N)$ of rank $O_\delta(1)$ and radius $\rho\gtrsim_{\delta,d}1$, $y\in\Z^d$, and a locally quadratic modulation $\phi:\Z^d\rightarrow\T\cup\{0\}$ supported on $\mc B+y$ such that
\[
\left|\E_{x\in[N]^d}f(x)\overline{\phi(x)}\right|\gtrsim_{\delta,d}1.
\]
\end{prop}
\begin{proof}
This is a version of \cite[Theorem 2.7]{GreenTao08(U3)}. Set $N'=2^5N+1$ and $G=\Z_{N'}^d$. Recalling Definition \ref{def:Gowers norm}, applying \cite[Theorem 2.7]{GreenTao08(U3)} to $G$ yields our statement once $\#\mc B\gtrsim_{\delta,d} N^d$ is shown. Indeed, $\#\mc B\gtrsim_{\delta,d} N^d$ is shown in \cite[Lemma 8.1]{GreenTao08(U3)}.
\end{proof}
We note that \cite{GreenTao08(U3)} indeed showed a more general version concerning arbitrary finite group of odd order. (For the even-order case, see \cite{Jamneshan-evenU3(G)}.) In Section \ref{app-sec:deg}, we will see an inverse Gowers $U^{k+1}([N])$-theorem for general $k\in\N$. For $U^3([N]^d)$, these two inverse theorems essentially describe the same object and one can be almost expressed by a linear combination of the other. This will be further discussed in Section \ref{app-sec:deg}.
\begin{defn}[Affine Bohr sets]\label{defn:affine Bohr sets}
Let $d,r,N\in\N$. An \emph{affine Bohr set} $\mc B$ in $[N]^d$ of \emph{rank} $r$ is any set $\mc B$ of the form
\[
\mc B=\{n\in[N]^d:\theta_j\cdot n\in I_j+\Z\text{ for }j=1,\ldots,r\},
\]
where $\theta_1,\ldots,\theta_r\in\R^d$ and $I_1,\ldots,I_r\subset(-1,1)$ are intervals such that $|I_j|<1$.
\end{defn}
In particular, Bohr sets are affine Bohr sets. One conventional perspective is to regard an affine Bohr set $\mc B$ as a projection of a lattice-convex set. More precisely, we have the following expression of $\mc B$:
\begin{rem}\label{rem:Bohr set by lattice-convex proj}
Let $N,\mc B,r,\theta_j,I_j$ be as in Definition \ref{defn:affine Bohr sets}.
Let $P:\R^d\rightarrow\R^r$ be the linear operator
\[
P(x):=(x\cdot\theta_j)_{j\le r}\in\R^r.
\]
Denote $p_{e_j}=(e_j,P(e_j))$ for $j=1,\ldots,d$. Let
\begin{align*}
\ul\Omega=\ul\Omega_{I_1,\ldots,I_r}&:=\left\{(x,y)\in[-N,N]^d\times\R^r:P(x)-y\in I_1\times\cdots\times I_r\right\}
\\&=p_{e_1}[-N,N]+\cdots+p_{e_d}[-N,N]-(0\times I_1\times\cdots\times I_r)
\end{align*}
and
\[
\Omega:=(\Z^d\times\Z^r)\cap\ul\Omega.
\]
Denote by $\pi_{\R^d}:\R^d\times\R^r\rightarrow\R^d$ the canonical projection. Since $|I_1|,\ldots,|I_r|<1$,
$\pi_{\R^d}:\Omega\rightarrow\mc B$ is a bijection.
Given any locally polynomial modulation $\phi$ of degree $k\ge1$ supported on $\mc B$, the map $\phi\circ\pi_{\R^d}:\Omega\rightarrow\T$ is also locally polynomial. Thus, for the case that $\Omega$ is thick, by Lemma \ref{lem:lift}, there exists a polynomial $F:\R^d\times\R^r\rightarrow\R$ of degree $k$ such that for $w\in\Omega$,
\[
e^{iF(w)}=\phi\circ\pi_{\R^d}(w).
\]
\end{rem}
In Remark \ref{rem:Bohr set by lattice-convex proj}, we regarded $F$ as a lift of $\log\phi$. Conversely, we can also descend a polynomial $F$ to a locally polynomial modulation.
\begin{lem}\label{lem:descend F to phi}
In Remark \ref{rem:Bohr set by lattice-convex proj}, assume $|I_1|,\ldots,|I_r|<1/2$ holds. Then, for any polynomial $F:\R^d\times\R^r\rightarrow\R$ of degree $k$, $\phi:\Z^d\rightarrow\T$ supported on $\mc B$ defined by
\[
\phi(\pi_{\R^d}(w)):=e^{iF(w)},\qquad w\in\Omega
\]
is a locally polynomial modulation of degree $k$.
\end{lem}
\begin{proof}
Since $|I_j|<1/2$, $\pi_{\R^d}$ is injective on $\Omega+\Omega$. Thus, we have the Freiman property
\[
u+v=u'+v',\qquad\text{whenever }\pi_{\R^d}(u)+\pi_{\R^d}(v)=\pi_{\R^d}(u')+\pi_{\R^d}(v'),
\]
where $u,v,u',v'\in\Omega$. Hence the conditional identity $\Alt_{w_1,\ldots,w_{k+1}}\phi(u)=1$ for $u+\s_1w_1+\cdots+\s_{k+1}w_{k+1}\in\Omega$, $\s_j\in\{0,1\}$ descends to $\mc B$, finishing the proof.
\end{proof}
\begin{lem}\label{lem:reduced Bohr}
In Remark \ref{rem:Bohr set by lattice-convex proj}, let $I_j'\subset I_j$ be an interval and $m\in\N$. Then,
\[
\mc B':=\pi_{\R^d}(m\Z^{d+r}\cap\ul\Omega_{I_1',\ldots,I_r'})
\]
is an affine Bohr set of rank $(d+r)$.
\end{lem}
\begin{proof}
This is immediate from the identity
\[
\mc B'=\left\{n\in[N]^d:\tfrac{1}{m}\theta_j\cdot n\in\tfrac1m I_j'+\Z\text{ and }\tfrac{1}{m}e_k\cdot n\in(-\tfrac{1}{m},\tfrac{1}{m})+\Z\right\},
\]
where $j=1,\ldots,r$ and $k=1,\ldots,d$.
\end{proof}
\begin{defn}[$m$-partition of a Bohr set]\label{defn:m-partition}
Let $\mc B$ be an affine Bohr set as in Remark \ref{rem:Bohr set by lattice-convex proj} and $m\in\N$. The following family of affine Bohr sets is \emph{the $m$-partition of $\mc B$}:
\[
\pi_m(\mc B):=\left\{\pi_{\R^d}\left((m\Z^{d+r}+x_0)\cap\ul\Omega_{I_{1,l_1},\ldots,I_{r,l_r}}\right)\right\}_{l_1,\ldots,l_r\in\{0,\ldots,m-1\},x_0\in\{0,\ldots,m-1\}^{d+r}}.
\]
Here, $I_{j,0},\ldots,I_{j,m-1}$ denote the subintervals of equal lengths partitioning $I_j$.
\end{defn}
In particular, if $\Omega$ is contained in $k\in\N$ translates of a subspace $\mc P\le\R^{d+r}$, each member of $\pi_k(\mc B)$ corresponds to $\Omega'$ lying on a single translate of $\mc P$. This process will be used to reduce to the case that $\Omega$ is thick.

Immediately from Definition \ref{defn:affine Bohr sets}, an intersection of affine Bohr sets of ranks $r_1$ and $r_2$ is again an affine Bohr set of rank $r_1+r_2$. By this property, we can rephrase Proposition \ref{prop:inverse U^3} as follows:
\begin{prop}[profile decomposition in $U^3$]\label{prop:profile U3 Bohr}
Let $d\in\N$ and $\delta>0$. There exist $r,J\in\N$ satisfying the following:

Let $N\in\N$ and $f:[N]^d\rightarrow\D$ be a function. Then, there exist affine  Bohr sets $\mc B_1,\ldots,\mc B_J$ in $[N]^d$ of ranks at most $r$, $c_1,\ldots,c_J\in\D$, and locally quadratic modulations $\phi_1,\ldots,\phi_J$ supported on $\mc B_j$ such that
\[
h=\sum_{j\le J}c_j\phi_j
\]
satisfies
\[
\norm{h}_{\ell^\infty}\le 1,
\]
\[
\norm{f-h}_{U^3}<\delta,
\]
and
\[
\left|\jp{f-h,h}_{\ell^2([N]^d)}\right|<\delta N^d.
\]
\end{prop}
\begin{proof}
A product of two locally quadratic modulations supported on affine  Bohr sets of ranks $r_1$ and $r_2$ is again such (of rank $r_1+r_2$). Thus, applying Lemma \ref{lem:multiplicative decomp} and Proposition \ref{prop:inverse U^3} finishes the proof.
\end{proof}
\begin{lem}\label{lem:Bohr char is sum of linear}
Let $r\in\N$ and $\epsilon>0$. There exists $J\in\N$ satisfying the following:

Let $N\in\N$. Let $\phi$ be a locally linear modulation supported on an affine Bohr set $\mc B$ in $[N]^2$ of rank $r$. Then, there exist $\xi_1,\ldots,\xi_J\in\R^2$ and $c_1,\ldots,c_J\in\D$ such that
\begin{equation}\label{eq:Bohr char is approx in l2}
\norm{\phi-\sum_{j\le J}c_je^{ix\cdot\xi_j}}_{\ell^2([N]^2)}\le\epsilon N
\end{equation}
and
\begin{equation}\label{eq:Bohr char is bdd in l infty}
\norm{\sum_{j\le J}c_je^{ix\cdot\xi_j}}_{\ell^\infty([N]^2)}\le1.
\end{equation}
\end{lem}
\begin{proof}
Let us adopt notations in Definition \ref{defn:affine Bohr sets} and Remark \ref{rem:Bohr set by lattice-convex proj}. We first satisfy \eqref{eq:Bohr char is approx in l2}; the condition \eqref{eq:Bohr char is bdd in l infty} will be satisfied later by an argument similar to Lemma \ref{lem:multiplicative decomp}.

Firstly, when $r=1$ and $\phi=\chi_{\mc B}$ is a characteristic function, we can write
\[
\chi_{\mc B}(x)=\chi_{I_1+\Z}(\theta_1\cdot x).
\]
By Lemma \ref{lem:Weyl bound implies Lipschitz}, there exists $\delta=\delta(\epsilon)>0$ such that either
\begin{equation}\label{eq:equidist over [N]^2}
\#\{n\in [N]^2:n\cdot\theta_1\in\d I_1+(-\delta,\delta)+\Z\}\le\frac{\epsilon}{10}N^2
\end{equation}
or
\begin{equation}\label{eq:thetas are silly}
\dist(\theta_1,\frac{1}{m}\Z^2)\les_\epsilon\frac1N
\end{equation}
holds for some $m=O_\epsilon(1)$. If \eqref{eq:thetas are silly} holds, $\chi_{\mc B}$ can be written as an $O_\epsilon(1)$ sum of the form $\chi_{x_0+m\Z^2}\cdot\chi_{\mc P\cap[N]^2}$, where $x_0\in\{0,\ldots,m-1\}^2$ and $\mc P\subset\R^2$ is a strip of width $\epsilon$-comparable to $N$, and can be easily approximated as \eqref{eq:Bohr char is approx in l2}. Thus, we assume \eqref{eq:equidist over [N]^2}.

Approximating $\chi_{I_1+\Z}:\R/\Z\rightarrow\C$ by a Fej{\'e}r kernel, there exist $K=O_{\epsilon,\delta}(1)$ and $c_{-K},\ldots,c_K\in\D$ such that
\[
\norm{\chi_{I_1+\Z}-\sum_{|k|\le K}c_ke^{2\pi ikx}}_{L^\infty(\R\setminus (\d I+(-\delta,\delta)+\Z))}\le\frac{\epsilon}{10}
\]
and
\[
\norm{\chi_{I_1+\Z}-\sum_{|k|\le K}c_ke^{2\pi ikx}}_{L^\infty(\R)}\le2.
\]
Then, by \eqref{eq:equidist over [N]^2}, \eqref{eq:Bohr char is approx in l2} is satisfied.

More generally, for $r\in\N$, we can represent $\chi_{\mc B}=\prod_{j\le r}\chi_{\{x\in[N]^2:x\cdot\theta_j\in I_j+\Z\}}$ and thus we are done for the case that $\phi=\chi_{\mc B}$.

Now we consider a general locally linear modulation $\phi$ supported on $\mc B$ of rank $r$. Up to subpartitioning $\mc B$, we assume that $\Omega$ in Remark \ref{rem:Bohr set by lattice-convex proj} is sufficiently thick. Let $F:\R^2\times\R^r\rightarrow\R$ be the degree $1$ polynomial in Remark \ref{rem:Bohr set by lattice-convex proj}. Write $F$ as
\[
F(x,y)=\theta_{\R^2}\cdot x+\theta_{\R^r}\cdot y+c,
\]
where $\theta_{\R^2}\in\R^2$, $\theta_{\R^r}\in\R^r$, and $c\in\R$. Here, since we work on $y\in\Z^{r}$, there is no harm in assuming $\theta_{\R^r}\in[0,2\pi)^r$. For $(x,y)\in\Omega$, since $y\in P(x)-I_1\times\cdots\times I_r$, we have
\[
F(x,y)=\theta_{\R^2}\cdot x+\theta_{\R^r}\cdot P(x)+c'+O(|I_1|+\cdots+|I_r|),
\]
where $c'\in\R$.
Since $\theta_{\R^2}\cdot x+\theta_{\R^r}\cdot P(x)$ is a linear form of $x$, there exists $\theta'\in\R^2$ such that
\[
F(x,y)=\theta'\cdot x+c'+O(|I_1|+\cdots+|I_r|).
\]
With this $\theta'$, we have
\begin{equation}\label{eq:local Bohr is approx by single eix}
|e^{ic'}e^{i\theta'\cdot x}-\phi(x)|=|e^{i(\theta'\cdot x+c')}-e^{iF(x,y)}|\les |I_1|+\cdots+|I_r|,\qquad x\in\mc B.
\end{equation}
Let $m\in\N$ be a number to be fixed shortly. Perform an $m$-partition of $\mc B$ into sub-Bohr sets $\mc B'$. Since $|I_j'|=|I_j|/m$, choosing $m$ big enough, we can approximate $\phi$ in $\ell^\infty(\mc B')$ by a linear modulation up to $\epsilon/2$ error. Since each $\chi_{\mc B'}$ is already arbitrarily approximable in $\ell^2([N]^2)$ as a linear combination of linear modulations, this finishes the construction for \eqref{eq:Bohr char is approx in l2}.

We satisfy \eqref{eq:Bohr char is bdd in l infty}, mimicking the proof of Lemma \ref{lem:multiplicative decomp}. Let
\[
\psi_0(z):=\begin{cases}
z&,\qquad |z|\le 1\\
z/|z|&,\qquad |z|>1.
\end{cases}
\]
By the Stone-Weierstra{\ss} Theorem, there exists a polynomial $\psi=\psi_{\epsilon,J}:\C\rightarrow\C$ of $z,\bar z$ such that $\norm{\psi-\psi_0}_{C^0(J\D)}\le\frac{\epsilon}{10}$ and $\psi(J\D)\subset\D$.

Let $h=\sum_{j\le J}c_je^{ix\cdot\xi_j}$ be a function satisfying \eqref{eq:Bohr char is approx in l2}. For every $x\in\mc B$, since $|\phi(x)|\le1$, we have $|\phi(x)-\psi_0(h(x))|\le|\phi(x)-h(x)|$, hence by the triangle inequality, we have
\[
\norm{\phi(x)-\psi(h(x))}_{\ell^2([N]^2)}\le\epsilon N+\norm{\frac{\epsilon}{10}}_{\ell^2([N]^2)}\le 2\epsilon N.
\]
Since $\psi$ is $\D$-valued on $J\D$, $\psi(h)$ has $\ell^\infty$-norm bounded by $1$. Since $\psi$ is a polynomial of $z$ and $\ol z$, $\psi(h)$ can also be written as an $O_{\epsilon,J}(1)$ linear combination of linear modulations. Reparametrizing $\epsilon$ by $\epsilon/2$, this finishes the proof.
\end{proof}
\begin{lem}\label{lem:bias lem for local quad on Bohr}
Let $r\in\N$ and $\epsilon>0$. There exists $J\in\N$ satisfying the following:

Let $N\in\N$ be an integer. Let $\phi$ be a locally quadratic modulation supported on an affine Bohr set $\mc B$ in $[N]^2$ of rank $r$, such that
\begin{equation}\label{eq:bias lem Bohr assump}
\left|\E_{n\in[N]^2}\phi\right|\ge\epsilon.
\end{equation}
Then, there exist $\xi_1,\ldots,\xi_J\in\R^2$ and $c_1,\ldots,c_J\in\D$ such that
\begin{equation}\label{eq:bias lem Bohr l2}
\norm{\phi-\sum_{j\le J}c_je^{ix\cdot\xi_j}}_{\ell^2([N]^2)}\le\epsilon N
\end{equation}
and
\begin{equation}\label{eq:bias lem Bohr l infty}
\norm{\sum_{j\le J}c_je^{ix\cdot\xi_j}}_{\ell^\infty([N]^2)}\le1.
\end{equation}
\end{lem}
\begin{proof}
We claim the existence of affine Bohr sets $\mc B_1,\ldots,\mc B_{J_0}$, $J_0=O_{r,\epsilon}(1)$ partitioning $\mc B$ and $c_1,\ldots,c_{J_0}\in\T$ such that $\norm{\phi-c_j}_{\ell^\infty(\mc B_j)}\le\frac{\epsilon}{10}$. Once we show this, the proof will conclude as follows: $\norm{\phi-\sum_j c_j\chi_{\mc B_j}}_{\ell^\infty([N]^2)}\le\frac{\epsilon}{10}$ then holds. By Lemma \ref{lem:Bohr char is sum of linear}, each $c_j\chi_{\mc B_j}$ can be approximated by an $O_{\epsilon,J}(1)$-linear combination of linear modulations, up to an $\frac{\epsilon}{2J}N$ error in $\ell^2([N]^2)$. Then, by the triangle inequality, \eqref{eq:bias lem Bohr l2} will be satisfied. Then, \eqref{eq:bias lem Bohr l infty} will be satisfied by the Stone-Weierstra{\ss} argument used in Lemma \ref{lem:Bohr char is sum of linear}.

We adopt notations in Remark \ref{rem:Bohr set by lattice-convex proj}. Up to an $O_{\epsilon,r}(1)$-partition, we can assume $\Omega$ to be thick.
In terms of the lifted quadratic polynomial $F$, \eqref{eq:bias lem Bohr assump} can be rewritten as
\[
\left|\sum_{x\in\Omega}e^{iF(x)}\right|\ge\epsilon N^2.
\]
Up to a constant modulation, we may write
\[
F=Q+L,
\]
where $Q$ and $L$ are quadratic and linear forms.
Then, since $\#\Omega=\#\mc B\les N^2$, by Lemma \ref{lem:convex Weyl}, there exists $m=O_{\epsilon,r}(1)$ such that
\[
\sup_{\substack{x,y\in\Omega\\ x-y\in\frac{1}{m}([-N_1,N_1]w_1+\cdots+[-N_{2+r},N_{2+r}]w_{2+r})\\ m\mid x-y}}|e^{iF(x)}-e^{iF(y)}|\le\epsilon.
\]
Here, $N_1,\ldots,N_{r+2}\gg_{\epsilon,r}1$ and $w_1,\ldots,w_{r+2}\in\Z^{r+2}$ are as in Lemma \ref{lem:convex Weyl}. Let $\{\mc B_j\}=\pi_m(\mc B)$, then for each $j$, $\phi\mid_{\mc B_j}=e^{iF}\mid_{\mc B_j}$ varies by at most $\epsilon$, finishing the proof.
\end{proof}
\section{Degree-lowering on inverse Gowers inequalities}\label{app-sec:deg}
In this section, we build a degree-lowering theorem that will play a key role in Section \ref{subsec:analytic}.
We newly define a property of norms on $[N]$ (namely the \emph{alt-stable property}; see Definition \ref{defn:alt-stable}). Then we prove Theorem \ref{thm:degree reduction}, which is the main theorem of this section. Theorem \ref{thm:degree reduction} shows that, for any norm $\mc N$ in such class and integers $1\le d_0<d$, any $U^{d+1}$-inverse element with a large $\mc N$-norm should also have a large $U^{d_0+1}$-norm, once that is shown for the case $d=d_0+1$.

The early part of this section is devoted to recalling known facts in a self-contained manner. In particular, we recall the key inverse Gowers theorem in \cite{GreenTaoZiegler12(Ud)} and equidistribution theory in \cite{GreenTao12Quantitative}.
Then we introduce the definition of our new concept, the alt-stable property, and show our main theorem of this section (Theorem \ref{thm:degree reduction}).

This section builds on the theory developed in \cite{GreenTao08(U3), GreenTao12Quantitative,GreenTaoZiegler11(U4),GreenTaoZiegler12(Ud)} and we will use some of the notations and results in a crucial way.
We start with introducing generic concepts.
\begin{defn}[Nilmanifolds]
A \emph{nilmanifold} is a closed manifold of the form $G/\Gamma$, where $G$ is a connected, simply-connected nilpotent Lie group and $\Gamma$ is a discrete subgroup of $G$.

Since $G$ is nilpotent, by the unimodularity there exists a unique Haar measure $\mu_G$ with a normalized induced measure $\mu_{G/\Gamma}$. We denote by $\mu_G$ and $\mu_{G/\Gamma}$ such measures.
\end{defn}
Note that the nilmanifold $G/\Gamma$ above is not a quotient Lie group in general; we do not impose $\Gamma$ to be normal in $G$.
\begin{defn}[Rational subgroup, \cite{GreenTao12Quantitative}]
A \emph{rational subgroup} of a nilmanifold $G/\Gamma$ is a subgroup $G'\le G$ which is closed, connected, and makes $G'/(G'\cap\Gamma)$ compact (or equivalently, $G'/(G'\cap\Gamma)$ is also a nilmanifold).\footnote{That $G'$ is simply connected follows from \cite{Malcev51}. (Indeed, such $G'$ is homeomorphic to an $\R$-vector space.)}
\end{defn}
\begin{defn}[Rational element, \cite{GreenTao12Quantitative}]
An element $\gamma\in G$ is a \emph{rational element} of a nilmanifold  $G/\Gamma$ if there exists $k\in\N$ such that $\gamma^k\in\Gamma$ holds.

For $Q\in\N$, $\gamma$ is called \emph{$Q$-rational} if there exists $k\le Q$ such that $\gamma^k\in\Gamma$ holds.
\end{defn}
\begin{defn}[Filtered nilmanifolds, \cite{GreenTao12Quantitative}]\label{defn:filtered}
A \emph{filtered nilmanifold $X$ of degree at most $d$}, $d\ge 1$, is a nilmanifold $G/\Gamma$ equipped with a \emph{filtration} $G_\N=\{G_0,G_1,\ldots\}$ of rational subgroups of $G$ such that
\[
G=G_0=G_1\ge G_2\ge\cdots\ge G_{d+1}=\cdots=\{1_G\}
\]
and for every $j,k\ge 0$,
\[
[G_j,G_k]\le G_{j+k}.
\]
Here, $[G_j,G_k]$ denotes the commutator subgroup of $G_j$ and $G_k$.

A simple example of filtration is the lower central series $G_j:=[G,G_{j-1}],j\ge 2$.

For simplicity, we also denote $X=(G/\Gamma,G_\N)$.
We use the following conventions:
\begin{itemize}
    \item For a rational subgroup $G'\le G$, $G_\N\cap G':=\{G_0\cap G',G_1\cap G',\ldots\}$.
    \item For a rational normal subgroup $N\trianglelefteq G$, $G_\N/N:=\{G_0N/N,G_1N/N,\ldots\}$.
\end{itemize}
\end{defn}
In Definition \ref{defn:filtered}, $G_d$ plays an important role. Throughout this section, we denote $k_d:=\dim G_d$ and $\Gamma_d:=G_d\cap\Gamma$. Since $G_d$ is a connected, simply connected abelian Lie group, it is an $\R$-vector space of dimension $k_d=\dim G_d$. By the rationality of $G_d$, $G_d/\Gamma_d$ can be viewed as a torus and we naturally introduce the Pontryagin dual $\widehat{G_d/\Gamma_d}$, identifying it with $\Z^{k_d}$.

Similar to $\mu_{G/\Gamma}$, $\mu_{G_d/\Gamma_d}$ denotes the normalized Haar measure on the torus $G_d/\Gamma_d$.

\begin{defn}[Mal'cev basis, \cite{Malcev51,GreenTao12Quantitative}]\label{defn:Malcev}
Let $X=(G/\Gamma,G_\N)$ be a filtered nilmanifold. Let $m=\dim G$. A basis $\mc X=\{X_1,\ldots,X_m\}$ for the Lie algebra $\mathfrak{g}$ over $\R$ is called a \emph{Mal'cev basis} for $X$ if \cite[Definitions 2.1 and 2.4]{GreenTao12Quantitative} are satisfied.\footnote{We omit the precise definition here, since we do not work with the definition within this paper.}

Corresponding to such $\mc X$, a \emph{Mal'cev} coordinate map $\psi=\psi_{\mc X}:G\rightarrow\R^m$ is defined as the inverse of the bijection $(t_1,\ldots,t_m)\mapsto e^{t_1X_1}\cdots e^{t_mX_m}$. The following are known:
\begin{enumerate}
    \item\label{enu:Malcev diffeom} $\psi$ is a diffeomorphism, \cite[Definition 2.1(iii)]{GreenTao12Quantitative}.
    \item\label{enu:Malcev Z^m} $\psi(\Gamma)=\Z^m$, \cite[Definition 2.1(iv)]{GreenTao12Quantitative}.
    \item\label{enu:Malcev G_k} $\psi(G_k)=\{0\}^{m-\dim G_k}\times\R^{\dim G_k}$ holds, \cite[Definition 2.1(i)--(ii)]{GreenTao12Quantitative}.
    \item\label{enu:Malcev poly} For $g,h\in G$, $\psi(gh)$ and $\psi(g^{-1})$ can be written as rational polynomials of $\psi(g)$ and $\psi(h)$, \cite[Lemma A.3]{GreenTao12Quantitative}.
    \item\label{enu:Malcev rational element} For $Q\ge 2$ and a $Q$-rational element $\gamma\in G$, there exists an integer $Q'\le Q^{O_X(1)}$ such that $\psi(\gamma)\in\frac{1}{Q'}\Z^m$, \cite[Lemma A.11(iii)]{GreenTao12Quantitative}.
    \item\label{enu:Malcev rational subgp} A connected Lie subgroup $G'\le G$ is rational if and only if the Lie algebra $\mathfrak g'$ has a basis consisting of rational combinations of $X_1,\ldots,X_m$, \cite[Lemma 4.3]{Zorin-ergodic}.
\end{enumerate}
\end{defn}
\begin{prop}\label{prop:Malcev existence}
Any filtered nilmanifold has a Mal'cev basis for it.
\end{prop}
\begin{proof}
This is an immediate consequence of \cite[Proposition A.9]{GreenTao12Quantitative} taking the initial weak basis in there as the coordinates in \cite{Malcev51} (i.e., a Mal'cev basis adapted to the lower central series of $G$).
\end{proof}

In the sense of Proposition \ref{prop:Malcev existence}, in the rest of this section, we regard each filtered nilmanifold $X$ equipped with a fixed Mal'cev basis $\mc X$ and the corresponding Mal'cev coordinate map $\psi_X=\psi_{\mc X}$.
Having fixed a Mal'cev basis, we further equip a right-invariant metric $d_G(\cdot,\cdot)$ on $G$ as given in \cite[Definition 2.2]{GreenTao12Quantitative}. We denote by $d_X(\cdot,\cdot)$ the metric on $X$ defined as
\[
d_X(g\Gamma,h\Gamma):=\inf_{\la,\la'\in\Gamma}d_G(g\la,h\la').
\]
See \cite[Lemma A.15]{GreenTao12Quantitative} for the proof that such $d_X$ is indeed a metric.

It is known that there exists $C=C(X)\in\N$ such that for $M\ge2$ and $\epsilon,g,h\in G$ such that $d_G(\epsilon,1_G)\le M$ and $d_G(g,h)\le M$, the inequality
\begin{equation}\label{eq:prep for Lip}
d_G(\epsilon g,\epsilon h)\le M^C d_G(g,h)
\end{equation}
holds. \eqref{eq:prep for Lip} is immediate from \cite[Lemmas A.5 and A.14]{GreenTao12Quantitative}.
\begin{defn}[$Q$-rational subgroup, \cite{GreenTao12Quantitative}]
Let $X=(G/\Gamma,G_\N)$ be a filtered nilmanifold and $Q\in\N$. Let $\mc X=\{X_1,\ldots,X_m\}$ be the Mal'cev basis equipped to $X$. A rational subgroup $G'\le G$ is $Q$-rational if the Lie algebra $\mathfrak g'$ has a basis $\mathcal X'$ consisting of linear combinations $\sum_{j=1}^m c_jX_j$, where each $c_j$ can be written as $m/n$ with $\max\{|m|,|n|\}\le Q$.
\end{defn}
\begin{lem}\label{lem:commensurate}
Let $X=(G/\Gamma,G_\N)$ be a filtered nilmanifold and $Q\ge 2$. Let $\gamma\in G$ be a $Q$-rational element. Then, we have $[\Gamma:\Gamma\cap\gamma\Gamma\gamma^{-1}]\le Q^{O_X(1)}$.
\end{lem}
\begin{proof}
Equivalently, we show the existence of $L=O_X(1)$ such that for any $\la_1,\ldots,\la_{Q^L}\in\Gamma$, there exist $j\neq k$ such that $\la_k^{-1}\la_j\in\gamma\Gamma\gamma^{-1}$.
Let $m=\dim G$.
Let $\psi$ be as in Definition \ref{defn:Malcev}. By Definition \ref{defn:Malcev} \eqref{enu:Malcev Z^m}, we have $\psi(\Gamma)=\Z^m$. By Definition \ref{defn:Malcev} \eqref{enu:Malcev rational element} and \eqref{enu:Malcev poly}, $\psi(\gamma^{-1}\la_k^{-1}\la_j\gamma)$ is a polynomial of $\psi(\la_j)$ and $\psi(\la_k)$ with rational coefficients whose denominators are bounded by $Q^{O_X(1)}$. This shows that, for $L\gg_X1$, we can find $j\neq k$ such that $\psi(\gamma^{-1}\la_k^{-1}\la_j\gamma)\in\Z^m=\psi(\Gamma)$, finishing the proof.
\end{proof}
\begin{lem}[\cite{GreenTaoZiegler12(Ud), GreenTao12Quantitative}]\label{lem:filter updates}
Let $G_\N\cap G'$ and $G_\N/G_d$ be as in Definition \ref{defn:filtered}. They are filtrations for $G'/(G'\cap\Gamma)$ and $(G/G_d)/(\Gamma G_d/G_d)$, respectively.
\end{lem}
\begin{proof}
The only nontrivial part is the rationality of the members.

For $G_\N\cap G'$, it suffices to show that an intersection of any two rational subgroups is rational. This is immediate from Definition \ref{defn:Malcev} \eqref{enu:Malcev rational subgp}.

For $G_\N/G_d$, it suffices to check the rationalities of each $G_j/G_d$ over $\Gamma G_d/G_d$. Compactness is straightforward; we show only the discreteness of $\Gamma G_d/G_d$ in $G/G_d$. By Definition \ref{defn:Malcev} \eqref{enu:Malcev Z^m} and \eqref{enu:Malcev G_k}, the coordinate map $\psi_X$ maps $\Gamma G_d\setminus G_d$ to $(\Z^{\dim G-k_d}\setminus\{0\})\times\R^{k_d}$. This implies that $\Gamma G_d/G_d$ is discrete as claimed, finishing the proof.
\end{proof}
In view of Lemma \ref{lem:filter updates}, for a filtered nilmanifold $X=(G/\Gamma,G_\N)$ of degree $d$, we denote $X\cap G'=(G'/(G'\cap\Gamma),G_\N\cap G')$ and $X/G_d=((G/G_d)/(\Gamma G_d/G_d),G_\N/G_d)$.

Next, we recall ingredients from \cite{GreenTaoZiegler12(Ud)} and \cite{GreenTao12Quantitative} with minor modifications.
\begin{defn}[Polynomial sequence, \cite{GreenTao12Quantitative}]
Let $(G/\Gamma,G_\N)$ be a filtered nilmanifold. $\poly(G_\N)$ denotes the set of sequences $g:\Z\rightarrow G$ such that for every $k\in\N$, $a_1,\ldots,a_k\in \Z$, and $n\in\Z$,
\[
\De_{a_1}\cdots\De_{a_k}g(n)\in G_{k}
\]
holds. Here, we denoted $\De_a g(n):=g(n+a)g(n)^{-1}$.
\end{defn}
We define the nilsequence used in \cite{GreenTaoZiegler12(Ud)} with minor modification.
In our setting, the compactness of the set $\F$ replaces the role of Lipschitz constraint in \cite{GreenTaoZiegler12(Ud)}. This replacement is not necessary and used merely for conciseness within this section.
\begin{defn}[Nilsequence]
Let $X=(G/\Gamma,G_\N)$ be a filtered nilmanifold. Let $\F$ be a compact subset of $C^0(X;\C)$. We denote
\[
\mc S_{X,\F}:=\left\{\{F(g(n)\Gamma)\}_{n\in\Z}:F\in\mc F,g\in \poly(G_\N)\right\}.
\]
\end{defn}
The following notion brings the concept of vertical oscillation used, e.g., in \cite{GreenTao12Quantitative}:
\begin{defn}\label{defn:C*}
Let $X=(G/\Gamma,G_\N)$ be a filtered nilmanifold of degree $d\in\N$. We denote
\begin{align*}
C^0_*(X;\C) :=&\{F\in C^0(X;\C):\text{there exists }\xi\in\widehat{G_d/\Gamma_d}\text{ such that}
\\& F(g_dg\Gamma)=\xi(g_d)F(g\Gamma)\text{ holds for }(g_d,g)\in G_d\times G\}.
\end{align*}
\end{defn}
\begin{defn}
For $L\in\N$, denote by $X_{L\T}$ the circle $\R/L\Z$ equipped with the trivial filtration $\R_\N:=\{\R,\R,\{0\},\cdots\}$, which has degree $1$.

We denote by $\varphi_{L\T}:X_{L\T}\rightarrow[0,1]$ a smooth function such that
\[
\varphi_{L\T}(m):=\begin{cases}
1,\qquad m\in L\Z
\\
0,\qquad m\in\Z\setminus L\Z.
\end{cases}
\]
\end{defn}

As a particular example, for $a\in\N$ and $b\in\Z$, one has
\begin{equation}\label{eq:Zm is nilseq}
\chi_{a\Z+b}=\varphi_{a\T}(\cdot-b)\in\mc S_{X_{a\T},\{\varphi_{a\T}\}}.
\end{equation}
The following is the main result of the breakthrough \cite{GreenTaoZiegler12(Ud)} (extending \cite{GreenTao08(U3),GreenTaoZiegler11(U4)}).
\begin{prop}[Inverse Gowers $U^{d+1}$ theorem, \cite{GreenTao08(U3),GreenTaoZiegler11(U4),GreenTaoZiegler12(Ud)}]\label{prop:SXF inverse Ud}
Let $d\in\N$ and $\delta>0$. There exist a nilmanifold $X=G/\Gamma$ equipped with the lower central series of step $d$ and a compact set $\F\subset C^0(X;\C)$ satisfying the following:

For $N\in\N$ and $f:[N]\rightarrow \D$ such that $\norm{f}_{U^{d+1}}\ge\delta$, there exists $f'\in\mc S_{X,\F}$ such that
\[
\left|\E_{n\in[N]}f(n)\overline{f'(n)}\right|\gtrsim_\delta 1.
\]
\end{prop}
\begin{proof}
This was essentially shown in \cite[Theorem 1.3]{GreenTaoZiegler12(Ud)} (for $d\ge 3$), extending previous works \cite{GreenTao08(U3)} ($d=2$) and \cite{GreenTaoZiegler11(U4)} ($d=3$).

The precise statement in \cite{GreenTaoZiegler12(Ud)} involves a finite collection of nilmanifolds; however, as is also mentioned in \cite{GreenTaoZiegler12(Ud)}, the collection can be reduced to a single nilmanifold by taking a Cartesian product.

In \cite[Theorem 1.3]{GreenTaoZiegler12(Ud)}, $\F$ is a collection of $O_{d,\delta}(1)$-Lipschitz functions. Such $\F$ is precompact in the topology $C^0(X;\C)$, thus our statement also holds.
\end{proof}
The converse direction is also known.
\begin{prop}[{\cite[Proposition 1.4]{GreenTaoZiegler11(U4)}}]\label{prop:SXF anti-inverse Ud}
Let $d\in\N$. Let $X=(G/\Gamma,G_\N)$ be a filtered nilmanifold of degree at most $d$. There exists a dense set $\mc D_{X}\subset C^0(X;\C)$ of $F:X\rightarrow\C$ satisfying the following:

Let $f'\in\mc S_{X,\{F\}}$ and $\delta>0$. Then, for $N\in\N$ and $f:[N]\rightarrow\D$ such that
\[
\left|\E_{n\in[N]}f'(n)\overline{f(n)}\right|\ge\delta,
\]
$\norm{f}_{U^{d+1}}\gtrsim_{\delta,X,F}1$ holds.
\end{prop}
\begin{proof}
This is provided in \cite[Proposition 1.4]{GreenTaoZiegler11(U4)}, precisely for the set of $F$ which is Lipschitz with respect to $d_X(\cdot,\cdot)$. By the Stone-Weierstra{\ss} Theorem, this set is dense in $C^0(X;\C)$.
\end{proof}
We provide the closedness (up to updates on $X$ and $\F$) of $\mc S_{X,\F}$ under additions and multiplications.
\begin{lem}[{\cite[Corollary E.2]{GreenTaoZiegler12(Ud)}}]\label{lem:SXF linear}
Let $X_1$ and $X_2$ be filtered nilmanifolds. Let $\F_k\subset C^0(X_k;\C)$, $k=1,2$ be compact sets. For every $f_1\in\mc S_{X_1,\F_1}$ and $f_2\in\mc S_{X_2,\F_2}$, the following hold:
\begin{enumerate}
\item\label{enum:f_1+f_2} $f_1+f_2\in\mc S_{X_1\times X_2,\F_1\otimes \{1_{X_2}\}+\{1_{X_1}\}\otimes \F_2}$
\item\label{enum:f_1f_2} $f_1f_2\in\mc S_{X_1\times X_2,\F_1\otimes\F_2}$
\end{enumerate}
\end{lem}
Here, we denoted $\F_1\otimes \F_2=\{f_1\otimes f_2:f_1\in\F_1,f_2\in\F_2\}$.
\begin{proof}
For $f_k=F_k(g_k(n)\Gamma_k)\in\mc S_{X_k,\F_k}$, $k=1,2$, we can write
\[
(f_1+f_2)(n)=(F_1\otimes 1_{X_2}+1_{X_1}\otimes F_2)((g_1,g_2)(n)\Gamma_1\times\Gamma_2),
\]
thus \eqref{enum:f_1+f_2} holds. \eqref{enum:f_1f_2} can be shown similarly.
\end{proof}
\begin{lem}\label{lem:profile decomp Ud}
Let $d\in\N$ and $\epsilon>0$. There exists a filtered nilmanifold $X=G/\Gamma$ of degree $d$ and a compact set $\F\subset C^0(X;\C)$ satisfying the following:

For $N\in\N$, $f:[N]\rightarrow\D$, and $S\supset\supp(f)$, there exists $h=h_0\chi_S$, $h_0\in\mc S_{X,\F}$ such that
\[
\norm{h}_{\ell^\infty([N])}\le 1,
\]
\[
\norm{f-h}_{U^{d+1}([N])}<\epsilon,
\]
and
\[
|\jp{f-h,h}_{\ell^2([N])}|<\epsilon N.
\]
Furthermore, $\F$ can be assumed to be a finite subset of $\mc D_X$ (given in Proposition \ref{prop:SXF anti-inverse Ud}).
\end{lem}
\begin{proof}
By Lemma \ref{lem:multiplicative decomp}, Lemma \ref{lem:SXF linear}, and Proposition \ref{prop:SXF inverse Ud}, the existence of such $X$ and compact set $\F$ is immediate.
Approximating $\F$ up to an $o(\epsilon)$-error in $C^0(X)$, $\F$ can be reduced to a finite subset of $\mc D_X$, as claimed.
\end{proof}
The next lemma recalls {\cite[Lemma E.8 (iv)]{GreenTaoZiegler12(Ud)}} in a standard-analysis version.
\begin{lem}[{\cite[Lemma E.8 (iv)]{GreenTaoZiegler12(Ud)}}]\label{lem:SXF Alt reduce}
Let $\epsilon>0$ and $X=(G/\Gamma,G_\N)$ be a filtered nilmanifold of degree $d\ge 1$. Let $\F\subset C^0_*(X;\C)$ be any compact set.

Then, there exist a filtered nilmanifold $\tilde X$ of degree at most $(d-1)$ and a compact set $\tilde\F\subset C^0(\tilde X;\C)$ such that for every $\eta\in\Z$,
\begin{equation}\label{eq:Alt:CS}
\Alt_\eta:\mc S_{X,\F}\rightarrow\mc S_{\tilde X,\tilde\F}.
\end{equation}
\end{lem}
\begin{proof}
This is essentially shown in \cite[Lemma E.8]{GreenTaoZiegler12(Ud)}; in this proof, we only recall the explicit forms of $\tilde X$ and $\F$ to display the compactness. Let $G^\square_j:=\{(g,g')\in G_j^2:g^{-1}g'\in G_{j+1}\}$. As discussed in \cite[Lemma E.8]{GreenTaoZiegler12(Ud)} and \cite[Proposition 7.2]{GreenTao12Quantitative}, $\tilde X^0=(G^\square_0/\Gamma^2,G^\square_\N)$ is a filtered nilmanifold of degree $d$. We set $\tilde X=(\tilde G/\tilde\Gamma,\tilde G_\N)=\tilde X^0/G_d^\square$
and
\[
\tilde\F=\{\tilde F\in C^0(\tilde X;\C):\tilde F((g,h)\tilde\Gamma)=F(g)\ol{F(h)},\quad F\in\F\}.
\]
As shown in Lemma \ref{lem:filter updates}, $\tilde X$ is a filtered nilmanifold of degree $(d-1)$. $\tilde F$ is well-defined since $G^\square_d=\{(g_d,g_d):g_d\in G_d\}$ and $F(g_dg)\ol{F(g_dh)}=F(g)\ol{F(h)}$ holds for every $g_d\in G_d$. With these $\tilde X$ and $\tilde\F$, \eqref{eq:Alt:CS} is satisfied and $\tilde\F$ is compact in $C^0(\tilde X;\C)$, thus we finish the proof.
\end{proof}

We emphasize in Lemma \ref{lem:SXF Alt reduce} that $\tilde X$ has degree at most $(d-1)$. This is the key to our dimension reduction argument in this section.

The next lemma recalls the Fourier decomposition given in the proof of \cite[Proposition 5.6]{TaoTeravainen19}. For self-containedness, we provide a separate proof.
\begin{lem}\label{lem:SXF Fourier decomp}
Let $\epsilon>0$ and $X=(G/\Gamma,G_\N)$ be a filtered nilmanifold of degree $d\ge 1$. Let $\F$ be a compact subset of $C^0(X;\C)$. There exist a number $J\in\N$ and a compact set $\F_*\subset C^0_*(X;\C)$ satisfying the following:

For every $f\in\mc S_{X,\F}$, there exist $f_1,\ldots,f_J\in\mc S_{X,\F_*}$ such that
\[
\norm{f-\sum_{j\le J} f_j}_{\ell^\infty(\Z)}\le\epsilon.
\]
\end{lem}
\begin{proof}
For $F\in \F$ and $\xi\in\widehat{G_d/\Gamma_d}$, denote by $F_\xi:X\rightarrow\C$ the function
\[
F_\xi(g\Gamma):=\int\xi(h^{-1})F(hg\Gamma)d\mu_{G_d/\Gamma_d}(h).
\]
Since $G_d$ is in the center of $G$, $F_\xi\in C^0_*(X;\C)$ holds.

Identify $\widehat{G_d/\Gamma_d}$ with $\Z^{k_d}$. Consider the $k_d$-dimensional Fej\'{e}r kernel approximation
\[
S_mF=\sum_{\xi=(\xi_1,\ldots,\xi_{k_d})\in\Z^{k_d}}\prod_{j\le k_d}\max\left\{1-|\xi_j|/m,0\right\}\cdot F_\xi
\]
then since $\F\subset C^0(X;\C)$ is equicontinuous, there exists $m\in\N$ such that $\sup_{F\in\F}\norm{S_mF-F}_{C^0(X)}\le\epsilon$.
Setting
\[
\F_*:=\{cF_\xi:c\in[0,1],\xi\in[m]^{k_d},F\in\F\}
\]
and $J:=\#[m]^{k_d}=(2m+1)^{k_d}$
finishes the proof.
\end{proof}
We provide a version of \cite[Theorem 10.2]{GreenTao12Quantitative} in the following proposition:
\begin{prop}\label{prop:egr original}
Let $X=(G/\Gamma,G_\N)$ be a filtered nilmanifold of degree $d\ge 1$. Let $M_0\ge2$ and $N\in\N$. Let $g\in\poly(G_\N)$.
Then, there exist an integer $M_0\le M\le M_0^{O_X(1)}$, an $M$-rational subgroup $G'\le G$, $g'\in\poly(G_\N\cap G')$, and $\epsilon,\gamma:[N]\rightarrow G$ satisfying the following:
\begin{itemize}
    \item For $n\in[N]$, $g(n)=\epsilon(n)g'(n)\gamma(n)$ holds.
    \item For $n\in[N]$, $d_G(\epsilon(n),1_G)\le M$ and $d_G(\epsilon(n),\epsilon(n-1))\le\frac{M}{N}$.
    \item For $n\in[N]$, $\gamma(n)$ is an $M$-rational element.
    \item $\gamma(\cdot)$ is periodic with a period $l\le M^{O_X(1)}$.
    \item Let $\gamma_0$ be an $M$-rational element. Let $\mc P\subset [N]$ be any arithmetic progression of length $\#\mc P\ge\frac{N}{M^{C+1}l}$. Let $F:G'\rightarrow\C$ be a function invariant under right-multiplication by $\Gamma_{\gamma_0}':=G'\cap(\gamma_0\Gamma\gamma_0^{-1})$ and $1$-Lipschitz with respect to the subspace metric induced from $d_G(\cdot,\cdot)$. We have
    \begin{equation}\label{eq:G bias g'}
    \left|\E_{n\in\mc P}F(g'(n))-\int_{G'/\Gamma_{\gamma_0}'}Fd\mu_{G'/\Gamma_{\gamma_0}'} \right|\le\frac{1}{M^{C+1}}.
    \end{equation}
\end{itemize}
Here, $\mu_{G'/\Gamma_{\gamma_0}'}$ denotes the normalized quotient Haar measure and $C=C(X)$ is as in \eqref{eq:prep for Lip}.
\end{prop}
\begin{proof}
This is a consequence of \cite[Proposition 10.2]{GreenTao12Quantitative} (setting the index $t=1$ in there). We adopt notions in \cite{GreenTao12Quantitative}. Indeed, we use a stronger version of \cite[Proposition 10.2]{GreenTao12Quantitative}, previously discussed in \cite[Theorem D.4]{GreenTaoZiegler12(Ud)}. That is, we generalize the setting $\Gamma'=G'\cap\Gamma$ in there (not directly applicable to our case of $\Gamma_{\gamma_0}'$) to
\begin{equation}\label{eq:Gamma'}
\Gamma'\le G'\cap\Gamma\text{ of index }O(M^c),
\end{equation}
where $c=c_X\in\N$ is a number to be fixed shortly. Although the statement is more general, the proof of \cite[Proposition 10.2]{GreenTao12Quantitative} works identically even if one allows $\Gamma'$ to be any of \eqref{eq:Gamma'}. \footnote{This works since quantitative rationalities of subgroups, elements, and metrics are comparable (up to $M^{O_X(1)}$-comparabilities) over $G'\cap\Gamma$ and $\Gamma'$; see \cite{GreenTao12Quantitative} for details.}

For every $M$-rational element $\gamma_0$, by Lemma \ref{lem:commensurate}, we have 
\[
[G'\cap\Gamma:\Gamma_{\gamma_0}'\cap\Gamma]\le[\Gamma:\gamma_0\Gamma\gamma_0^{-1}\cap\Gamma]\le M^c
\]
for some $c=c_X\in\N$; setting $\Gamma'=\Gamma'_{\gamma_0}\cap\Gamma$, \eqref{eq:Gamma'} holds.
Now our statement is just a paraphrase of \cite[Theorem 10.2]{GreenTao12Quantitative} choosing a large parameter $A\gg_X1$ in there. (Note here that $\gamma(n)$ in our statement should be read as a representative of $\gamma(n)\Gamma$ in the original statement of \cite[Theorem 10.2]{GreenTao12Quantitative}; see \cite[Definition 1.17]{GreenTao12Quantitative} for the comparison. The periodicity of $\gamma(n)\Gamma$ is a consequence of \cite[Lemma A.12(ii)]{GreenTao12Quantitative} and \cite[Theorem 10.2(iii)]{GreenTao12Quantitative}.) This finishes the proof.
\end{proof}
Next, we show a version of \cite[Corollary E.6]{GreenTaoZiegler12(Ud)}.
\begin{lem}\label{lem:Cd E Sd-1}
Let $d\ge 2$ and $\delta>0$. Let $X=(G/\Gamma,G_\N)$ be a filtered nilmanifold of degree $d$. Let $\F\subset C^0_*(X;\C)$ be compact. There exist a filtered nilmanifold $\tilde X=(\tilde G/\tilde\Gamma,\tilde G_\N)$ of degree at most $(d-1)$ and a compact set $\tilde \F\subset C^0(\tilde X;\C)$ satisfying the following:

Let $N\in\N$ and $f\in\mc S_{X,\F}$. Assume that
\begin{equation}\label{eq:EF>0}
\left|\E_{n\in[N]}f(n)\right|\ge\delta.
\end{equation}
Then, there exists $f'\in\mc S_{\tilde X,\tilde\F}$ such that $f(n)=f'(n)$ holds for $n\in[N]$.
\end{lem}
Such $f'$ as above is said to \emph{extend} $f\mid_{[N]}$.
\begin{proof}[Proof of Lemma \ref{lem:Cd E Sd-1}]
Let $M_0\gg 1/\delta$ be an integer to be fixed later. Let $f=F(g(\cdot)\Gamma)\in\mc S_{X,\F}$. We will explicitly construct $\tilde X$ and $\tilde\F$ satisfying the statement, which will not depend on $f=F(g(\cdot)\Gamma)$ although we first fixed $f$ for convenience.
Let $M_0\le M\le M_0^{O_X(1)}$, $G'$, and $l$ be as in Proposition \ref{prop:egr original}. Here, up to taking $O_{M_0,X}(1)$ Cartesian products, we fix $M$, $G'$, and $l$. Decompose $g(n)=\epsilon(n)g'(n)\gamma(n)$ as in Proposition \ref{prop:egr original}. Let $C=C(X)$ be the number in \eqref{eq:prep for Lip}.
We first claim that
\begin{equation}\label{eq:claim for A}
F\text{ is invariant under multiplication by }G'\cap G_d.
\end{equation}
By the Stone-Weierstra{\ss} Theorem, there exists a $B=B(\F,\delta,X)$-Lipschitz map $F_*:X\rightarrow\C$ such that $\norm{F-F_*}_{C^0(X)}\le\delta/3$.
Partition $[N]$ into $[N]\cap(l\Z+k),k=0,\ldots,l-1$, then subpartition each by intervals of lengths $N/M^{C+1}$. Enumerating each progression by $\mc P_j$, we have a partition $\cup_j\mc P_j=[N]$. By \eqref{eq:EF>0} and $|F_*-F|\le\delta/3$ we have $|\E_{n\in[N]}F_*(g(n)\Gamma)|\ge2\delta/3$, then by pigeonholing there exists an index $j$ such that
\begin{equation}\label{eq:EF 2/3}
|\E_{n\in\mc P_j}F_*(g(n)\Gamma)|\ge\frac{2\delta}3.
\end{equation}
Choose any $n_j\in\mc P_j$ and denote $\epsilon_j:=\epsilon(n_j)$ and $\gamma_j:=\gamma(\mc P_j)$.
Since $\epsilon(\cdot)$ is $M/N$-Lipschitz, $F_*$ is $B$-Lipschitz, and the diameter of $\mc P_j$ is $O(N/M^{C+1})$, by the right-invariance of $d_G$, we have
\begin{equation}\label{eq:E error small}
\left|\E_{n\in\mc P_j}F_*(\epsilon(n)g'(n)\gamma(n))-\E_{n\in\mc P_j}F_*(\epsilon_jg'(n)\gamma_j)\right|\les B\cdot\frac{M}{N}\cdot\frac{N}{M^{C+1}}\les\frac{B}{M^{C}}.
\end{equation}
By \eqref{eq:prep for Lip} and $d_G(\epsilon_j,1_G)\le M$, the map $h\mapsto F_*(\epsilon_jh\gamma_j)$ is $BM^C$-Lipschitz, and is invariant under right-multiplication by $\gamma_j\Gamma\gamma_j^{-1}$. Thus, by \eqref{eq:G bias g'} we have
\begin{equation}\label{eq:G bias g' cor}
\left|\E_{n\in\mc P_j}F_*(\epsilon_jg'(n)\gamma_j)-\int_{G'/{\Gamma'}_{\gamma_j}}F_*(\epsilon_jh\gamma_j)d\mu_{G'/{\Gamma'}_{\gamma_j}}(h)\right|\les\frac{BM^C}{M^{C+1}}\les\frac{B}{M}.
\end{equation}
Thus, for $M_0\gg_{\delta,B}1$ large enough, by the triangle inequality on \eqref{eq:EF 2/3}, \eqref{eq:E error small}, and \eqref{eq:G bias g' cor}, we have
\begin{equation}\label{eq:int G' dh}
\left|\int_{G'/{\Gamma'}_{\gamma_j}}F_*(\epsilon_jh\gamma_j)d\mu_{G'/{\Gamma'}_{\gamma_j}}(h)\right|>\frac{\delta}3.
\end{equation}
Now since $|F-F_*|\le\delta/3$, we have
\begin{equation}\label{eq:integral neq 0}
\int_{G'/{\Gamma'}_{\gamma_j}}F(\epsilon_jh\gamma_j)d\mu_{G'/{\Gamma'}_{\gamma_j}}(h)\neq 0.
\end{equation}
Since $F\in\F\subset C^0_*(X;\C)$, there exists $\xi\in\widehat{G_d/\Gamma_d}$ such that $F(g_dg\Gamma)=\xi(g_d)F(g\Gamma)$. If $\xi\mid_{G'\cap G_d}$ were nontrivial, there would exist $g_*\in G'\cap G_d$ such that $\xi(g_*)=-1$. Then, substituting $h\mapsto g_* h$ would contradict \eqref{eq:integral neq 0}. Hence, $F$ should be invariant over $G'\cap G_d$. This implies \eqref{eq:claim for A} as claimed.

Now we construct $\tilde X$ and $\tilde\F$, independent of $N$ and $f$, satisfying this lemma. We construct for each $k=0,\ldots,l-1$ a $(d-1)$-degree filtered nilmanifold $X_k$ and a compact set $\F_k\subset C^0(X_k;\C)$, not depending on $N$ and $f$, containing $f_k\in\mc S_{X_k,\F_k}$ such that
\begin{equation}\label{eq:f'_k=f_k}
f_k(n)=F(\epsilon(n)g'(n)\gamma_k\Gamma)=f(n),\qquad n\in[N]\cap(l\Z+k),
\end{equation}
where we denoted $\gamma_k=\gamma(k)$.
Once we do this, the proof finishes by applying Lemma \ref{lem:SXF linear} and \eqref{eq:Zm is nilseq} to the right-hand side of
\[
f(n)=\sum_{k=0}^{l-1}\chi_{l\Z+k}\cdot f_k(n),\quad n\in[N].
\]
Let
\[
X_k^0:=(G/\gamma_k\Gamma\gamma_k^{-1},G_\N)
\]
and
\[
X_k:=(X_k^0\cap G')/(G_d\cap G')\times X_{10\T}.
\]
Let $\mc K\subset C^0(X_{10\T};G)$ be the set of functions $\tilde\epsilon:\R/10\Z\rightarrow G$ satisfying the Lipschitz bound
\[
\sup_{x,y\in\R/10\Z}d_G(\tilde\epsilon(x),\tilde\epsilon(y))\le M\cdot \dist(x,y)
\]
and
\[
\sup_{x\in\R/10\Z}d_G(\tilde\epsilon(x),1_G)\le M.
\]
Let $\F_k$ be the set
\[
\F_k:=\{(g'(G'\cap G_d)\gamma_k\Gamma\gamma_k^{-1},x)\mapsto F(\tilde\epsilon(x)g'\gamma_k\Gamma):\tilde\epsilon\in\mc K\}\subset C^0(X_k;\C).
\]
Now observe that $\epsilon$ extends to $\tilde\epsilon(\cdot/N)$, $\tilde\epsilon\in\mc K$ and thus
\[
\left(n\mapsto F(\tilde\epsilon(n/N)g'(n)\gamma_k\Gamma)\right)\in\mc S_{X_k,\F_k},
\]
which proves \eqref{eq:f'_k=f_k} and finishes the proof.
\end{proof}
The following is a version of \cite[Corollary E.12]{GreenTaoZiegler12(Ud)}:
\begin{lem}\label{lem:Cd Ud Sd-1}
Let $d\ge 2$ and $\delta>0$. Let $X=(G/\Gamma,G_\N)$ be a filtered nilmanifold of degree $d$. Let $\F\subset C^0_*(X;\C)$ be compact. There exist a filtered nilmanifold $\tilde X=(\tilde G/\tilde\Gamma,\tilde G_\N)$ of degree $(d-1)$ and a compact set $\tilde \F\subset C^0(\tilde X;\C)$ satisfying the following:

Let $N\in\N$ and $f\in\mc S_{X,\F}$. Assume
\begin{equation}\label{eq:Ud!=0}
\norm{f}_{U^d([N])}\ge\delta.
\end{equation}
Then, there exists $\tilde f\in\mc S_{\tilde X,\tilde\F}$ such that $f(n)=\tilde f(n)$ holds for $n\in[N]$.
\end{lem}
\begin{proof}
By Proposition \ref{prop:SXF inverse Ud}, there exist a filtered nilmanifold $X_\delta$ of degree $(d-1)$ and a compact set $\F_\delta\subset C^0(X_\delta;\C)$ such that for $N\in\N$ and $f:[N]\rightarrow\D$ such that $\norm{f}_{U^d}\ge\delta$, there exists $f'\in\mc S_{X_\delta,\F_\delta}$ such that
\begin{equation}\label{eq:Ud*>d'}
\left|\E_{n\in[N]}f(n)f'(n)\right|\gtrsim_\delta1.
\end{equation}
Let $M=\sup_{F\in\F}\norm{F}_{C^0(X_\delta)}$. Up to replacing $\F_\delta$ by $\{2M\}\cup\{F-2M:F\in \F_\delta\}$, where \eqref{eq:Ud*>d'} stays true by a triangle inequality, we may assume $|F|\ge M$ for $F\in\F_\delta$. Denote $\F_\delta^{-1}:=\{F^{-1}:F\in\F_\delta\}\subset C^0(X_\delta;\C)$, which is compact. Then, $1/f'\in\mc S_{X_\delta,\F_\delta^{-1}}$ holds.
By Lemma \ref{lem:Cd E Sd-1}, there exist $X'$ and $\F'$ such that $f(n)f'(n),n\in[N]$ extends to a member $f^0\in\mc S_{X',\F'}$. Then $f$ extends to
\[
\tilde f=f^0/f'\in\mc S_{X'\times X_\delta,\F'\otimes\F_\delta^{-1}},
\]
finishing the proof.
\end{proof}
So far we have recalled previous results. Below we introduce a new inductive degree-lowering principle and a class of norms enjoying such property.
\begin{defn}\label{defn:alt-stable}
A sequence $\{\mc N_N\}_{N\in\N}$ of norms on functions $f:[N]\rightarrow\C$ is \emph{alt-stable} if the following are satisfied:
\begin{itemize}
\item $\norm{f}_{\mc N_N}\les\norm{f}_{\ell^\infty}$ holds.
\item For $\epsilon>0$, $N\gg_\epsilon1$, and $f:[N]\rightarrow\D$ such that
$\norm{f}_{\mc N_N}\ge\epsilon$,
\begin{equation}\label{eq:alt-stable}
\E_{\eta\in[2N]}\norm{\Alt_{\eta}f}_{\mc N_N}\gtrsim_\epsilon 1.
\end{equation}
\end{itemize}
\end{defn}
\begin{defn}\label{defn:reducible}
Let $d\ge d_0+1$ be positive integers. A sequence $\{\mc N_N\}_{N\in\N}$ of norms on functions $f:[N]\rightarrow\C$ is \emph{$(d,d_0)$-*-reducible} if for every $\epsilon>0$, there exists $\epsilon'=\epsilon'(\epsilon)>0$ satisfying the following:

Let $X$ be any filtered nilmanifold of degree at most $d$ and $\F\subset C^0_*(X;\D)$ be compact. For $N\gg_{\epsilon,X,\F}1$ and $f\in\mc S_{X,\F}$ such that $\norm{f}_{\mc N_N}\ge\epsilon$,
\begin{equation}\label{eq:reduced Ud+1 large}
\norm{f}_{U^{d_0+1}([N])}\ge\epsilon'.
\end{equation}
$\{\mc N_N\}_{N\in\N}$ is \emph{$(d,d_0)$-reducible} if above holds for arbitrary compact set $\F\subset C^0(X;\D)$.
\end{defn}
\begin{lem}\label{lem:d-1->d*}
Let $d\ge d_0+2$ be positive integers. Let $\{\mc N_N\}_{N\in\N}$ be alt-stable and $(d-1,d_0)$-reducible. Then, $\{\mc N_N\}_{N\in\N}$ is $(d,d_0)$-*-reducible.
\end{lem}
\begin{proof}
Let $\epsilon>0$. Let $X=(G/\Gamma,G_\N)$ be a filtered nilmanifold of degree $d$. Let $\F\subset C^0_*(X;\D)$ be compact. Let $N\in\N$ and $f\in\mc S_{X,\F}$ be such that
\[
\norm{f}_{\mc N_N}\ge\epsilon.
\]
Then, since $\{\mc N_N\}$ is alt-stable, we have
\begin{equation}\label{eq:many eta}
\#\{\eta\in[2N]:\norm{\Alt_\eta f}_{\mc N_N}\gtrsim_\epsilon1\}\gtrsim_\epsilon N.
\end{equation}
Let $\tilde X$ and $\tilde\F$ be as in Lemma \ref{lem:SXF Alt reduce}. Then, for each $\eta\in[2N]$, we have $\Alt_\eta f\in\mc S_{\tilde X,\tilde\F}$. Thus, by the assumption that $\{\mc N_N\}$ is $(d-1,d_0)$-reducible, we can rewrite \eqref{eq:many eta} as
\begin{equation}\label{eq:many eta'}
\#\{\eta\in[2N]:\norm{\Alt_\eta f}_{U^{d_0+1}}\gtrsim_\epsilon1\}\gtrsim_\epsilon N.
\end{equation}
By \eqref{eq:many eta'} and \eqref{eq:Uk def inductive}, $\norm{f}_{U^{d_0+2}}\gtrsim_\epsilon1$ holds. By \eqref{eq:Ud monotone} and $d\ge d_0+2$, we have $\norm{f}_{U^d}\gtrsim_\epsilon1$. Now using Lemma \ref{lem:Cd Ud Sd-1} and the $(d-1,d_0)$-reducibility of $\{\mc N_N\}$, we have $\norm{f}_{U^{d_0+1}}\gtrsim_\epsilon1$ for $N\gg1$, finishing the proof.
\end{proof}
\begin{lem}\label{lem:d*->d}
Let $d\ge d_0+1$ be positive integers. Let $\{\mc N_N\}_{N\in\N}$ be $(d,d_0)$-*-reducible. Then, $\{\mc N_N\}$ is $(d,d_0)$-reducible.
\end{lem}
\begin{proof}
We start with setting parameters. Let $\epsilon>0$. Since $\{\mc N_N\}$ is $(d,d_0)$-*-reducible, there exists $\epsilon'=\epsilon'(\frac\epsilon{10})>0$ as in Definition \ref{defn:reducible}.
Since $\{\mc N_N\}$ is weaker than $\ell^\infty$, there exists $\epsilon_\infty>0$ such that for $f:[N]\rightarrow\epsilon_\infty\D$,
\begin{equation}\label{eq:eps' def}
\norm{f}_{\mc N_N}\le\frac{\epsilon}{10}\text{ and }\norm{f}_{U^{d_0+1}}\le\frac{\epsilon'}{10}.
\end{equation}
Let $X$ be a filtered nilmanifold of degree $d$ and $\F\subset C^0(X;\D)$ be compact. By Lemma \ref{lem:SXF Fourier decomp}, there exist a compact set $\F_*\subset C^0_*(X;\D)$ and $J=J_{\epsilon_\infty,X,\F}\in\N$ such that for $f\in\mc S_{X,\F}$, there exist $f_1,\ldots,f_J\in\mc S_{X,\F_*}$ such that
\begin{equation}\label{eq:f-sum fj}
\norm{f-\sum_{j\le J}f_j}_{\ell^\infty(\Z)}\le\epsilon_\infty.
\end{equation}
By \eqref{eq:eps' def} and \eqref{eq:f-sum fj}, we have
\begin{equation}\label{eq:eps' conseq}
\norm{f-\sum_{j\le J}f_j}_{\mc N_N}\le\frac{\epsilon}{10}\text{ and }\norm{f-\sum_{j\le J}f_j}_{U^{d_0+1}([N])}\le\frac{\epsilon'}{10}.
\end{equation}
Since $\{\mc N_N\}$ is $(d,d_0)$-*-reducible, there exists $\epsilon''=\epsilon''(\frac{\epsilon}{10J})$ as in Definition \ref{defn:reducible}.

By Lemma \ref{lem:Cd Ud Sd-1} and \eqref{eq:Ud monotone}, there exist a filtered nilmanifold $\tilde X$ of degree $(d-1)$ and a compact set $\tilde\F\subset C^0(\tilde X;\C)$ such that for every $f_*\in\mc S_{X,\F_*}$ satisfying
\[
\norm{f_*}_{U^{d_0+1}([N])}\ge\min\left\{\epsilon'',\tfrac{\epsilon'}{10J}\right\},
\]
$f_*\mid_{[N]}$ extends to a member of $\mc S_{\tilde X,\tilde\F}$.
By Lemma \ref{lem:SXF linear}, there exists a compact set $\tilde\F_\Sigma\subset C^0(\tilde X^J;\C)$ such that any sum of at most $J$ members of $\mc S_{\tilde X,\tilde\F}$ lies in $\mc S_{\tilde X^J,\tilde\F_{\Sigma}}$.

Now we start our proof. Let $f\in\mc S_{X,\F}$ and $N\gg_{X,\F,\epsilon}1$ be such that $\norm{f}_{\mc N_N}\ge\epsilon$. Let $f_1,\ldots,f_J\in\mc S_{X,\F_*}$ be as in \eqref{eq:f-sum fj}. Let
\[
\mc J:=\left\{j\in\{1,\ldots,J\}:\norm{f_j}_{U^{d_0+1}([N])}\ge\min\left\{\epsilon'',\tfrac{\epsilon'}{10J}\right\}\right\}.
\]
By the definition of $\epsilon''$, we have
\begin{equation}\label{eq:fjNN small}
\norm{f_j}_{\mc N_N}<\frac{\epsilon}{10J},\qquad j\notin\mc J.
\end{equation}
By the triangle inequality on $\norm{f}_{\mc N_N}\ge\epsilon$, \eqref{eq:fjNN small}, and \eqref{eq:eps' conseq}, we have
\[
\norm{\sum_{j\in\mc J}f_j}_{\mc N_N}
\ge\epsilon-J\cdot\frac{\epsilon}{10J}-\frac{\epsilon}{10}\ge\frac{\epsilon}{10}.
\]
Here, since $\sum_{j\in\mc J}f_j\in\mc S_{\tilde X^J,\tilde\F_\Sigma}$ and $\tilde X^J$ is of degree $(d-1)$, by the definition of $\epsilon'$, we have
\[
\norm{\sum_{j\in\mc J}f_j}_{U^{d_0+1}}\ge\epsilon'.
\]
Thus, by the triangle inequality and \eqref{eq:eps' conseq}, we have
\begin{align*}
\norm{f}_{U^{d_0+1}}&\ge\norm{\sum_{j\in\mc J}f_j}_{U^{d_0+1}}-\norm{\sum_{j\notin\mc J}f_j}_{U^{d_0+1}}-\norm{f-\sum_{j\le J}f_j}_{U^{d_0+1}}
\\&\ge\epsilon'-J\cdot\frac{\epsilon'}{10J}-\frac{\epsilon'}{10}\ge\frac{\epsilon'}{10}.
\end{align*}
Since $\epsilon'$ depends only on $\epsilon$, this finishes the proof.
\end{proof}
\begin{thm}\label{thm:degree reduction}
Let $d_0\in\N$. Let $\{\mc N_N\}_{N\in\N}$ be an alt-stable sequence of norms. If $\{\mc N_N\}_{N\in\N}$ is $(d_0+1,d_0)$-*-reducible, then it is $(d,d_0)$-reducible for every $d\in\N$.
\end{thm}
\begin{proof}
By Lemma \ref{lem:d*->d}, $(d_0+1,d_0)$-reducibility holds. Then, by Lemma \ref{lem:d-1->d*}, $(d_0+2,d_0)$-*-reducibility also holds. Iterating this process finishes the proof.
\end{proof}
For adaptation to the two-dimensional setting in Theorem \ref{thm:inverse L4 Stri}, we recall a technique that enables us to identify Gowers norms on $[N]^2$ and $[N^2]$.
\begin{defn}\label{defn:dimension map}
For $N\in\N$, denote $\tilde N=2^9N$. $\varphi_N:[N]^2\rightarrow[N]+\tilde N[N]$ denotes the map
\[
\varphi_N(n_1,n_2)=n_1+\tilde N n_2.
\]
For $N\in\N$ and $g:[N]^2\rightarrow\C$, $\iota_Ng:[N+\tilde NN]\rightarrow\C$ denotes the map
\[
\iota_Ng(z):=
\begin{cases}
g(\varphi_N^{-1}(z))&,\qquad z\in[N]+\tilde N[N]
\\
0&,\qquad \text{otherwise}.
\end{cases}
\]
\end{defn}
The multiplier $2^9=2^{7+2}$ is to avoid overlaps between copies of $[N]$ in $\Alt$-calculations in Gowers norms up to $U^7$ (which is the highest Gowers norm used throughout this paper).
For $d\le 6$ and $g:[N]^2\rightarrow\C$, one can easily check
\begin{equation}\label{eq:transfer N^d}
\norm{g}_{U^{d+1}([N]^2)}\sim\norm{\iota_Ng}_{U^{d+1}([N+\tilde NN])}.
\end{equation}
\begin{lem}\label{lem:Ud pullback}
Let $N\in\N$ and $d\le 6$. For $\epsilon>0$ and $f:[N+\tilde NN]\rightarrow\D$ such that
\[
\norm{f\circ\varphi_N}_{U^{d+1}([N]^2)}\ge\epsilon,
\]
we have
\[
\norm{f}_{U^{d+1}([N+\tilde NN])}\gtrsim_\epsilon1.
\]
\end{lem}
\begin{proof}
Let $\{F_m\}$ be a sequence of continuous functions $F_m:X_{2^9\T}\rightarrow[0,1]$ converging uniformly to $\chi_{[-1,1]+2^9\Z}$ outside any neighborhood of $\{\pm1
+2^9\Z\}$. 
Since $[N]+\tilde N[N]=[N+\tilde NN]\cap([N]+2^9N\Z)$, we have $\norm{f\chi_{[N]+\tilde N[N]}-fF_m(\cdot/N)}_{U^{d+1}([N+\tilde NN])}\rightarrow o_m(1)$ as $N\rightarrow\infty$. Thus, choosing $m$ big enough, for $N\gg1$, by the triangle inequality, we have $\norm{f F_m(\cdot/N)}_{U^{d+1}([N+\tilde N N])}\gtrsim\epsilon$. Then, by Proposition \ref{prop:SXF inverse Ud}, there exist a filtered nilmanifold $X$ of degree $d$ and compact $\F\subset C^0(X;\C)$, depending only on $\epsilon$, and $f'\in\mc S_{X,\F}$ such that
\[
|\E_{n\in[N+\tilde NN]}f(n)F_m(n/N)\ol{f'(n)}|\gtrsim_\epsilon1.
\]
Now since $F_m(\cdot/N)\in\mc S_{X_{2^9\T};\{F_m\}}$, by Lemma \ref{lem:SXF linear}, $\ol{F_m(\cdot/N)}f'(\cdot)\in\mc S_{X_{2^9\T}\times X;\{\ol{F_m}\}\otimes\F}$ holds. By Proposition \ref{prop:SXF anti-inverse Ud}, the proof finishes.
\end{proof}
\begin{lem}\label{lem:S^2=sum of LP}
Let $\epsilon>0$. Let $X$ be a filtered nilmanifold of degree $2$ and $\F\subset C^0(X;\D)$ be compact. Then, there exist $r,J\in\N$ satisfying the following:

For $N\in\N$ and $f\in\mc S_{X,\F}$, there exist $c_1,\ldots,c_J\in\D$ and locally quadratic modulations $\phi_1,\ldots,\phi_J$ supported on affine Bohr sets of ranks at most $r$ such that
\begin{equation}\label{eq:S^2=sum of LP, h def}
h:=\sum_{j\le J} c_j\phi_j
\end{equation}
satisfies
\[
\norm{h}_{\ell^\infty}\le 1
\]
and
\[
\norm{f\circ\varphi_N-h}_{\ell^2([N]^2)}<\epsilon N.
\]
\end{lem}
\begin{proof}
Let $\mc D_X\subset C^0(X;\C)$ be the dense set in Proposition \ref{prop:SXF anti-inverse Ud}. Since $\F$ is compact, up to a small perturbation, it suffices to show this lemma when $\F$ is a finite subset of $\mc D_X$. Up to taking maxima of $r$ and $J$ over $F\in\F$, we assume $\F$ is a singleton $\F=\{F\}$. For any $g:[N]^2\rightarrow\D$ such that
\[
\left|\jp{g,f\circ\varphi_N}_{\ell^2([N]^2)}\right|\ge\frac{\epsilon^2}{2}N^2,
\]
we have
\[
\left|\jp{\iota_Ng,f}_{\ell^2([N+\tilde NN])}\right|\ge\frac{\epsilon^2}{2}N^2
\]
and thus $\norm{\iota_Ng}_{U^3}\gtrsim_{\epsilon,X,F}1$ holds.
Thus, by \eqref{eq:transfer N^d}, $\norm{g}_{U^3}\ge\delta$ holds for some $\delta=\delta(\epsilon,X,F)>0$.
By Proposition \ref{prop:profile U3 Bohr}, there exists $h$ as in \eqref{eq:S^2=sum of LP, h def} such that
\[
\norm{h}_{\ell^\infty}\le 1,
\]
\[
\norm{f\circ\varphi_N-h}_{U^3([N]^2)}<\delta,
\]
and
\begin{equation}\label{eq:f-h,h<e/2}
\left|\jp{f\circ\varphi_N-h,h}_{\ell^2([N]^2)}\right|<\frac{\epsilon^2}{2}N^2.
\end{equation}
Plugging $g=f\circ\varphi_N-h$, since $\norm{g}_{U^3([N]^2)}<\delta$, we have
\begin{equation}\label{eq:f-h,f<e/2}
\left|\jp{f\circ\varphi_N-h,f\circ\varphi_N}_{\ell^2([N]^2)}\right|<\frac{\epsilon^2}{2}N^2.
\end{equation}
By the triangle inequality, \eqref{eq:f-h,f<e/2}, and \eqref{eq:f-h,h<e/2}, the proof finishes.
\end{proof}
\section{Norms and inverse theorems concerning rectangular resonances}\label{sec:norms-rec}
\subsection{Norms concerning tensor products}
In this subsection, we focus on structures of functions positively correlated to tensor products of bounded functions. This object naturally appears from the resonance consideration in Section \ref{subsec:analytic}. In particular, Lemma \ref{lem:two tensor => U3} plays a key role relating the rectangular resonance and the Gowers uniformity.

For a set $S\neq\emptyset$ and functions $g,h:S\rightarrow\C$, we denote by $g\otimes h:S\times S\rightarrow\C$ the function $(g\otimes h)(x,y):=g(x)h(y)$.

For functions $f_{jk}:\Z^2\rightarrow\C$, $j,k=1,2$, we denote
\[
\Pi(f_{11},f_{12},f_{21},f_{22}):=\sum_{x_1,x_2,y_1,y_2\in\Z}f_{11}(x_1,y_1)\overline{f_{12}(x_1,y_2)f_{21}(x_2,y_1)}f_{22}(x_2,y_2).
\]
Correspondingly, for $f:\Z^2\rightarrow\C$, we define the norm
\[
\norm{f}_\Pi:=\Pi(f,f,f,f)^{1/4}.
\]
That $\norm{\cdot}_\Pi$ is indeed a norm can be shown by a conventional argument introduced, e.g., in \cite[p419-420]{tao2006additive}. We provide the proof for completeness. By the Cauchy-Schwarz inequality, for $f_{jk}:\Z^2\rightarrow\C$, $j,k=1,2$, we have
\begin{equation}\label{eq:f11f12f21f22}
\left|\Pi(f_{11},f_{12},f_{21},f_{22})\right|\le\prod_{j=1,2}\Pi(f_{j1},f_{j2},f_{j1},f_{j2})^{1/2}\le\prod_{j,k=1,2}\Pi(f_{jk},f_{jk},f_{jk},f_{jk})^{1/4},
\end{equation}
thus for $f,g:\Z^2\rightarrow\C$, we have the estimate
\[\norm{f+g}_{\Pi}^4=\Pi(f+g,f+g,f+g,f+g)
\le\sum_{k=0}^4\binom{4}{k}\norm{f}_\Pi^k\norm{g}_\Pi^{4-k}
\le(\norm{f}_\Pi+\norm{g}_\Pi)^4.\]
\begin{lem}\label{lem:tensor => q~1}
For any $N\in\N$ and functions $f:[N]^2\rightarrow\C$ and $g,h:[N]\rightarrow\D$, we have
\[
\left|\jp{f,g\otimes h}_{\ell^2(\Z^2)}\right|\les N\norm f_\Pi.
\]
\end{lem}
\begin{proof}
In \eqref{eq:f11f12f21f22}, we set $f_{11}=f$, $f_{12}=g\otimes\chi_{\{0\}}$, $f_{21}=\chi_{\{0\}}\otimes h$, and $f_{22}=\chi_{\{0\}}$. Then, $\norm{f_{12}}_\Pi,\norm{f_{21}}_\Pi\les N^{1/2}$ and $\norm{f_{22}}_\Pi=1$ hold. Now since $\Pi(f_{11},f_{12},f_{21},f_{22})=\jp{f,g\otimes h}_{\ell^2}$, the proof finishes.
\end{proof}
\begin{lem}\label{lem:q~1 => tensor}
For any function $f:\Z^2\rightarrow\D$, there exist $g,h:\Z\rightarrow\D$ such that
\[
\#\supp(f)\cdot\left|\jp{f,g\otimes h}_{\ell^2(\Z^2)}\right|\ge\norm f_\Pi^4.
\]
\end{lem}
\begin{proof}
By the pigeonhole principle, there exists $z_0\in \Z^2$ such that
\[
\#\supp(f)\cdot\left|\sum_{m,n\in\Z}\Alt_{me_1,ne_2}f(z_0)\right|\ge\sum_{z\in\Z^2}\sum_{m,n\in\Z}\Alt_{me_1,ne_2}f(z)=\norm f_\Pi^4.
\]
Up to a translation, we assume $z_0=0$. Set $g(x):=f(xe_1)$ and $h(y):=f(ye_2)$. We have
\[
\left|\sum_{m,n\in\Z}\Alt_{me_1,ne_2}f(0)\right|=\left|f(0)\sum_{m,n\in\Z}\overline{g(m)h(n)}f(me_1+ne_2)\right|\le\left|\jp{g\otimes h,f}_{\ell^2(\Z^2)}\right|,
\]
finishing the proof.
\end{proof}
We also use rotated versions of $\norm{\cdot}_\Pi$. For $\eta\in\Z^2\setminus\{0\}$, we denote
\begin{equation}\label{eq:Pi_eta}
\norm{f}_{\Pi_\eta}^4:=\sum_{\xi\in\Z^2/\eta\Z^2}\norm {f(\eta \cdot+\xi)}_{\Pi}^4=\sum_{m,n\in\Z}\sum_{z\in\Z^2}\Alt_{m\eta,n\eta^\perp}f(z).
\end{equation}
Throughout this paper, we use the convention of identifying a coset $H\in\Z^2/\eta\Z^2$ with its representative $\xi\in([0,1)\eta+[0,1)\eta^\perp)\cap\Z^2$.

Immediately from the definition of $\Pi_\eta$-norm, we have the following identity:
\begin{lem}\label{lem:Pi_eta1 eta2}
For $f:\Z^2\rightarrow\C$ and $\eta_1,\eta_2\in\Z^2\setminus\{0\}$, denoting by $\eta:=\eta_1\eta_2$ the product in $\Z[i]\simeq\Z^2$, we have
\begin{equation}\label{eq:Pi_eta1 eta2}
\norm f_{\Pi_{\eta}}^4=\sum_{\xi\in\Z^2/\eta_1\Z^2}\norm {f(\eta_1\cdot+\xi)}_{\Pi_{\eta_2}}^4.
\end{equation}
\end{lem}
\begin{lem}\label{lem:two tensor => U3}
Let $\eta=(\eta_1,\eta_2)\in\Z^2$ with $\eta_1,\eta_2\neq 0$ and $N\in\N$. For functions $g,h:[N]\rightarrow\C$, we have
\[
\norm{g\otimes h}_{\Pi_\eta}\les_\eta N\norm g_{U^3}\norm h_{U^3}.
\]
\end{lem}
\begin{proof}
We have
\begin{align*}
\norm{g\otimes h}_{\Pi_\eta}^4&=\sum_{m,n\in\Z}\sum_{x,y\in\Z}\Alt_{m\eta,n\eta^\perp}(g\otimes h)(x,y)
\\&
=\sum_{m,n\in\Z}\left(\sum_{x\in \Z}\Alt_{m\eta_1,-n\eta_2}g(x)\cdot\sum_{y\in \Z}\Alt_{m\eta_2,n\eta_1}h(y)\right)
\end{align*}
and applying the Cauchy-Schwarz inequality gives
\begin{align*}
&\le\left(\sum_{m,n\in \Z}\left|\sum_{x\in \Z}\Alt_{m\eta_1,n\eta_2}g(x)\right|^2\cdot \sum_{m,n\in \Z}\left|\sum_{y\in \Z}\Alt_{m\eta_2,n\eta_1}h(y)\right|^2\right)^{1/2}.
\end{align*}
Now using that
\begin{equation}\label{eq:two tensor => U3 claim}
\sum_{m,n\in \Z}\left|\sum_{x\in \Z}\Alt_{m\eta_1,n\eta_2}g(x)\right|^2\le\sum_{m,n\in \Z}\left|\sum_{x\in \Z}\Alt_{m,n}g(x)\right|^2\les N^4\norm g_{U^3}^8,
\end{equation}
the proof finishes.
\end{proof}
\subsection{An inverse theorem on the rectangle resonance}
In this subsection, we prove Lemma \ref{lem:U3>>Q(x) main lem}. This enables the final reduction in the proof of Theorem \ref{thm:inverse L4 Stri} in Section \ref{subsec:analytic}.
\begin{defn}
Let $r\in\N$. Let $\mc B$ be an affine Bohr set in $[N]^2$, $N\in\N$. $\mc B$ has \emph{rotation-symmetric rank at most $2r$} if $\mc B$ can be represented as an affine Bohr set of rank $2r$ in the form of Definition \ref{defn:affine Bohr sets} with the rotational symmetry
\[
(\theta_1,\ldots,\theta_{2r})=(\theta_1,\ldots,\theta_r,\theta_1^\perp,\ldots,\theta_r^\perp).
\]
\end{defn}
\begin{lem}\label{lem:Bohr can be made rotation-sym}
Let $r\in\N$. Let $\mc B$ be an affine Bohr set of rank $r$ in $[N]^2$, $N\in\N$. Then, $\mc B$ has rotation-symmetric rank at most $2r$.
\end{lem}
\begin{proof}
Let $\theta_1,\ldots,\theta_r$ and $I_1,\ldots,I_r$ be as in Definition \ref{defn:affine Bohr sets}. Since $\mc B$ is finite, we can set $I_{r+1},\ldots,I_{2r}\subset(-1,1)$ as intervals of lengths slightly less than $1$ such that $\theta_j^\perp\cdot x\in I_{r+j}+\Z$ holds for every $x\in\mc B$. Since $(\theta_1,\ldots,\theta_r,\theta_1^\perp,\ldots,\theta_r^\perp),(I_1,\ldots,I_{2r})$ leads to the same $\mc B$, the proof finishes.
\end{proof}
For $d,Q\in\N$, a subspace or a lattice in $\R^d$ is \emph{$Q$-rational} if it is generated by members of $[Q]^d$.
\begin{lem}\label{lem:U3>>Q(x) main lem}
Let $\{a_N\}$ be a sequence of integers such that $a_N\rightarrow\infty$.
Let $\{\mc N_N\}_{N\in\N}$ be a sequence of norms for $f:[N]^2\rightarrow\C$ satisfying the following:
\begin{itemize}
    \item
We have
\begin{equation}\label{eq:N weaker than l2}
\norm{f}_{\mc N_N}\le\norm{f}_{\ell^2([N]^2)}/N.
\end{equation}
    \item
For $\epsilon'>0$, $N\gg_{\epsilon'}1$, an integer $a_*\le a_N$, and $f:[N]^2\rightarrow\D$ such that $\norm{f}_{\mc N_N}\ge\epsilon'$, there exist positive integers $a\sim_{\epsilon'} a_*$ and $b=O_{\epsilon'}(1)$ such that
\begin{equation}\label{eq:U3>>Q(x) main x,eta}
\left|\sum_{x\in\Z^2}\sum_{\eta\in b\Z^2}\Alt_{\eta,a\eta^\perp}f(x)\right|\gtrsim_{\epsilon'}\frac{N^4}{a^2}.
\end{equation}
\end{itemize}
Let $\epsilon>0$ and $r,m_*\in\N$. Let $\mc B_*$ be an affine Bohr set in $[N]^2$, $N\gg_{\epsilon,r,\{a_N\}}1$ of rotation-symmetric rank at most $2r$. Let $\phi$ be a locally quadratic modulation supported on $\mc B_*$. Assume there exists $\mc B\in\pi_{m_*}(\mc B_*)$ such that
\[
\norm{\chi_{\mc B}\phi}_{\mc N_N}\ge\epsilon.
\]
Then, for every $\delta>0$, there exist an integer $m=O_{r,\epsilon,\delta}(1)$, $\mc B'\in\pi_m(\mc B)$, a locally linear modulation $\psi$ supported on $\mc B'$, and a locally quadratic modulation $\tilde\phi$ supported on an affine Bohr set $\tilde{\mc B}\supset\mc B'$ of rotation-symmetric rank at most $2(r-1)$, such that
\begin{equation}\label{eq:U3>>Q(x) main goal}
|\phi(x)-\psi(x)\tilde\phi(x)|\le\delta,\qquad x\in\mc B'
\end{equation}
and
\begin{equation}\label{eq:U3>>Q(x) keep positive}
\norm{\chi_{\mc B'}\phi}_{\mc N_N}\gtrsim_{\epsilon,r,\delta}1.
\end{equation}
\end{lem}
\begin{proof}
Throughout this proof, every comparability depends in default on the sequence $\{\mc N_N\}$ and the parameters $r,\epsilon,\delta$; we keep track of dependencies only on $a$-parameters to appear within this proof. We denote $\mc N=\mc N_N$ for simplicity. We use \eqref{eq:U3>>Q(x) main x,eta} for $a=a_0,a_1$, where $a_0\gg1$ and $a_1\gg_{a_0}1$ are $O(1)$-integers to be fixed later. Assuming $N\gg_{a_0,a_1}1$, such choices of $a_0$ and $a_1$ are available by the assumption of this lemma.

Since $\mc N_N$ is a norm, by the pigeonhole principle, for any integer $m_1=O(1)$, there exists $\mc B_1\in\pi_{m_1}(\mc B)$ such that $\norm{\chi_{\mc B_1}\phi}_{\mc N_N}\gtrsim 1$ holds. This process will be referred to as \emph{passing to an $m_1$-partition}.
Such passing will be repeated at most $100$ times throughout this proof; once we show \eqref{eq:U3>>Q(x) main goal} with $\mc B'=\mc B_k$, $k\le 100$, where $\mc B_j$ is a member of the $m_j=O(1)$-partition of $\mc B_{j-1}$ for $j=1,\ldots,k$, then since $\mc B_k$ is an $(m_1\cdots m_k)$-partition member of $\mc B$, the proof finishes. In this sense, we freely pass to an $O(1)$-partition in this proof. Hereafter we denote the partition member $\mc B_k$ considered at each step of the proof by $\mc B$ for convenience, omitting the subscript.

We use notations from Definition \ref{defn:affine Bohr sets} and Remark \ref{rem:Bohr set by lattice-convex proj} throughout this proof. In particular, we write $\mc B=\pi_{\R^2}(\Omega)$ with $\Omega=\ul\Omega\cap Z$, where $Z\subset\Z^{2+2r}$ is a translate of $m_*\Z^{2+2r}$. Note that passing to an $m$-partition updates $m_*$ to $mm_*$. Up to translations by integer points, we assume $\theta_1,\ldots,\theta_{2r}$ are $O(1)$.

By pigeonholing on the $x$-variable of \eqref{eq:U3>>Q(x) main x,eta}, there exists $x_0$ such that
\begin{equation}\label{eq:U3>>Q(x) main lem assump}
\left|\sum_{\eta\in b\Z^2}\Alt_{\eta,a\eta^\perp}f(x_0)\right|\gtrsim\frac{N^2}{a^2}.
\end{equation}
Hereafter we assume $x_0=0$ for simplicity; this proof focuses on the quadratic part and is uninfluenced by translating $x_0$. As a consequence, $0\in\Omega\subset Z$ holds, i.e., $Z=m_*\Z^{2+2r}$.

Firstly, we consider the case $|I_j|\ll 1$ for some $j=1,\ldots,2r$. Without loss of generality, assume $|I_1|\ll1$ holds. Since $\#\mc B\gtrsim N^2$, by Lemma \ref{lem:Weyl bound implies Lipschitz}, we have $\dist(\theta_1,\frac{1}{m}\Z^2)\les1/N$ for some $m=O(1)$. Since $\theta_{r+1}=\theta_1^\perp$, similar does the $(r+1)$-th coordinate. Thus, passing to an $O(1)$-partition of $\mc B$, up to translating $\theta_1$ and $\theta_{r+1}$ by integers, $\Omega$ has fixed first and $(r+1)$th coordinates. Passing to an $O(1)$-partition $\Omega$, $\Omega$ is thickly contained in an affine subspace and we can set $F$ as in Remark \ref{rem:Bohr set by lattice-convex proj}. Keeping $F\mid_\Omega$ fixed, we can further assume $F$ is invariant of first and $(r+1)$th coordinates. Up to a $2$-partition, assume $|I_1|,\ldots,|I_{2r}|<1/2$. Let $\Pi:\R^{2+2r}\rightarrow\R^{2+2(r-1)}$ be the canonical projection annihilating the first and $(r+1)$th coordinates. Let $\tilde\pi_{\R^2}:\R^{2+2(r-1)}\rightarrow\R^2$ be the canonical projection, so that $\tilde\pi_{\R^2}\circ\Pi=\pi_{\R^2}$. Set $\tilde{\mc B}:=\tilde\pi_{\R^2}(\tilde\Omega)$, where $\tilde\Omega:=\Pi(\ul\Omega)\cap\Z^{2+2(r-1)}$. Then, $\tilde{\mc B}$ is an affine Bohr set of rotation-symmetric rank at most $2(r-1)$. Let $\tilde F=F(0_{1,r+1},\cdot):\R^{2+2(r-1)}\rightarrow\R$. 
By Lemma \ref{lem:descend F to phi}, $e^{i\tilde F}:\tilde\Omega\rightarrow\T$ descends to a locally quadratic modulation $\tilde\phi$ on $\tilde{\mc B}$. For $x\in\mc B$, let $u\in\Omega$ be such that $x=\pi_{\R^2}(u)$. We have $\Pi(u)\in\Pi(\Omega)=\Pi(\ul\Omega\cap Z)\subset\Pi(\ul\Omega)\cap\Z^{2+2(r-1)}\subset\tilde\Omega$, thus denoting $v=(0_{1,r+1},\Pi(u))$, we have $\tilde\phi(x)=e^{iF(v)}$.
Here, since $\Pi(u)=\Pi(v)$ and (by the first and $(r+1)$th coordinate-invariance of $F$)
\begin{equation}\label{eq:goal for steps 1 and 3}
e^{iF(u)}=e^{iF(v)},\qquad u,v\in Z\text{ such that }\Pi(u)=\Pi(v),
\end{equation}
$\phi(x)=e^{iF(u)}=e^{iF(v)}=\tilde\phi(x)$ holds and we are done for this case with the trivial choice $\psi=\chi_{\mc B}$.

Hereafter we consider the case $|I_j|\sim1$ for every $j=1,\ldots,2r$.
For $w=(x,u,v)\in\R^2\times\R^r\times\R^r$, we denote the \emph{rotation}
\[
w^\perp:=(x^\perp,-v,u).
\]
Call a set $S\subset\R^2\times\R^r\times\R^r$ \emph{rotation-invariant} if $S^\perp=\{s^\perp:s\in S\}$ equals $S$.

We show that the problem can be reduced to the case that $\Omega$ is thick in $Z$ (compared to $a_0$ and $a_1$). Assume for contrary that $\Omega$ is relatively thickly contained in a union of $k=O_{a_0,a_1}(1)$ affine translates of a proper subspace $\mc P\lneq\R^{2+2r}$. 
Then, since $Z\subset\Z^{2+2r}$ has bounded index, the comparably enlarged set $\Omega^0=\ul\Omega_{[0,1)^{2r}}\cap \Z^{2+2r}$ is also contained in $O_k(1)$ translates of $\mc P$. Denote by $\lfloor\cdot\rfloor:\R^m\rightarrow\R^m,m\in\N$ the coordinate-wise floor function. The map $x\mapsto(x,\lfloor P(x)\rfloor)$ is $O(1)$-Lipschitz from $[N]^2$ onto $\Omega^0$. Thus, by the pigeonhole principle there exists an $O_k(1)$-bounded $v_1\in\mc P\cap \Z^{2+2r}$; modding out by $v_1\R$ and repeating yields a spanning set of $\mc P$ by $O_k(1)$-bounded integer points, i.e., $\mc P$ is $O_k(1)$-rational.

Up to a $k$-partition, we assume $\Omega\subset\mc P$. Since $|I_j|\sim1$, $\ul\Omega^\perp$ is contained in a translate of $O(1)$-scaling of $\ul\Omega$. Since thickness is translation and $O(1)$-scaling-invariant, $\mc P$ is rotation-invariant.

Let $\La:=\mc P\cap \Z^{2+2r}$, which is $O_k(1)$-rational and rotation-invariant. Regarding $\La$ as a $\Z[i]$-module defined by $(m+in)\cdot\la:=m\la+n\la^\perp$, $\La$ is a free module.
Thus, there exists a generator $(\la_0,\ldots,\la_{d},\la_0^\perp,\ldots,\la_{d}^\perp)$ for $\La$ as a lattice, where $d=\rank(\La)/2-1\le r-1$ and $\pi_{\R^2}(\la_1)=\cdots=\pi_{\R^2}(\la_{d})=0$.
Here, since the collection of $O_k(1)$-rational $\La$ is $O_k(1)$-finite, by assigning a fixed generator for each $\La$, we can assume $|\la_j|\sim_k1$.
Let $\tilde\pi_{\R^2}:\R^{2+2d}\rightarrow\R^2$ be the canonical projection.
Denote by $\{e_1,\ldots,e_d,e_1^\perp,\ldots,e_d^\perp\}$ the standard basis for $\R^{2d}$.
Let $T:\R^{2+2d}\rightarrow\mc P$ be the linear operator mapping $T(\pi_{\R^2}(\la_0),0)=\la_0$, $T(0,e_j)=\la_j$ for $j=1,\ldots,d$, and $T(u^\perp)=T(u)^\perp$.
Then, $\norm{T},\norm{T^{-1}}\les_k1$, $T(\Z^{2+2d})\supset\La$, and $\pi_{\R^2}\circ T=\tilde\pi_{\R^2}$ hold.

Since $\pi_{\R^2}\mid_{\mc P}$ is surjective, there exists a surjective projection $K:\R^{2+2r}\rightarrow\mc P$ such that 
$\pi_{\R^2}\circ K=\pi_{\R^2}$. Up to replacing by $(K(u)-K(u^\perp)^\perp)/2$, $K(u^\perp)=K(u)^\perp$ further holds.
Here, since $K$ depends only on $\mc P$ and $\mc P$ is $O_k(1)$-rational, we may assume $\norm{K}\les_k1$. Thus, passing to an $O_k(1)$-partition, there exists a translate $\mc C$ of $(-\tfrac{1}{10},\tfrac{1}{10})^{2d}$ such that
\[
T^{-1}\circ K(0\times I_1\times\cdots\times I_{2r})\subset 0\times\mc C,
\]
where $I_1,\ldots,I_{2r}$ are as in Remark \ref{rem:Bohr set by lattice-convex proj}.
Denote $e_0=(1,0)\in\R^2$ and let $p_{e_0}$ and $p_{e_0^\perp}=p_{e_0}^\perp$ be as in Remark \ref{rem:Bohr set by lattice-convex proj}. Let $\tilde p_{e_0}=T^{-1}\circ K(p_{e_0})$ and
\[
\ul{\tilde\Omega}=\tilde p_{e_0}[-N,N]+\tilde p_{e_0}^\perp[-N,N]-0\times\mc C.
\]
Set $\tilde\Omega=\ul{\tilde\Omega}\cap\Z^{2+2d}$. Since $\La\subset\mc P$ and $K$ is a projection onto $\mc P$, we have
\[
\Omega=\ul\Omega\cap\La=(p_{e_0}[-N,N]+p_{e_0}^\perp[-N,N]-0\times I_1\times\cdots\times I_{2r})\cap\La\subset T(\ul{\tilde\Omega}\cap\Z^{2+2d})\subset T(\tilde\Omega),
\]
thus by $\pi_{\R^2}\circ T=\tilde\pi_{\R^2}$,
\[
\mc B=\pi_{\R^2}(\Omega)\subset\tilde\pi_{\R^2}(\tilde\Omega)=:\tilde{\mc B}.
\]
$\tilde{\mc B}$ is an affine Bohr set of rotation-symmetric rank at most $2d\le2(r-1)$. By Lemma \ref{lem:descend F to phi}, $e^{iF\circ T}$ on $\tilde\Omega\supset T^{-1}(\Omega)$ descends to a locally quadratic modulation $\tilde\phi$ supported on $\tilde{\mc B}$, which equals $\phi$ on $\mc B$ and thus satisfies the sharp equality for \eqref{eq:U3>>Q(x) main goal} with the trivial choice $\psi=\chi_{\mc B}$. Hence, we assume $\Omega$ is thick in $Z$.

Let $F$ be as in Remark \ref{rem:Bohr set by lattice-convex proj}. Since \eqref{eq:U3>>Q(x) main x,eta} and \eqref{eq:U3>>Q(x) main lem assump} are invariant under locally linear modulations, up to a change of the locally linear modulation $\psi$, we may assume that $F$ is a quadratic form. Let $B=B_F:\R^{2+2r}\times\R^{2+2r}\rightarrow\R$ be the symmetric bilinear form
\[
B[u,v]:=F(u+v)-F(u)-F(v).
\]
Passing to a $b$-partition, we assume $b\mid m_*$.
Since $\pi_{\R^2}(\Omega)=\mc B$, substituting $\eta=\pi_{\R^2}(u)$ and $a\eta^\perp=\pi_{\R^2}(v)$, one can rewrite \eqref{eq:U3>>Q(x) main lem assump} as
\begin{equation}\label{eq:U3>>Q(x) in B}
\left|\sum_{\substack{u,v\in\Omega\\ u+v\in\Omega\\ au^\perp-v\in(0\times\Z^{2r})\cap Z}}e^{iB[u,v]}\right|\gtrsim\frac{N^2}{a^2}.
\end{equation}
Since $\pi_{\R^2}$ is injective on $\Omega$, for each $u\in\Omega$, there is at most one $v\in\Omega$ such that $au^\perp-v\in0\times\Z^{2r}$. Here, we have
\[
au^\perp-v\in (a\Omega^\perp-\Omega)\cap(0\times\Z^{2r})\subset0\times[2a]^{2r}.
\]
Thus, substituting $v=au^\perp-\la$, \eqref{eq:U3>>Q(x) in B} can be rewritten as
\begin{equation}\label{eq:U3>>Q(x) wrt C_la}
\left|\sum_{\la\in(0\times[2a]^{2r})\cap Z}\sum_{u\in C_\la}e^{iB[u,au^\perp-\la]}\right|\gtrsim\frac{N^2}{a^2},
\end{equation}
where we denoted by $C_\la=\ul C_\la\cap Z$ the lattice-convex set with
\[
\ul C_\la:=\{u\in\ul{\Omega}:au^\perp-\la,u+(au^\perp-\la)\in\ul{\Omega}\}.
\]
We have
\begin{equation}\label{eq:C_la containment}
C_\la\subset-\frac{1}{a}(\ul\Omega+\la)^\perp\cap Z.
\end{equation}
Let $\mc L_a\subset(0\times[2a]^{2r})\cap Z$, $a\in\{a_0,a_1\}$ be the set
\begin{equation}\label{eq:mc L def}
\mc L_a:=\left\{\la\in(0\times[2a]^{2r})\cap Z:\left|\sum_{u\in C_\la}e^{iB[u,au^\perp-\la]}\right|\gtrsim \frac{N^2}{a^{2+2r}}\right\}.
\end{equation}
By \eqref{eq:U3>>Q(x) wrt C_la}, \eqref{eq:C_la containment}, and
\begin{equation}\label{eq:few C_la}
\#\left(-\frac{1}{a}(\ul\Omega+\la)^\perp\cap Z\right)\les\frac{\#\Omega}{a^{2+2r}}\les\frac{N^2}{a^{2+2r}},
\end{equation}
(where we used the thickness of $\Omega\subset Z$,) we have
\begin{equation}\label{eq:many la saturating C_la}
\#\mc L_a\gtrsim a^{2r}.
\end{equation}
By \eqref{eq:mc L def} plugging $a=a_0$ and any $\la\in\mc L_{a_0}$, \eqref{eq:C_la containment}, \eqref{eq:few C_la}, and Corollary \ref{cor:convex Weyl}, we have
\begin{equation}\label{eq:B[u,au perp] <u0}
\dist(a_0B[u,u^\perp],\frac{2\pi}{m}\Z)\les\norm{u}_{\frac1{a_0}(\ul\Omega+\la)^\perp-\frac1{a_0}(\ul\Omega+\la)^\perp}^2\les a_0^2\norm{u}_{\ul\Omega-\ul\Omega}^2,\qquad u\in Z
\end{equation}
for some $m=O(1)$. Passing to an $O_{a_0}(1)$-partition of $\mc B$, we strengthen \eqref{eq:B[u,au perp] <u0} to
\begin{equation}\label{eq:B[u,au perp] <u}
\dist(B[u,u^\perp],8\pi\Z)\les\norm{u}_{\ul\Omega-\ul\Omega}^2,\qquad u\in Z,
\end{equation}
at the cost of allowing dependencies on $a_0$ for all comparabilities hereafter.

Since $\ul\Omega+\ul\Omega^\perp\subset O(1)(\ul\Omega-\ul\Omega)$ holds by $|I_j|\sim1$, by \eqref{eq:B[u,au perp] <u}, for $u\in Z$, we have
\[
|e^{\frac i2(F(u)+F(u^\perp))}-e^{iF(u)}|=|e^{\frac i2(F(u^\perp)-F(u))}-1|=|e^{\frac i4B[u+u^\perp,(u+u^\perp)^\perp]}-1|\les\norm{u}_{\ul{\Omega}-\ul{\Omega}}^2.
\]
Thus, passing to an $O(1)$-partition of $\mc B$, up to a linear modulation $\psi$ and triangle inequalities, we may reduce the problem to the rotationally symmetric case $F(u)=\frac{1}{2}F(u)+\frac{1}{2}F(u^\perp)$. Then, $B[u,v]=B[u^\perp,v^\perp]$ holds.

Next, we make use of the larger parameter $a_1$. Denote by $\{e_1,\ldots,e_r,e_1^\perp,\ldots,e_r^\perp\}$ the standard basis for $\R^{2r}$. Denote $\la_k=(0,m_*e_k)\in Z$ for $k=1,\ldots,r$. Since $B[u,a_1u^\perp]=0$, by \eqref{eq:mc L def}, \eqref{eq:C_la containment}, \eqref{eq:few C_la}, and Corollary \ref{cor:convex Weyl}, we have
\[
\dist(B[u,\la],\frac{2\pi}{m}\Z)\les\norm{u}_{\frac{1}{a_1}(\ul\Omega+\la)^\perp-\frac{1}{a_1}(\ul\Omega+\la)^\perp}\les a_1\norm{u}_{\ul\Omega-\ul\Omega},\qquad u\in Z,\qquad\la\in\mc L_{a_1},
\]
where $m=O(1)$. Thus, by \eqref{eq:many la saturating C_la} and Lemma \ref{lem:Weyl bound implies Lipschitz} plugging $u\in c(\ul\Omega-\ul\Omega)\cap Z$ with sufficiently small number $c\gtrsim1$ (which spans $\R^{2+2r}$ since $\Omega$ is thick), we have
\begin{equation}\label{eq:B[u,la] <u,la}
\dist(B[u,\la_1^\perp],\frac{2\pi}{m}\Z)\les \norm{u}_{\ul\Omega-\ul\Omega},\qquad u\in Z,
\end{equation}
where $m=O(1)$. Passing to an $O(1)$-partition, we assume $m=1$ in \eqref{eq:B[u,la] <u,la}.
Since $\Omega$ is thick, there exists
\[
v_1\in (a_1 Z-\la_1)\cap\frac{1}{a_1}(\ul\Omega-\ul\Omega).
\]
By the rotational symmetry $B[v^\perp,v]=0$, substituting $u=-\tfrac{1}{a_1}(v+\la)^\perp$ into \eqref{eq:U3>>Q(x) in B} yields
\begin{equation}\label{eq:U3>>Q(x) wrt v}
\left|\sum_{\substack{\la\in(0\times[2a_1]^{2r})\cap Z\\ v\in \Omega\cap(a_1 Z-\la)\\ -\frac{1}{a_1}(v+\la)^\perp,v-\frac{1}{a_1}(v+\la)^\perp\in\Omega}}e^{i\frac{1}{a_1}B[v,\la^\perp]}\right|\gtrsim\frac{N^2}{a_1^2}.
\end{equation}
Parametrizing $(\la,v)=(\la_0+k\la_1,v_0+kv_1)$, where $\la_0\in0\times[2a_1]^{2r-1}$ and $k\in[2a_1]$, we have
\begin{align*}
&\#\{(\la,v)+(\la_1,v_1)\Z:\la\in0\times[2a_1]^{2r}\text{ and }v\in\Omega\cap(a_1 Z-\la)\}
\\\le&\#\{(\la_0,v_0):\la_0\in0\times[2a_1]^{2r-1}\text{ and }v_0\in(\Omega+2(\ul\Omega-\ul\Omega))\cap(a_1 Z-\la_0)\}
\\\les& a_1^{2r-1}\cdot\frac{N^2}{a_1^{2+2r}}\les\frac{N^2}{a_1^3}.
\end{align*}
Thus, by the pigeonhole principle and \eqref{eq:U3>>Q(x) wrt v}, there exist $\la_0$ and $v_0$ such that
\[
\left|\sum_{k\in I}e^{i\frac{1}{a_1}B[v_0+kv_1,(\la_0+k\la_1)^\perp]}\right|\gtrsim a_1,
\]
where $I\subset[2a_1]$ is an interval (explicitly the set of $k\in[2a_1]$ such that $(\la,v)=(\la_0+k\la_1,v_0+kv_1)$ satisfies the summand conditions in \eqref{eq:U3>>Q(x) wrt v}). Thus, by Lemma \ref{lem:convex Weyl}, we have
\begin{equation}\label{eq:number-theoretic info}
\dist(\frac{1}{a_1}B[v_1,\la_1^\perp],\frac{2\pi}{m}\Z)\les\frac{1}{a_1^2}
\end{equation}
where $m=O(1)$ is an integer. Again, passing to an $m$-partition, we assume $m=1$.

For $k\in[a_1]$ and $v\in(\ul\Omega-\ul\Omega)\cap(a_1 Z-k\la_1)$, since $\frac1{a_1}(v-kv_1)\in Z$ and $\norm{\frac1{a_1}(v-kv_1)}_{\ul\Omega-\ul\Omega}\les\frac1{a_1}$, by \eqref{eq:B[u,la] <u,la} and \eqref{eq:number-theoretic info}, we have
\begin{equation}\label{eq:B[u,la] in a La+la_1 Z}
\dist(\frac{1}{a_1}B[v,\la_1^\perp],2\pi\Z)\le\dist(\frac{1}{a_1}B[v-kv_1,\la_1^\perp],2\pi\Z)+\dist(\frac{1}{a_1}B[kv_1,\la_1^\perp],2\pi\Z)\les\frac{1}{a_1}.
\end{equation}
Since \eqref{eq:B[u,la] in a La+la_1 Z} holds for every $k\in[a_1]$ and $a_1\Z-[a_1]\la_1=a_1\Z+\la_1\Z$, we have
\begin{equation}\label{eq:B[u,la] in a La+la_1 Z strong}
\dist(\frac{1}{a_1}B[v,\la_1^\perp],2\pi\Z)\les\frac{1}{a_1}\norm{v}_{\ul\Omega-\ul\Omega},\qquad v\in a_1 Z+\la_1\Z.
\end{equation}
By \eqref{eq:B[u,la] in a La+la_1 Z strong} and the rotation invariance of $B$, choosing $a_1$ large enough, there exists a unique linear operator $B^*:\R^{2+2r}\times(\la_1\R+\la_1^\perp\R)\rightarrow\R$ such that for $\la\in\{\la_1,\la_1^\perp\}$,
\[
B^*[u,\la^\perp]\in 2\pi a_1\Z \text{ is nearest to }B[u,\la^\perp],\qquad u\in (a_1 Z+\la\Z)\cap(\Omega-\Omega).
\]
Since $a_1 Z\cap(\Omega-\Omega)$ spans $\R^{2+2r}$, by \eqref{eq:B[u,la] in a La+la_1 Z strong}, for $\la\in\{\la_1,\la_1^\perp\}$ we have
\begin{equation}\label{eq:B*-B bound}
|(B^*-B)[u,\la^\perp]|\les\norm{u}_{\ul\Omega-\ul\Omega},\qquad u\in\R^{2+2r}.
\end{equation}
In particular, we have $|(B^*-B)[\la_1,\la_1^\perp]|\les 1$; since $B[\la_1,\la_1^\perp]=0$ and $B^*[\la_1,\la_1^\perp]\in2\pi a_1\Z$, $B^*[\la_1,\la_1^\perp]=0$ holds. Similarly, $B^*[\la_1^\perp,\la_1]=0$ holds. Thus, $B^*$ is symmetric on $(\la_1\R+\la_1^\perp\R)\times(\la_1\R+\la_1^\perp\R)$.

Now define $\tilde B:\R^{2+2r}\times\R^{2+2r}\rightarrow\R$ 
as the symmetric operator such that
\[
\tilde B[u,v]=B[u,v],\qquad u,v\in\Span(\{\la_2,\ldots,\la_r,\la_2^\perp,\ldots,\la_r^\perp,p_{e_0},p_{e_0^\perp}\})
\]
where $p_{e_0},p_{e_0^\perp}$ are as earlier (i.e., as in Remark \ref{rem:Bohr set by lattice-convex proj}) and
\[
\tilde B[u,v]=B^*[u,v],\qquad u\in\R^{2+2r},\qquad v\in\la_1\R+\la_1^\perp\R.
\]
By \eqref{eq:B*-B bound}, passing to an $O(1)$-partition, we have
\begin{equation}\label{eq:consequence of strong}
|(\tilde B-B)[u,u]|\le\delta/10,\qquad u\in\ul\Omega.
\end{equation}
Thus, up to a perturbation of $F(u)$, $u\in\Omega$ to $\frac{1}{2}\tilde B[u,u]$, we may reduce to the case $\tilde B=B$. 
Then $B(a_1 Z\times(\la_1\Z+\la_1^\perp\Z))\subset 2\pi\Z$ holds; we have
\begin{equation}\label{eq:Pi annihilates}
B[u,u]-B[v,v]\in4\pi\Z,\qquad u,v\in Z\text{ such that }u-v\in 2a_1(\la_1\Z+\la_1^\perp\Z).
\end{equation}
Thus, passing to an $2a_1$-partition, $u\mapsto e^{iF(u)}=e^{\frac{1}{2}iB[u,u]}$ is invariant under addition by $\la_1\Z+\la_1^\perp\Z$. This implies \eqref{eq:goal for steps 1 and 3} and finishes the proof.
\end{proof}
\begin{lem}\label{lem:U3>>Q(x) final lem}
Let $\{a_N\}$, $\{\mc N_N\}$, $\epsilon>0$, $r\in\N$, and $\phi$ be as in Lemma \ref{lem:U3>>Q(x) main lem}. Then, for every $N\gg_{\epsilon,r,\{a_N\}}1$, there exist $t\in\R$ and $\xi\in\R^2$ such that
\begin{equation}\label{eq:U3>>Q(x) final lem}
\left|\jp{\phi(x),e^{i(t|x|^2+\xi\cdot x)}}_{\ell^2([N]^2)}\right|\gtrsim_{\epsilon,r}N^2.
\end{equation}
Moreover, for $\delta>0$, there exist $J=O_{\epsilon,r,\delta}(1)$, $c_1,\ldots,c_J\in\D$, and $\xi_1,\ldots,\xi_J\in\R^2$ such that
\begin{equation}\label{eq:U3>>Q(x) profile}
\norm{\phi(x)-\sum_{j\le J}c_je^{i(t|x|^2+\xi_j\cdot x)}}_{\ell^2([N]^2)}\le\delta N\text{ and }
\norm{\sum_{j\le J}c_je^{i(t|x|^2+\xi_j\cdot x)}}_{\ell^\infty([N]^2)}\le 1.
\end{equation}
\end{lem}
\begin{proof}
We adopt the notations used in Lemma \ref{lem:U3>>Q(x) main lem}. Note that for each $r$, \eqref{eq:U3>>Q(x) final lem} implies \eqref{eq:U3>>Q(x) profile} by applying Lemma \ref{lem:bias lem for local quad on Bohr} to $\phi(x)e^{-i(t|x|^2+\xi\cdot x)}$. Thus, it suffices to show only \eqref{eq:U3>>Q(x) final lem}.

We prove by an induction on $r$.
If $r=0$, recalling the proof of Lemma \ref{lem:U3>>Q(x) main lem}, there exist an affine Bohr set $\mc B'$ in $[N]^2$ satisfying $\norm{\phi\chi_{\mc B'}}_{\mc N_N}\gtrsim_\epsilon 1$ (thus $\#\mc B'\gtrsim_\epsilon N^2$) and a quadratic polynomial $F=Q+L:\R^2\rightarrow\R$ such that $Q(x)=Q(x^\perp)$ and
\begin{equation}\label{eq:phi-e^iF}
|\phi(x)-e^{iF(x)}|\le1/2,\qquad x\in\mc B'.
\end{equation}
Here, $L:\R^2\rightarrow\R$ is a linear map, which appears by unfolding the assumption that $F$ was a pure quadratic form. The symmetry $Q(x)=Q(x^\perp)$ implies $Q(x)=t|x|^2$ for some $t\in\R$. Thus, \eqref{eq:phi-e^iF} yields
\[
\left|\jp{\phi,e^{it|x|^2+L(x)}\chi_{\mc B'}}_{\ell^2([N]^2)}\right|\gtrsim_\epsilon N^2.
\]
Applying Lemma \ref{lem:Bohr char is sum of linear} to $\chi_{\mc B'}$ and pigeonholing yields \eqref{eq:U3>>Q(x) final lem}, concluding the case $r=0$.

We show the inductive step; let $r\ge 1$. Assume that \eqref{eq:U3>>Q(x) final lem} can be satisfied for $\mc B$ of rotation-symmetric rank at most $2(r-1)$.
By Lemma \ref{lem:U3>>Q(x) main lem}, there exists an $O_{\epsilon,r}(1)$-partition member $\mc B'$ of $\mc B$, a locally linear modulation $\psi$ supported on $\mc B'$, and a locally quadratic modulation $\tilde\phi$ supported on $\tilde{\mc B}\supset\mc B'$ of rotation-symmetric rank at most $2(r-1)$ such that
\begin{equation}\label{eq:U3>>Q(x) main goal*}
|\phi(x)-\psi(x)\tilde\phi(x)|\le1/2,\qquad x\in\mc B'
\end{equation}
and
\begin{equation}\label{eq:U3>>Q(x) keep positive*}
\norm{\phi\chi_{\mc B'}}_{\mc N_N}\gtrsim_{\epsilon,r}1.
\end{equation}
Since $\psi$ is supported on $\mc B'$, by \eqref{eq:U3>>Q(x) main goal*}, we have
\begin{equation}\label{eq:E phi psi phitilde}
\left|\E_{x\in[N]^2}\phi\ol{\psi\tilde\phi}\right|\gtrsim_{\epsilon,r}1.
\end{equation}
Since $\phi\ol{\psi\tilde\phi}$ is locally quadratic, by Lemma \ref{lem:bias lem for local quad on Bohr}, for $\delta>0$, there exist $J=O_{\epsilon,r,\delta}(1)$, $c_1,\ldots,c_J\in\D$, and $\xi_1,\ldots,\xi_J\in\R^2$ such that
\[
\norm{\phi\ol{\psi\tilde\phi}-\sum_{j\le J}c_je^{ix\cdot\xi_j}}_{\ell^2([N]^2)}\le\delta N,
\]
which can be rewritten as
\[
\norm{\phi\chi_{\mc B'}-\psi\tilde\phi\sum_{j\le J}c_je^{ix\cdot\xi_j}}_{\ell^2([N]^2)}\le\delta N.
\]
Thus, choosing $\delta=\delta(\epsilon,r)\ll1$, by \eqref{eq:U3>>Q(x) keep positive*} and \eqref{eq:N weaker than l2}, we have
\[
\norm{\psi\tilde\phi\sum_{j\le J}c_je^{ix\cdot\xi_j}}_{\mc N_N}\gtrsim_{\epsilon,r}1.
\]
Then, by pigeonholing over the index $j$, there exists $j$ such that
\[
\norm{\psi\tilde\phi e^{ix\cdot\xi_j}}_{\mc N_N}\gtrsim_{\epsilon,r}1.
\]
Since $\psi$ is locally linear, applying Lemma \ref{lem:Bohr char is sum of linear} to $\psi$ and pigeonholing as earlier, there exists $\xi_*\in\R^2$ such that
\[
\norm{\tilde\phi e^{ix\cdot\xi_*}}_{\mc N_N}\gtrsim_{\epsilon,r}1.
\]
Then, since $\tilde\phi e^{ix\cdot\xi_*}$ is a locally quadratic modulation supported on $\tilde{\mc B}$, which has the rotation-symmetric rank at most $2(r-1)$, by the induction hypothesis, \eqref{eq:U3>>Q(x) profile} is applicable to $\tilde\phi e^{ix\cdot\xi_*}$. Thus, by applying Lemma \ref{lem:Bohr char is sum of linear} to $\psi$ and taking a product, $\psi\tilde\phi$ can be approximated in the form of \eqref{eq:U3>>Q(x) profile}. Now by pigeonholing on \eqref{eq:E phi psi phitilde} as earlier, \eqref{eq:U3>>Q(x) final lem} holds, finishing the proof.
\end{proof}
\section{Limit properties of profiles}\label{sec:prof}
In this section, we introduce terminologies to detect distributional concentration of Schr\"{o}dinger evolutions and provide limiting behavior of profiles appearing in Theorem \ref{thm:inverse L4 Stri}. To some extent, we follow \cite{ionescu2012energy,ionescu2012global}.
Then, we show Lemma \ref{lem:inverse thm for almost sum of profiles}, which is equivalent to Theorem \ref{thm:inverse L4 Stri} for the special case that $e^{it\De}\phi$ is approximated in $L^4$ by a finite sum of profiles. Lemma \ref{lem:inverse thm for almost sum of profiles} is a consequence of the inverse $L^4$-Strichartz inequality on $\R^2$ \cite{Bourgain-refinements,MerleVega} and conventional profile decomposition arguments.
\subsection*{An extinction lemma}
We show a version of the extinction lemma in \cite{ionescu2012energy}. As a preparation, we recall a kernel estimate in \cite{bourgain1993fourier}.
\begin{prop}[{\cite[Lemma 3.18]{bourgain1993fourier}}]
\label{prop:kernel bound bourgain}Let $N\in2^\N$. Let $(a,q)$ be a pair of coprime integers such that
\begin{equation}\label{eq:a/q assump}
1\le q<N\qquad\text{and }\left|t-\frac aq\right|\le\frac{1}{qN}.
\end{equation}
Then, we have
\begin{equation}
\norm{e^{it\De}\delta_{N}}_{L^{\infty}(\T^{2})}\les\left(\frac{N}{\sqrt{q}\left(1+N\left|t-\frac{a}{q}\right|^{1/2}\right)}\right)^{2}.\label{eq:Bourgain bound}
\end{equation}
\end{prop}
The following lemma is a version of \cite[Lemma 4.3]{ionescu2012energy}:
\begin{lem}[Extinction lemma]\label{lem:extinct}
\label{prop:kernel 0}We have
\begin{equation}
\limsup_{\substack{\epsilon\rightarrow0\\
T\rightarrow\infty
}
}\sup_{\substack{N\in2^\N\\ TN^{-2}<\frac{\epsilon}{\log N}}}N^{-1}\norm{e^{it\De}\delta_{N}}_{L_{t,x}^{4}([TN^{-2},\frac{\epsilon}{\log N}]\times\T^{2})}=0.\label{eq:kernel 0}
\end{equation}
Moreover, for $\ep>0$, $Q_*\in\N$, $N\in2^\N$, and $t_*\in\R$ such that $|t_*-a/q|\ge2\ep/\log N$ for every $(a,q)\in\Z\times\{1,\ldots,Q_*\}$, we have
\begin{equation}\label{eq:kernel a/q}
N^{-1}\norm{e^{i(t-t_*)\De}\delta_N}_{L^4_{t,x}([-\frac\ep{\log N},\frac\ep{\log N}]\times\T^2)}=o_\ep(1)+o_{Q_*}(1).
\end{equation}
\end{lem}
\begin{proof}
By Dirichlet's Lemma, for each $t\in[0,1]$, there exists $(a,q)$ satisfying \eqref{eq:a/q assump}. Interpolating the $L^{2}$-conservation of $e^{it\De}$ and
\eqref{eq:Bourgain bound} yields
\begin{equation}
\norm{e^{it\De}\delta_{N}}_{L^{4}(\T^{2})}\les\frac{N^{3/2}}{\sqrt{q}\left(1+N\left|t-\frac{a}{q}\right|^{1/2}\right)}.\label{eq:Bourgain bound-1}
\end{equation}
For $t\in[TN^{-2},\frac{\epsilon}{\log N}]$,
since $\frac{a}{q}\le\frac{1}{qN}+\frac{\epsilon}{\log N}$,
either $1\le a\le\frac{\epsilon}{\log N}q+\frac{1}{N}$ or $(a,q)=(0,1)$ holds.
Thus, by \eqref{eq:Bourgain bound-1}, taking a summation over $Q\in2^\N$, we have
\begin{align*}
&\norm{e^{it\De}\delta_{N}}_{L_{t,x}^{4}([TN^{-2},\frac{\epsilon}{\log N}]\times\T^2)}^{4} =\int_{TN^{-2}}^{\frac{\epsilon}{\log N}}\norm{e^{it\De}\delta_{N}}_{L^{4}(\T^{2})}^{4}dt\\
  \les&\sum_{Q\les N
}\sum_{q\sim Q}\sum_{1\le a\le\frac{\epsilon}{\log N}q+\frac{1}{N}}\int_{\R}\left(\frac{N^{3/2}}{\sqrt{q}\left(1+N\left|t-\frac{a}{q}\right|^{1/2}\right)}\right)^{4}dt
 +\int_{TN^{-2}}^{N^{-1}}\left(\frac{N^{3/2}}{1+Nt^{1/2}}\right)^{4}dt,
\end{align*}
then by direct calculations, we can estimate which by
\begin{align*}
 & \les\sum_{Q\les N
}Q\left(\frac{\epsilon}{\log N}Q+\frac{1}{N}\right)\int_{\R}\left(\frac{N^{3/2}}{\sqrt{Q}\left(1+N|s|^{1/2}\right)}\right)^{4}ds+\int_{TN^{-2}}^{\infty}\left(\frac{N^{3/2}}{Nt^{1/2}}\right)^{4}dt\\
 & \les\sum_{Q\les N
}Q\left(\frac{\epsilon}{\log N}Q+\frac{1}{N}\right)\cdot\frac{N^{4}}{Q^{2}}+\frac{N^{4}}{T}\les N^{4}\left(\epsilon+\frac{1}{T}\right),
\end{align*}
which yields \eqref{eq:kernel 0}.

Similarly, for $t_*$ as in the statement, we have
\begin{align*}
\norm{e^{i(t-t_*)\De}\delta_{N}}_{L^4_{t,x}([-\frac\ep{\log N},\frac\ep{\log N}]\times\T^2)}^{4} \les\sum_{Q_*\les Q\les N
}Q\left(\frac{\epsilon}{\log N}Q+1\right)\cdot\frac{N^{4}}{Q^{2}}\les N^4\left(\ep+\frac1{Q_*}\right),
\end{align*}
where $Q_*\les Q$ follows from the assumption of $t_*$ and $(\frac{\ep}{\log N}Q+1)$ reflects the cardinality of $[t_*-\frac{\ep}{\log N},t_*+\frac\ep{\log N}]\cap\frac1q\Z$ for $q\sim Q$.
Hence, we conclude \eqref{eq:kernel a/q}.
\end{proof}
\subsection*{\label{sec:Spacetime-profile-decomposition}Periodic extensions and
frames}
The symmetries of the Schr\"{o}dinger operator to be considered for inverse Strichartz estimates are the spacetime translations, scalings, and Galilean transforms.
We denote the Galilean transform with a shift $\xi\in\Z^{2}$ by $I_{\xi}:L^1_{t,x,\loc}(\R\times\T^{2})\rightarrow L^1_{t,x,\loc}(\R\times\T^2)$,
mapping a function $u:\R\times\T^{2}\rightarrow\C$ to 
\begin{equation}
I_{\xi}u(t,x)=e^{ix\cdot\xi-it|\xi|^{2}}u(t,x-2t\xi).\label{eq:Galilean}
\end{equation}
The linear Schr\"odinger flow is preserved by Galilean transforms.

We denote a quadruple $\left(N_{*},t_{*},x_{*},\xi_{*}\right)\in2^{\N}\times\R\times\T^{2}\times\Z^{2}$
of scale, time, space, and Galilean boost parameters. For $f\in L_{t,x,\loc}^{1}(\R\times\T^{2})$
and $\left(N_{*},t_{*},x_{*},\xi_{*}\right)\in2^{\N}\times\R\times\T^{2}\times\Z^{2}$,
we denote by $\iota_{(N_{*},t_{*},x_{*},\xi_{*})}f$ the $C^{0}L^{2}$-critically
rescaled periodic extension
\[
\iota_{(N_{*},t_{*},x_{*},\xi_{*})}f(t,x):=N_{*}^{-1}I_{\xi_*}f\left(N_{*}^{-2}t+t_{*},N_{*}^{-1}x+x_{*}+2\pi\Z^{2}\right).
\]
\begin{defn}
\label{def:frame}A sequence of quadruples of parameters $\left\{ \left(N_{n},t_{n},x_{n},\xi_{n}\right)\right\} _{n\in\N}\subset2^{\N}\times\R\times\T^{2}\times\Z^{2}$
is said to be a \emph{frame} if $\lim_{n\rightarrow\infty}N_{n}=\infty$.

Two frames $\{\mc O_n\}=\{(N_n,t_n,x_n,\xi_n)\}$ and $\{\mc O_n'\}=\{(N_n',t_n',x_n',\xi_n')\}$ are \emph{orthogonal} if
\[
\lim_{n\rightarrow\infty}|\log (N_n/N_n')|+N_n^2\cdot \dist(t_n-t_n',2\pi\Z)+N_n\cdot\dist(x_n-x_n',2\pi\Z^2)+N_n^{-1}\cdot|\xi_n-\xi_n'|=\infty
\]
and \emph{comparable} if
\[
\limsup_{n\rightarrow\infty}|\log (N_n/N_n')|+N_n^2\cdot \dist(t_n-t_n',2\pi\Z)+N_n\cdot\dist(x_n-x_n',2\pi\Z^2)+N_n^{-1}\cdot|\xi_n-\xi_n'|<\infty.
\]
\end{defn}
We frequently work on distributional weak limits of $\iota_{\mc O_{n}}f_{n}$
for a sequence of functions $f_{n}:\R\times\T^{2}\rightarrow\C$.
Indeed, every weak limit we consider in this paper is in a distributional
sense. Given a sequence of functions $\left\{ f_{n}\right\} $, we
denote by $\lim_{n}f_{n}$ the weak limit of $f_{n}$ (if it exists).
\begin{defn}
A family of distributions $\left\{ f_{n}\right\} $ on either $\R^{2}$
or $\R\times\R^{2}$ is said to be \emph{weakly nonzero} if for every
subsequence $\left\{ n_{k}\right\} $, $f_{n_{k}}$ does not converge
weakly to zero.
\end{defn}
\subsection*{Inverse $L^4$-Strichartz estimate for a bounded sum of profiles}
\begin{lem}\label{lem:single profile conv}
Let $\{\mc O_n\}=\{(N_n,t_n,x_n,\xi_n)\}$ be a frame. Let $\psi_n$ be the profile
\begin{equation}\label{eq:profile O_n}
\psi_n(x)=N_n^{-1}e^{-it_n\De}e^{i\xi_n\cdot x}P_{N_n}\delta(x-x_n).
\end{equation}
Then, for every sequence $\{T_n\}$ in $(0,1]$ such that $T_n\log N_n\rightarrow0$, we have
\begin{equation}\label{eq:single profile conv}
\limsup_{n\rightarrow\infty}\norm{\iota_{\mc O_n}e^{it\De}\psi_n-e^{it\De}P_1\delta}_{L^4([-T_nN_n^2,T_nN_n^2]\times[-\pi N_n,\pi N_n]^2)}=0.
\end{equation}
\end{lem}
\begin{proof}
By Lemma \ref{lem:extinct}, it suffices to show \eqref{eq:single profile conv} when $T_n=TN_n^{-2}$, where $T>0$ is arbitrary finite number. Up to symmetries of the Schr\"{o}dinger evolution, we may assume $t_n=0$, $x_n=0$, and $\xi_n=0$. Then, \eqref{eq:single profile conv} can be rewritten as
\begin{equation}\label{eq:single profile conv 1}
\limsup_{n\rightarrow\infty}\norm{\sum_{\xi\in2\pi N_n\Z^2}e^{it\De}P_1\delta(\cdot-\xi)-e^{it\De}P_1\delta}_{L^4([-T,T]\times[-\pi N_n,\pi N_n]^2)}=0.
\end{equation}
\eqref{eq:single profile conv 1} is immediate from the rapid spatial decay of $e^{it\De}P_1\delta$, $|t|\le T$, finishing the proof.
\end{proof}
\begin{lem}\label{lem:inverse thm for almost sum of profiles}
Let $\{I_n\}$ be a sequence of intervals on $\R$ and $\epsilon>0$. Let $\{\phi_n\}$ be a sequence in $L^2(\T^2)$ such that $\norm{\phi_n}_{L^2}\le1$ and
\[
\norm{e^{it\De}\phi_n}_{L^4_{t,x}(I_n\times\T^2)}\ge\epsilon.
\]
Assume there exist $J\in\N$, $c_1,\ldots,c_J\in\D$, and frames $\{\mc O_{n,j}\}_{j\le J}=\{(N_{n,j},t_{n,j},x_{n,j},\xi_{n,j})\}$ such that $|I_n|\cdot\max_j\log N_{n,j}\rightarrow0$ and
\begin{equation}\label{eq:profile sum almost equal}
\norm{e^{it\De}\phi_n-\sum_{j\le J}c_jN_{n,j}^{-1}e^{i(t-t_{n,j})\De}e^{i\xi_{n,j}\cdot x}P_{N_{n,j}}\delta(\cdot-x_{n,j})}_{L^4_{t,x}(I_n\times\T^2)}\le\epsilon/4.
\end{equation}
Then, there exist an index $j_0\le J$ and a frame $\mc O=\{(N_n,t_n,x_n,\xi_n)\}$ such that $N_n\sim N_{n,j_0}$ and, along a subsequence,
\begin{equation}\label{eq:profile sum main}
|\jp{(\iota_{\mc O_n}e^{it\De}\phi_n)(0),P_1\delta}_{L^2(\R^2)}|\gtrsim_\epsilon1
\end{equation}
holds and $\{\chi_{\tilde I_n}\}$ is weakly nonzero.

Here, we denoted by $\tilde I_n$ the image of $I_n$ under $\mc O_n$, i.e. $\tilde I_n=\{N_n^2(t-t_n):t\in I_n\}$.
\end{lem}
In Section \ref{sec:proof-inv}, we will show that \eqref{eq:profile sum almost equal} holds whenever $|I_n|\cdot\log\# S_n\rightarrow0$, $S_n=\supp(\widehat{\phi_n})$ (Proposition \ref{prop:profile'}). Consequently, when $|I_n|\cdot\log\# S_n\rightarrow0$, the assumption \eqref{eq:profile sum almost equal} becomes redundant and this lemma yields \eqref{eq:profile sum main} together with the weak nonzeroness of $\{\chi_{\tilde I_n}\}$. (Note in particular that \eqref{eq:profile sum main} corresponds to Theorem \ref{thm:inverse L4 Stri} written in terms of sequences.) This observation will be used in Lemma \ref{lem:inverse Z_R stri}, which is a crucial ingredient in a concentration argument for the global well-posedness.
\begin{proof}
Denote
\begin{equation}\label{eq:phi_n def}
\tilde\phi_n(x)=\sum_{j\le J} c_jN_{n,j}^{-1}e^{-it_{n,j}\De}e^{i\xi_{n,j}\cdot x}P_{N_{n,j}}\delta(x-x_{n,j}).
\end{equation}
By \eqref{eq:profile sum almost equal} and the triangle inequality, we have
\begin{equation}\label{eq:bdd sum L4 bound}
\limsup_{n\rightarrow\infty}\norm{e^{it\De}\tilde\phi_n}_{L^4_{t,x}(I_n\times\T^2)}\ge\epsilon/2.
\end{equation}
We use a conventional profile decomposition argument.
Passing to a subsequence, we may assume $\mc O_j=\{\mc O_{n,j}\}_n$ are pairwisely either orthogonal or comparable.
Up to merging all comparable frames, we may assume pairwise orthogonality between $\mc O_j$. Then, \eqref{eq:bdd sum L4 bound} is simplified to
\begin{equation}\label{eq:bdd sum L4 bound, psi}
\limsup_{n\rightarrow\infty}\norm{\sum_je^{it\De}\psi_{n,j}}_{L^4(I_n\times\T^2)}\ge\epsilon/2,
\end{equation}
where $\psi_{n,j}$ is the partial sum of comparable summands in \eqref{eq:phi_n def} for each $j$. Let $\phi_j$ and $\psi_j$ be the weak limits
\[
(\iota_{\mc O_{n,j}}e^{it\De}\phi_n)(0)\weak\phi_j\quad\text{and}\quad(\iota_{\mc O_{n,j}}e^{it\De}\psi_{n,j})(0)\weak \psi_j.
\]
Passing to a subsequence, we assume the weak convergence $\chi_{\tilde I_{n,j}}\weak\chi_{\tilde I_j}$ for each $j$, $\tilde I_j\subset\R$ being some interval. Passing to a further subsequence, we may assume $\dist(t_{n,j}-t_j,I_n+\Z)\to0$ for some $t_j\in\R$. If $t_j\in\Q$, $\psi_{n,j}$ can be rewritten as a linear combination of spatial translates of $e^{-it_j\De}\psi_{n,j}$. Hence, up to adjusting the profiles and frames, we may assume $\dist(t_{n,j},I_n)\to0$ or $t_j\notin\Q$. By Lemmas \ref{lem:extinct} and \ref{lem:single profile conv}, we have
\[
\limsup_{n\rightarrow\infty}\norm{\iota_{\mc O_{n,j}}e^{it\De}\psi_{n,j}-e^{it\De}\psi_j}_{L^4(\tilde I_{n,j}\times[-\pi N_{n,j},\pi N_{n,j}]^2)}=0,
\]
which can be rewritten as
\begin{equation}\label{eq:prep for brezis}
\chi_{\tilde I_{n,j}\times[-\pi N_{n,j},\pi N_{n,j}]^2}\iota_{\mc O_{n,j}}e^{it\De}\psi_{n,j}\rightarrow \chi_{\tilde I_j}e^{it\De}\psi_j\text{ in }L^4(\R\times\R^2).
\end{equation}
Passing to a subsequence, by \eqref{eq:prep for brezis} and the frame orthogonality, we have the almost everywhere convergence $\chi_{\tilde I_{n,j}\times[-\pi N_{n,j},\pi N_{n,j}]^2}\iota_{\mc O_{n,j}}e^{it\De}\tilde\phi_n\rightarrow \chi_{\tilde I_j}e^{it\De}\psi_j$. Thus, by Brezis-Lieb \cite[(1)]{brezis1983relation} on $\R\times\R^2$, we have
\begin{align*}
&\limsup_{n\rightarrow\infty}\norm{\chi_{\tilde I_{n,j}\times[-\pi N_{n,j},\pi N_{n,j}]^2}\iota_{\mc O_{n,j}}e^{it\De}\tilde\phi_n}_{L^4(\R\times\R^2)}^4
\\=&\limsup_{n\rightarrow\infty}\norm{\chi_{\tilde I_{n,j}\times[-\pi N_{n,j},\pi N_{n,j}]^2}\iota_{\mc O_{n,j}}e^{it\De}\tilde\phi_n-\chi_{\tilde I_j}e^{it\De}\psi_{j}}_{L^4(\R\times\R^2)}^4
+\norm{\chi_{\tilde I_j}e^{it\De}\psi_j}_{L^4(\R\times\R^2)}^4,
\end{align*}
which can be rewritten by \eqref{eq:prep for brezis} as
\[
\limsup_{n\rightarrow\infty}\norm{e^{it\De}\tilde\phi_n}_{L^4(I_n\times\T^2)}^4=\limsup_{n\rightarrow\infty}\norm{e^{it\De}(\tilde\phi_n-\psi_{n,j})}_{L^4(I_n\times\T^2)}^4+\norm{e^{it\De}\psi_j}_{L^4(\tilde I_j\times\R^2)}^4.
\]
Repeating the process, we have the $\ell^4$-decoupling identity
\begin{equation}\label{eq:l4 decoupled}
\sum_j\norm{e^{it\De}\psi_j}_{L^4(\tilde I_j\times\R^2)}^4=\limsup_{n\rightarrow\infty}\norm{e^{it\De}\tilde\phi_n}_{L^4(I_n\times\T^2)}^4\ge(\ep/2)^4.
\end{equation}
By \eqref{eq:profile sum almost equal} and the orthogonality between frames, we have the $\ell^4$-bound
\begin{equation}\label{eq:l4 difference small}
\sum_j\norm{e^{it\De}(\phi_j-\psi_j)}_{L^4(\tilde I_j\times\R^2)}^4\le\limsup_{n\to\infty}\norm{e^{it\De}(\phi_n-\tilde\phi_n)}_{L^4(I_n\times\T^2)}^4\le{(\ep/4)}^4.
\end{equation}
By \eqref{eq:l4 decoupled} and \eqref{eq:l4 difference small}, we have
\begin{equation}\label{eq:l4 largeness phi}
\sum_j\norm{e^{it\De}\phi_j}_{L^4(\tilde I_j\times\R^2)}^4\ge(\ep/4)^4.
\end{equation}
By the orthogonality between frames, we also have the $\ell^2$-bound
\begin{equation}\label{eq:l2 decoupled phi}
\sum_j\norm{\phi_j}_{L^2(\R^2)}^2\le\limsup_{n\to\infty}\norm{\phi_n}_{L^2(\T^2)}^2\le1.
\end{equation}
By \eqref{eq:l4 largeness phi} and \eqref{eq:l2 decoupled phi}, there exists an index $j_0$ such that
\begin{equation}\label{eq:j_0 def}
\norm{e^{it\De}\phi_{j_0}}_{L^4(\tilde I_{j_0}\times\R^2)}\gtrsim_\ep1.
\end{equation}
In particular, $|\tilde I_{j_0}|>0$ and thus $\{\chi_{\tilde I_{n,j_0}}\}$ is weakly nonzero.
Then, by the $L^4$-inverse Strichartz inequality on $\R^2$ \cite{Bourgain-refinements,MerleVega}, there exists a quadruple $(N_0,t_0,x_0,\xi_0)\in 2^\Z\times\R\times\R^2\times\R^2$ such that
\begin{equation}\label{eq:used R2 inverse}
\left|\jp{e^{it_0\De}\phi_{j_0},N_0^{-1}e^{i\xi_0\cdot x}P_{N_0}\delta(\cdot-x_0)}_{L^2(\R^2)}\right|\gtrsim_\epsilon1.
\end{equation}
Thus, we can choose a frame $\mc O=\{\mc O_n\}$ comparable to $\mc O_{j_0}$ such that the weak limit
\[
(\iota_{\mc O_n}e^{it\De}\phi_n)(0)\weak \phi
\]
satisfies $|\jp{\phi,P_1\delta}_{L^2(\R^2)}|\gtrsim_\epsilon1$. Since $\mc O$ is comparable to $\mc O_{{j_0}}$, we also obtain that $\{\chi_{\tilde I_n}\}$ is weakly nonzero, as desired.
\end{proof}
\section{Proof of Theorem \ref{thm:inverse L4 Stri}}\label{sec:proof-inv}
In this section, we prove Theorem \ref{thm:inverse L4 Stri}. We split the proof of Theorem \ref{thm:inverse L4 Stri} into four separate propositions, describing on the Fourier side the concentration of the modulus, of the support, and of the modulation.

We start with a lemma which will be used in the proof of Proposition \ref{prop:profile'}.
\begin{lem}[\cite{herr2024strichartz}]
\label{lem:norm}Let $\delta>0$. For $f:\Z^{2}\rightarrow\C$ supported on a finite set $\supp(f)=S\subset\Z^{2}$ and $M=\lfloor\delta^{-1}\log\#S\rfloor$, we have
\begin{equation}\label{eq:norm}
\norm{e^{it\De}\F^{-1}f}_{L_{t,x}^{4}([0,\frac{\delta}{\log\#S}]\times\T^{2})}^{4}\les\frac{1}{M}\sum_{Q\in\cup_{\tau=1}^{M}\mc Q^{\tau}(S)}|f(Q)|+\delta\norm{\F^{-1}f}_{L^{2}(\T^2)}^{4}.
\end{equation}
\end{lem}
\begin{proof}
Let $g:\R\rightarrow[0,\infty)$ be the $2\pi$-periodic function
\[
g(t):=\int_{\T^2}\left|e^{it\De}\F^{-1}f(x)\right|^4dx.
\]
Then, \eqref{eq:norm} can be rewritten as
\begin{equation}\label{eq:norm2}
\int_0^{1/M}g(t)dt\les\frac{1}{M}\sum_{Q\in\cup_{\tau=1}^{M}\mc Q^{\tau}(S)}|f(Q)|+\frac{\log\#S}{M}\norm{\F^{-1}f}_{L^{2}}^{4}.
\end{equation}
Denote by $F_M:\R/2\pi\Z\rightarrow[0,\infty)$ the Fej\'{e}r kernel
\[
\widehat{F_M}(\tau):=\max\left\{ 0,1-\frac{|\tau|}{M}\right\}.
\]
Since $F_M(t)\gtrsim M$ holds for $|t|\le1/M$, we have
\begin{align}\label{eq:norm I}
\int_0^{1/M}g(t)dt&\les\frac{1}{M}\sum_{\tau\in\Z}\widehat{g}(\tau)\cdot\overline{\widehat{F_M}(\tau)}
\\&\nonumber
\les\frac{1}{M}\sum_{|\tau|\le M}|\widehat{g}(\tau)|
\\&\nonumber
\les\frac{1}{M}\sum_{1\le|\tau|\le M}\sum_{Q\in\mc Q^\tau(S)}|f(Q)|+\frac{1}{M}\widehat{g}(0).
\end{align}
Here, the summation $1\le|\tau|\le M$ can be reduced to $1\le \tau\le M$ since relabeling $Q$ flips the sign of $\tau(Q)$ and conjugates $f(Q)$.
Since $\widehat g(0)$ measures the $L^4$-norm of $e^{it\De}\F^{-1}f$ over $[0,2\pi]\times\T^2$, by \eqref{eq:L4 Stri}, \eqref{eq:norm I} can be reduced to \eqref{eq:norm2}, finishing the proof.
\end{proof}
The following are the four main propositions of this section; showing these implies Theorem \ref{thm:inverse L4 Stri}, as will be shown shortly.
\begin{prop}
\label{prop:modulus conc}Let $\epsilon>0$. There exists $\delta>0$ satisfying the following:

For any function $\phi\in L^{2}(\T^{2})$ with a finite Fourier support $\supp(\widehat{\phi})=S$ such that
\begin{equation}
\norm{e^{it\De}\phi}_{L_{t,x}^{4}([0,\frac{\delta}{\log\#S}]\times\T^{2})}\ge\epsilon\norm{\phi}_{L^{2}(\T^2)},\label{eq:modulus suppose}
\end{equation}
there exists $\rho>0$ such that
\begin{equation}\label{eq:layer conc}
\norm{1_{\rho/2\le\left|\widehat{\phi}\right|<\rho}\widehat{\phi}}_{\ell^{2}(\Z^{2})}\sim_\epsilon\norm{\phi}_{L^{2}(\T^2)}.
\end{equation}
\end{prop}
\begin{prop}
\label{prop:conc to P}Let $\epsilon>0$. For any finite
set $S\subset\Z^{2}$ and integer $M_{0}\gg_\epsilon1$ such that
\begin{equation}
\#\left(\cup_{\tau=1}^{M_{0}}\mc Q^{\tau}(S)\right)\ge\epsilon M_{0}\cdot\left(\#S\right)^{2},\label{eq:P claim}
\end{equation}
there exists a multiprogression $(P,\Omega)\sim_\epsilon S$ of rank $r=O_\epsilon(1)$.
\end{prop}
\begin{prop}\label{prop:conc to [N]^2}Let $\epsilon>0$ and $r\in\N$. There exist $M_0,k\in\N$ satisfying the following:

Let $(P,\Omega)$ be a $k$-injective thick multiprogression of rank $r$ into $\Z^2$. Assume there exists $M_1\ge M_0$ such that
\[
\#\left(\cup_{\tau=1}^{M_1}\mc Q^{\tau}(P(\Omega))\right)\ge\epsilon M_1\cdot\left(\#\Omega\right)^{2}.\label{eq:P claim2}
\]
Then, $(P,\Omega)\sim_{\epsilon,r}[\sqrt{\#\Omega}]^2$ holds.
\end{prop}
\begin{prop}
\label{prop:conc to Q}Let $\epsilon>0$. There exist $\delta>0$ and $J\in\N$ satisfying the following:

For $N\gg_{\epsilon}1$ and $f:[N]^{2}\rightarrow\D$,
there exist $c_j\in\D$, $t_j\in\R$, $\xi_j\in\R^2$, $j=1,\ldots,J$ such that
\begin{equation}\label{eq:Q goal}
\frac{1}{N}\norm{e^{it\De}\F^{-1}(f-\sum_{j\le J}c_je^{i(t_j|x|^2+\xi_j\cdot x)})}_{L^4_{t,x}([0,\frac{\delta}{\log N}]\times\T^2)}\le\epsilon.
\end{equation}
\end{prop}
Before proving these propositions, which will be done in the upcoming subsections, we show why Propositions \ref{prop:modulus conc}-\ref{prop:conc to Q} imply Theorem \ref{thm:inverse L4 Stri}. Indeed, proving these enables the following profile decomposition property:
\begin{prop}\label{prop:profile'}
Let $\epsilon>0$. There exist $\delta=\delta(\epsilon)>0$ and $J=J(\epsilon)\in\N$ satisfying the following:

For every $\phi\in L^2(\T^2)$ such that $\norm{\phi}_{L^2(\T^2)}\le 1$ and $S=\supp(\widehat\phi)\subset\Z^2$ is finite, there exist $N_j\in2^\N$, $t_j\in\R$, $x_j\in\T^2$, $\xi_j\in\Z^2$, and $c_j\in\D$, $j=1,\ldots,J$ such that $\log N_j\les\log \#S$ and
\begin{equation}\label{eq:profile' goal}
\norm{e^{it\De}\phi-\sum_{j\le J}c_jN_j^{-1}e^{i(t-t_j)\De}e^{i\xi_j\cdot x}P_{N_j}\delta(x-x_j)}_{L^4([0,\frac{\delta}{\log\#S}]\times\T^2)}\le\epsilon.
\end{equation}
\end{prop}
As noted after Lemma \ref{lem:inverse thm for almost sum of profiles}, showing Proposition \ref{prop:profile'} implies Theorem \ref{thm:inverse L4 Stri}.
\begin{proof}[Proof of Proposition \ref{prop:profile'}, assuming Propositions \ref{prop:modulus conc}, \ref{prop:conc to P}, \ref{prop:conc to [N]^2}, and \ref{prop:conc to Q}.]
Let $\epsilon>0$, $\phi\in L^2(\T^2)$, and $S\subset\Z^2$ be as in Proposition \ref{prop:profile'}. Let $\delta>0$ be a small number to be fixed shortly. Denote $T:=\frac{\delta}{\log\#S}$.
For $N\in2^\N$, we denote
\[
S_N:=\{\xi\in S:1/N\le|\widehat\phi(\xi)|<2/N\}.
\]
By Proposition \ref{prop:modulus conc}  there exist $J_0=O_\epsilon(1)$ and $N_1,\ldots,N_{J_0}\in2^\N$ such that
\begin{equation}\label{eq:proof-modulus conc}
\norm{e^{it\De}\phi-\sum_{j\le J_0}e^{it\De}P_{S_{N_j}}\phi}_{L^4_{t,x}([0,T]\times\T^2)}<\frac{\epsilon}{10}.
\end{equation}
Here, $\log N_j\les\log \#S$ can be assumed since higher $N_j$ contribute $o(1)$ to \eqref{eq:proof-modulus conc} by \eqref{eq:L4 Stri}.

Let $j\in\{1,\ldots,J_0\}$. For $E\subset S_{N_j}$ such that
\[
\norm{e^{it\De}P_E\phi}_{L^4_{t,x}([0,T]\times\T^2)}\ge\frac{\epsilon}{10J_0},
\]
assuming $\delta=\delta(\epsilon)>0$ is small enough, by Lemma \ref{lem:norm}, there exists $M_j\gg_\epsilon1$ such that
\[
\#\left(\cup_{\tau=1}^{M_j}\mc Q^\tau(E)\right)\gtrsim_{\epsilon} M_jN_j^4\gtrsim M_j(\#E)^2.
\]
Thus, by Proposition \ref{prop:conc to P}, there exists a multiprogression $(P_0,\Omega_0)\sim_{\epsilon} E$ of rank at most $r=O_{\epsilon}(1)$. Let $k=k(\epsilon,r)=O_\epsilon(1)$ be the number in Proposition \ref{prop:conc to [N]^2}. By Proposition \ref{prop:rank reduction}, there exists a $k$-injective multiprogression $(P,\Omega)\sim_\epsilon(P_0,\Omega_0)$ of rank at most $r$. Ignoring thin coordinates in $\Omega$ (i.e., coordinates of heights $O_\epsilon(1)$), we can further assume $(P,\Omega)$ is thick.

Since $(P,\Omega)\sim_\epsilon(P_0,\Omega_0)\sim_\epsilon E$, up to a translation we may assume $\#(P(\Omega)\cap E)\gtrsim_\epsilon\#E$. Thus, repeating the extraction of $(P,\Omega)$ starting from $E=S_{N_j}$ and passing to $E\setminus P(\Omega)$, there exist $m_j=O_\epsilon(1)$ and $k$-injective affine multiprogressions $(P_{j,1},\Omega_{j,1}),\ldots,(P_{j,m_j},\Omega_{j,m_j})\sim_\epsilon S_{N_j}$ of ranks $r_j\le r$ such that for every $E\subset S_{N_j}\setminus\cup_{m=1}^{m_j}P_{j,m}(\Omega_{j,m})$,
\begin{equation}\label{eq:phi S-ImPjk}
\norm{e^{it\De}P_E\phi}_{L^4_{t,x}([0,T]\times\T^2)}<\frac{\epsilon}{10J_0}
\end{equation}
holds. For $m\le m_j$ such that $(P_{j,m},\Omega_{j,m})\nsim_\epsilon[N_j]^2$ and $M\gg_\epsilon1$, by Proposition \ref{prop:conc to [N]^2}, we have
\[
\#\left(\cup_{\tau=1}^{M}\mc Q^\tau(P_{j,m}(\Omega_{j,m}))\right)\ll_\epsilon MN_j^4,
\]
which implies by Lemma \ref{lem:norm} that for every $E^*\subset S_{N_j}\cap P_{j,m}(\Omega_{j,m})$,
\begin{equation}\label{eq:E* Stri small}
\norm{e^{it\De}P_{E^*}\phi}_{L^4_{t,x}([0,T]\times\T^2)}<\frac{\epsilon}{10m_jJ_0}+O_\epsilon(\delta^{1/4}).
\end{equation}
Thus, choosing $\delta=\delta(\epsilon)>0$ small enough, by the triangle inequality between \eqref{eq:phi S-ImPjk} and \eqref{eq:E* Stri small}, there exists $\mc M_j\subset\{1,\ldots,m_j\}$ such that $(P_{j,m},\Omega_{j,m})\sim_\epsilon[N_j]^2$ holds for $m\in\mc M_j$ and for every $E\subset S_{N_j}\setminus\cup_{m\in\mc M_j}P_{j,m}(\Omega_{j,m})$,
\begin{equation}\label{eq:phi S-ImPjk'}
\norm{e^{it\De}P_E\phi}_{L^4_{t,x}([0,T]\times\T^2)}<\frac{3\epsilon}{10J_0}.
\end{equation}
For each $m\in\mc M_j$, since $(P_{j,m},\Omega_{j,m})\sim_\epsilon [N_j]^2=\id_{\R^2}([N_j]^2)$, $P_{j,m}(\Omega_{j,m})$ can be covered by $O_\epsilon(1)$ translates of $[N_j]^2$. Thus, there exists $\Xi_j\subset(2N_j+1)\Z^2$ such that $\#\Xi_j=O_\epsilon(1)$ and
\[
\cup_{m\in\mc M_j}P_{j,m}(\Omega_{j,m})\subset\cup_{\xi\in\Xi_j}([N_j]^2+\xi).
\]
Plugging $E=S_{N_j}\setminus\cup_{\xi\in\Xi_{j}}([N_{j}]^2+\xi)$ into \eqref{eq:phi S-ImPjk'} yields
\begin{equation}\label{eq:phi S-IMpjk' final}
\norm{e^{it\De}P_{S_{N_j}\setminus\cup_{\xi\in\Xi_{j}}([N_j]^2+\xi)}\phi}_{L^4_{t,x}([0,T]\times\T^2)}<\frac{3\epsilon}{10J_0}.
\end{equation}
By \eqref{eq:proof-modulus conc}, \eqref{eq:phi S-IMpjk' final}, and using the triangle inequality, we have
\[
\norm{e^{it\De}\phi-\sum_{j\le J_0}\sum_{\xi\in\Xi_j}e^{it\De}P_{S_{N_j}\cap([N_j]^2+\xi)}\phi}_{L^4_{t,x}([0,T]\times\T^2)}<\frac{4\epsilon}{10}.
\]
Here, by Proposition \ref{prop:conc to Q}, each summand $e^{it\De}P_{S_{N_j}\cap([N_j]^2+\xi)}\phi$ can be approximated up to arbitrary $\epsilon'$-error in $L^4([0,T]\times\T^2)$ by $e^{it\De}$ of an $O_{\epsilon'}(1)$ linear combination of the forms
\[
N_j^{-1}\F^{-1}(\chi_{[N_j]^2+\xi}e^{i(t_*|\xi|^2+x_*\cdot \xi)})=N_j^{-1}e^{-it_*\De}e^{i\xi\cdot x}\F^{-1}(\chi_{[N_j]^2})(x-x_*),\qquad (t_*,x_*)\in\R\times\T^2.
\]
Taking $\epsilon'=\epsilon'(\epsilon)\ll1$ yields \eqref{eq:profile' goal} by the triangle inequality, except that sharp Fourier cutoffs $\F^{-1}(\chi_{[N_j]^2})$ remain to be replaced by smooth Littlewood-Paley kernels $P_{N_j}\delta$. Approximating each $N_j^{-1}\F^{-1}(\chi_{[N_j]^2})$ in $L^2(\T^2)$ by a linear combination of smooth Littlewood-Paley kernels finishes the proof by \eqref{eq:L4 Stri}.
\end{proof}
\subsection{\label{subsec:Concentration-of-modulus}Proof of Proposition \ref{prop:modulus conc}}

In this subsection, we show Proposition \ref{prop:modulus conc}. By the pruning argument in \cite[Prop.\ 3.1]{herr2024strichartz}, Proposition \ref{prop:modulus conc} reduces to the following lemma:
\begin{lem}\label{lem:conc la*}
Let $\epsilon>0$ and $m,C\in\N$. Let $f:\Z^{2}\rightarrow[0,\infty)$
be a function of the form
\[
f=\sum_{j\le m}\la_{j}2^{-j/2}\chi_{S_{j}},
\]
where $S_{0},\ldots,S_{m},m\ge1$ are disjoint subsets of $\Z^{2}$
such that $\#S_{j}\le 2^{j}$, and $\la_{0},\ldots,\la_{m}\ge0$. Suppose
that for each $j=0,\ldots,m$ and $\xi\in S_{j}$, there exists at
most one line $\ell\ni\xi$ such that $\#(\ell\cap S_{j})\ge2^{j/2+C}$.
Assume further that
\begin{equation}
\frac{1}{M}\sum_{Q\in\mc Q^{\le M}}f(Q)\ge\epsilon\norm{\la_{j}}_{\ell_{j\le m}^{2}}^{4}\label{eq:modulus suppose*}
\end{equation}
for some $M\gg_{\epsilon,C}m$. Then,
\[
\max_{j=0,\ldots,m}\la_{j}\gtrsim_{\epsilon,C}\norm{\la_{j}}_{\ell_{j\le m}^{2}}.
\]
\end{lem}
\begin{proof}[Proof of Proposition \ref{prop:modulus conc} assuming Lemma \ref{lem:conc la*}]
We follow the proof of \cite[Proposition 3.1]{herr2024strichartz}.
Let $\epsilon$, $\phi$, and $S$ be as in Proposition \ref{prop:modulus conc}. Let $m=\lceil\log\#S\rceil$. We choose an enumeration $\xi_1,\xi_2,\ldots$ of $\Z^2$ such that $|\widehat\phi(\xi_1)|\ge |\widehat\phi(\xi_2)|\ge\cdots$. Let $S_j^0:=\{\xi_{2^j},\ldots,\xi_{2^{j+1}-1}\}$ and $\la_j:=2^{j/2}|\widehat\phi(\xi_{2^j})|$ for $j=0,\ldots,m$. By \cite[(3.6)]{herr2024strichartz}, we have
\begin{equation}\label{eq:la<f}
\norm{\la_j}_{\ell^2_{j\le m}}\sim\norm{\phi}_{L^2(\T^2)}.
\end{equation}
For $j=0,\ldots,m$, let $E_{j}\subset S_{j}^{0}$ be the set
of intersections $\xi\in S_{j}^{0}$ of two lines $\ell_{1},\ell_{2}$
such that
\[
\#\left(\ell_{1}\cap S_{j}^{0}\right),\#\left(\ell_{2}\cap S_{j}^{0}\right)\ge2^{j/2+C},
\]
where $C=C(\epsilon)\in\N$ is a constant to be fixed shortly.
By \cite[(3.7)]{herr2024strichartz}, we have
\begin{equation}\label{eq:E_j bound}
\sqrt{\#E_{j}}\les 2^{j/2-C}.
\end{equation}
Since $|\widehat{\phi}(\xi)|\le\la_{j}2^{-j/2}$ holds for
$\xi\in E_{j}\subset S_{j}^{0}$, by \eqref{eq:E_j bound} and \eqref{eq:la<f},
we have
\[
\norm{\chi_{E}\widehat{\phi}}_{\ell^{2}(\Z^{2})}\les\norm{\la_{j}2^{-j/2}\cdot\sqrt{\#E_{j}}}_{\ell_{j\le m}^{2}}\les2^{-C}\norm{\phi}_{L^2(\T^2)}.
\]
Here we denote $E=\cup_j E_j$. By \eqref{eq:L4 Stri}, we have
\[
\norm{e^{it\De}P_E\phi}_{L^4_{t,x}([0,\frac{\delta}{\log\#S}]\times\T^2)}\les\norm{\chi_E\widehat{\phi}}_{\ell^2(\Z^2)}\les 2^{-C}\norm{\phi}_{L^2(\T^2)}
\]
and fixing $C$ as a big number, a triangle inequality with \eqref{eq:modulus suppose} yields
\begin{equation}\label{eq:modulus phi-g}
\norm{e^{it\De}(\phi-P_E\phi)}_{L^4_{t,x}([0,\frac{\delta}{\log\#S}]\times\T^2)}\gtrsim\epsilon\norm{\phi}_{L^2(\T^2)}.
\end{equation}
Let $S_{j}:=S_{j}^{0}\setminus E_{j}$ and
\[
f:=\sum_{j\le m}\la_{j}2^{-j/2}\chi_{S_{j}}\ge|(1-\chi_E)\widehat{\phi}|.
\]
Let $M=\lfloor\delta^{-1}\log\#S\rfloor$. Applying Lemma \ref{lem:norm} to \eqref{eq:modulus phi-g}, for $\delta\ll_{C,\epsilon}1$, we have
\[
\frac{1}{M}\sum_{Q\in\mc Q^{\le M}}f(Q)\gtrsim\epsilon^4\norm{\phi}_{L^2(\T^2)}^4\gtrsim\epsilon^4\norm{\la_j}_{\ell^2_{j\le m}}^4,
\]
applying Lemma \ref{lem:conc la*} to which finishes the proof.
\end{proof}
The remainder of this subsection is devoted to the proof of Lemma \ref{lem:conc la*}. We use the notation from \cite{herr2024strichartz}.
We introduce a way to \emph{symmetrize} types of parallelograms without
reducing to $\tau=0$. Let $\mc C_1,\ldots,\mc C_J,J\in\N$ be subsets of $\{(\xi_1,\xi_4)\in\Z^2:\xi_1\neq\xi_4\}$.
Let $j,k=1,\ldots,J$ be any indices. We have
\begin{align*}
  \sum_{\substack{Q=(\xi_{1},\xi_{2},\xi_{3},\xi_{4})\in\mc Q^{\le M}\\
(\xi_{1},\xi_{4})\in\mc C_{j},\,
(\xi_{2},\xi_{3})\in\mc C_{k}
}
}f(Q)&\les\sum_{\xi\in\Z^{2}\setminus\left\{ 0\right\} }\sum_{\substack{Q=(\xi_{1},\xi_{2},\xi_{3},\xi_{4})\in\mc Q^{\le M}\\
\xi_{1}-\xi_{4}=\xi\\
(\xi_{1},\xi_{4})\in\mc C_{j},\,
(\xi_{2},\xi_{3})\in\mc C_{k}
}
}f(Q)\nonumber \\
 & \les\sum_{\xi\in\Z^{2}\setminus\left\{ 0\right\} }\sum_{\substack{\s_{1},\s_{2}\in\Z\\
\left|\s_{1}-\s_{2}\right|\le 2M
}
}\sum_{\substack{(\xi_{1},\xi_{4})\in\mc E_{\xi}^{\s_{1}}\cap\mc C_{j}\\
(\xi_{2},\xi_{3})\in\mc E_{\xi}^{\s_{2}}\cap\mc C_{k}
}
}f(\xi_{1})f(\xi_{4})f(\xi_{2})f(\xi_{3})\nonumber \\
 & \les\sum_{\substack{\xi\in\Z^{2}\setminus\left\{ 0\right\}\\ n\in\Z}}\left(\sum_{\substack{(n-2)M\le\s_{1}\le(n+2)M\\(\xi_{1},\xi_{4})\in\mc E_{\xi}^{\s_{1}}\cap\mc C_{j}}}f(\xi_{1})f(\xi_{4})\sum_{\substack{(n-2)M\le\s_{2}\le(n+2)M\\(\xi_{2},\xi_{3})\in\mc E_{\xi}^{\s_{2}}\cap\mc C_{k}}}f(\xi_{2})f(\xi_{3})\right)\nonumber,
\end{align*}
where we denote by $\mc{E}^\sigma_\xi $ the set of segments $(\xi_1,\xi_4)\in (\Z^2)^2$ such that $\xi_1-\xi_4=\xi$ and $\xi_1\cdot \xi =\sigma$. Now applying the Cauchy-Schwarz inequality and reversing the above expansion, the estimate continues as
\begin{equation}
\les\left(\sum_{\substack{Q=(\xi_{1},\xi_{2},\xi_{3},\xi_{4})\in\mc Q^{\le10M}\\
(\xi_{1},\xi_{4}),(\xi_{2},\xi_{3})\in\mc C_{j}
}
}f(Q)\cdot\sum_{\substack{Q=(\xi_{1},\xi_{2},\xi_{3},\xi_{4})\in\mc Q^{\le10M}\\
(\xi_{1},\xi_{4}),(\xi_{2},\xi_{3})\in\mc C_{k}
}
}f(Q)\right)^{1/2}. \label{eq:sym tau}
\end{equation}
Recall from \cite{herr2024strichartz} that a cross $(\xi_0,\xi_0+\R\xi,\xi_0+\R\xi^{\perp}),\xi_0\in(S_0\cup\cdots\cup S_m),\xi\neq 0$ is of
\[
\begin{cases}
\text{Type 1} & \text{ if } a\ge j/2+C\\
\text{Type 2} & \text{ if } 1\le a<j/2+C\\
\text{Type 3} & \text{ if }a=0,
\end{cases}
\]
where $j$ is the index such that $\xi_0\in S_{j}$, and $a$ is the number
\[
a=\log_{2}\max\left\{ \#\left((\xi_{0}+\R\xi)\cap S_{j}\right),\#\left((\xi_{0}+\R\xi^{\perp})\cap S_{j}\right)\right\}.
\]
For $\tau\ge0$ and $\al,\be=1,2,3$, we denote by $\mc Q_{\al,\be}^{\tau}$
the set of parallelograms $(\xi_{1},\xi_{2},\xi_{3},\xi_{4})\in\mc Q^{\tau}$
such that the crosses $(\xi_{k},\xi_{k}+\R\xi,\xi_{k}+\R\xi^{\perp}),\xi:=\xi_{1}-\xi_{4}\neq 0$
are of type $\al$ for $k=1,2$ and $\be$ for $k=3,4$. Also, let $j_k$ denote the index corresponding to $
\xi_k$, as above.

Let $h=h(\epsilon)\in\N$ be a number to be fixed later. Concerning the case $\max j_{k}-\min j_{k}\le h$, by applying \eqref{eq:para-count} to $\cup_{j=j_0-h}^{j_0+h}S_j$, $j_0=h,\ldots,m-h$, we have
\begin{equation}
\frac{1}{M}\sum_{\substack{Q=(\xi_{1},\xi_{2},\xi_{3},\xi_{4})\in\mc Q^{\le M}\\
\max j_{k}-\min j_{k}\le h
}
}f(Q)\les_{\epsilon, h}\norm{\la_{j}}_{\ell_{j\le m}^{4}}^{4},\label{eq:l4 case}
\end{equation}
which is (up to $\epsilon$-dependence) comparable to $\norm{\la_{j}}_{\ell_{j\le m}^{2}}^{4}$
only if $\max_{j\le m}\la_{j}\gtrsim_{\epsilon,h}\norm{\la_{j}}_{\ell_{j\le m}^{2}}$,
for which we are done since $h$ will depend only on $\epsilon$.
Thus, we assume that the left-hand side of \eqref{eq:l4 case} is sufficiently smaller than $\epsilon^2  \norm{\la_{j}}_{\ell_{j\le m}^{2}}^{4}$. (At the end, we will show a contradiction.)

For $\al,\be=1,2,3$, denote by $\mc C_{\al,\be,\prec}$ the
set of segments $(\xi_{1},\xi_{4})\in(\Z^{2})^{2}$ such that $\xi_{1}\neq\xi_{4}$
and
\[
(\xi_{1},\xi_{1}+(\xi_{4}-\xi_{1})\R,\xi_{1}+(\xi_{4}-\xi_{1})^{\perp}\R)\text{ is a cross of type }\al,
\]
\[
(\xi_{4},\xi_{4}+(\xi_{4}-\xi_{1})\R,\xi_{4}+(\xi_{4}-\xi_{1})^{\perp}\R)\text{ is a cross of type }\be,
\]
\[
\text{the indices }j_{1},j_{4}\text{ such that }\xi_{1}\in S_{j_{1}},\xi_{4}\in S_{j_{4}}\text{ satisfy }j_{1}\le j_{4}-h.
\]
The sets $\mc C_{\al,\be,\succ}$ and $\mc C_{\al,\be,\sim}$ are defined similarly, replacing
$j_{1}\le j_{4}-h$ by $j_1\ge j_4+h$ and $j_4-h<j_1<j_4+h$, respectively.
\begin{lem}
Let $\xi_1,\xi_3\in \Z^2$. Let $\ell\subset\R^2$ be any line. For $M\in\N$, we have
\begin{equation}\label{eq:quadratic poly has sqrt M roots}
\#\left\{\xi_2\in\ell:Q_{\xi_2}=(\xi_1,\xi_2,\xi_3,\xi_1+\xi_3-\xi_2)\in\mc Q^{\le M}\right\}\les\sqrt M.
\end{equation}
\end{lem}
\begin{proof}
Let $\eta\in\Z^2_\irr$ be a vector parallel to $\ell$. Then, every $\xi_2\in\ell\cap\Z^2$ can be written in the form $\xi_2=\xi_0+k\eta$, $k\in\Z$. We have
\[
\tau_{Q_{\xi_2}}=2(\xi_1-\xi_2)\cdot(\xi_2-\xi_3)=2(\xi_1-\xi_0-k\eta)\cdot(\xi_0+k\eta-\xi_3),
\]
which is a quadratic polynomial of $k$ with the quadratic coefficient $-2|\eta|^2$, hence contained in $[M]$ for $O(\sqrt{M})$ integers $k\in\Z$, finishing the proof.
\end{proof}
\begin{lem}\label{lem:count type1}
Let $\epsilon,C,M$, and $f$ be as in Lemma \ref{lem:conc la*}. For $\be=1,2,3$, we have
\begin{equation}
\sum_{Q\in\mc Q_{1,\be}^{\le M}}f(Q)\les_{\epsilon}\sqrt{M}\cdot\norm f_{\ell^{2}}^{4}.\label{eq:type1}
\end{equation}
\end{lem}
\begin{proof}
Let $(\xi_{1},\xi_{3})\in(\Z^{2})^{2}$ be any pair. By the
assumption of Lemma \ref{lem:conc la*}, there exists at most one line
$\ell_{\xi_{1}}$ such that $\#\left(\ell_{\xi_{1}}\cap S_{j_{1}}\right)\ge2^{j_{1}/2+C}$.
Since any parallelogram $(\xi_{1},\xi_{2},\xi_{3},\xi_{4})\in\mc Q_{1,\be}^{\le M}(\supp(f))$
requires $\xi_{1}-\xi_{4}$ to be either parallel or orthogonal to $\ell_{\xi_{1}}$, by \eqref{eq:quadratic poly has sqrt M roots}, we have
\[
\#\left\{ (\xi_{2},\xi_{4})\in(\Z^{2})^{2}:(\xi_{1},\xi_{2},\xi_{3},\xi_{4})\in\cup_{\tau\in[M]}\mc Q_{1,\be}^{\tau}\right\} \les\sqrt{M}.
\]
By the Cauchy-Schwarz inequality, we conclude
\begin{align*}
\sum_{Q\in\mc Q_{1,\be}^{\le M}}f(Q) & =\sum_{Q=(\xi_{1},\xi_{2},\xi_{3},\xi_{4})\in\mc Q_{1,\be}^{\le M}}f(\xi_{1})f(\xi_{3})\cdot f(\xi_{2})f(\xi_{4})\\
 & \le\sum_{Q=(\xi_{1},\xi_{2},\xi_{3},\xi_{4})\in\mc Q_{1,\be}^{\le M}}f(\xi_{1})^{2}f(\xi_{3})^{2}\\
 & \le\sum_{(\xi_{1},\xi_{3})\in(\Z^{2})^{2}}f(\xi_{1})^{2}f(\xi_{3})^{2}\cdot\sqrt M\les\sqrt{M}\cdot\norm f_{\ell^{2}(\Z^{2})}^{4}.
\end{align*}
\end{proof}
Applying \eqref{eq:sym tau} to \eqref{eq:modulus suppose*} with the partition $\left\{ \mc C_{\al,\be,\prec},\mc C_{\al,\be,\sim},\mc C_{\al,\be,\succ}\right\} _{\al,\be=1,2,3}$, we have
\begin{equation}\label{eq:primary Qab}
\frac{1}{M}\max_{\al,\be=2,3}\sum_{\substack{Q=(\xi_{1},\xi_{2},\xi_{3},\xi_{4})\in\mc Q_{\al,\be}^{\le 10M}\\
j_{1}\le j_{4}-h,j_{2}\le j_{3}-h
}
}f(Q)\gtrsim\epsilon\norm{\la_{j}}_{\ell_{j\le m}^{2}}^{4},
\end{equation}
where $j_{k},k=1,\ldots,4$ denotes the index such that $S_{j_{k}}\ni\xi_{k}$. This corresponds to $\mc C_{\al,\be,\prec}$, $\al,\be=2,3$ (or $\mc C_{\al,\be,\succ}$, symmetrically); the cases with $\al=1$ or $\be=1$ are removed by Lemma \ref{lem:count type1} and the cases $\mc C_{\al,\be,\sim}$ cannot be the only large contribution due to the smallness of \eqref{eq:l4 case}.

Let $\mc A$ be the set of segments $(\xi_1,\xi_2)\in(\Z^2)^2$ such that $\{(\xi_4,\xi_3)\in S_{j_1}\times S_{j_2}:(\xi_1,\xi_2,\xi_3,\xi_4)\in\mc Q^0\}=\{(\xi_1,\xi_2)\}$. Let $\mc B:=(\Z^2)^2\setminus\mc A$. Then, once $\tau,j_3,j_4$, and $(\xi_1,\xi_2)$ of a parallelogram $(\xi_1,\xi_2,\xi_3,\xi_4)$ are fixed, the choice of the opposite edge $(\xi_3,\xi_4)\in\mc A$ is unique.
Thus, we have
\begin{align}\label{eq:no AA}
&\frac{1}{M}\max_{\al,\be=2,3}\sum_{\substack{Q=(\xi_{1},\xi_{2},\xi_{3},\xi_{4})\in\mc Q_{\al,\be}^{\le 10M}\\
j_{1}\le j_{4}-h,j_{2}\le j_{3}-h\\ (\xi_3,\xi_4)\in\mc A
}
}f(Q)
\\\nonumber\les&\sum_{\substack{j_1,j_2,j_3,j_4\le m\\ j_{1}\le j_{4}-h,j_{2}\le j_{3}-h}}\la_{j_1}\la_{j_2}\la_{j_3}\la_{j_4}2^{-\frac{1}{2}(j_1+j_2+j_3+j_4)}\cdot 2^{j_1+j_2}
\\\les&\nonumber\left(\sum_{j\le k-h}\la_j\la_k2^{-(k-j)/2}\right)^2
\les o_h(1)\cdot\norm{\la_j}_{\ell^2_{j\le m}}^4.
\end{align}
For $h$ large enough, by the triangle inequality between \eqref{eq:primary Qab} and \eqref{eq:no AA}, we have
\[
\frac{1}{M}\sum_{\substack{Q=(\xi_1,\xi_2,\xi_3,\xi_4)\in\mc Q^{\le 10M}\\ (\xi_3,\xi_4)\in\mc B}}f(Q)\gtrsim_\epsilon\norm{\la_j}_{\ell^2_{j\le m}}^4.
\]
Then, by \eqref{eq:sym tau}, we have
\[
\frac{1}{M}\sum_{\substack{Q=(\xi_1,\xi_2,\xi_3,\xi_4)\in\mc Q^{\le 100M}\\ (\xi_1,\xi_2),(\xi_3,\xi_4)\in\mc B}}f(Q)\gtrsim_\epsilon\norm{\la_j}_{\ell^2_{j\le m}}^4.
\]
By \eqref{eq:sym tau} with the partition $\{\mc C_{\al,\be,\prec}\cap\mc B,\mc C_{\al,\be,\sim}\cap\mc B,\mc C_{\al,\be,\succ}\cap\mc B\}$, we have either
\begin{equation}\label{eq:Qab BB}
\frac{1}{M}\max_{\al,\be=2,3}\sum_{\substack{Q=(\xi_{1},\xi_{2},\xi_{3},\xi_{4})\in\mc Q_{\al,\be}^{\le 1000M}\\
j_{1}\le j_{4}-h,j_{2}\le j_{3}-h\\ (\xi_1,\xi_4),(\xi_2,\xi_3)\in\mc B
}
}f(Q)\gtrsim_\epsilon\norm{\la_{j}}_{\ell_{j\le m}^{2}}^{4},
\end{equation}
or
\begin{equation}\label{eq:Qab BB sim}
\frac{1}{M}\max_{\al,\be=2,3}\sum_{\substack{Q=(\xi_{1},\xi_{2},\xi_{3},\xi_{4})\in\mc Q_{\al,\be}^{\le 1000M}\\
|j_{1}-j_{4}|<h,|j_{2}-j_{3}|<h\\ (\xi_1,\xi_4),(\xi_2,\xi_3)\in\mc B
}
}f(Q)\gtrsim_\epsilon\norm{\la_{j}}_{\ell_{j\le m}^{2}}^{4},
\end{equation}
where the cases where $\al=1$ or $\be=1$ are removed by Lemma \ref{lem:count type1}.

Now we reduce to rectangle-counting just as in \cite{herr2024strichartz}.
Let $(\al,\be)\in\{2,3\}^2$ be a pair saturating \eqref{eq:Qab BB}. Since $\tau_{Q}=2(\xi_{1}-\xi_{2})\cdot(\xi_{1}-\xi_{4})$
is a multiple of $\gcd(\xi_{1}-\xi_{4})$ for any parallelogram $Q=(\xi_{1},\xi_{2},\xi_{3},\xi_{4})$
such that $\xi_{1}\neq\xi_{4}$, we have
\begin{align}\label{eq:esti1}
 & \sum_{\substack{Q=(\xi_{1},\xi_{2},\xi_{3},\xi_{4})\in\mc Q_{\al,\be}^{\le 1000M}\\
j_1\le j_4-h,j_2\le j_3-h\\ (\xi_1,\xi_4),(\xi_2,\xi_3)\in\mc B 
}
}f(Q)=\sum_{\substack{\tau\le 1000M\\ \xi\in\Z^{2}\setminus\left\{ 0\right\} \\
\gcd(\xi)\mid\tau
}
}\sum_{\substack{Q=(\xi_{1},\xi_{2},\xi_{3},\xi_{4})\in\mc Q_{\al,\be}^{\tau}\\
\xi_{1}-\xi_{4}=\xi\\
j_1\le j_4-h,j_2\le j_3-h\\ (\xi_1,\xi_4),(\xi_2,\xi_3)\in\mc B  
}
}f(Q) \\
= & \sum_{\substack{\tau\le 1000M\\ \xi\in\Z^{2}\setminus\left\{ 0\right\} \\
\gcd(\xi)\mid\tau
}
}\sum_{\substack{\s_{1},\s_{2}\in\Z\\
\s_{1}-\s_{2}=\pm\tau/2
}
}\sum_{\substack{
j_1\le j_4-h,j_2\le j_3-h\\ (\xi_1,\xi_4)\in\mc E^{\s_1}_{\xi,\al,\be}\cap \mc B\\(\xi_2,\xi_3)\in\mc E^{\s_2}_{\xi,\al,\be}\cap \mc B}
}f(\xi_1)f(\xi_4)f(\xi_2)f(\xi_3),\nonumber
\end{align}
where $\mc E_{\xi,\al,\be}^{\s},\s\in\Z,\xi\in\Z^{2}\setminus\left\{ 0\right\} $
denotes the set of segments $(\xi_{1},\xi_{4})$ such that $\xi_{1}-\xi_{4}=\xi$,
$\xi_{1}\cdot\xi=\s$, and $(\xi_{k},\xi_{k}+\xi\R,\xi_{k}+\xi^{\perp}\R),k=1,4$
are crosses of type $\al,\be$, respectively. 

Thus, if \eqref{eq:Qab BB} holds, by \eqref{eq:esti1} and Cauchy-Schwarz, we have
\[
\norm{\la_j}_{\ell^2_{j\le m}}^4\les\frac{1}{M}\sum_{\substack{\tau\le 1000M\\ \xi\in\Z^{2}\setminus\left\{ 0\right\} \\
\gcd(\xi)\mid\tau}}\sum_{\s\in\Z}\left(\sum_{\substack{(\xi_1,\xi_4)\in\mc E^{\s}_{\xi,\al,\be}\cap\mc B\\ j_1\le j_4-h}}f(\xi_1)f(\xi_4)\right)^2.
\]
Rewriting in rectangle notation, this is bounded by
\begin{equation}\label{eq:esti2}
\les\frac{1}{M}\sum_{\substack{\tau\le 1000M\\ \xi\in\Z^2\setminus\{0\}\\ \gcd(\xi)\mid\tau}}\sum_{\substack{Q=(\xi_{1},\xi_{2},\xi_{3},\xi_{4})\in\mc Q_{\al,\be}^{0*}\\
\xi_1-\xi_4=\xi\\
j_{1}\le j_{4}-h\text{ and }j_{2}\le j_{3}-h
}
}f(Q),
\end{equation}
where $\mc Q_{\al,\be}^{0*}$ denotes the set of rectangles in $\mc Q_{\al,\be}^0$ whose vertices are all distinct.
Observe that the partial sum of \eqref{eq:esti2} for $\tau=0$ is bounded by
\[
\frac1M\sum_{Q\in\mc Q^0(S)}f(Q)\sim\frac1M\norm{e^{it\De}\F^{-1}f}_{L^4([0,2\pi]\times\T^2)}^4\les\frac{m}M\norm{f}_{\ell^2(S)}^4.
\]
(See \cite[Theorem 1.1]{herr2024strichartz}.)
Hence, in \eqref{eq:esti2}, the summand $\tau=0$ is $o(1)$-negligible for sufficiently large $M\gg m$. Thus, by the estimate
\[
\frac{1}{M}\#\{1\le\tau\le 1000M:\gcd(\xi)\mid\tau\}\les\frac{1}{\gcd(\xi)},
\]
\eqref{eq:esti2} reduces to
\begin{equation}
\sum_{\substack{\substack{Q=(\xi_{1},\xi_{2},\xi_{3},\xi_{4})\in\mc Q_{\al,\be}^{0*}\\
j_{1}\le j_{4}-h\text{ and }j_{2}\le j_{3}-h
}
}
}\frac{f(Q)}{\gcd(\xi_{1}-\xi_{4})}\gtrsim\norm{\la_{j}}_{\ell_{j\le m}^{2}}^{4}.\label{eq:esti3}
\end{equation}
Similarly, if \eqref{eq:Qab BB sim} holds, it holds that
\begin{equation}
\sum_{\substack{\substack{Q=(\xi_{1},\xi_{2},\xi_{3},\xi_{4})\in\mc Q_{\al,\be}^{0*}\\
|j_{1}-j_{4}|<h\text{ and }|j_{2}-j_{3}|<h
}
}
}\frac{f(Q)}{\gcd(\xi_{1}-\xi_{4})}\gtrsim\norm{\la_{j}}_{\ell_{j\le m}^{2}}^{4}.\label{eq:esti3'}
\end{equation}
In summary, either \eqref{eq:esti3} or \eqref{eq:esti3'} holds.
To finish the proof, we recall the main counting inequality of \cite{herr2024strichartz}.
\begin{lem}[Cases II-IV in proof of Lemma 3.3 in \cite{herr2024strichartz}]\label{lem:count23}
Let $j_{1},j_{2},j_{3},j_{4}\in\N$. Denoting $\delta=\frac{1}{10000}$, for $\al,\be=2,3$, we have
\begin{equation}
\sum_{Q=(\xi_{1},\xi_{2},\xi_{3},\xi_{4})\in\mc Q_{\al,\be}^{0*}\cap\left(S_{j_{1}}\times S_{j_{2}}\times S_{j_{3}}\times S_{j_{4}}\right)}\frac{f(Q)}{\gcd(\xi_{1}-\xi_{4})}\les\la_{j_1}\la_{j_2}\la_{j_3}\la_{j_4}2^{-\delta\left(\left|j_{1}-j_{3}\right|+\left|j_{2}-j_{4}\right|\right)}.\label{eq:diag penalty}
\end{equation}
\end{lem}
Since $|j_1-j_3|+|j_2-j_4|\ge(j_4-j_1)+(j_3-j_2)$, we have
\[
\sum_{\substack{j_{1},j_{2},j_{3},j_{4}\le m\\ j_1\le j_4-h,j_2\le j_3-h}}\la_{j_{1}}\la_{j_{2}}\la_{j_{3}}\la_{j_{4}}2^{-\delta(|j_1-j_3|+|j_2-j_4|)}\les\left(\sum_{\substack{j_1,j_4\le m\\ j_1\le j_4-h}}\la_{j_1}\la_{j_4}2^{-\delta(j_4-j_1)}\right)^2 \les o_h(1)\norm{\la_{j}}_{\ell_{j\le m}^{2}}^{4},
\]
which contradicts \eqref{eq:esti3} by \eqref{eq:diag penalty} for large $h$. Thus, \eqref{eq:esti3'} must hold; by \eqref{eq:diag penalty}, we have
\begin{align}\label{eq:esti3' conseq}
\norm{\la_j}_{\ell^2_{j\le m}}^4&\les_h\sum_{\substack{j_{1},j_{2},j_{3},j_{4}\le m\\ |j_1-j_4|<h,|j_2-j_3|<h}}\la_{j_{1}}\la_{j_{2}}\la_{j_{3}}\la_{j_{4}}2^{-\delta(|j_1-j_3|+|j_2-j_4|)}
\\&\les\left(\sum_{\substack{j_1,j_4\le m\\|j_1-j_4|<h}}\la_{j_1}^2\la_{j_4}^2\sum_{\substack{j_3,j_2\le m\\|j_3-j_2|<h}}\la_{j_3}^2\la_{j_2}^2\right)^{1/2}\les_h\sum_{j\le m}\la_j^4,\nonumber
\end{align}
for the second line of which we also used Young's convolution inequality on $\Z^2$.
By \eqref{eq:esti3' conseq}, we have $\norm{\la_j}_{\ell^2_{j\le m}}\les_h\norm{\la_j}_{\ell^4_{j\le m}}$, which yields $\norm{\la_j}_{\ell^2_{j\le m}}\les_h\norm{\la_j}_{\ell^\infty_{j\le m}}$ by an interpolation.
This finishes the proof of Lemma \ref{lem:conc la*}.
\subsection{\label{subsec:conc to multiprogression}Proof of Proposition \ref{prop:conc to P} and Proposition \ref{prop:conc to [N]^2}}

In this subsection, we show Proposition \ref{prop:conc to P} and Proposition \ref{prop:conc to [N]^2}. The key observation is Proposition \ref{prop:combinatorial goal-1}, which deduces from \eqref{eq:P claim} the almost maximality of the number of arithmetic progressions of length $3$ in the Fourier support $S$. The underlying idea is at the discrete-geometric level; loosely speaking, if there exist many parallelograms in $\mc Q^{\le M}(S)$ containing a generic common edge $e\in S^2$, then the set of fourth vertices saturates a positive portion of a long arithmetic progression, to which we apply Roth's Theorem.

For $\xi_1,\xi_2\in S$ such that $\xi_1\neq \xi_2$, we denote by $\kappa(\xi_1,\xi_2)$ the cross count
\[
\kappa(\xi_1,\xi_2):=\max\{\#(S\cap(\xi_1+(\xi_2-\xi_1)^\perp\R)),\#(S\cap(\xi_1+(\xi_2-\xi_1)\R))\}.
\]
For $e=(\eta_{0},\eta_{1})\in(\R^{2})^2$, we denote $\overrightarrow{e}=\eta_{1}-\eta_{0}$.
\begin{lem}
\label{lem:c>m}Let $S\subset\Z^{2}$ be a finite set and $e\in S^{2}$.
Let $(e,e_1),\ldots,(e,e_m)\in\mc Q^{\tau}(S),\tau\in\Z$ be parallelograms. We have $\kappa(e_1)\ge m$.
\end{lem}
\begin{proof}
Denote $e=(\eta,\eta')$ and $e_j=(\eta_j',\eta_j)$.
For each $j=1,\ldots,m$, we have
\[
2(\eta_{1}-\eta_{1}')\cdot(\eta_{j}-\eta_{1})=2(\eta_{1}-\eta_{1}')\cdot(\eta_{j}-\eta)-2(\eta_{1}-\eta_{1}')\cdot(\eta_{1}-\eta)=\tau-\tau=0.
\]
Thus, $\eta_{1},\ldots,\eta_{m}\in\eta_{1}+(\eta_{1}-\eta_{1}')^{\perp}\R$ holds
and hence $\kappa(e_1)=\kappa(\eta_{1},\eta_{1}')\ge m$.
\end{proof}
\begin{lem}
Let $S\subset\Z^{2}$ be a finite set. Let $n=\#S$ and $m,M\in\N$. Let $l$ be the maximum of $\#(S\cap\ell)$ for a line $\ell\subset\R^2$. We have
\begin{equation}\label{eq:SzTr,c, logM}
\#\left\{ (\xi_1,\xi_2,\xi_3,\xi_4)\in\cup_{\tau=1}^{M}\mc Q^{\tau}(S):\kappa(\xi_1,\xi_2)\ge m\right\} \les Mn^{2}\left(\frac{\log m}{m}+\frac{\log M}{n/l}\right).
\end{equation}
\end{lem}
\begin{proof}
For $k\in2^\N$, denote by $\mc C_k$ the set of crosses $(\xi_1,\ell,\ell^\perp)$ such that
\[
k\le \max\{\#(\ell\cap S),\#(\ell^\perp\cap S)\}\le 2k.
\]
Note that $\mc C_k=\emptyset$ for $k> l$. Since $\#\mc C_k$ is bounded by
\[
2\#\{(\xi_1,\ell):\text{$\ell$ is a line through $\xi_1\in S$ and }k\le\#(\ell\cap S)\le 2k\},
\]
by \eqref{eq:SzTr'}, we have
\begin{equation}\label{eq:c_k type I}
\#\mc C_k\les\frac{n^2}{k^2},\qquad k\le\sqrt n
\end{equation}
and
\begin{equation}\label{eq:c_k type II}
\sum_{k\ge\sqrt n}\#\mc C_k\les n.
\end{equation}
For $k\in2^\N$ and $(\xi_1,\ell,\ell^\perp)\in\mc C_k$, we can write
\[
S\cap\ell=\xi_1+\{0,k_1\eta,\ldots,k_r\eta\},
\]
where $r\le 2k$, $\eta\in\Z^2_\irr$, and $k_1,\ldots,k_r\in\Z\setminus\{0\}$.
For each $j\le r$ and $\tau\in k_j\Z$, by Lemma \ref{lem:c>m}, there exist at most $2k$ segments $(\xi_3,\xi_4)$ such that $(\xi_1,\xi_1+k_j\eta,\xi_3,\xi_4)\in\mc Q^\tau(S)$ and $\kappa(\xi_3,\xi_4)\le2k$. Thus, by $r\le 2k$, we have
\begin{align}
&\#\{(\xi_2,\xi_3,\xi_4)\in (\ell\cap S)\times S^2:(\xi_1,\xi_2,\xi_3,\xi_4)\in\cup_{\tau=1}^M\mc Q^\tau(S),\kappa(\xi_3,\xi_4)\le 2k\}\nonumber
\\
\le&\sum_{j\le r} 2k\cdot\#\{\tau\in k_j\Z:1\le\tau\le M\}\les k\sum_{j\le r}\lfloor M/k_j\rfloor\les kM\log\min\{k,M\}.\label{eq:Mmlogm}
\end{align}
By \eqref{eq:c_k type I}, \eqref{eq:c_k type II}, and \eqref{eq:Mmlogm}, taking a summation over $k\in2^\N$ yields
\begin{align*}
&
\#\left\{ (\xi_1,\xi_2,\xi_3,\xi_4)\in\cup_{\tau=1}^{M}\mc Q^{\tau}(S):\kappa(\xi_1,\xi_2)\ge m\right\}
\\ \les&\sum_{m\le k\le l}\#\left\{ (\xi_1,\xi_2,\xi_3,\xi_4)\in\cup_{\tau=1}^{M}\mc Q^{\tau}(S):\kappa(\xi_3,\xi_4)\le\kappa(\xi_1,\xi_2)\in[k,2k]\right\}
\\ \les&\sum_{m\le k\le l}\#\mc C_k\cdot kM\log\min\{k,M\}
\\ \les&\sum_{m\le k\le\sqrt n}\frac{n^2}{k^2}\cdot kM\log \min\{k,M\}+n\cdot lM\log \min\{l,M\}
\\ \les&\frac{n^2}{m}M\log m+n\cdot lM\log M,
\end{align*}
which can be rewritten as \eqref{eq:SzTr,c, logM} and finishes the proof.
\end{proof}
\begin{lem}
Let $S\subset\Z^{2}$ be a finite set. Let $n=\#S$ and $m,M\in\N$. We have
\begin{equation}\label{eq:SzTr,c}
\#\left\{ (\xi_1,\xi_2,\xi_3,\xi_4)\in\cup_{\tau=1}^{M}\mc Q^{\tau}(S):\kappa(\xi_1,\xi_2)\ge m\right\} \les Mn^{2}\left(\frac{\log m}{m}+\frac{1}{\sqrt M}\right).
\end{equation}
\end{lem}
\begin{proof}
Recalling the proof of \eqref{eq:SzTr,c, logM}, \eqref{eq:SzTr,c} reduces to showing
\[
\#\left\{ (\xi_1,\xi_2,\xi_3,\xi_4)\in\cup_{\tau=1}^{M}\mc Q^{\tau}(S):\kappa(\xi_1,\xi_2)\ge\sqrt n\right\} \les \sqrt Mn^{2}.
\]
By \eqref{eq:SzTr'}, the number of crosses $(\xi_1,\ell,\ell^\perp)$ such that
\[
\max\{\#(\ell\cap S),\#(\ell^\perp\cap S)\}\ge\sqrt n
\]
is bounded by $O(n)$. For each such cross $(\xi_1,\ell,\ell^\perp)$ and $\xi_3\in S$, by \eqref{eq:quadratic poly has sqrt M roots} there exist at most $O(\sqrt{M})$ choices of $\xi_2\in\ell$ such that $(\xi_1,\xi_2,\xi_3,\xi_1+\xi_3-\xi_2)\in\cup_{\tau=1}^M\mc Q^\tau(S)$. Thus, \eqref{eq:SzTr,c} is bounded by $O(n\cdot n\cdot \sqrt M)$, finishing the proof.
\end{proof}
\begin{lem}
Let $\epsilon>0,N\gg_{\epsilon}1$, and $E\subset [N]$
be a set such that $\#E\ge\epsilon N$. Then, for any positive integer $K\le N$,
we have
\begin{equation}
\#\left\{(a,b)\in E^2:|a-b|\sim_{\epsilon}K\right\} \sim_{\epsilon}KN.\label{eq:a-b~K}
\end{equation}
\end{lem}
\begin{proof}
Let $I_{1}\cup\cdots\cup I_{n}=[N]$ be a partition into intervals of sizes $\#I_{1},\ldots,\#I_{n}\sim K$. Then $n\sim N/K$ holds. Since $\#\{j\le n:\#(I_j\cap E)\sim_\epsilon K\}\sim_\epsilon n$, we have
\[
\#\left\{ a,b\in E:|a-b|\les_{\epsilon}K\right\} \gtrsim_{\epsilon}K^{2}\cdot n\sim KN.
\]
Since $\#\left\{ a,b\in E:|a-b|\ll_{\epsilon}K\right\}\ll_\epsilon K\cdot\# E\les KN$,
we conclude \eqref{eq:a-b~K}.
\end{proof}
The following is the key ingredient to Propositions \ref{prop:conc to P} and \ref{prop:conc to [N]^2}: 
\begin{prop}
\label{prop:combinatorial goal-1}Let $\epsilon,S$, and $M_{0}$ be as in
Proposition \ref{prop:conc to P}. We have
\begin{equation}
\#\left\{ (\xi_{-1},\xi_{0},\xi_{1})\in S^{3}:\xi_{-1}+\xi_{1}=2\xi_{0}\right\} \sim_{\epsilon}(\#S)^{2}\label{eq:goal1}
\end{equation}
and for any integer $K\gg_\epsilon1$ such that $M_0\gg_K1$,
\begin{equation}
\#\left\{ \eta\in S-S:k\eta\in S-S\text{ for some }k\sim_\epsilon K\right\} \gtrsim_{\epsilon}K^{-\frac{5}{2}}\cdot\#S.\label{eq:goal2}
\end{equation}
\end{prop}
Here, the exponent $-\frac52$ is not special and can be replaced throughout this subsection by any $\rho\in(-3,-2)$.
\begin{proof}
Throughout this proof, every comparability depends on $\epsilon$ by default.

For any $\delta>0$, let $\ell_1,\ldots,\ell_L$ be lines such that $\#(\ell_j\cap S)\ge\delta\#S$, then by \eqref{eq:SzTr}, we have $L=O_\delta(1)$. By \eqref{eq:quadratic poly has sqrt M roots}, for each $j$, the number of $Q\in\cup_{\tau=1}^{M_0}\mc Q^\tau(S)$ containing a vertex on $\ell_j$ is $O(\#S^2\cdot\sqrt{M_0})$. Hence, if we assume $M_0\gg_\delta1$, on the reduced set $\tilde S=S\setminus(\ell_1\cup\cdots\cup\ell_L)$,
\[
\#\left(\cup_{\tau=1}^{M_0}\mc Q^\tau(\tilde S)\right)\gtrsim M_0(\#S)^2
\]
still holds. In this sense, we assume
\begin{equation}\label{eq:l cap S << S}
\text{for every line }\ell\subset\R^2,\qquad\#(\ell\cap S)=o_{M_0}(1)\cdot\#S.
\end{equation}
By \eqref{eq:P claim} and \eqref{eq:SzTr,c}, for $M_0\gg1$,
we can fix a number $m=O(1)$ such that the set
\begin{equation}\label{eq:A_0}
\mc A_{0}:=\left\{ Q\in\cup_{\tau=1}^{M_{0}}\mc Q^\tau(S):\kappa(e)<m\text{ for all edges }e\text{ of }Q\right\} 
\end{equation}
has size $\#\mc A_{0}\gtrsim M_{0}(\#S)^{2}.$
Let $M=M(\epsilon)\gg1$ be an integer to be fixed later. Partitioning $\mc A_{0}$ into $\cup_{\tau=\tau^{*}+1}^{\tau^{*}+M}\mc Q^{\tau}(S)$
for $\tau^{*}\in M\Z$, there exists $\tau^{*}\in\N$ such that
\[
\mc A:=\mc A_{0}\cap\left(\cup_{\tau=\tau^{*}+1}^{\tau^{*}+M}\mc Q^{\tau}(S)\right)
\]
satisfies $\#\mc A\gtrsim M(\#S)^{2}$.
By Lemma \ref{lem:c>m}, for each $e\in S^{2}$ and $\tau\in\gcd(\overrightarrow{e})\Z$,
there exist at most $m=O(1)$ parallelograms $Q\in\mc A\cap\mc Q^{\tau}(S)$
that contain $e$. Thus, we have
\begin{equation}\label{eq:no multiple}
\#\{ Q\in\mc A:e\text{ is an edge of }Q\}\les\#\left([\tau^{*}+1,\tau^{*}+M]\cap\gcd(\overrightarrow{e})\Z\right)\les\lceil M/\gcd(\overrightarrow{e})\rceil.
\end{equation}
Let $\mc E\subset S^{2}$ be the set of segments
\begin{equation}\label{eq:E def}
\mc E:=\left\{ e\in S^{2}:\#\left\{ Q\in\mc A:e\text{ is a longest edge of }Q\right\} \gtrsim M\right\}.
\end{equation}
By \eqref{eq:no multiple}, assuming $M$ large enough, we have 
\begin{equation}
\max_{e\in\mc E}\gcd(\overrightarrow{e})\les1.\label{eq:gcd~1}
\end{equation}
Also, since $\#\mc A\gtrsim M(\#S)^{2}$, by
\eqref{eq:no multiple}, we have
\begin{equation}
\#\mc E\gtrsim\#\mc A/M\gtrsim(\#S)^{2}.\label{eq:E big}
\end{equation}
For $e=(\eta_{0},\eta_{1})\in\mc E$, let $E_{e}$ be the set
\[
E_{e}:=\left\{ \xi\in S:(\eta_0,\eta_1,\xi+\overrightarrow{e},\xi)\in\mc A\text{ and }\left|\overrightarrow{e}\right|\ge\left|\xi-\eta_{0}\right|\right\} .
\]
Since $E_{e}$ is the collection of a vertex of all parallelograms in $\left\{ Q\in\mc A:e\text{ is a longest edge of }Q\right\} $
and $e\in\mc E$, we have $\#E_{e}\gtrsim M$. $E_{e}$ is nested
in the sets
\begin{align*}
E_{e} & \subset\eta_{0}+\left\{ \xi\in\Z^{2}:\left|2\xi\cdot\overrightarrow{e}-\tau_{*}\right|\le M\text{ and }\left|\xi\cdot\overrightarrow{e}^{\perp}\right|\le\left|\overrightarrow{e}\right|^{2}\right\} =:R_{e}\\
 & \subset\eta_{0}+\left\{ \xi\in\Z^{2}:\left|2\xi\cdot\overrightarrow{e}-\tau_{*}\right|\le2M\text{ and }\left|\xi\cdot\overrightarrow{e}^{\perp}\right|\le2\left|\overrightarrow{e}\right|^{2}\right\} =:\tilde R_{e}.
\end{align*}
For each $\tau\in\Z$, since the set $\{\xi\in\Z^2:2\xi\cdot\overrightarrow{e}=\tau\}$
is of the form $\xi_{0}+\frac{1}{\gcd(\overrightarrow{e})}\overrightarrow{e}^\perp\Z,\xi_{0}\in\Z^{2}$, we have
\begin{equation}\label{eq:Re cap tau bound}
\#\{\xi\in\tilde R_e-\eta_0:2\xi\cdot\overrightarrow{e}=\tau\}\les\gcd(\overrightarrow{e})\les 1.
\end{equation}
Thus, we have
\begin{equation}
\#\tilde R_{e}\les M.\label{eq:=000023Re}
\end{equation}
By \eqref{eq:Re cap tau bound}, there exists $\tilde{E_{e}}\subset E_{e}$ such that $\#\tilde E_{e}\gtrsim\#E_{e}\gtrsim M$ and $\left\{\xi\cdot\overrightarrow{e}\right\}_{\xi\in\tilde E_{e}}$
are all distinct. Hence, partitioning $R_{e}\supset\tilde{E_{e}}$
into $n\times n$ congruent rectangular regions for a number $n\gtrsim\sqrt{M}$, by the pigeonhole principle, there exists $\De\xi=\De\xi(e)\in\tilde E_e-\tilde E_e$ such that
\begin{equation}
1\le\De\xi\cdot\overrightarrow{e}\les\sqrt{M}\text{ and }\left|\De\xi\cdot\overrightarrow{e}^{\perp}\right|\les\frac{1}{\sqrt{M}}\left|\overrightarrow{e}\right|^{2}.\label{eq:d xi}
\end{equation}
Now we find a triple as in \eqref{eq:goal1} from $E_{e}$. We cover
$E_{e}$ by arithmetic progressions:
\[
\left\{ I_{j}\right\} _{j\le n}:=\left\{(\eta+(\De\xi)\Z)\cap R_e:\eta\in\Z^2\right\}.
\]
For each $j\le n$, there exists $\xi\in I_j\subset R_e$, then
by \eqref{eq:d xi} we have $\xi+k\De\xi\in\tilde R_{e}$ for $\left|k\right|\ll\sqrt{M}$,
and so the size of the intersection of the extension of $I_j\neq\emptyset$ and $\tilde R_e$ is at least comparable to $\sqrt M$. Thus, by \eqref{eq:=000023Re}, we have $n\les\sqrt{M}$.
Since
\begin{equation}
\sum_{j\le n}\#I_{j}=\#R_{e}\le\#\tilde R_{e}\les M\label{eq:I small}
\end{equation}
and
\begin{equation}
\sum_{j\le n}\#(I_{j}\cap E_{e})=\#E_{e}\gtrsim M,\label{eq:IE big}
\end{equation}
we can choose an index $j$ such that $\#(I_{j}\cap E_{e})\sim\#I_{j}\gtrsim M/n\gtrsim\sqrt{M}$.
Thus, for $M\gg1$, by Proposition \ref{prop:Sz+}, there exists $\{ \xi_{-1}(e),\xi_{0}(e),\xi_{1}(e)\} =\{ \xi_{-1},\xi_{0},\xi_{1}\} \subset I_{j}\cap E_{e}$
such that $\xi_{-1}+\xi_{1}=2\xi_{0}$. By \eqref{eq:d xi} we have
$\tau_{(e,\xi_{-1}+\overrightarrow{e},\xi_{-1})}\neq\tau_{(e,\xi_{1}+\overrightarrow{e},\xi_{1})}$,
so without loss of generality we assume $\tau_{(e,\xi_{-1}+\overrightarrow{e},\xi_{-1})}<\tau_{(e,\xi_{1}+\overrightarrow{e},\xi_{1})}$.
Note that
\begin{equation}
\tau_{(\xi_{-1},\xi_{1},\xi_{1}+\overrightarrow{e},\xi_{-1}+\overrightarrow{e})}=\tau_{(e,\xi_{1}+\overrightarrow{e},\xi_{1})}-\tau_{(e,\xi_{-1}+\overrightarrow{e},\xi_{-1})}\in\left\{ 1,\ldots,M\right\} .\label{eq:tau dif}
\end{equation}

We claim that there exist points $\xi_{\pm 1}^j=\xi_{\pm 1}^j(e)\in E_{e}$, $j=1,\ldots,J$, where $J\sim M\sqrt M$, such that
\begin{equation}
\xi_{1}^{j}-\xi_{-1}^{j}\in\{ k\De\xi:k\in\N,k\sim\sqrt{M}\} .\label{eq:xi1-xi-1}
\end{equation}
Indeed, by \eqref{eq:I small} and $n\les\sqrt{M}$, we have
\[
\sum_{j:\#(I_{j}\cap E_{e})\ll\max\{ \sqrt{M},\#I_{j}\} }\#(I_{j}\cap E_{e})\ll\sum_{j\le n}\max\{\sqrt M,\#I_{j}\}\les M,
\]
applying a triangle inequality to which and \eqref{eq:IE big} yields
\[
\sum_{j:\#(I_{j}\cap E_{e})\sim\#I_{j}\gtrsim \sqrt M}\#(I_{j}\cap E_e)\gtrsim M.
\]
Thus, applying \eqref{eq:a-b~K} to each $I_j$, we have $J\sim M\sqrt M$ such $\xi^j_{\pm 1}$'s. Note that since $\xi_{\pm1}^j\in E_e$,
\begin{equation}
\tau_{(\xi_{-1}^{j},\xi_{1}^{j},\xi_{1}^{j}+\overrightarrow{e},\xi_{-1}^{j}+\overrightarrow{e})}=\tau_{(e,\xi_{1}^{j}+\overrightarrow{e},\xi_{1}^{j})}-\tau_{(e,\xi_{-1}^{j}+\overrightarrow{e},\xi_{-1}^{j})}\in\left\{ 1,\ldots,M\right\}.\label{eq:tau dif-1}
\end{equation}
We are ready to prove \eqref{eq:goal1}. For each $e\in\mc E$, since
$\xi_{\pm1}(e)\in E_{e}$, we have
\[
(e,\xi_{-1}(e)+\overrightarrow{e},\xi_{-1}(e)),(e,\xi_{1}(e)+\overrightarrow{e},\xi_{1}(e))\in\mc A.
\]
Let $\mc B$ be the set
\[
\mc B:=\left\{ (\xi_{-1}(e),\xi_{1}(e),\xi_{1}(e)+\overrightarrow{e},\xi_{-1}(e)+\overrightarrow{e}):e\in\mc E\right\} \subset\cup_{\tau=1}^{M}\mc Q^{\tau}(S),
\]
where the inclusion holds by \eqref{eq:tau dif}.
Since $(\xi_1(e),\xi_1(e)+\overrightarrow{e})$ is an edge of $(e,\xi_{1}(e)+\overrightarrow{e},\xi_{1}(e))\in\mc A$, by \eqref{eq:no multiple},
at most $O(M)$ $e$ determine a common member of $\mc B$.
Thus, by \eqref{eq:E big}, we have
\begin{equation}\label{eq:B big}
\#\mc B\gtrsim\#\mc E/M\gtrsim(\#S)^{2}/M.
\end{equation}
For $M_0\gg_M1$ large enough, by \eqref{eq:SzTr,c, logM} and \eqref{eq:l cap S << S} there exists $m'=O_M(1)$ such that the set
\[
\mc B^{*}:=\left\{ Q\in\mc B:\text{all edges of }Q\text{ have }\kappa(e)\le m'\right\} 
\]
has size $\#\mc B^{*}\gtrsim(\#S)^{2}/M$.
Applying Lemma \ref{lem:c>m} to $\mc B^{*}$, we have
\[
\#\left\{ (\xi_{-1}(e),\xi_{1}(e)):e\in\mc E\right\} \gtrsim\#\mc B^{*}/m'\gtrsim_M(\#S)^{2},
\]
which yields \eqref{eq:goal1} fixing a large number $M=M(\epsilon)$, since $\frac{\xi_{-1}(e)+\xi_{1}(e)}{2}=\xi_{0}(e)\in S$.

Now we prove \eqref{eq:goal2}. We take a similar approach, but we
choose $M=K^{2}$. Let $\mc C$ be the set
\[
\mc C:=\{(\xi_{-1}^{j}(e),\xi_{1}^{j}(e),\xi_{1}^{j}(e)+\overrightarrow{e},\xi_{-1}^{j}(e)+\overrightarrow{e}):e\in\mc E,j\le J\} \subset\cup_{\tau=1}^{M}\mc Q^{\tau}(S),
\]
where the inclusion holds by \eqref{eq:tau dif-1}.
Since $(\xi_{1}^{j}(e),\xi_{1}^{j}(e)+\overrightarrow{e})$ is an
edge of $(e,\xi_{1}^{j}(e)+\overrightarrow{e},\xi_{1}^{j}(e))\in\mc A$,
by \eqref{eq:no multiple}, at most $O(M)$ $e$ determine a common member of $\mc C$. Thus, by \eqref{eq:E big}, we have
\begin{equation}
\#\mc C\gtrsim J\cdot\#\mc E/M\gtrsim\sqrt{M}\cdot(\#S)^{2}.\label{eq:C big}
\end{equation}
Applying \eqref{eq:SzTr,c, logM} to \eqref{eq:C big} with the choice
$m'=M^{6/10}$, by $M(\#S)^{2}\cdot\log m'/m'\ll\sqrt{M}(\#S)^{2}\les\#\mc C$ and \eqref{eq:l cap S << S}, assuming $M_0\gg_M1$
we have
\begin{equation}
\mc C^{*}:=\left\{ Q\in\mc C:\text{all edges of }Q\text{ have }\kappa(e)\le m'\right\} \gtrsim\#\mc C\gtrsim\sqrt{M}\cdot(\#S)^{2}.\label{eq:C* big}
\end{equation}
Since $\mc C^{*}\subset\mc C\subset\cup_{\tau=1}^{M}\mc Q^{\tau}(S)$,
there exists $\tau\in\left\{ 1,\ldots,M\right\} $ such that
\[
\mc C_{\tau}^{*}:=\mc C^{*}\cap\mc Q^{\tau}(S)
\]
has size $\#\mc C_{\tau}^{*}\gtrsim M^{-1/2}\cdot(\#S)^{2}$. By Lemma
\ref{lem:c>m}, we have
\begin{align}\label{eq:(x1j,x-1j) many}
&\#\left\{ (\xi_{1}^{j}(e),\xi_{-1}^{j}(e))\in S^{2}:(\xi_{-1}^{j}(e),\xi_{1}^{j}(e),\xi_{1}^{j}(e)+\overrightarrow{e},\xi_{-1}^{j}(e)+\overrightarrow{e})\in\mc C_{\tau}^{*},e\in\mc E,j\le J\right\} \\
\gtrsim&\#\mc C_{\tau}^{*}/m'\gtrsim M^{-11/10}\cdot(\#S)^{2}.\nonumber
\end{align}
We claim that for $\eta\in\Z^{2}\setminus\{0\}$,
\begin{equation}
\#\{ (\xi_{1}^{j}(e),\xi_{-1}^{j}(e)):\De\xi(e)=\eta,(\xi_{-1}^{j}(e),\xi_{1}^{j}(e),\xi_{1}^{j}(e)+\overrightarrow{e},\xi_{-1}^{j}(e)+\overrightarrow{e})\in\mc C_{\tau}^{*},e\in\mc E,j\le J\} \les M^{1/10}\cdot\#S.\label{eq:dxi multi bound}
\end{equation}
Since $(\xi_{-1}^{j}(e),\xi_{1}^{j}(e),\xi_{1}^{j}(e)+\overrightarrow{e},\xi_{-1}^{j}(e)+\overrightarrow{e})\in\mc C_{\tau}^{*}\subset\mc Q^{\tau}$
implies $\gcd(\xi_{1}^{j}(e)-\xi_{-1}^{j}(e))\mid\tau$, the number
of such $\xi_{1}^{j}(e)-\xi_{-1}^{j}(e)$ is bounded by
\[
\#\left\{ k\in\Z\setminus\left\{ 0\right\} :k\mid\tau\right\} \les\tau^{1/10}\les M^{1/10}.
\]
Also, the number of positions $\xi_{1}^{j}(e)$ is bounded by $\#S$,
so we have \eqref{eq:dxi multi bound}.

By \eqref{eq:(x1j,x-1j) many} and \eqref{eq:dxi multi bound}, we
conclude
\[
\#\left\{ \De\xi(e):e\in\mc E\right\} \gtrsim\frac{M^{-11/10}\cdot(\#S)^{2}}{M^{1/10}\cdot\#S}\gtrsim K^{-24/10}\cdot\#S\gtrsim K^{-5/2}\cdot\#S,
\]
finishing the proof of \eqref{eq:goal2}.
\end{proof}
Applying Proposition \ref{prop:Balog} to \eqref{eq:goal1}, Proposition \ref{prop:conc to P} is immediate. We finish this subsection with the proof of Proposition \ref{prop:conc to [N]^2}.
\begin{proof}[Proof of Proposition \ref{prop:conc to [N]^2}]
Let $(P,\Omega)$ be as in Proposition \ref{prop:conc to [N]^2}. Every comparability in this proof depends on $\epsilon$ in default. We show first that the rank $r$ is at most $2$.
By \eqref{eq:goal2} setting $S=P(\Omega)$, there exists $K\gg1$ such that $(P,\Omega)$ is $2K$-injective and
\[
\#\{\eta\in 2\cdot P(\Omega):k\eta\in 2\cdot P(\Omega) \text{ for some } K/2\le k\le K\}\gtrsim K^{-5/2}\#\Omega,
\]
which yields that the set
\[
\mc X:=\{x\in 2\cdot\Omega:\text{ there exist } K/2\le k\le K\text{ such that }kP(x)\in2\cdot P(\Omega)\}
\]
has size $\#\mc X\gtrsim K^{-5/2}\#\Omega$.
For $x\in\mc X$, since $kP(x)\in 2\cdot P(\Omega)$, there exists $y\in 2\cdot\Omega$ such that $P(y)=P(kx)$. Since $P$ is $2K$-injective and $y,kx\in2K\cdot\Omega$, $kx=y\in 2\cdot\Omega$ holds. Thus, we have $\#\mc X=O(K^{-r}\#\Omega)$, which implies $r\le 5/2$, i.e., $r=1$ or $r=2$.

Now we prove that $P\sim [\sqrt{\#\Omega}]^2$.
Denote $S=P(\Omega)$. We adopt notations used in the proof of Proposition \ref{prop:combinatorial goal-1}. Let $M$ be a sufficiently big constant to be fixed later, so that the proof of Proposition \ref{prop:combinatorial goal-1} works.
Let $\mc E'$ be the set of $e\in\mc E$ such that, for some $j=j(e)\le J$, $(\xi_{-1}^{j}(e),\xi_{1}^{j}(e),\xi_{1}^{j}(e)+\overrightarrow{e},\xi_{-1}^{j}(e)+\overrightarrow{e})\in\mc C^*$. By \eqref{eq:C* big} and $J\sim M\sqrt M$, we have $\#\mc E'\gtrsim_M \#S^2$.
Let $e\in\mc E'$ and $j=j(e)$. By the definition of $\mc C^*$,
the line through $\xi_{1}^{j}(e)$ and $\xi_{-1}^{j}(e)$ contains at most $m'=M^{6/10}$ points of $S$, while $\#E_e\gtrsim M$. Thus, choosing $M=O(1)$ big enough, there exists $\xi(e)\in E_e$ such that $\xi(e)$ and $\xi_{\pm1}^j(e)$ form a triangle. We label $\{\eta_1,\eta_2,\eta_3\}=\{\xi_{\pm 1}^j(e),\xi(e)\}$, so that $|\eta_1-\eta_2|\ge|\eta_1-\eta_3|\ge|\eta_2-\eta_3|$.
Since $\eta_1,\eta_2,\eta_3\in E_e\subset R_e$ and $M=O(1)$, we have
\begin{equation}\label{eq:not collinear}
|(\eta_k-\eta_l)\cdot\overrightarrow{e}|\les 1\text{ and }|\eta_k-\eta_l|\les|\overrightarrow{e}|,\qquad k\neq l.
\end{equation}
By \eqref{eq:not collinear}, we have
\begin{equation}\label{eq:small angle 14}
|\angle(\eta_k-\eta_l,\overrightarrow{e}^\perp)|\les\frac{|(\eta_k-\eta_l)\cdot\overrightarrow{e}|}{|\eta_k-\eta_l|\cdot|\overrightarrow{e}|}\les\frac{1}{|\eta_k-\eta_l|\cdot|\overrightarrow{e}|},\qquad k\neq l.
\end{equation}
Applying a triangle inequality to \eqref{eq:small angle 14}, we have
\begin{equation}\label{eq:small angle 14jk}
|\angle(\eta_1-\eta_2,\eta_1-\eta_3)|\les\frac{1}{|\eta_1-\eta_2|\cdot|\overrightarrow{e}|}+\frac1{|\eta_1-\eta_3|\cdot|\overrightarrow{e}|}\les\frac{1}{|\eta_1-\eta_2|\cdot|\eta_1-\eta_3|},
\end{equation}
(where we used \eqref{eq:not collinear} for the second inequality,) which implies $|\det P|\les1$.

Since the number of segments $e\in S^2$ of lengths $o(1)\cdot\diam (S)$ is $o(1)\cdot\#S^2$ and $\#\mc E'\sim\#S^2$, in the choice of $e\in\mc E'$ above, we may impose further that $|\overrightarrow{e}|\sim\diam(S)$. Observe that in \eqref{eq:small angle 14jk}, since $(\eta_1-\eta_2)\cdot(\eta_1-\eta_3)^\perp$ is a nonzero integer, the right-hand side is actually smaller than the left-hand side, thus all three terms in \eqref{eq:small angle 14jk} are comparable. Hence, we have $|\eta_1-\eta_2|\sim|\overrightarrow{e}|\sim\diam(S)$. By the almost orthogonality \eqref{eq:small angle 14}, this implies that $P(\ul\Omega)$ has $O(1)$ eccentricity, i.e., there exists $N\les\sqrt{|P(\ul\Omega)|}$ such that $P(\ul\Omega)\subset [-N,N]^2$. Now since
\[
|P(\ul\Omega)|\les|\det P|\cdot|\ul\Omega|\les\#\Omega,
\]
$P(\Omega)\subset[O(1)\sqrt{\#\Omega}]^2$ holds and finishes the proof.
\end{proof}
\subsection{Proof of Proposition \ref{prop:conc to Q}}\label{subsec:analytic}
In this subsection, we often view an integer point $z=(a,b)\in\Z^2$ as a complex number $z=a+bi\in \Z[i]\subset\C$ and vice versa. For $z\in\Z^2$, we denote $z\Z^2:=\{zw:w\in\Z^2\}$. For $z,w\in\Z[i]=\Z^2$, we denote $z\mid w$ if $w\in z\Z^2$.
We denote by $\mc P$ the set of primes $p=a+bi\in\Z[i]$ such that $a\neq 0$, $b\neq0$, and $a^2+b^2\neq2$.
\subsubsection{Inverse inequalities on sparse sublattices}
\begin{lem}\label{lem:sparse sublattice}
Let $\epsilon>0$. There exists $p\in\mc P$ such that the following holds:

Let $N\gg_\epsilon M$ be integers such that $\log N\ll_\epsilon M$. Let $I:=[-\pi/M,\pi/M]$. Let $r\ge0$ be any integer such that $|p^r|\ll_\epsilon \sqrt {\frac{M}{\log N}}$.
Let $f:[N]^2\rightarrow\D$ be any function such that
\begin{equation}\label{eq:Pi assump}
\norm{e^{it\De}\F^{-1}f}_{L^4_{t,x}(I\times\T^2)}\ge\epsilon N.
\end{equation}
Then, for any $R\in\N$ such that $N\sqrt{\frac{\log N}{M}}\ll_\epsilon R\ll_\epsilon\frac{N}{|p^r|}$, we have
\begin{equation}\label{eq:sparse sublattice}
\#\left\{\eta\in [R]^2:p^r\eta\in\Z^2_\irr\text{ and }\norm f_{\Pi_{p^r\eta}}\gtrsim_\epsilon \frac{N}{\sqrt{|p^r|R}}\right\}\gtrsim_\epsilon R^2.
\end{equation}
\end{lem}
\begin{proof}
We postpone the choice of $p$ to the end of this proof. For simplicity, we denote $z:=p^r$ in this proof.
Let
$u=e^{it\De}\F^{-1}f$.
For $S\subset \Z^2$, we denote
\[
\mc I(S):=\int_{I\times\T^2}\left|P_S(|u|^2)\right|^2dxdt-\frac{1}{M}\int_{[-\pi,\pi]\times\T^2}\left|P_S(|u|^2)\right|^2dxdt,
\]
which is a (signed) measure on $\Z^2$ supported in $[2N]^2\setminus\{0\}$ since $\widehat{|u|^2}(0)$ is a constant.
Denote by $\F_t$ the temporal Fourier series. For $F:[-\pi,\pi]\rightarrow[0,\infty)$, multiplying a Fej{\'e}r kernel yields
\[
\int_IFdt-\frac{1}{M}\int_{[-\pi,\pi]}Fdt\le\int_IFdt\les\sum_{\tau\in[2M]}|\F_tF(\tau)|,
\]
where the sum can be reduced to $\tau\in[2M]\setminus\{0\}$ since the left-hand side is invariant under constant addition.
Thus, for $\eta\in \Z^2\setminus\{0\}$, we have
\begin{align}\label{eq:|u|^2(eta)}
\mc I(\{\eta\})&\les\frac{1}{M}\sum_{0<|\tau|\le 2M}\left|\F_t\left(|P_{\{\eta\}}(|u|^2)|^2\right)(\tau)\right|
\\&\nonumber
\les\frac{1}{M}\#\{(\xi_1,\xi_2,\xi_3,\xi_4)\in\mc \cup_{0<|\tau|\le 2M}\mc Q^\tau([N]^2):\xi_1-\xi_2=\eta\}
\\&\nonumber
\les\frac{1}{M}\#([N]^2)\cdot\#\{\xi\in [2N]^2:0<|\xi\cdot\eta|\le 2M\}
\\&\nonumber
\les \frac{1}{M}\cdot N^2\cdot\frac{M}{\gcd(\eta)}\cdot\frac{N}{|\eta|/\gcd(\eta)}
\les \frac{N^3}{|\eta|}.
\end{align}
Denote $R_k:=\frac{4N}{k|z|}$. By \eqref{eq:|u|^2(eta)}, for $k,d\in\N$, we have
\begin{equation}\label{eq:|u|^2<}
\mc I(\{kz\eta:\gcd(\eta)\ge d\})\les\sum_{\eta\in  [R_k]^2:\gcd(\eta)\ge d}\frac{N^3}{|kz\eta|}
\end{equation}
and since $\gcd(\eta)\ge d$ implies $\eta\in l\Z^2\setminus\{0\}$ for some $l\ge d$, the estimate continues as
\[
\les \frac{N^3}{k|z|}\sum_{l\ge d}\sum_{\eta\in  [R_k]^2\cap l\Z^2\setminus\{0\}}\frac{1}{|\eta|}
\les\frac{N^3}{k|z|}\cdot R_k\sum_{l\ge d}\frac1{l^2}\les\frac{N^4}{dk^2|z|^2}.
\]
The containment $\eta\in [R_k]^2$ appears since only $\eta\in \Z^2$ satisfying $kz\eta\in\supp(\widehat{|u|^2})\subset [2N]^2$ contributes to \eqref{eq:|u|^2<}, which requires $\eta\in  [R_k]^2$.

For each $k\in\N$, by the Cauchy-Schwarz inequality and \eqref{eq:L4 Stri}, we have
\begin{align}\label{eq:|u|^2>}
\int_{I\times\T^2}\left|P_{kz\Z^2}(|u|^2)\right|^2dxdt
&
=\sum_{z_0\in\Z^2/kz\Z^2}\int_{I\times\T^2}\left|P_{kz\Z^2+z_0}u\right|^2|u|^2dxdt
\\&\nonumber
\ge\frac{1}{\#(\Z^2/kz\Z^2)}\int_{I\times\T^2}|u|^4dxdt\gtrsim_\epsilon \frac{N^4}{k^2|z|^2}.
\end{align}
By the Strichartz estimate on $[-\pi,\pi]\times\T^2$ \cite[(1.2)]{herr2024strichartz}, we have
\begin{equation}\label{eq:L4 Stri [-pi,pi]}
\int_{[-\pi,\pi]\times\T^2}|P_{kz\Z^2}(|u|^2)|^2dxdt\le\int_{[-\pi,\pi]\times\T^2}|u|^4dxdt\les N^4\log N.
\end{equation}
For $k\ll_\epsilon \sqrt {\frac{M}{\log N}}/|z|$, by \eqref{eq:|u|^2>} and \eqref{eq:L4 Stri [-pi,pi]}, we have
\[
\mc I(kz\Z^2)=\int_{I\times\T^2}\left|P_{kz\Z^2}(|u|^2)\right|^2dxdt
-\frac{1}{M}\int_{[-\pi,\pi]\times\T^2}|P_{kz\Z^2}(|u|^2)|^2dxdt\gtrsim_\epsilon\frac{N^4}{k^2|z|^2},
\]
subtracting \eqref{eq:|u|^2<} from which shows that, choosing sufficiently large $d_0=d_0(\epsilon)$,
\begin{equation}\label{eq:I(kzZ^2)}
\mc I(\{kz\eta:\gcd(\eta)\le d_0\})\gtrsim_\epsilon\frac{N^4}{k^2|z|^2}.
\end{equation}
Let $K:=\frac{4N}{R|z|}$. By the range condition of $R$ in this lemma, $1\ll_\epsilon K\ll_\epsilon\sqrt{\frac{M}
{\log N}}/|z|$ holds. Note that $R_K=R$. Let $K':=\lfloor (1+\frac{1}{d_0}) K\rfloor$. Since the products of elements from $\{1,\ldots,d_0\}$ and $\{K,\ldots,K'\}$ are all distinct, by \eqref{eq:I(kzZ^2)}, we have
\begin{equation}\label{eq:IKK'>}
\sum_{\gcd(\eta)\le d_0}\mc I(\{Kz\eta,\ldots,K'z\eta\})=\sum_{k=K}^{K'}\mc I(\{kz\eta:\gcd(\eta)\le d_0\})\gtrsim_\epsilon\frac{N^4}{K|z|^2}.
\end{equation}
On the opposite side, by \eqref{eq:|u|^2(eta)}, for $\eta\in\Z^2\setminus\{0\}$, we have
\begin{equation}\label{eq:|u|^2 eta Z upper bound}
\mc I(\{Kz\eta,\ldots,K'z\eta\})\les\sum_{k=K}^{K'}\frac{N^3}{|kz\eta|}\les\frac{N^3}{|z\eta|}.
\end{equation}
Let $\mc E$ be the set of the almost maximizers to \eqref{eq:|u|^2 eta Z upper bound}, that is,
\[
\mc E:=\left\{\eta\in  [ R]^2:\mc I(\{Kz\eta,\ldots,K'z\eta\})\sim_\epsilon \frac{N^3}{|z\eta|}\right\}.
\]
If $\#\mc E\ll_\epsilon R^2$, by \eqref{eq:|u|^2 eta Z upper bound}, we have
\[
\sum_{\eta\in[R]^2}\mc I(\{Kz\eta,\ldots,K'z\eta\})
\les o(1)\cdot \sum_{\eta\in [R]^2}\frac{N^3}{|z\eta|}
+\sum_{\eta\in\mc E}\frac{N^3}{|z\eta|}
\ll_\epsilon \frac{N^4}{K|z|^2},
\]
which contradicts \eqref{eq:IKK'>}.
Thus, we have $\#\mc E\sim_\epsilon R^2$.

We fit the condition $z\eta\in\Z^2_\irr$. For any $p\in\mc P$ and $d\ge1$, since
\[
\gcd(z\eta)=\gcd(p^r\eta)=\gcd(\eta)\text{ for every }\eta\notin\bar p\Z^2,
\]
we have
\begin{align*}
\#\{\eta\in\mc E:\gcd(z\eta)\ge d\}&\le\#\left([ R]^2\cap\left(\cup_{k\ge d}k\Z^2\cup\bar p \Z^2\right)\right)
\\&
\les R^2\cdot\left(\frac{1}{d}+\frac{1}{|p|^2}\right)\les_\epsilon\#\mc E\cdot\left(\frac{1}{d}+\frac{1}{|p|^2}\right).
\end{align*}
Thus, there exist $d$ and $p\in\mc P$, both depending only on $\epsilon$, such that $|p|>d$ and
\[
\#\{\eta\in\mc E:\gcd(z\eta)\le d\}\ge\#\mc E/2.
\]
Since $z$ is a pure power of $p$, $\gcd(z\eta)\le d<|p|$ implies $\eta/\gcd(z\eta)=\eta/\gcd(\eta)\in \Z^2$. Thus,
\[
\mc E':=\left\{\frac{\eta}{\gcd(z\eta)}:\eta\in\mc E\text{ and }\gcd(z\eta)\le d\right\}\subset\Z^2_\irr
\]
has size $\#\mc E'\gtrsim\#\mc E/d\gtrsim_\epsilon R^2$.
Let $\eta'=\frac{\eta}{\gcd(z\eta)}\in\mc E'$. For each $k\in\N$ and $\xi\in \Z^2$, since
\[
\supp(\F_t^{-1}(\widehat{|u|^2}(k\xi)))\subset \{|\xi'+k\xi|^2-|\xi'|^2:\xi'\in\Z^2\}\subset k\Z,
\]
$\widehat{|u|^2}(k\xi)$ is $\frac{2\pi}{k}$-periodic. Thus, by the shortness $|I|=\frac{2\pi}{M}\les\frac{1}{K}$, we have
\begin{align}\label{eq:N^4/R^2<f(Q)}
\frac{N^3}{R|z|}\les\frac{N^3}{|z\eta|}&\les_\epsilon\mc I(\{Kz\eta,(K+1)z\eta,\ldots,K'z\eta\})
\\&\nonumber
\les\sum_{k=K}^{K'}\int_I\left|\widehat{|u|^2}(kz\eta)\right|^2dt
\les\frac{1}{K}\sum_{k=K}^{K'}\int_{[-\pi,\pi]}\left|\widehat{|u|^2}(kz\eta)\right|^2dt
\\&\nonumber
\les\frac{1}{K}\int_{[-\pi,\pi]\times\T^2}\left|P_{z\eta\Z}(|u|^2)\right|^2dxdt
\les\frac{1}{K}\int_{[-\pi,\pi]\times\T^2}\left|P_{z\eta'\Z}(|u|^2)\right|^2dxdt,
\end{align}
which equals $\frac{1}{K}\norm{f}_{\Pi_{z\eta'}}^4$. \eqref{eq:N^4/R^2<f(Q)} can be rewritten as
\[
\#\left\{\eta'\in [R]^2:z\eta'\in\Z^2_\irr\text{ and }\norm{f}_{\Pi_{z\eta'}}^4\gtrsim_\epsilon \frac{KN^3}{R|z|}\right\}
\gtrsim\#\mc E'\gtrsim_\epsilon R^2,
\]
which implies \eqref{eq:sparse sublattice} by $\frac{KN^3}{|Rz|}\gtrsim\frac{N^4}{|Rz|^2}$, finishing the proof.
\end{proof}
Lemma \ref{lem:sparse sublattice} has two roles. First, \eqref{eq:sparse sublattice} contributes by itself to the quadratic structure of locally quadratic modulations (Section \ref{subsubsec:inverse for locally quadratic}).
The second role is to extract the largeness of Gowers $U^7$-norm. As a preparation to this, we first define the norm to work on:
\begin{defn}\label{defn:N-norm}
For $M,N\in\N$, $\mc N_{M,N}$ is the norm on $f:[N]^2\rightarrow\C$ defined as
\begin{align}
\norm{f}_{\mc N_{M,N}}&:=\frac{1}{N}\left(\frac{1}{M}\int_{[-\pi,\pi]\times\T^2}F_{M}(t)|e^{it\De}\F^{-1}f|^4dxdt\right)^{1/4}
\\&\nonumber=\frac{1}{N}\left(\frac{1}{M}\sum_{\tau\in[M]}\frac{M-|\tau|}{M}\sum_{Q\in\mc Q^\tau}f(Q)\right)^{1/4},
\end{align}
where we denoted by $F_M:\R/2\pi\Z\rightarrow[0,\infty)$ the Fej\'{e}r kernel
\[
\widehat{F_M}(\tau):=\max\{ 0,1-|\tau|/M\}.
\]
\end{defn}
Note that by \eqref{eq:Fejer bound}, for $M\ge\log N$ and $f:[N]^2\rightarrow\C$, we have
\begin{equation}\label{eq:N is indeed weaker than l2}
\norm{f}_{\mc N_{M,N}}\les\norm{f}_{\ell^2}/N.
\end{equation}
For adaptation to Theorem \ref{thm:degree reduction}, we will also use the following induced norm $\ol{\mc N_{M,N}}$ on functions $f:[N+\tilde NN]\rightarrow\C$:
\[
\norm{f}_{\ol{\mc N_{M,N}}}:=\sup_{\mc R\subset[N]^2}\norm{\chi_{\mc R}(f\circ\varphi_N)}_{\mc N_{M,N}},
\]
where $\mc R$ ranges over rectangles $\mc R=I_1\times I_2$, $I_1,I_2\subset[N]$. Here we are using notations in Section \ref{app-sec:deg}. 
The cutoff $\chi_{\mc R}$ is involved for the following alt-stability:
\begin{lem}\label{lem:N-norm is alt-stable}
Let $\{M_N\}$ be any sequence in $\N$ such that $M_N/\log N\rightarrow\infty$.
Then, $\{\ol{\mc N_{M_N,N}}\}_{N\in\N}$ is alt-stable (in the sense of Definition \ref{defn:alt-stable}).
\end{lem}
\begin{proof}
By \eqref{eq:N is indeed weaker than l2}, $\{\mc N_{M_N,N}\}_{N\in\N}$ is $\ell^\infty$-bounded and so is $\{\ol{\mc N_{M_N,N}}\}_{N\in\N}$. Let $\epsilon>0$, $N\in\N$, and $\tilde f:[N+\tilde NN]\rightarrow\D$ be a function such that $\norm{\tilde f}_{\ol{\mc N_{M_N,N}}}\ge\epsilon$. Then, there exists a rectangle $\mc R\subset[N]^2$ such that $f=\chi_{\mc R}(\tilde f\circ\varphi_N)$ satisfies
\[
\norm{f}_{\mc N_{M_N,N}}\ge\epsilon.
\]
Let $M=M_N$ and for $Q\in\mc Q$, denote
\[
g(Q):=\sum_{x\in\Z^2}f(x+Q).
\]
Denote by $\mc Q^\tau_0$, $\tau\in\Z$ the set of $Q=(0,\xi_2,\xi_3,\xi_4)\in\mc Q^\tau$.
We have
\begin{equation}\label{eq:N-norm is alt-stable, 1}
\epsilon^4\le\norm{f}_{\mc N_{M,N}}^4\le\frac{1}{MN^4}\sum_{\tau\in[M]}\widehat{F_M}(\tau)\sum_{Q\in\mc Q^\tau}f(Q)\le\frac{1}{MN^4}\sum_{\tau\in[M]}\widehat{F_M}(\tau)\sum_{Q\in\mc Q^\tau_0}g(Q).
\end{equation}
By \eqref{eq:para-count} setting $S=[4N]^2$ in there, we have
\begin{equation}\label{eq:N-norm is alt-stable, 2}
\frac{1}{MN^2}\sum_{\tau\in[M]}\sum_{Q\in\mc Q^\tau_0([2N]^2)}1\les\frac{1}{MN^4}\#\mc Q^{\le M}([4N]^2)
\les 1.
\end{equation}
Since $0\le\widehat{F_M}\le1$, by \eqref{eq:N-norm is alt-stable, 1}, \eqref{eq:N-norm is alt-stable, 2}, and the Cauchy-Schwarz inequality, we have
\begin{equation}\label{eq:N-norm is alt-stable, 3}
\frac{1}{MN^6}\sum_{\tau\in[M]}\widehat{F_M}(\tau)\sum_{Q\in\mc Q^\tau_0}\left|g(Q)\right|^2\gtrsim_\epsilon1
\end{equation}
because $g(Q)=0$ if $Q\notin\mc Q([2N]^2)$.
By the identity
\[
\sum_{Q\in\mc Q^\tau_0}|g(Q)|^2=\sum_\eta\sum_{Q\in\mc Q^\tau}\Alt_\eta f(Q),\qquad \tau\in\Z,
\]
\eqref{eq:N-norm is alt-stable, 3} can be rewritten as
\begin{equation}\label{eq:N-norm is 2D alt-stable}
\E_{\eta\in[2N]^2}\norm{\Alt_\eta f}_{\mc N_{M_N,N}}^4\gtrsim_\epsilon1.
\end{equation}
By the identity
\begin{equation}\label{eq:Alt R identity}
\Alt_{\eta}f=\Alt_\eta\left(\chi_{\mc R}(\tilde f\circ\varphi_N)\right)=\chi_{\mc R\cap(\mc R-\eta)}\cdot\left((\Alt_{\varphi_N(\eta)} \tilde f)\circ\varphi_N\right),\qquad\eta\in[2N]^2
\end{equation}
and that $\mc R\cap(\mc R-\eta)$ is a rectangle, \eqref{eq:N-norm is 2D alt-stable} transfers to \eqref{eq:alt-stable}, finishing the proof.
\end{proof}
\begin{lem}\label{lem:N-norm largeness}
Let $\{M_N\}$ be a sequence in $\N$ such that $M_N/\log N\rightarrow\infty$. Let $\epsilon>0$. Then, for any $a\in\N$, $N\gg_{\epsilon,a}1$, and function $f:[N]^2\rightarrow\D$ such that
\[
\norm{f}_{\mc N_{M_N,N}}\ge\epsilon,
\]
\begin{equation}\label{eq:N-norm large >> L4 large}
\norm{e^{it\De}\F^{-1}f}_{L^4_{t,x}([-C/M_N,C/M_N]\times\T^2)}\gtrsim_\epsilon N
\end{equation}
holds for some $C\les_\epsilon1$ and
\begin{equation}\label{eq:N-norm large >> Pi norm large}
\#\{\eta\in[N/a]^2\cap\Z^2_\irr:\norm{f}_{\Pi_\eta}\gtrsim_\epsilon\sqrt{aN}\}\gtrsim_\epsilon (N/a)^2.
\end{equation}
\end{lem}
\begin{proof}
Since the Fej\'er kernel $F_{M_N}(t)$ decays rapidly for $|t|\ge C/M_N$, repeating the proof of \eqref{eq:Fejer bound} yields
\[
\int_{[-\pi,\pi]\setminus[-C/M_N,C/M_N]}F_{M_N}(t)|e^{it\De}\F^{-1}f|^4dxdt=o_C(1)\cdot M_N\norm{f}_{\ell^2(\Z^2)}^4=o_C(1)\cdot M_NN^4.
\]
Since $\norm{f}_{\mc N_{M_N,N}}\ge\epsilon$, the same integral on $[-\pi,\pi]$ is $\epsilon$-comparable to $M_NN^4$.
Thus, by the triangle inequality, \eqref{eq:N-norm large >> L4 large} holds. Then, \eqref{eq:N-norm large >> Pi norm large} holds by \eqref{eq:sparse sublattice} plugging in $r=0$ and $R=N/\max\{a,O_\epsilon(1)\}$.
\end{proof}
\begin{lem}\label{lem:Pi}
Let $\epsilon>0$ and $p\in\mc P$. There exists $K\in\N$ satisfying the following:

Let $N\gg_{\epsilon,p}1$ and $f:[N]^2\rightarrow\D$. Assume that
\[
\mc K:=\left\{k\in\{1,\ldots,K\}:\norm f_{\Pi_{p^k}}\ge\epsilon N/|p^k|^{1/2}\right\}
\]
satisfies $\#\mc K\ge\epsilon K$.
Then, there exist $k\in\mc K$ and $\xi\in\Z^2/p^k\Z^2$ such that
\[
\E_{y\in [10N/|p|^k]}\norm{f(p^k(\cdot,y)+\xi)}_{U^3([10N/|p^k|])}\gtrsim_{\epsilon,p} 1.
\]
\end{lem}
\begin{proof}
For each $k\in\mc K$, since $\#\supp(f(p^k\cdot+\xi))\les N^2/|p^k|^2$ for $\xi\in\Z^2/p^k\Z^2$ and
\[
\sum_{\xi\in\Z^2/p^k\Z^2}\norm{f(p^k\cdot+\xi)}_{\Pi}^4=\norm{f}_{\Pi_{p^k}}^4\gtrsim_\epsilon N^4/|p^k|^2,
\]
by Lemma \ref{lem:q~1 => tensor}, there exist $g_{k,\xi},h_{k,\xi}:[10N/|p^k|]\rightarrow\D$, $\xi\in\Z^2/p^k\Z^2$ such that
\begin{equation}\label{eq:<f,F_k>}
\sum_{\xi\in\Z^2/p^k\Z^2}\Re\jp{f(p^k\cdot+\xi),g_{k,\xi}\otimes h_{k,\xi}}_{\ell^2}\gtrsim_\epsilon \frac{N^4/|p^k|^2}{N^2/|p^k|^2}\gtrsim N^2.
\end{equation}
Let $F_k:[N]^2\rightarrow\D$ be the function defined as
\[
F_k(p^k\cdot+\xi):=g_{k,\xi}\otimes h_{k,\xi},\qquad \xi\in\Z^2/p^k\Z^2.
\]
Taking a summation of \eqref{eq:<f,F_k>} over $k\in\mc K$, we have
\begin{equation}\label{eq:<f,sum F_k>}
\Re\jp{f,\sum_{k\in\mc K} F_k}_{\ell^2([N]^2)}\gtrsim_\epsilon KN^2.
\end{equation}
Let $\delta=\delta(\epsilon,p)>0$ be a number to be fixed later. For each $k\in\mc K$, applying Lemma \ref{lem:profile decomp Ud} to $g_{k,\xi}$ setting $S=[10N/|p^k|]$ in there, there exists a function $g_{k,\xi}^*:[10N/|p^k|]\rightarrow\D$ such that
\begin{equation}\label{eq:g* pos}
\inf_{\substack{g:[10N/|p^k|]\rightarrow\D \\ \left|\jp{g,g_{k,\xi}^*}_{\ell^2([10N/|p^k|])}\right|\ge cN/|p^k|}}\norm{g}_{U^3([10N/|p^k|])}\gtrsim_{c,\delta} 1
\end{equation}
holds for every $c>0$ and the remainder $g_{k,\xi,\err}:=g_{k,\xi}-g_{k,\xi}^*$ satisfies
\begin{equation}\label{eq:g err small}
\norm{g_{k,\xi,\err}}_{U^3}\le\delta.
\end{equation}
Define $h_{k,\xi}^*$ and $h_{k,\xi,\err}$ similarly. For each $k\in\mc K$, we decompose $F_k=F_k^*+F_k^\err$ as follows:
\begin{equation}\label{eq:F_k^J}
F_k^*(p^k\cdot+\xi):=g_{k,\xi}^*\otimes h_{k,\xi}^*,\qquad\xi\in\Z^2/p^k\Z^2
\end{equation}
and
\begin{equation}\label{eq:F_k^err}
F_k^\err (p^k\cdot+\xi):=g_{k,\xi}\otimes h_{k,\xi,\err}+g_{k,\xi,\err}\otimes h_{k,\xi}-g_{k,\xi,\err}\otimes h_{k,\xi,\err},\qquad\xi\in\Z^2/p^k\Z^2.\end{equation}
By \eqref{eq:<f,sum F_k>}, we have either
\begin{equation}\label{eq:poss1}
\Re\jp{f,\sum_{k\in\mc K} F_k^*}_{\ell^2([N]^2)}\gtrsim_\epsilon KN^2
\end{equation}
or
\begin{equation}\label{eq:poss2}
\Re\jp{f,\sum_{k\in\mc K} F_k^\err}_{\ell^2([N]^2)}\gtrsim_\epsilon KN^2.
\end{equation}
We claim that \eqref{eq:poss2} does not happen. Since $\norm f_{\ell^2([N]^2)}\les N$, \eqref{eq:poss2} implies
\[
\norm{\sum_{k\in\mc K} F_k^\err}_{\ell^2([N]^2)}\gtrsim_\epsilon KN.
\]
Here since $|F_k^\err|\le 10$, $\max_k\norm{F_k^\err}_{\ell^2([N]^2)}\les N$ also holds, thus if we assume $K\gg_\epsilon1$, there exist $k,k'\in\mc K$ such that $0<k'-k\les_\epsilon1$ and $F_k$ and $F_{k'}$ are not almost $\ell^2$-orthogonal, i.e.,
\begin{equation}\label{eq:not almost ortho}
\left|\jp{F_k^\err,F_{k'}^\err}_{\ell^2([N]^2)}\right|\gtrsim_{\epsilon}N^2.
\end{equation}
Since $F_{k'}^\err(p^{k'}\cdot+\xi)$ is a sum of three tensor products of bounded functions on $[10N/|p^{k'}|]$ for each $\xi\in\Z^2/p^{k'}\Z^2$, by Lemma \ref{lem:tensor => q~1}, \eqref{eq:not almost ortho} implies
\begin{equation}\label{eq:pi' sum}
\sum_{\xi'\in\Z^2/p^{k'}\Z^2}\norm{F_{k}^\err(p^{k'}\cdot+\xi')}_{\Pi}\gtrsim\frac{N^2}{10N/|p^{k'}|}\gtrsim_\epsilon N|p^{k'}|.
\end{equation}
Since $p^{k'}/p^k=O_{\epsilon,p}(1)$, taking an $\ell^4$-partial sum in \eqref{eq:pi' sum} over cosets of $\Z^2/p^{k'-k}\Z^2$ yields
\begin{equation}\label{eq:pi sum}
\sum_{\xi\in\Z^2/p^{k}\Z^2}\norm{F_{k}^\err(p^{k}\cdot+\xi)}_{\Pi_{p^{k'-k}}}\gtrsim_{\epsilon,p}N|p^k|.
\end{equation}
Thus, by pigeonholing there exists $\xi\in\Z^2/p^k\Z^2$ such that
\[
\norm{F_k^\err(p^k\cdot+\xi)}_{\Pi_{p^{k'-k}}}\gtrsim_{\epsilon,p}N/|p^k|,
\]
which can be rewritten as
\[
\norm{g_{k,\xi}\otimes h_{k,\xi,\err}+g_{k,\xi,\err}\otimes h_{k,\xi}-g_{k,\xi,\err}\otimes h_{k,\xi,\err}}_{\Pi_{p^{k'-k}}}\gtrsim_{\epsilon,p}N/|p^k|.
\]
Thus, by Lemma \ref{lem:two tensor => U3} and $p^{k'-k}=O_{\epsilon,p}(1)$, we have
\[
\norm{g_{k,\xi}}_{U^3}\norm{h_{k,\xi,\err}}_{U^3}+\norm{g_{k,\xi,\err}}_{U^3}\norm{h_{k,\xi}}_{U^3}+\norm{g_{k,\xi,\err}}_{U^3}\norm{h_{k,\xi,\err}}_{U^3}\gtrsim_{\epsilon,p}1,
\]
then since $g_{k,\xi}$ and $h_{k,\xi}$ are $O(1)$-bounded, either
\[
\norm{g_{k,\xi,\err}}_{U^3}\gtrsim_{\epsilon,p}1 \text{ or } \norm {h_{k,\xi,\err}}_{U^3}\gtrsim_{\epsilon,p}1
\]
should hold. Thus, choosing $\delta=\delta(\epsilon,p)>0$ small enough, the case \eqref{eq:poss2} can be avoided.

For the case \eqref{eq:poss1}, there exists $k\in\mc K$ such that
\[
\Re\jp{f,F_k^*}_{\ell^2([N]^2)}\gtrsim_\epsilon N^2,
\]
which can be rewritten as
\[
\Re\sum_{\xi\in\Z^2/p^k\Z^2}\jp{f(p^k\cdot+\xi),g_{k,\xi}^*\otimes h_{k,\xi}^*}_{\ell^2}\gtrsim_\epsilon N^2.
\]
Then, by the pigeonhole principle, there exists $\xi\in\Z^2/p^k\Z^2$ such that
\begin{equation}\label{eq:<f,g*h*>}
\Re\jp{f(p^k\cdot+\xi),g_{k,\xi}^*\otimes h_{k,\xi}^*}_{\ell^2}\gtrsim_\epsilon N^2/|p^k|^2.
\end{equation}
By \eqref{eq:<f,g*h*>} and the $\ell^\infty$-boundedness of $g_{k,\xi}^*,h_{k,\xi}^*$, and $f$, we have
\[
\#\left\{y\in [10N/|p^k|]: \left|\jp{f(p^k(\cdot,y)+\xi,g_{k,\xi}^*}_{\ell^2}\right|\gtrsim_\epsilon N/|p^k|\right\}\gtrsim_\epsilon N/|p^k|.
\]
For each such $y$, by \eqref{eq:g* pos}, we have
\[
\norm{f(p^k(\cdot,y)+\xi)}_{U^3}\gtrsim_{\epsilon,p}1,
\]
finishing the proof.
\end{proof}
\begin{lem}\label{lem:8seg positivity}
Let $\epsilon,p,N,M,I,f$ be as in Lemma \ref{lem:sparse sublattice}.
For $N\sqrt {\frac{\log N}{M}}\ll_\epsilon R\ll_\epsilon N$, we have
\begin{equation}\label{eq:U3 positivity}
\#\{(\eta,\zeta)\in [R]^2\times [N]^2:|\eta|\sim_\epsilon R\text{ and }\norm{f(\cdot\eta+\zeta)}_{U^3([10N/|\eta|])}\gtrsim_\epsilon1\}\gtrsim_{\epsilon} R^2N^2.
\end{equation}
\end{lem}
\begin{proof}
Let $K=K(\epsilon,p)\in\N$ be a number to be fixed later.
For each $k\le K$, since $N\sqrt{\frac{\log N}{M}}\ll_\epsilon R\ll_\epsilon\frac{N}{|p^k|}$ holds, by Lemma \ref{lem:sparse sublattice}, we have
\[
\#\left\{\eta\in \Z^2_\irr\cap [R]^2:|\eta|\sim_\epsilon R\text{ and }\norm f_{\Pi_{p^k\eta}}\gtrsim_\epsilon N/(|p^k|^{1/2}R^{1/2})\right\}\gtrsim_\epsilon R^2.
\]
By \eqref{eq:Pi_eta1 eta2}, for each such $\eta$, we have
\[
\sum_{\xi\in\Z^2/\eta\Z^2}\norm {f(\cdot\eta+\xi)}_{\Pi_{p^k}}^4=\norm f_{\Pi_{p^k}}^4\gtrsim_\epsilon N^4/(|p^k|^{2}R^{2}).
\]
On the converse direction, since $|f|\le 1$ and $|\eta|\sim_\epsilon R$, we have the trivial bound
\[
\sup_{\xi\in\Z^2/\eta\Z^2}\norm{f(\cdot\eta+\xi)}_{\Pi_{p^k}}\les_\epsilon N/(|p^k|^{1/2}R).
\]
Thus, denoting $\mc A:=\{(\eta,\xi):\eta\in\Z^2_\irr\cap [R]^2,\,|\eta|\sim_\epsilon R,\text{ and }\xi\in\Z^2/\eta\Z^2\}$, we have
\[
\#\left\{(\eta,\xi)\in \mc A:\norm {f(\cdot\eta+\xi)}_{\Pi_{p^k}}\gtrsim_\epsilon N/(|p^k|^{1/2}R)\right\}\gtrsim_\epsilon R^4.
\]
Taking a union over $k\le K$, we have
\[
\#\left\{(k,\eta,\xi)\in \{1,\ldots,K\}\times\mc A:\norm {f(\cdot\eta+\xi)}_{\Pi_{p^k}}\gtrsim_\epsilon N/(|p^k|^{1/2}R)\right\}\gtrsim_\epsilon KR^4.
\]
Thus, the set
\[
\mc E:=\left\{(\eta,\xi)\in \mc A:\#\left\{k\in\{1,\ldots,K\}:\norm {f(\cdot\eta+\xi)}_{\Pi_{p^k}}\gtrsim_\epsilon N/(|p^k|^{1/2}R)\right\}\gtrsim_\epsilon K\right\}
\]
has size $\#\mc E\sim_\epsilon R^4$. By Lemma \ref{lem:Pi}, fixing $K=K(\epsilon,p)$ large enough, for each $(\eta,\xi)\in\mc E$, there exist $k\le K$ and $\xi'\in\Z^2/p^k\Z^2$ such that
\[
\E_{y\in [10N/|p^k\eta |]}\norm{f((p^k(\cdot,y)+\xi')\eta+\xi)}_{U^3([10N/|p^k\eta|])}\gtrsim_\epsilon1,
\]
which can be rewritten as
\[
\E_{x,y\in [10N/|p^k\eta |]}\norm{f((p^k(\cdot+x,y)+\xi')\eta+\xi)}_{U^3([10N/|p^k\eta|])}\gtrsim_\epsilon1.
\]
Thus, the number of $\zeta\in\xi+\eta\Z^2$ such that $(p^k\eta,\zeta)$ is contained in \eqref{eq:U3 positivity} for some $k\le K\les_\epsilon1$ is at least $\epsilon$-comparable to $(N/R)^2$. Taking a union over $(\eta,\xi)\in\mc E$ finishes the proof.
\end{proof}
Now bringing in Corollary \ref{cor:von Neumann Alt}, we deduce a global $U^7$-largeness of $f$ satisfying \eqref{eq:Pi assump}.
\begin{lem}\label{lem:U^7 large}
Let $\epsilon>0$.
For any $N\gg_\epsilon1$, $M\gg_\epsilon\log N$, and any function $f:[N]^2\rightarrow\mathbb{D}$ satisfying \eqref{eq:Pi assump}, we have
\begin{equation}\label{eq:U8 positivity}
\norm{f}_{U^7}\gtrsim_\epsilon1.
\end{equation}
\end{lem}
\begin{proof}
Let $L=L(\epsilon)$ be a number to be fixed later. Assuming $L\gg1$ and setting $R=\lfloor N/L\rfloor$, by Lemma \ref{lem:8seg positivity} we have
\[
\#\{(\eta,\zeta)\in [R]^2\times [N]^2:|\eta|\ge cR\text{ and }\norm{f(\cdot\eta+\zeta)}_{U^3([10N/cR])}\gtrsim_\epsilon1\}\gtrsim_\epsilon R^2N^2
\]
for some constant $c=c(\epsilon)>0$. Then, we have
\begin{equation}\label{eq:3seg U3}
\sum_{(\eta,\zeta)\in[R]^2\times[N]^2:|\eta|\ge cR}\norm{f(\cdot\eta+\zeta)}_{U^3([10N/cR])}^8\gtrsim_{\epsilon} R^2N^2.
\end{equation}
By \eqref{eq:Uk def explicit}, \eqref{eq:3seg U3} can be rewritten as
\begin{equation}\label{eq:3seg U3 expanded}
\sum_{(\eta,\zeta)\in[R]^2\times[N]^2:|\eta|\ge cR}\E_{a_1,a_2,a_3,k\in[100N/cR]}\Alt_{a_1\eta,a_2\eta,a_3\eta}f(k\eta+\zeta)\gtrsim_{\epsilon} R^2N^2.
\end{equation}
For $m\in\N$, denote by $\mu_m:\Z\rightarrow\R$ the function
\[
\mu_m(n):=\frac{1}{m}\max\left\{0,1-\frac{|n|}{m}\right\}.
\]
Then, since $\Alt_{a_1\eta,a_2\eta,a_3\eta}f(k\eta+\zeta)\neq0$ only if $|a_1|\les N/cR$, by the triangle inequality, there exists a positive integer $K=O_\epsilon(1)$ such that
\begin{equation}\label{eq:3seg U3 expanded, K}
\sum_{(\eta,\zeta)\in[R]^2\times[N]^2:|\eta|\ge cR}\sum_{a_1\in\Z}\mu_{KN/R}(a_1)\E_{a_2,a_3,k\in[100N/cR]}\Alt_{a_1\eta,a_2\eta,a_3\eta}f(k\eta+\zeta)\gtrsim_{\epsilon} R^2N^2.
\end{equation}
Since the Fourier transform of $\mu_{KN/R}$ is nonnegative, each partial sum in \eqref{eq:3seg U3 expanded, K} is nonnegative and the summand set $(\eta,\zeta)$ can be relaxed to $\Z^2\times\Z^2$. Thus, fixing $L$ large enough, by pigeonholing we can choose $a_1,a_2,a_3$, and $k$ such that $0,a_1,\ldots,a_1+a_2+a_3$ are all distinct and
\[
\Re\sum_{\eta,\zeta\in\Z^2}\Alt_{a_1\eta,a_2\eta,a_3\eta}f(\zeta)=\Re\sum_{\eta,\zeta\in\Z^2}\Alt_{a_1\eta,a_2\eta,a_3\eta}f(k\eta+\zeta)\gtrsim_\epsilon R^2N^2\gtrsim_\epsilon N^4.
\]
By Corollary \ref{cor:von Neumann Alt}, the proof finishes.
\end{proof}
\subsubsection{Inverse property of the $\mc N_{M,N}$-norm for degree $2$ nilsequences}\label{subsubsec:inverse for locally quadratic}
In this subsubsection, we provide the inverse property of $\mc N_{M,N}$-norms for nilsequences of degree $2$ (Lemma \ref{lem:LQ, large N => Q, inverse}). The key ingredient is Lemma \ref{lem:U3>>Q(x) final lem}; as a preparation to bring it, we start with showing that $\mc N_{M,N}$-norms satisfy the condition for Lemma \ref{lem:U3>>Q(x) final lem} to hold.
\begin{lem}\label{lem:Pi to (a,1)}
For $\epsilon>0$, $M,a_*\in\N$, and $f:[a_*]^2\rightarrow\D$, we have
\begin{equation}\label{eq:Pi to (a,1)}
M\sum_{a\in M\Z\cap[2a_*]}\sum_{y\in \Z}\left|\sum_{x\in\Z}f(x,y)\ol{f(x,y+a)}\right|^2\ge \norm{f}_\Pi^4.
\end{equation}
\end{lem}
\begin{proof}
By the Cauchy-Schwarz inequality, we have
\[
M\sum_{x,x'\in\Z}\sum_{y_0\in \Z/M\Z}\left|\sum_{y\in\Z\cap(M\Z+y_0)}f(x,y)\ol{f(x',y)}\right|^2\ge \sum_{x,x'\in\Z}\left|\sum_{y\in\Z}f(x,y)\ol{f(x',y)}\right|^2=\norm{f}_\Pi^4,
\]
which can be rewritten as
\[
M\sum_{x,x'\in\Z}\sum_{a\in M\Z}\sum_{y\in\Z}f(x,y)\ol{f(x,y+a)}\ol{f(x',y)}f(x',y+a)\ge\norm{f}_\Pi^4.
\]
This is equivalent to \eqref{eq:Pi to (a,1)} since only $a\in[2a_*]\cap M\Z$ participates in the sum.
\end{proof}
The next lemma shows that the sequence of norms $\{\mc N_{M_N,N}\}_{N\in\N}$ in Lemma \ref{lem:N-norm is alt-stable} satisfies the condition in Lemma \ref{lem:U3>>Q(x) main lem}. We will use the following version of the van der Corput inequality: for $M,N\in\N$ and $F:[N]\rightarrow\C$, since $F(x_0)+\cdots+F(x_0+M-1)\neq 0$ holds only if $x_0\in[M+N]$, by the Cauchy-Schwarz inequality, we have
\begin{align}\label{eq:Van der Corput ineq}
\left|\sum_{x\in\Z}F(x)\right|^2&=\frac{1}{M^2}\left|\sum_{x_0\in\Z}\sum_{x=x_0}^{x_0+M-1}F(x)\right|^2
\les\frac{M+N}{M^2}\sum_{x_0\in\Z}\left|\sum_{x=x_0}^{x_0+M-1}F(x)\right|^2
\\&\nonumber
\les\frac{M+N}{M^2}\sum_{m\in[M]}(M-|m|)\sum_{x\in\Z}F(x)\ol{F(x+m)}.
\end{align}
\begin{lem}\label{lem:cor of Pi to (a,1)}
Let $\{M_N\}$ be a sequence in $\N$ such that $M_N/\log N\rightarrow\infty$. For any $\epsilon>0$, $a_*\in\N$, $N\gg_{\epsilon,\{M_N\},a_*}1$, and $f:[N]^2\rightarrow\D$ such that $\norm{f}_{\mc N_{M_N,N}}\ge\epsilon$, there exist positive integers $a\sim_\epsilon a_*$ and $b=O_{\epsilon}(1)$ such that
\begin{equation}\label{eq:cor of Pi to (a,1)}
\left|\sum_{\eta\in b\Z^2}\sum_{x\in\Z^2}\Alt_{\eta,a\eta^\perp}f(x)\right|\gtrsim_\epsilon\frac{N^4}{a^2}.
\end{equation}
\end{lem}
\begin{proof}
In this proof, every comparability depends on $\epsilon$ in default. Let $M=M(\epsilon)\in\N$ be a number to be fixed later. Up to a comparable update of $a_*$, we assume $a_*\gg M!$. By \eqref{eq:N-norm large >> Pi norm large}, we have
\[
\#\{\eta\in[N/a_*]^2:\norm{f}_{\Pi_\eta}\gtrsim\sqrt{a_*N}\}\gtrsim(N/a_*)^2,
\]
where we can further impose $|\eta|\sim N/a_*$. Thus, denoting $f_{\eta,\zeta}=f(\eta\cdot+\zeta)$, we have
\[
\sum_{|\eta|\sim N/a_*}\sum_{\zeta\in\Z^2/\eta\Z^2}\norm{f_{\eta,\zeta}}_{\Pi}^4\gtrsim N^4.
\]
Whenever $|\eta|\sim N/a_*$, $\diam(\supp(f_{\eta,\zeta}))\les a_*$ holds. Thus, by Lemma \ref{lem:Pi to (a,1)} and pigeonholing on $a$, there exists $a_0\in M!\Z$ such that $|a_0|\sim a_*$ and
\[
\sum_{|\eta|\sim N/a_*}
\sum_{\zeta\in\Z^2/\eta\Z^2}\sum_{y\in\Z}\left|\sum_{x\in\Z}f_{\eta,\zeta}(x,y)\ol{f_{\eta,\zeta}(x,y+a_0)}\right|^2\gtrsim N^4/a_*.
\]
Applying \eqref{eq:Van der Corput ineq}, removing the restriction $|\eta|\sim N/a_*$ (since the summands are positive), fixing $M=O(1)$ big enough, then pigeonholing over $m$, there exists $0\neq m\in[M]$ such that
\[
\Re\sum_{\eta\in\Z^2}
\sum_{\zeta\in\Z^2/\eta\Z^2}\sum_{x,y\in\Z}f_{\eta,\zeta}(x,y)\ol{f_{\eta,\zeta}(x,y+a_0)}\ol{f_{\eta,\zeta}(x+m,y)}f_{\eta,\zeta}(x+m,y+a_0)\gtrsim N^4/a_*^2,
\]
which can be rewritten as
\[
\Re\sum_{\eta\in\Z^2}\sum_{z\in\Z^2}\Alt_{m\eta,a_0\eta^\perp}f(z)\gtrsim N^4/a_*^2.
\]
Setting $a=a_0/m$ and $b=m$ yields \eqref{eq:cor of Pi to (a,1)}. Up to conjugations, we may switch signs to $a,b>0$, finishing the proof.
\end{proof}
\begin{lem}\label{lem:LQ, large N => Q, inverse}
Let $\epsilon>0$. Let $\{M_N\}$ be a sequence such that $M_N/\log N\rightarrow\infty$. Let $X$ be a filtered nilmanifold of degree $2$ and $\F\subset C^0(X;\D)$ be a compact set. For $N\gg_{\epsilon,X,\F}1$, a rectangle $\mc R\subset[N]^2$, and $f\in\mc S_{X,\F}$ such that
\begin{equation}\label{eq:LQ, large N => Q, assump}
\norm{\chi_{\mc R}\cdot f\circ\varphi_N}_{\mc N_{M_N,N}}\ge\epsilon ,
\end{equation}
the following hold:
\begin{enumerate}
    \item\label{enu:enu11}
    For $\delta>0$, assuming further that $N\gg_\delta1$, there exist an integer $J=O_{\epsilon,\delta,X,\F}(1)$, $c_1,\ldots,c_J\in\D$, $t_1,\ldots,t_J\in\R$, and $\xi_1,\ldots,\xi_J\in\R^2$ such that
    \[
    \norm{f\circ\varphi_N-\sum_{j\le J}c_je^{i(t_j|x|^2+\xi_j\cdot x)}}_{\mc N_{M_N,N}}\le\delta .
    \]
    \item\label{enu:enu22}
    There exist $t\in\R$ and $\xi\in\R^2$ such that
\[
\left|\jp{f\circ\varphi_N,e^{i(t|x|^2+\xi\cdot x)}}_{\ell^2([N]^2)}\right|\gtrsim_\epsilon N^2.
\]
\end{enumerate}
\end{lem}
\begin{proof}
Since $\chi_{\mc R}$ can be $\ell^2$-approximated by a sum of linear modulations and $\mc N_{M,N}$-norm is weaker than $\ell^2$-norm \eqref{eq:N is indeed weaker than l2}, by pigeonholing there exists $\xi_0\in\R$ such that
\[
\norm{e^{i\xi_0\cdot x}\cdot f\circ\varphi_N}_{\mc N_{M_N,N}}\gtrsim_\epsilon1.
\]
Since $\mc N_{M,N}$-norm is invariant under linear modulations, we may simply assume $\xi_0=0$.
By Lemma \ref{lem:S^2=sum of LP}, $f\circ\varphi_N$ is $\ell^2$-approximable by a linear combination of locally quadratic modulations $\phi_j$ supported on affine Bohr sets of ranks $O_{\epsilon,\delta,X,\F}(1)$. Among them, for each index $j$ such that $\norm{\phi_j}_{\mc N_{M_N,N}}\gtrsim_{\epsilon,\delta,X,\F}1$, by \eqref{eq:U3>>Q(x) profile} and Lemma \ref{lem:cor of Pi to (a,1)}, $\phi_j$ can be $\ell^2$-approximated in the form of \eqref{enu:enu11}. By triangle inequalities, \eqref{enu:enu11} is immediate.

We prove \eqref{enu:enu22}. By \eqref{eq:N-norm large >> L4 large}, there exists $C\les_\epsilon1$ such that $\norm{e^{it\De}\F^{-1}(f\circ\varphi_N)}_{L^4([-C/M_N,C/M_N]\times\T^2)}\gtrsim_\epsilon N$ holds. By the Fej\'{e}r kernel estimate $F_{M_N}(t)\gtrsim M_N$ for $|t|\les1/M_N$, \eqref{enu:enu11} can be rewritten as
\[
\norm{e^{it\De}\F^{-1}\Bb{f\circ\varphi_N-\sum_{j\le J}c_je^{i(t_j|x|^2+\xi_j\cdot x)}}}_{L^4([-C/M_N,C/M_N]\times\T^2)}\les_\epsilon\delta N.
\]
Thus, choosing $\delta=\delta(\epsilon)>0$ small enough, by Lemma \ref{lem:inverse thm for almost sum of profiles}, there exists a quadruple $(N_*,t_*,x_*,\xi_*)\in2^\N\times\R\times\R^2\times\Z^2$ such that $N_*\sim_\epsilon N$ and
\[
\left|\jp{f\circ\varphi_N(x),\psi(\frac{x-\xi_*}{N_*})e^{i(t_*|x|^2+x_*\cdot x)}}_{\ell^2([N]^2)}\right|\gtrsim_\epsilon N^2,
\]
where $\psi\in C^\infty_0(\R^2)$ is the Littlewood-Paley multiplier. Approximating $\psi(\frac{x-\xi_*}{N_*})$ by a linear combination of linear modulations and pigeonholing as earlier finishes the proof.
\end{proof}
\subsubsection{$(3,2)$-*-reducibility of the $\mc N_{M,N}$-norm}
The goal of this subsubsection is Lemma \ref{lem:(3,2)-*-reducibility}, which shows the $(3,2)$-*-reducibility of the $\mc N_{M,N}$-norm. For $d,k\in\N$ and any set $D\subset\Z^d$, a function $\la:D\rightarrow\R^k$ is \emph{locally linear} if for every $a,b,c,d\in D$ such that $a-b=c-d$,
\[
\la(a)-\la(b)=\la(c)-\la(d).
\]
The following preliminary fact is a standard application of sum set theory on graphs, in accordance with \cite{Gowers-szemeredi4,GreenTao08(U3),GreenTaoZiegler12(Ud)} and many others:
\begin{lem}\label{lem:comp to multi}
Let $d,k\in\N$ and $\epsilon>0$. Let $D\subset\Z^d$ be a finite set and $\la:D\rightarrow\R^k$ be a function such that
\[
\#\{ (a,b,c,d)\in D^{4}: a-b=c-d\text{ and }\la(a)-\la(b)=\la(c)-\la(d)\}\ge\epsilon(\#D)^{3}.
\]
Then there exists  $D'\subset D$ such that $\la\mid_{D'}$ is locally linear and $\#D'\gtrsim_\epsilon\#D$.
\end{lem}
\begin{proof}
Let $\Gamma\subset\R^{d+k}$ be the graph $\Gamma:=\left\{ (a,\la(a)):a\in D\right\}$. Allow every comparability to depend on $\epsilon$. By Proposition \ref{prop:Balog}, there exists a multiprogression $(S,\Omega)\sim\Gamma$ of rank $r=O(1)$. Let $\pi_{\R^d}:\R^{d+k}\rightarrow\R^d$ and $\pi_{\R^k}:\R^{d+k}\rightarrow\R^k$ be the canonical projections.
If there exists $\delta\ll1$ such that
\[
\ker \pi_{\R^d}S\cap(\delta\cdot\Omega)\neq\{0\},
\]
due to large multiplicity we have $\#\pi_{\R^d}S(\Omega)\ll\#D$, contradicting that $(S,\Omega)\sim\Gamma$. Thus, up to shrinking $\Omega$ comparably, we can assume $\pi_{\R^d}S$ to be $10$-injective on $\Omega$. Up to a translation, allowing $(S,\Omega)$ to be affine, we can further assume $\#(S(\Omega)\cap\Gamma)\gtrsim\#\Gamma$. Then, the set
\[
D':=\pi_{\R^d}(S(\Omega)\cap\Gamma)\subset\pi_{\R^d}\Gamma=D
\]
has size $\#D'=\#(S(\Omega)\cap\Gamma)\gtrsim\#\Gamma\gtrsim\#D$. For $a,b,c,d\in D'$ such that $a-b=c-d$, let $u_a\in\Omega$ be such that $\pi_{\R^d}S(u_a)=a$ and define $u_b,u_c,u_d$ similarly. Then, since $\pi_{\R^d}S$ is $10$-injective, $u_a-u_b=u_c-u_d$ holds and thus $\la(a)-\la(b)=\pi_{\R^k}S(u_a)-\pi_{\R^k}S(u_b)=\la(c)-\la(d)$ holds. This finishes the proof.
\end{proof}
\begin{lem}\label{lem:Alt CS}
Let $d,N\in\N$. Let $f:[N]^d\rightarrow\D$ and $\{P_a\}_{a\in\mc A},\mc A\subset[2N]^d$ be a family of functions $P_a:[N]^d\rightarrow\R$. Assume that
\begin{equation}\label{eq:Alt CS assump}
\sum_{a\in\mc A}\left|\sum_{x\in \Z^d}f(x)\overline{f(x+a)}e^{iP_a(x)}\right|\gtrsim N^{2d}.
\end{equation}
Then, we have
\begin{equation}\label{eq:Alt CS}
\sum_{\substack{a,b,c,d\in\mc A\\ a-b=c-d}}\left|\sum_{\substack{y\in\Z^d\\ y,y+d,y+c,y-a+c\in[10N]^d}}e^{i(P_a-P_b)(y-a+c)-i(P_c-P_d)(y)}\right|\gtrsim N^{4d}.
\end{equation}
\end{lem}
\begin{proof}
Observe that \eqref{eq:Alt CS assump} and \eqref{eq:Alt CS} are invariant under constant addition to each $P_a,a\in\mc A$; up to adding a constant $\theta_a\in\R$ to each $P_a$, we assume
\[
\sum_{x\in\Z^d}f(x)\ol{f(x+a)}e^{iP_a(x)}\ge 0,\qquad a\in\mc A.
\]
Then, since $\supp(f)\subset[N]^d$, \eqref{eq:Alt CS assump} can be rewritten as
\[
\sum_{x,a\in[10N]^d}f(x)\ol{f(x+a)}e^{iP_a(x)}\cdot1_{a\in\mc A}\gtrsim N^{2d}.
\]
By Cauchy-Schwarz, we have
\begin{align*}
 & \sum_{u,v\in [10N]^d}f(u)\overline{f(v)}\sum_{x\in [10N]^d}e^{i(P_{v-x}-P_{u-x})(x)}\cdot1_{u-x,v-x\in\mc A}\\
  =&\sum_{x,a,b\in [10N]^d}f(x+b)\overline{f(x+a)}e^{i(P_a-P_b)(x)}\cdot1_{a,b\in\mc A}\\
  \ge&\sum_{x\in [10N]^d}\left|\sum_{a\in [10N]^d}f(x)\overline{f(x+a)}e^{iP_a(x)}\cdot1_{a\in\mc A}\right|^{2}\\
  \ge&\frac{1}{\# [10N]^d}\left|\sum_{x,a\in [10N]^d}f(x)\overline{f(x+a)}e^{iP_a(x)}\cdot1_{a\in\mc A}\right|^{2}\gtrsim N^{3d}.
\end{align*}
Thus, we have
\[
N^{4d} \les\frac{1}{N^{2d}}\left|\sum_{u,v\in [10N]^d}f(u)\overline{f(v)}\sum_{x\in [10N]^d}e^{i(P_{v-x}-P_{u-x})(x)}\cdot1_{u-x,v-x\in\mc A}\right|^2,
\]
from which we can continue to estimate by Cauchy-Schwarz as
\begin{align*}
&\les\sum_{u,v\in[10N]^d}\left|\sum_{x\in[10N]^d}e^{i(P_{v-x}-P_{u-x})(x)}\cdot1_{u-x,v-x\in\mc A}\right|^2\\
& \les\sum_{u,v,x,y\in [10N]^d}e^{i(P_{v-x}-P_{u-x})(x)-i(P_{v-y}-P_{u-y})(y)}\cdot1_{u-x,v-x,u-y,v-y\in\mc A}\\
 & \les\sum_{\substack{a,b,c,d\in \mc A \\ a-b=c-d}}\sum_{\substack{y\in\Z^d\\ y,y+d,y+c,y-a+c\in[10N]^d}}e^{i(P_a-P_b)(y-a+c)-i(P_c-P_d)(y)},
\end{align*}
finishing the proof.
\end{proof}
\begin{lem}
\label{lem:Alt 4}Let $\epsilon>0$, $N\in\N$, and $f:[N]^2\rightarrow\D$. Let $\left\{ q_{a}\right\}_{a\in [2N]^{2}}\subset\R$ and $\left\{ r_{a}\right\}_{a\in [2N]^{2}}\subset\R^2$ be sequences. Assume that
\begin{equation}\label{eq:Alt4}
\sum_{a\in [2N]^2}\left|\jp{\Alt_{a}f,e^{i(q_{a}\left|x\right|^{2}+r_{a}\cdot x)}}_{\ell^2(\Z^2)}\right|\ge\epsilon N^4.
\end{equation}
Then, we have
\begin{equation}
\norm{f}_{U^3}\gtrsim_\epsilon1.\label{eq:Alt 4 goal}
\end{equation}
\end{lem}
\begin{proof}
Throughout the proof, we allow all comparabilities to depend on $\epsilon$. We loosen \eqref{eq:Alt4} as follows: there exists $\mc A\subset [2N]^2$ satisfying
\begin{equation}\label{eq:Alt4 claim}
\sum_{a\in\mc A}\left|\jp{\Alt_{a}f,e^{i(q_{a}\left|x\right|^{2}+r_{a}\cdot x)}}_{\ell^2(\Z^2)}\right|\gtrsim N^4.
\end{equation}
Note that \eqref{eq:Alt4 claim} implies $\#\mc A\gtrsim N^2$. Perturbing to nearest elements in $\frac{2\pi}{CN^{2}}\Z,\frac{2\pi}{CN}\Z^2$, $C\gg 1$ we assume further that $q_{a}\in\frac{2\pi}{CN^{2}}\Z\cap[-\pi,\pi]$ and $r_{a}\in\frac{2\pi}{CN}\Z^2
\cap[-\pi,\pi]$
for $a\in \mc A$. We claim the existence of $q^*\in\R$ such that 
\[
\{a\in\mc A:q_a=q^*\}\gtrsim N^2.
\]

Denote $\mc A^4_*:=\{(a,b,c,d)\in\mc A^4:a-b=c-d\}$.
By Lemma \ref{lem:Alt CS}, we have
\begin{equation}\label{eq:Alt4 wrt P(y)}
\sum_{(a,b,c,d)\in\mc A^4_*}\left|\sum_{\substack{y\in\Z^2\\ y,y+d,y+c,y-a+c\in[10N]^2}}e^{i(q_{a}-q_{b}-q_{c}+q_{d})\left|y\right|^{2}+i\left(2(q_{a}-q_{b})(c-a)+(r_{a}-r_{b}-r_{c}+r_{d})\right)\cdot y}\right|\gtrsim N^8.
\end{equation}
By Lemma \ref{lem:convex Weyl}, there exists an integer $m=O(1)$ such that
\begin{equation}\label{eq:a-b=c-d q0}
\#\{(a,b,c,d)\in\mc A^4_*:\dist(q_a-q_b-q_c+q_d,\frac{2\pi}m\Z)\les\frac{1}{N^2}\}\gtrsim N^6.
\end{equation}
\eqref{eq:a-b=c-d q0} can be rewritten as
\[
\sum_{E\in\mc I}\sum_{\De\in\Z^2}\#\{(a,b)\in\mc A^2:a-b=\De,\quad q_a-q_b\in2\pi\Z+E\}^2\gtrsim N^6,
\]
where $\mc I$ is a partition of $[-\pi,\pi]$ into sets of the form $E=I+\{0,\frac{2\pi}{m},\ldots,\frac{2\pi(m-1)}{m}\}$ where $I\subset[-\pi,\pi]$ is an interval of size $O(\frac{1}{N^2})$. Then, subpartitioning $\mc I$ into $\mc I'$ of intervals of sizes $\frac{2\pi}{CN^2}$, since each $E\in\mc I$ gets partitioned into $O(1)$ intervals, by Cauchy-Schwarz, we have
\[
\sum_{I\in\mc I'}\sum_{\De\in\Z^2}\#\{(a,b)\in\mc A^2:a-b=\De,\quad q_a-q_b\in2\pi\Z+I\}^2\gtrsim N^6,
\]
which can be rewritten as (since $q_a\in\frac{2\pi}{CN^2}\Z\cap[-\pi,\pi]$)
\begin{equation}\label{eq:a-b=c-d q}
\#\{(a,b,c,d)\in\mc A^4_*:q_a-q_b=q_c-q_d\}\gtrsim N^6.
\end{equation}
This argument will be repeated throughout this proof and will be referred to as a \emph{symmetrization}.

By Lemma \ref{lem:comp to multi} and \eqref{eq:a-b=c-d q}, there exists $\mc B\subset\mc A$ such that $\#\mc B\gtrsim N^2$ and $\{q_a\}_{a\in\mc B}$ is locally linear. If \eqref{eq:Alt4 claim} fails over the sum $a\in\mc B$, we replace $\mc A$ by $\mc A\setminus\mc B$ and iterate the process. This process stops in a finite number of steps, and when it stops, \eqref{eq:Alt4 claim} holds with $\mc A$ replaced by $\mc B$. Thus, once we show the existence of $q^*$ for the case $\mc A=\mc B$, that generalizes to arbitrary $\mc A$. 

Hereafter we assume $\mc A=\mc B$.
Applying Lemma \ref{lem:convex Weyl} to \eqref{eq:Alt4 wrt P(y)}, we have
\begin{equation}\label{eq:q_ar_a}
\#\{(a,b,c,d)\in\mc A^4_*:\dist(2(q_a-q_b)(c-a)+r_a-r_b-r_c+r_d,\frac{2\pi}{m}\Z^2)\les\frac{1}{N}\}\gtrsim N^6.\end{equation}
for some $m=O(1)$.
Denote by $\pi_{1},\pi_{2}:\R^2\rightarrow\R$ the projections to the first and second coordinate, respectively. Using the symmetric representation for $(a,b,c,d)\in\mc A^4_*$
\[
2(q_a-q_b)(c-a)+r_a-r_b-r_c+r_d=\left((r_a-r_b)-2(q_a-q_b)a\right)-\left((r_c-r_d)-2(q_c-q_d)c\right),
\]
(where we used $q_a-q_b=q_c-q_d$,) \eqref{eq:q_ar_a} can be symmetrized by Cauchy-Schwarz to
\[
\#\{(a,b,c,d)\in\mc A^4_*:\pi_1a=\pi_1c,\quad\dist(2(q_a-q_b)(c-a)+r_a-r_b-r_c+r_d,\frac{2\pi}{m}\Z^2)\les\frac{1}{N}\}\gtrsim N^5.
\]
Here, $\pi_1(2(q_a-q_b)(c-a))=0$ holds by $\pi_1a=\pi_1c$. Thus, we have
\begin{equation}\label{eq:CS N5}
\#\{(a,b,c,d)\in\mc A^4_*:\pi_1a=\pi_1c,\quad \dist(\pi_1(r_a-r_b-r_c+r_d),\frac{2\pi}{m}\Z)\les\frac{1}{N}\}\gtrsim N^5.
\end{equation}
Since $r_a\in\frac{2\pi}{CN}\Z\cap[-\pi,\pi]$ for $a\in\mc A$, we can symmetrize \eqref{eq:CS N5} with respective to each coordinate as follows:
\begin{equation}\label{eq:CS pi2 1111}
\#\{(a,b,c,d)\in\mc A^4_*:\pi_1a=\pi_1c,\quad \pi_1a=\pi_1b,\quad \pi_1(r_a-r_b-r_c+r_d)=0\}\gtrsim N^4
\end{equation}
and
\begin{equation}\label{eq:CS pi2 1122}
\#\{(a,b,c,d)\in\mc A^4_*:\pi_1a=\pi_1c,\quad \pi_2a=\pi_2b,\quad \pi_1(r_a-r_b-r_c+r_d)=0\}\gtrsim N^4.
\end{equation}
For $k\in[2N]$, denote $\mc A_k=\{l\in[2N]:(k,l)\in\mc A\}$. Let $\mc E$ be the set of $k\in[2N]$ such that
\[
\#\{(l_1,l_2,l_3,l_4)\in\mc A_k^4:l_1-l_2=l_3-l_4\text{ and }\pi_1(r_{(k,l_1)}-r_{(k,l_2)}-r_{(k,l_3)}+r_{(k,l_4)})=0\}\ge\delta N^3.
\]
Choosing $\delta=\delta(\epsilon)>0$ small enough, by \eqref{eq:CS pi2 1111}, $\#\mc E\gtrsim N$ holds. For each $k\in\mc E$, by Lemma \ref{lem:comp to multi}, there exists $\mc B_k\subset\mc A_k$ such that $\#\mc B_k\gtrsim N$ and $\{\pi_1r_{(k,\cdot)}\}_{\mc B_k}$ is locally linear. Let $\mc B:=\cup_{k\in\mc E}(\{k\}\times\mc B_k)\subset\mc A$; since $\#\mc B\sim N^2$, as earlier, we can reduce to the case $\mc A=\mc B$.

Hereafter we use the convention that an identity containing subscripts such as $(k,l)$ is true only if every subscript lies in $\mc A$.
Pigeonholing on \eqref{eq:CS pi2 1122}, there exist $k_0,l_0\in[2N]$ such that
\[
\#\{(a,b,c,d)\in\mc A^4_*:\pi_1a=\pi_1c=k_0,\quad \pi_2a=\pi_2b=l_0,\quad \pi_1(r_a-r_b-r_c+r_d)=0\}\gtrsim N^2.
\]
Equivalently, the set
\[
\mc B=\{(k,l)\in[2N]^2:\pi_1(r_{(k_0,l_0)}-r_{(k_0,l)})=\pi_1(r_{(k,l_0)}-r_{(k,l)})\}
\]
has size $\#\mc B\gtrsim N^2$. Reducing to $\mc A=\mc B$ as earlier, we may assume for every $(k,l)\in\mc A$ that
\begin{equation}\label{eq:CS pi2 kl forall}
\pi_1(r_{(k_0,l_0)}-r_{(k_0,l)})=\pi_1(r_{(k,l_0)}-r_{(k,l)}).
\end{equation}
Then, for every $l_1,l_2,l_3,l_4\in[2N]$ such that $l_1-l_2-l_3+l_4=0$ and $k_1,k_2\in[2N]$, if $(k_1,l_1),(k_1,l_3),(k_2,l_2),(k_2,l_4)\in\mc A$, by \eqref{eq:CS pi2 kl forall} and the local linearity of $\pi_1 r_{(k_0,\cdot)}$, we have
\begin{equation}\label{eq:CS pi2=0 final}
\pi_1(r_{(k_1,l_1)}-r_{(k_1,l_3)}-r_{(k_2,l_2)}+r_{(k_2,l_4)})=\pi_1(r_{(k_0,l_1)}-r_{(k_0,l_3)}-r_{(k_0,l_2)}+r_{(k_0,l_4)})=0.
\end{equation}
Now we symmetrize \eqref{eq:q_ar_a} matching $\pi_1a=\pi_1b$. We use the symmetric representation
\[
2(q_a-q_b)(c-a)+r_a-r_b-r_c+r_d=(2q_a(c-a)+r_a-r_c)-(2q_b(d-b)+r_b-r_d),
\]
then symmetrizing \eqref{eq:q_ar_a} as earlier and simplifying by \eqref{eq:CS pi2=0 final}, we have
\begin{equation}\label{eq:a1=b1 implies qa-qb (c-a) <1/N}
\#\{(a,b,c,d)\in\mc A_*^4:\pi_1a=\pi_1b,\quad\dist(2(q_a-q_b)\pi_1(c-a),\frac{2\pi}{m}\Z)\les\frac{1}{N}\}\gtrsim N^5
\end{equation}
for some $m=O(1)$. Applying Lemma \ref{lem:Weyl bound implies Lipschitz} to \eqref{eq:a1=b1 implies qa-qb (c-a) <1/N} with respect to $c$ yields
\[
\#\{(a,b)\in\mc A^2:\pi_1a=\pi_1b,\quad\dist(q_a-q_b,\frac{2\pi}{m}\Z)\les\frac{1}{N^2}\}\gtrsim N^3
\]
for some $m=O(1)$. Since $q_a\in\frac{2\pi}{CN^2}\Z\cap[-\pi,\pi]$, symmetrizing as earlier, we have
\[
\#\{(a,b)\in\mc A^2:\pi_1a=\pi_1b,\quad q_a=q_b\}\gtrsim N^3.
\]
Up to a row-wise further reduction of $\mc A$, we may assume that $q_a=q_b$ holds for every $a,b\in\mc A$ such that $\pi_1a=\pi_1b$. Working similarly on $\pi_2$, we further assume $q_a=q_b$ for every $a,b\in\mc A$ such that $\pi_2a=\pi_2b$. By these, partitioning $\mc A$ into equivalence classes $\mc A_1,\mc A_2,\ldots$ of common $q_a$, $\pi_1(\mc A_j)\cap\pi_1(\mc A_k)=\emptyset$ and $\pi_2(\mc A_j)\cap\pi_2(\mc A_k)=\emptyset$ hold. Since $\cup_j\mc A_j=\mc A\subset[2N]^2$ has size $\#\mc A\sim N^2$, there exists an index $j$ such that $\#\mc A_j\sim N^2$. Reducing to $\mc A=\mc A_j$, $q_c=q^*\in\R$ holds for every $c\in\mc A$, as claimed. Now \eqref{eq:Alt4 claim} reduces to
\[
\sum_{a\in\mc A}\left|\jp{\Alt_af,e^{i(q^*|x|^2+r_a\cdot x)}}_{\ell^2(\Z^2)}\right|\gtrsim N^4.
\]
For each $a\in[2N]^2$, assign $\theta_a\in\R$ such that
\[
\jp{\Alt_af,e^{i(q^*|x|^2+r_a\cdot x+\theta_a)}}_{\ell^2(\Z^2)}\ge 0
\]
is real-valued. Then, we have
\[
\sum_{a\in[2N]^2}\jp{\Alt_af,e^{i(q^*|x|^2+r_a\cdot x+\theta_a)}}_{\ell^2(\Z^2)}\gtrsim N^4,
\]
which can be rewritten as
\begin{equation}\label{eq:CS q*}
\E_{x\in[2N]^2}\E_{a\in[2N]^2}f(x)\ol{f(x+a)}e^{-i(q^*|x|^2+r_a\cdot x+\theta_a)}\gtrsim1.
\end{equation}
Applying Cauchy-Schwarz to \eqref{eq:CS q*}, we have
\begin{align}\label{eq:CS q* killing}
1&\les\E_{x\in[2N]^2}\left|\E_{a\in[2N]^2}f(x)\ol{f(x+a)}e^{-i(q^*|x|^2+r_a\cdot x+\theta_a)}\right|^2\\
&\les\E_{x\in[2N]^2}\E_{a,a'\in[2N]^2}f(x+a')\ol{f(x+a)}e^{-i((r_a-r_{a'})\cdot x+\theta_a-\theta_{a'})}.\nonumber
\end{align}
Since the $U^2$-norm is invariant over linear modulations and stronger than $U^1$, \eqref{eq:CS q* killing} can be rewritten as
\[
\E_{a,a'\in[2N]^2}\norm{f(x+a')\ol{f(x+a)}}_{U^2}\gtrsim1.
\]
Since the $U^2$-norm is invariant over translations, we have
\[
\E_{\eta\in[4N]^2}\norm{\Alt_\eta f}_{U^2}\gtrsim1,
\]
which implies by \eqref{eq:Uk def inductive} that $\norm{f}_{U^3}\gtrsim1$, finishing the proof.
\end{proof}
\begin{lem}\label{lem:(3,2)-*-reducibility}
Let $\{M_N\}$ be any sequence such that $M_N/\log N\rightarrow\infty$. Then, $\{\ol{\mc N_{M_N,N}}\}$ is $(3,2)$-*-reducible.
\end{lem}
\begin{proof}
Let $X$ be a filtered nilmanifold of degree at most $3$ and $\F\subset C^0_*(X;\D)$ be compact. For $f\in\mc S_{X,\F}$ and $N\gg1$ such that $\norm{f}_{\ol{\mc N_{M_N,N}}}\ge\epsilon$, by \eqref{eq:N-norm is 2D alt-stable},
\begin{equation}\label{eq:(3,2)-*-reducibility, a count}
\#\{a\in[2N]^2:\norm{\Alt_a(\chi_{\mc R}\cdot f\circ\varphi_N)}_{\mc N_{M_N,N}}\gtrsim_\epsilon1\}\gtrsim_\epsilon N^2
\end{equation}
holds for some rectangle $\mc R\subset[N]^2$. Here, by the identity
\[
\Alt_a(\chi_{\mc R}\cdot f\circ\varphi_N)=\chi_{\mc R\cap(\mc R-a)}\cdot\Alt_a( f\circ\varphi_N)=\chi_{\mc R\cap(\mc R-a)}\cdot(\Alt_{\varphi_N(a)}f)\circ\varphi_N,
\]
viewing $\Alt_{\varphi_N(a)}f$ as a nilsequence of degree $2$ by Lemma \ref{lem:SXF Alt reduce}, Lemma \ref{lem:LQ, large N => Q, inverse} is applicable; applying Lemma \ref{lem:LQ, large N => Q, inverse} to \eqref{eq:(3,2)-*-reducibility, a count}, there exist $\{t_a\}\subset\R$ and $\{\xi_a\}\subset\R^2$ such that
\[
\#\{a\in[2N]^2:|\jp{\Alt_a(f\circ\varphi_N),e^{i(t_a|x|^2+\xi_a\cdot x)}}_{\ell^2([N]^2)}|\gtrsim_\epsilon N^2\}\gtrsim_\epsilon N^2.
\]
Then, by Lemma \ref{lem:Alt 4}, we have
\[
\norm{f\circ\varphi_N}_{U^3([N]^2)}\gtrsim_\epsilon1,
\]
which implies $\norm{f}_{U^3([N+\tilde NN])}\gtrsim_\epsilon1$ by Lemma \ref{lem:Ud pullback} and finishes the proof.
\end{proof}
\subsubsection{Proof of Proposition \ref{prop:conc to Q}}\label{subsubsec:proof of conc to Q}

By Definition \ref{defn:N-norm} and the estimate of the Fej\'{e}r kernel $|F_M(t)|\gtrsim M$ for $|t|\le 1/M$, it suffices to show the following:
\begin{prop}\label{prop:conc to Q wrt N}
Let $\epsilon>0$. Let $\{M_N\}$ be any sequence such that $M_N/\log N\rightarrow\infty$. Then, there exists $J\in\N$ satisfying the following:

Let $N\gg1$ and $f:[N]^2\rightarrow\D$. There exist $c_j\in\D,t_j\in\R,\xi_j\in\R^2$, $j=1,\ldots,J$ such that
\begin{equation}\label{eq:final approx}
\norm{f-\sum_{j\le J}c_j e^{i(t_j|x|^2+\xi_j\cdot x)}}_{\mc N_{M_N,N}}\le\epsilon.
\end{equation}
\end{prop}
\begin{proof}
Let $\delta=\delta(\epsilon)>0$ be a number to be fixed later.
By Lemma \ref{lem:profile decomp Ud}, there exist a filtered nilmanifold $X$ of degree $6$ and a compact set $\F\subset C^0(X;\C)$ such that for each $N$, there exists $g\in\mc S_{X,\F}$ such that $\norm{\iota_Nf-\chi_{[N]+\tilde N[N]}g}_{U^7([N+\tilde NN])}\le\delta$, which can be transferred to $\norm{f-g\circ\varphi_N}_{U^7([N]^2)}\les\delta$. Then, by Lemma \ref{lem:U^7 large} and \eqref{eq:N-norm large >> L4 large}, $\norm{f-g\circ\varphi_N}_{\mc N_{M_N,N}}=o_\delta(1)$ holds.

By Lemma \ref{lem:profile decomp Ud}, there exists a filtered nilmanifold $X'$ of degree $2$ and a compact set $\F'\subset C^0(X';\C)$ such that for each $N$, there exists $h\in\mc S_{X',\F'}$ such that $\norm{g-h}_{U^3([N+\tilde NN])}\le\delta$.
Since $\{\ol{\mc N_{M_N,N}}\}$ is alt-stable (Lemma \ref{lem:N-norm is alt-stable}) and $(3,2)$-*-reducible (Lemma \ref{lem:(3,2)-*-reducibility}), by Theorem \ref{thm:degree reduction}, $\{\ol{\mc N_{M_N,N}}\}$ is $(6,2)$-reducible. Thus, $\norm{(g-h)\circ\varphi_N}_{\mc N_{M_N,N}}\le\norm{g-h}_{\ol{\mc N_{M_N,N}}}=o_\delta(1)$ holds.

By Lemma \ref{lem:LQ, large N => Q, inverse}, $h\circ\varphi_N$ can be approximated in the form \eqref{eq:final approx}. Taking $\delta=\delta(\epsilon)>0$ small enough, by triangle inequalities, the proof finishes.
\end{proof}
\section{Global well-posedness of \eqref{eq:NLS}}\label{sec:GWP}
In this section, we prove the large data global well-posedness of
(\ref{eq:NLS}). This is consistent with the expectation that an inverse Strichartz estimate together with local well-posedness shown via a fixed-point argument leads to the large-data GWP based on the following corresponding GWP and $L^4$-norm bounds known on $\R^2$:
\begin{prop}[\cite{dodson-def}]
\label{prop:defocusing R^2}Let $M>0$. Let $u_{0}\in L^{2}(\R^{2})$
be a data such that $\norm{u_{0}}_{L^{2}}^{2}\le M$. There exists
a unique Duhamel solution $u\in C^{0}L^{2}\cap L_{t,x}^{4}(\R\times\R^{2})$
to the cubic defocusing NLS on $\R^{2}$, which is global and scatters
in both time directions. Moreover, we have
\[
\norm u_{L_{t,x}^{4}}\les_{M}1.
\]
\end{prop}
\begin{prop}[\cite{dodson-foc}]
\label{prop:focusing R^2}Let $0<M<\norm Q_{L^{2}}^{2}$. Let $u_{0}\in L^{2}(\R^{2})$
be a data such that $\norm{u_{0}}_{L^{2}}^{2}\le M$. Then, there
exists a unique Duhamel solution $u\in C^{0}L^{2}\cap L_{t,x}^{4}(\R\times\R^{2})$
to the cubic focusing NLS on $\R^{2}$, which is global and scatters
in both time directions. Moreover, we have
\[
\norm u_{L_{t,x}^{4}}\les_{M}1.
\]
\end{prop}
For the discussion on the energy-critical cases, see \cite{ionescu2012energy,ionescu2012global,YUE2021754,Kwak2024global}. In this section, we follow the argument in \cite{Kwak2024global}, for the benefit that we do not need to estimate interactions between profiles.
\subsection{Cutoff solutions}
To consider a solution $u$ to \eqref{eq:NLS} locally within a short time
interval $I\subset\R$, it is often convenient to consider an extension
of $u\mid_{I}$ by linear evolutions on $\R\setminus I$. For this,
we introduce the concept of a cutoff solution.

Denote by $\mc S'(\R^2)$ the space of tempered distributions. A pair $(u,I)$, where $u\in C^{0}\mc S'\cap L^3_{t,x,\loc}(\R\times\R^2)$
and $I\subset\R$ is an interval, is a \emph{cutoff solution }to \eqref{eq:NLS} if $u$ is a Duhamel solution to
\[
i\partial_t u+\De u=\chi_{I}\mc N(u),
\]
where $\mc N(u)=\pm|u|^2u$ denotes the nonlinearity of \eqref{eq:NLS}.
Here, $I$ is possibly empty or $\R$ itself (i.e., both linear evolutions
and global nonlinear solutions are cutoff solutions). A cutoff solution
on $\T^2$ is defined similarly. Equivalently, $u$ can be regarded
as the continuous extension of $u\mid_{I}$ by linear evolutions. We also denote by $\mc I_{\R^{2}}$ the retarded Schr\"{o}dinger operator on $\R^{2}$,
\[
\mc I_{\R^2}(f)(t):=-i\int_{-\infty}^t e^{i(t-t')\De}f(t')\,dt.
\]
Similarly, $\mc I$ denotes the retarded Schr\"{o}dinger operator on $\T^{2}$.

For $R>0$ and $x_{0}\in\R^{2}$, we define $B_{R}(x_{0})=\left\{ x\in\R^2:|x-x_{0}|<R\right\}$.
\begin{defn}
Let $q,r\in[1,\infty]$. A sequence of functions $\left\{ f_{n}\right\} $
in $L_{t,x,\loc}^{1}(\R\times\R^{2})$ is said to be \emph{uniformly
locally bounded }in $L^{q}L^{r}$ if
\[
\sup_{R\in2^{\N}}\limsup_{n\rightarrow\infty}\norm{f_{n}}_{L^{q}L^{r}\left(\R\times B_{R}(0)\right)}<\infty.
\]
\end{defn}
For instance, for any frame $\left\{ \mc O_{n}\right\} $ and any
bounded sequence of functions $f_{n}$ in $L^{q}L^{r}(\R\times\T^{2})$, $1/q+1/r=1/2$,
since the support of $\iota_{\mc O_{n}}(\chi_{\R\times[-\pi,\pi]^{2}})$
grows to $\R\times\R^{2}$ as $N_{n}\rightarrow\infty$, $\iota_{\mc O_{n}}f_{n}$
is uniformly locally bounded in $L^{q}L^{r}$.
\begin{lem}
\label{lem:weak lim is sol}Let $\left\{(v_{n},I_n)\right\} $ be a sequence
of cutoff solutions to the cubic NLS on $\R^{2}$ which is
uniformly locally bounded in $C^{0}L^{2}\cap L_{t,x}^{4}(\R\times\R^{2})$.
After passing to a subsequence, $v_{n}$ converges weakly and almost everywhere
to a cutoff solution $(v_*,I_*)$, $v_{*}\in C^{0}L^{2}\cap L_{t,x}^{4}(\R\times\R^{2})$
to the cubic NLS on $\R^{2}$, which scatters in both time directions.
Furthermore, for every $T\in\R$, we have $v_{n}(T)\weak v_{*}(T)$.
\end{lem}
The proof of Lemma \ref{lem:weak lim is sol} proceeds conventionally
by first showing the weak convergence to a distributional cutoff solution
$v_{*}$, then using a mollification argument to show that such $v_{*}$
is also a Duhamel solution. For a gain of regularity, we use local smoothing effects of the Schr{\"o}dinger operator, found by Sj{\"o}lin, Constantin-Saut, and Vega in \cite{sjolin, constantin-saut, vega}.
\begin{lem}\label{lem:local smoothings}
Let $\psi\in C^\infty_0(\R^2)$. For $\phi\in L^2(\R^2)$ and $f\in L^2H^{-1}(\R\times\R^2)$, we have the following homogeneous and retarded local smoothing inequalities:
\begin{equation}\label{eq:local smoothing hom}
\norm{\psi e^{it\De}\phi}_{L^2H^{1/2}}\les_\psi\norm{\phi}_{L^2(\R^2)}
\end{equation}
and
\begin{equation}\label{eq:local smoothing inhom}
\norm{\psi\mc I_{\R^2}(\psi f)}_{L^2_{t,x}(\R\times\R^2)}\les_\psi\norm{f}_{L^2H^{-1}(\R\times\R^2)}.
\end{equation}
\end{lem}
\begin{proof}
\eqref{eq:local smoothing hom} follows from \cite[Theorem 2.1]{kenig-ponce-vega}.
The retarded estimate \eqref{eq:local smoothing inhom} is shown in \cite[Theorem 2.3 (b)]{kenig-ponce-vega} (up to a duality argument).
\end{proof}
\begin{proof}[Proof of Lemma \ref{lem:weak lim is sol}]
For $n\in\N$, since $(v_n,I_n)$ is a cutoff solution, we can write
\begin{equation}
(i\d_{t}+\De)v_{n}=\chi_{I_{n}}\mc N(v_{n}).\label{eq:cutoff NLS R^2}
\end{equation}
Let $\psi_{1}\in C_{0}^{\infty}(\R^{2})$ be a function such that
$\psi_{1}(x)=1$ holds for $x\in B_{1}(0)$ and $\supp(\psi_1)\subset B_{2}(0)$.
For $R\in\N$, denote $\psi_{R}(x):=\psi_{1}(x/R)$. Multiplying $\psi_{R}$
to (\ref{eq:cutoff NLS R^2}), we have
\begin{equation}
(i\d_{t}+\De)(\psi_{R}v_{n})=\psi_{R}\chi_{I_{n}}\mc N(v_{n})+2\na\psi_{R}\cdot\na v_{n}+(\De\psi_{R})v_{n}.\label{eq:cutoff NLS R^2 psi}
\end{equation}
Let $I\subset\R$ be a bounded interval. For each $R\in2^\N$, by the estimate $\norm{\na\psi_R\cdot\na \phi}_{H^{-1}(\R^2)}\les R^{-1}\norm{\phi}_{L^2(\R^2)}$, we have
\begin{equation}\label{eq:R psi v sum1}
R\norm{\na\psi_R\cdot\na v_n}_{L^2H^{-1}(I\times\R^2)}\les\norm{v_n}_{L^2_{t,x}(I\times\supp(\na\psi_R))}.
\end{equation}
For $R_0\in2^\N$, taking a square summation of \eqref{eq:R psi v sum1} over $R\le R_0$ yields
\begin{equation}\label{eq:R psi v sum2}
\sum_{R\le R_0}\left(R\norm{\na\psi_R\cdot\na v_n}_{L^2H^{-1}(I\times\R^2)}\right)^2\les\sum_{R\le R_0}\norm{v_n}_{L^2_{t,x}(I\times\supp(\na\psi_R))}^2.
\end{equation}
The right-hand side of \eqref{eq:R psi v sum2} is asymptotically $O_I(1)$ as $n\rightarrow\infty$ since $\supp(\na\psi_R)$ are disjoint for $R\le R_0$ and $\{v_n\}$ is uniformly locally bounded in $C^0L^2\hook L^2_{t,x}(I\times\R^2)$. Thus, we have
\begin{equation}\label{Eq:R psi v sum3}
\sup_{R_0}\limsup_{n\rightarrow\infty}\sum_{R\le R_0}\left(R\norm{\na\psi_R\cdot\na v_n}_{L^2H^{-1}(I\times\R^2)}\right)^2\les_I1.
\end{equation}
We pass to a subsequence of $n$, so that $\limsup$ in \eqref{Eq:R psi v sum3} is replaced by $\lim$, i.e.,
\begin{equation}\label{Eq:R psi v sum4}
\sup_{R_0}\lim_{n\rightarrow\infty}\sum_{R\le R_0}\left(R\norm{\na\psi_R\cdot\na v_n}_{L^2H^{-1}(I\times\R^2)}\right)^2\les_I1.
\end{equation}
By the Monotone Convergence Theorem and \eqref{Eq:R psi v sum4}, we have
\begin{equation}
\lim_{R\rightarrow\infty}\lim_{n\rightarrow\infty}R\norm{\na\psi_{R}\cdot\na v_{n}}_{L^{2}H^{-1}(I\times\R^{2})}=0.\label{eq:d psi}
\end{equation}

For each $R\in2^{\N}$, connecting the homogeneous local smoothing estimate (rescaled version of \eqref{eq:local smoothing hom})
\begin{equation}
\norm{\psi_{2R}e^{it\De}}_{L^{2}\rightarrow L^{2}H^{1/2}}\les R^{1/2}\label{eq:free smoothing}
\end{equation}
and the Strichartz estimate
\[
\norm{e^{it\De}}_{L^{2}\rightarrow L_{t,x}^{4}}\les1
\]
by the $TT^{*}$ argument and the Christ-Kiselev Lemma, we have
\begin{equation}
\norm{\psi_{2R}\mc I_{\R^{2}}f}_{L^{2}H^{1/2}}\les R^{1/2}\norm f_{L_{t,x}^{4/3}}.\label{eq:stri smoothing}
\end{equation}
Rewriting \eqref{eq:cutoff NLS R^2 psi} in the Duhamel form yields
\begin{align*}
\psi_{R}v_{n}=\psi_{2R}\psi_{R}v_{n}&=\psi_{2R}\left(e^{it\De}\left(\psi_{R}v_{n}(0)\right)+\mc I_{\R^{2}}\left(\psi_{R}\chi_{I_{n}}\mc N(v_{n})+(\De\psi_{R})v_{n}\right)\right)
\\&+\psi_{2R}\mc I_{\R^{2}}\left(2\na\psi_{R}\cdot\na v_{n}\right)=:v_{R,n}^{\text{conv}}+v_{R,n}^{\text{err}}.
\end{align*}
By \eqref{eq:local smoothing inhom} and \eqref{eq:d psi},
we have
\begin{align*}
\limsup_{n\rightarrow\infty}\norm{v_{R,n}^{\err}}_{L_{t,x}^{2}(I\times\R^2)}&\les\limsup_{n\rightarrow\infty}R\norm{\na\psi_{R}\cdot\na v_{n}}_{L^{2}H^{-1}(I\times\R^2)}\xrightarrow{R\rightarrow\infty}0.
\end{align*}
By (\ref{eq:free smoothing}), (\ref{eq:stri smoothing}), and the
uniform local boundedness of $\left\{ v_{n}\right\} $ in $C^{0}L^{2}\cap L_{t,x}^{4}$,
we have
\begin{align}\label{eq:L2 H1/2}
&\limsup_{n\rightarrow\infty}\norm{v_{R,n}^{\text{conv}}}_{L^{2}H^{1/2}(I\times\R^{2})}
\\\les_{I,R}&\limsup_{n\rightarrow\infty}\norm{\psi_{R}v_{n}(0)}_{L^{2}(\R^{2})}+\norm{\psi_{R}\chi_{I_{n}}\mc N(v_{n})+(\De\psi_{R})v_{n}}_{L_{t,x}^{4/3}(I\times\R^{2})}\les_{I,R}1.
\nonumber
\end{align}
By the identity
\[
i\d_{t}v_{R,n}^{\text{conv}}=-\De v_{R,n}^{\text{conv}}+\psi_{R}\chi_{I_{n}}\mc N(v_{n})+(\De\psi_{R})v_{n}
\]
and (\ref{eq:L2 H1/2}), we also have
\begin{align}\label{eq:H1 H-3/2}
&\limsup_{n\rightarrow\infty}\norm{\d_{t}v_{R,n}^{\text{conv}}}_{L^{2}H^{-3/2}(I\times\R^{2})}
\\\nonumber\le&\limsup_{n\rightarrow\infty}\norm{v_{R,n}^{\text{conv}}}_{L^{2}H^{1/2}(I\times\R^{2})}+\norm{\psi_{R}\chi_{I_{n}}\mc N(v_{n})+(\De\psi_{R})v_{n}}_{L^{2}H^{-3/2}(I\times\R^{2})}\les_{I,R}1.
\end{align}
By (\ref{eq:L2 H1/2}), (\ref{eq:H1 H-3/2}), and the compactness
of the embedding $L^{2}H^{1/2}\cap H^{1}H^{-3/2}(I\times\R^{2})\hook L_{\loc}^{2}(I\times\R^2)$,
passing to a subsequence of $n$, $v_{R,n}^{\text{conv}}$ is convergent
in $L_{t,x}^{2}(I\times\R^{2})$ for each $R\in2^\N$.
Therefore, $\psi_{R}v_{n}$ is a sum of a convergent sequence plus
an $o_{R}(1)$ error term in $L^{2}(I\times\R^{2})$. Taking $R\rightarrow\infty$
and $I\uparrow\R$, $v_{n}$ is convergent in $L_{\loc}^{2}(\R\times\R^{2})$;
let $v_{n}\rightarrow v_{*}$ be the limit. By the uniform local boundedness
of $\left\{ v_{n}\right\}$, we have $v_{*}\in L^{\infty}L^{2}\cap L_{t,x}^{4}$.

Since $i\d_{t}v_{n}=-\De v_{n}+\mc N(v_{n})$ is bounded in $C^{0}H_{\loc}^{-2}+L_{t,x,\loc}^{4/3}$,
$\left\{ v_{n}\right\} $ is uniformly equicontinuous in $(H^{-2}_\loc+L^{4/3}_\loc)(\R^{2})$.
Thus, for every $T\in\R$, $v_{n}(T)\weak v_{*}(T)$ holds once we
show the continuity of $v_{*}$, which follows immediately once
we show that $v_{*}$ is a Duhamel solution.

It only remains to show that $v_{*}$ is a Duhamel cutoff solution.
We proceed conventionally by first showing that $v_{*}$ is a distributional
cutoff solution, then using a mollification argument to show that
such $v_{*}$ is also a solution in Duhamel sense.

Firstly, we show that $v_{*}$ is a distributional cutoff solution. Passing to a subsequence, there exists an interval $I$ such that
$\chi_{I_{n}}$ converges almost everywhere to $\chi_{I}$. Since
$v_{n}$ is bounded in $L_{t,x,\loc}^{4}$ and converges almost everywhere
to $v_{*}$, $\mc N(v_{n})$ is bounded in $L_{t,x,\loc}^{4/3}$ and
converges almost everywhere to $\mc N(v_{*})$. Thus, taking weak
limits of both sides of \eqref{eq:cutoff NLS R^2} gives
\begin{equation}
(i\d_{t}+\De)v_{*}=\chi_{I}\cdot\mc N(v_{*}).\label{eq:dist}
\end{equation}
We check that $v_{*}$ is a Duhamel cutoff solution. We use a mollification
argument. Let $\phi\in C_{0}^{\infty}(\R\times\R^{2})$ be a function
such that $\int_{\R\times\R^{2}}\phi dxdt=1$. For $n\in\N$, let
$\phi_{n}(t,x):=n^{3}\phi(nt,nx)$. Taking a convolution with $\phi_{n}$
to \eqref{eq:dist}, we have
\[
(i\d_{t}+\De)\left(v_{*}*\phi_{n}\right)=\left(\chi_{I}\cdot\mc N(v_{*})\right)*\phi_{n},
\]
which can be rewritten in the Duhamel form
\begin{equation}
\left(v_{*}*\phi_{n}\right)(t)=e^{it\De}\left(v_{*}*\phi_{n}\right)(0)-i\int_{0}^{t}e^{i(t-s)\De}\left(\left(\chi_{I}\cdot\mc N(v_{*})\right)*\phi_{n}\right)(s)ds.\label{eq:weak Duhamel}
\end{equation}
Since $v_{*}\in L^{\infty}L^{2}$ is continuous in $\left(H^{-2}+L^{4/3}\right)_{\loc}(\R^{2})$,
we have $\left(v_{*}*\phi_{n}\right)(0)\weak v_{*}(0)$. Thus, taking a weak limit by $n\rightarrow\infty$
of (\ref{eq:weak Duhamel}) yields
\[
v_{*}=e^{it\De}v_{*}(0)-i\int_{0}^{t}e^{i(t-s)\De}\left(\chi_{I}\cdot\mc N(v_{*})\right)(s)ds,
\]
which implies that $v_{*}$ is a Duhamel cutoff solution. In the case of an unbounded interval $I^*$ the scattering claim follows from the standard $L^4$ Strichartz estimate, see e.g. \cite[Theorem 1.22]{DodBook}.
\end{proof}

\subsection{The solution norm and its inverse property}
In this subsection, we recall and employ the space $Z_R$ used in \cite{herr2024strichartz}. Based on this $Z_R$ space, we provide a modified version of the cubic nonlinear estimate in \cite{herr2024strichartz} and an inverse property for the $L^4$-estimate of $Z_R$.

We briefly introduce the relevant definitions and elementary properties and refer to \cite{herr2024strichartz} for more details.
For $1\leq p<\infty$ let\[\|u\|_{V^p} =\sup \Big(\sum_{k=1}^K|u(t_k)-u(t_{k-1})|^p\Big)^{\frac1p}\]
be the $p-$variation of $u:\R\to\C$, where the supremum is taken over all finite increasing sequences $(t_k)$, and let $V^p$ be the space of all functions with finite $p-$variation which tend to zero at $-\infty$. For given Sobolev regularity parameter $s\in \R$, we define $Y^s$ as the space of all $u:\R\times \T^2\to \C$ such that $t\mapsto\widehat{e^{-it\Delta} u(t)}(\xi)\in V^2$ for any fixed $\xi \in \Z^2$ and
\[
\|u\|_{Y^s}=\sup \Big(\sum_{\xi \in \Z^2}(1+|\xi|^2)^s\|e^{it|\xi|^2} \widehat{u(t)}(\xi)\|_{V^2}^2\Big)^{\frac12}<\infty.
\]
For a given interval $I\subset \R$, let  $Y^s(I)$ be the corresponding restriction space.
The space $Y^{s}$ is used in \cite{herr2011global} and later works on critical regularity theory of Schr{\"o}dinger equations on periodic domains. Some well-known properties are the following:
\begin{prop}[{\cite[Section 2]{herr2011global}}]\label{prop:Ys props}
For $s\in\R$, the $Y^{s}$-norm has the following properties:
\begin{itemize}
\item For an interval $I\subset\R$ and $u\in Y^s$, we have
\[
\norm{\chi_I u}_{Y^s}\le2\norm u_{Y^s}.
\]
\item For a function $\phi\in L^2(\T^2)$, we have
\begin{equation}
\norm{e^{it\De}\phi}_{Y^s}\sim\norm{\phi}_{H^s}
\end{equation}
and for a function $u\in Y^s$,
\begin{equation}
\norm{u}_{Y^s}\gtrsim\norm{u}_{L^\infty H^s}.
\end{equation}
\item For $T>0$ and a function $f\in L^{1}H^{s}$,
we have 
\begin{equation}
\norm{\chi_{[0,T)}\cdot\mc I(f)}_{Y^{s}}\les\sup_{v\in Y^{-s}:\norm v_{Y^{-s}}\le1}\left|\int_0^T\int_{\T^{2}}f\overline{v}dxdt\right|.\label{eq:U2V2}
\end{equation}
By a density argument, $f$ is further allowed to be just integrable with bounded right-hand side. By \cite[Proposition 2.4]{hadac2009well}, the left-hand side is a continuous function of $T$ and $\int_{0}^{t}e^{i(t-t')\De}f(t')dt'$ is a continuous function of $t$ in $H^s$.
\end{itemize}
\end{prop}
Recall that  $\mc I$ denotes the retarded Schr\"{o}dinger operator on $\T^{2}$,
\[
\mc I(f)(t):=-i\int_{-\infty}^t e^{i(t-t')\De}f(t')\,dt.
\]

Now we introduce the space $Z_R$. Given $s>0$, let $Z_R=Z_R^s$ be the norm
\[
\norm{u}_{Z_R}:=\norm{u}_{Y^0}+R^{-s}\norm{u}_{Y^s}.
\]
The norm $Z_R$ was used as a solution norm in \cite{herr2024strichartz}. We point out that in \cite{herr2024strichartz} sharp Littlewood-Paley cutoffs were used, while smooth Littlewood-Paley cutoffs are used in our setting here. The analysis does not change, however, since all estimates in  \cite{herr2024strichartz} were $L^4$ and $L^2$-based and do not change (up to comparability) by replacing  sharp by smooth Fourier cutoffs.
\begin{lem}[{\cite[(4.10) and (4.11)]{herr2024strichartz}}]\label{lem:Z_R L4 stri}
For $N,M\in2^{\N}$ and an interval $I\subset\R$ such
that $\left|I\right|\le\frac{1}{\log N}$, we have
\begin{equation}
\norm{P_{\le M}u}_{L^{4}_{t,x}(I\times\T^{2})}\les\left(1+\frac{\log M}{\log N}\right)^{1/4}\norm u_{Y^{0}}\label{eq:Y^0 Stri}
\end{equation}
and
\begin{equation}\label{eq:Z_R stri}
\norm{u}_{L_{t,x}^{4}(I\times\T^2)}\les\norm u_{Z_N}.
\end{equation}
\end{lem}
We show our main nonlinear estimate:
\begin{lem}\label{lem:nu bound}
Let $s>0$. For $C,R\in2^\N$ such that $R\gg_C 1$ and $u\in Y^s$ supported in $[0,1/\log R)\times\T^2$, we have
\begin{equation}\label{eq:trilinear}
\norm{\mc I(|u|^2u)}_{Z_R}\les_s\left(C^{-s}\norm{u}_{Z_R}^2+\norm{u}_{L^4_{t,x}}^2\right)\norm{u}_{Z_R}.
\end{equation}
\end{lem}
\begin{proof}
In this proof, every comparability depends on $s$.
By \eqref{eq:U2V2}, \eqref{eq:trilinear} is reduced to showing
\begin{equation}\label{eq:bootstrap1}
\left|\int_{\R\times\T^2}\overline{v_{<R}}|u|^2u dxdt\right|\les \norm{u}_{L^4_{t,x}}^2\norm{u}_{Z_R}\norm{v}_{Y^0}
\end{equation}
and
\begin{equation}\label{eq:bootstrap2}
\left|\int_{\R\times\T^2}\overline{v_{\ge R}}|u|^2u dxdt\right|\les R^s\left(C^{-s}\norm{u}_{Z_R}^2+\norm{u}_{L^4_{t,x}}^2\right)\norm{u}_{Z_R}\norm{v}_{Y^{-s}}.
\end{equation}
\eqref{eq:bootstrap1} holds directly by \eqref{eq:Y^0 Stri} and \eqref{eq:Z_R stri}. We prove \eqref{eq:bootstrap2} mimicking the proof in \cite{herr2024strichartz}. In \cite[(4.12)]{herr2024strichartz}, it was shown that for $M\in2^{\N}$ and $u,v\in Y^{0}$, we have
\begin{equation}\label{eq:bilinear}
\norm{P_{\le M}\left(uv\right)}_{L_{t,x}^{2}([0,\frac{1}{\log R})\times\T^2)}\les\left(1+\frac{\log M}{\log R}\right)^{1/2}\norm u_{Y^{0}}\norm v_{Y^{0}}.
\end{equation}
Let $C\in2^\N$. By \eqref{eq:bilinear} and Young's convolution inequality on $(L,K)$ using that $\sum_{R\in2^\N}R^{-s}\les 1$, for $R\gg_C1$, we have
\begin{align}
 & \sum_{K\ge R}\sum_{L\gtrsim K}\left|\int_{[0,\frac{1}{\log R})\times\T^{2}}P_{<CR}(u_1u_2)P_{<CR}(w_{ L}v_{ K})dxdt\right|\label{eq:quar1}\\
  \les&\norm {u_1}_{L^4_{t,x}}\norm {u_2}_{L^4_{t,x}}\sum_{K\ge R}\sum_{L\gtrsim K}\norm{w_{ L}}_{Y^{0}}\norm{v_{ K}}_{Y^{0}}\nonumber \\
  \les&\norm {u_1}_{L^4_{t,x}}\norm {u_2}_{L^4_{t,x}}\sum_{K\ge R}\sum_{L\gtrsim K}(L/K)^{-s}\norm{w_{ L}}_{Y^{s}}\norm{v_{ K}}_{Y^{-s}}\nonumber \\
  \les&\norm {u_1}_{L^4_{t,x}}\norm {u_2}_{L^4_{t,x}} \norm w_{Y^{s}} \norm v_{Y^{-s}}.\nonumber
\end{align}
\eqref{eq:bilinear} also yields the following version of \cite[(4.14)]{herr2024strichartz}:
\begin{align}\label{eq:quar2}
 & \sum_{M\ge CR}\sum_{K\ge R}\sum_{L\gtrsim K}\left|\int_{[0,\frac{1}{\log R})\times\T^{2}}P_{ M}\left(u_1u_2\right)P_{ M}\left(w_{ L}v_{ K}\right)dxdt\right|\\ \nonumber
\les &\sum_{M\ge CR}\frac{\log M}{\log R}(\norm{P_{\ge M/4}u_1}_{Y^0}\norm{u_2}_{Y^0}+\norm{u_1}_{Y^0}\norm{P_{\ge M/4}u_2}_{Y^0})\sum_{K\ge R}\sum_{L\gtrsim K}\norm{w_{ L}}_{Y^0}\norm{v_{ K}}_{Y^0}\\ \nonumber
\les & \sum_{M\ge CR}\frac{\log M}{\log R}\frac{R^s}{M^s}\norm {u_1}_{Z_R} \norm {u_2}_{Z_R}\sum_{K\ge R}\sum_{L\gtrsim K}\norm{w_{ L}}_{Y^{0}}\norm{v_{ K}}_{Y^{0}},\nonumber
\end{align}
from which we continue to estimate
\[
\les C^{-s}\norm {u_1}_{Z_R} \norm {u_2}_{Z_R}\sum_{K\ge R}\sum_{L\gtrsim K}(L/K)^{-s}\norm{w_{ L}}_{Y^s}\norm{v_{ K}}_{Y^{-s}}\les C^{-s}\norm {u_1}_{Z_R} \norm {u_2}_{Z_R} \norm w_{Y^s} \norm v_{Y^{-s}}.
\]
By \eqref{eq:quar1}, \eqref{eq:quar2}, and that $\norm{u}_{Y^s}\le R^s\norm{u}_{Z_R}$, we have
\begin{align}\label{eq:quar 1+2}
&\sum_{K\ge R}\sum_{L\gtrsim K}\left|\int_{[0,\frac{1}{\log R})\times\T^{2}}(u_1u_2)w_{ L}v_{ K}dxdt\right|
\\\nonumber \les& R^s\left(C^{-s}\norm{u_1}_{Z_R}\norm{u_2}_{Z_R}+\norm{u_1}_{L^4_{t,x}}\norm{u_2}_{L^4_{t,x}}\right)\norm{w}_{Z_R}\norm{v}_{Y^{-s}}.
\end{align}
Note that in \eqref{eq:bilinear}, \eqref{eq:quar1}, \eqref{eq:quar2}, and \eqref{eq:quar 1+2}, each function on the left-hand side could be replaced by its complex conjugate. Since $u$ is supported in $[0,\frac{1}{\log R})\times\T^2$, applying \eqref{eq:quar 1+2} to each term of the bound
\begin{align*} 
 \left|\int_{\R\times\T^{2}}\overline{v_{\ge R}}|u|^2u dxdt\right|
 \le&\sum_{K\ge R}\left|\int_{\R\times\T^{2}}P_{\ge K/4}u\cdot \overline{u}\cdot u\cdot  v_{ K}dxdt\right|
\\
& +\sum_{K\ge R}\left|\int_{\R\times\T^{2}}P_{<K/4}u\cdot \overline{P_{\ge K/4}u}\cdot u\cdot v_{ K}dxdt\right|
\\
& +\sum_{K\ge R}\left|\int_{\R\times\T^{2}}P_{<K/4}u\cdot \overline{P_{<K/4}u}\cdot P_{\ge K/4}u\cdot v_{ K}dxdt\right|,
\end{align*}
we conclude \eqref{eq:bootstrap2}.
\end{proof}
\begin{lem}\label{lem:inverse Z_R stri}
Let $s,\epsilon>0$ and $\{R_n\}$ be a sequence in $2^\N$ such that $R_n\rightarrow\infty$. Let $\{I_n\}$ be a sequence of intervals such that $|I_n|\cdot\log R_n\rightarrow0$. Let $\{\phi_n\}$ be a sequence in $H^s(\T^2)$ such that
\begin{equation}\label{eq:inverse Z_R stri assump}
\norm{\phi_n}_{L^2(\T^2)}+R_n^{-s}\norm{\phi_n}_{H^s(\T^2)}\le 1
\end{equation}
and
\[
\norm{e^{it\De}\phi_n}_{L^4(I_n\times\T^2)}\ge\epsilon.
\]
Then, there exists a frame $\{\mc O_n\}$ such that both $\{\iota_{\mc O_n}(e^{it\De}\phi_n)(0)\}$ and $\{\chi_{\tilde{I_n}}\}$ are weakly nonzero. Here, $\tilde{I_n}$ denotes the interval mapped from $I_n$ by $\mc O_n$.
\end{lem}
\begin{proof}
By \eqref{eq:inverse Z_R stri assump}, we have
\[
\lim_{n\rightarrow\infty}\norm{P_{>R_n^2}\phi_n}_{L^2(\T^2)}=0.
\]
Thus, any sequential weak limits of $\iota_{\mc O_n}(e^{it\De}\phi_n)(0)$ and $\iota_{\mc O_n}(e^{it\De}P_{\le R_n^2}\phi_n)(0)$ are equal; we may reduce to the case $\supp(\widehat{\phi_n})\subset[R_n^2]^2$. Now applying Lemma \ref{lem:inverse thm for almost sum of profiles} and Proposition \ref{prop:profile'} finishes the proof.
\end{proof}
\subsection{Weak scattering behavior}
In this subsection, we show uniform convergences of scattering limits over sequences of solutions.
For $\R^{2}$, we show that for any bounded sequence of solutions to
\eqref{eq:NLS} on $\R^{2}$, weak convergence to the scattering limit as $t\rightarrow-\infty$
uniformly occurs. A similar result is shown for $\T^{2}$ with respect to any
frame $\left\{ \mc O_{n}\right\} $.
\begin{lem}
\label{lem:scatter R^d}Let $\left\{(v_n,I_n)\right\} $ be a bounded
sequence of Duhamel cutoff solutions to (\ref{eq:NLS}) in $C^{0}L^{2}\cap L_{t,x}^{4}(\R\times\R^{2})$.
Then, 
\[
\left\{ e^{iT\De}v_{n}(-T)\right\} _{n\in\N}
\]
is uniformly convergent (in the weak sense) as $T\rightarrow\infty$.
\end{lem}

\begin{proof}
Explicitly, we show that for every $\phi\in C_{0}^{\infty}(\R^{2})$,
\begin{equation}
\limsup_{T_{1},T_{2}\rightarrow\infty}\sup_{n\in\N}\left|\jp{\phi,e^{iT_{1}\De}v_{n}(-T_{1})-e^{iT_{2}\De}v_{n}(-T_{2})}_{L^{2}}\right|=0.\label{eq:weak scatter R^d claim}
\end{equation}
For $T_{1},T_{2}>0$ and $n\in\N$, we have
\begin{align*}
\left|\jp{\phi,e^{iT_{1}\De}v_{n}(-T_{1})-e^{iT_{2}\De}v_{n}(-T_{2})}_{L^{2}}\right| & =\left|\jp{\phi,\int_{[T_1,T_2]\cap I_n}e^{is\De}\mc N(v_{n}(-s))ds}_{L^{2}}\right|\\
 & =\left|\int_{[T_1,T_2]\cap I_n}\jp{e^{-is\De}\phi,\mc N(v_{n}(-s))}_{L^{2}}ds\right|,
\end{align*}
then by using that $\left\{ v_{n}\right\} $ is bounded in $C^{0}L^{2}\cap L_{t,x}^{4}(\R\times\R^{2})$,
we continue to estimate
\begin{align}
 & \les\norm{\chi_{[-T_{1},-T_{2}]}e^{it\De}\phi}_{L_{t,x}^{4}}\norm{\mc N(v_{n})}_{L_{t,x}^{4/3}}\nonumber \\
 & \les\norm{\chi_{[-T_{1},-T_{2}]}e^{it\De}\phi}_{L_{t,x}^{4}}.\label{eq:T1T2}
\end{align}
Taking a limit $T_{1},T_{2}\rightarrow\infty$ to (\ref{eq:T1T2})
yields
\begin{align*}
 & \limsup_{T_{1},T_{2}\rightarrow\infty}\sup_{n\in\N}\left|\jp{\phi,e^{iT_{1}\De}v_{n}(-T_{1})-e^{iT_{2}\De}v_{n}(-T_{2})}_{L^{2}}\right|\\
  \les&\limsup_{T_{1},T_{2}\rightarrow\infty}\norm{\chi_{[-T_{1},-T_{2}]}e^{it\De}\phi}_{L_{t,x}^{4}},
\end{align*}
which is $0$ since $e^{it\De}\phi$ lies in the Strichartz space
$L^4(\R\times\R^2)$, finishing the proof.
\end{proof}
The following states the weak scattering property on $\T^{2}$. Because of the resonances we assume the time-convergence $t_{n}\rightarrow0$ of the frame and the negative-time linearity.
\begin{lem}
\label{lem:scatter T^d}Let $s>0$ and $C<\infty$. Let $\{R_n\}$ be a sequence of dyadic numbers such that $R_n\rightarrow\infty$. Let $\left\{(u_{n},I_n)\right\} $ be
a sequence of cutoff solutions to \eqref{eq:NLS}
in $C^{0}H^{s}\cap Y^{s}(\R\times\T^{2})$. Let $\left\{ \mc O_{n}\right\} =\left\{ (N_{n},t_{n},x_{n},\xi_{n})\right\}$
be a frame such that $t_{n}\cdot\log R_{n}\rightarrow0$. If
\begin{equation}
\sup_{n}\norm{u_{n}}_{Z_{R_{n}}}\le C\label{eq:scatter T^d f bound}
\end{equation}
and
\begin{equation}
u_{n}(t)=e^{it\De}u_{n}(0)\text{ holds for }t\le0,\label{eq:t<0 zero}
\end{equation}
then
\[
\left\{ e^{iT\De}\left(\iota_{\mc O_{n}}u_{n}\right)(-T)\right\} _{n\in\N}
\]
is uniformly convergent (in weak sense) as $T\rightarrow\infty$.
\end{lem}
\begin{proof}
Explicitly, we show that for every Schwartz class function $\phi\in \mc S(\R^2)$,
\begin{equation}
\left\{ \jp{\phi,e^{iT\De}\left(\iota_{\mc O_{n}}u_{n}\right)(-T)}_{L^{2}(\R^{2})}\right\} _{n\in\N}\label{eq:scatter T^d claim0}
\end{equation}
is uniformly convergent as $T\rightarrow\infty$. Since the span of
$\left\{ \delta_{N_{*}}(\cdot-x_{*}):N_{*}\in2^{\Z}\text{ and }x_{*}\in\R^{2}\right\} $
is dense in the Schwartz space $\mc S(\R^{2})$, one may assume further that 
\[
\phi=\delta_{N_{*}}(\cdot-x_{*}),\qquad N_{*}\in2^{\Z}\text{ and }x_{*}\in\R^{2}.
\]
Up to a comparable choice of frame $\tilde{\mc O_n}=\left\{ (N_{n}N_{*},t_{n},x_{n}+N_{n}^{-1}x_{*},\xi_{n})\right\} $,
we can further reduce to the case $(N_{*},x_{*})=(1,0)$.

Along the subset of $n\in\N$ such that $N_{n}>R_{n}^{3}$, \eqref{eq:scatter T^d claim0} is bounded by the sum of
\[
|\jp{\delta_1,e^{iT\De}(\iota_{\mc O_n}P_{\le R_n^2}u_n)(-T)}_{L^2(\R^2)}|\les (R_n^2/N_n)\norm{u_n}_{C^0L^2}\les (R_n^2/N_n)\norm{u_n}_{Z_{R_n}}
\]
(where we used the Bernstein inequality) and
\[
|\jp{\delta_1,e^{iT\De}(\iota_{\mc O_n}P_{>R_n^2}u_n)(-T)}_{L^2(\R^2)}|\les\norm{P_{>R_n^2}u_n}_{C^0L^2}\les R_n^{-s}\norm{u_n}_{Z_{R_n}},
\]
both of which are $o_n(1)$ by \eqref{eq:scatter T^d f bound} and $N_n,R_n\rightarrow\infty$. Also, for each fixed $n\in\N$, by the negative-time linearity \eqref{eq:t<0 zero}, $\jp{\delta_1,e^{iT\De}(\iota_{\mc O_n}u_n)(-T)}_{L^2}$ is a constant for $T\gg_n1$. Thus, the uniform convergence holds.

Along the subset of $n\in\N$ such that $N_n\le R_n^3$, since $-T$ and $P_{1;\R^{2}}$ on $\R\times\R^2$ correspond to $t_{n}-TN_{n}^{-2}$
and $P_{N_{n};\T^{2}}$ on $\R\times\T^2$, it suffices to show
the uniform convergence as $T\rightarrow\infty$ of
\[
\left\{ \jp{\delta_{N_{n}}(\cdot-x_{n}),N_{n}^{-1}I_{\xi_n}e^{iTN_{n}^{-2}\De}u_{n}(t_{n}-TN_{n}^{-2})}_{L^{2}(\T^{2})}\right\}_{n\in\N:N_n\le R_n^3}.
\]
For $n\in\N$ and $T_{1},T_{2}>0$, denoting $f_{n}=\chi_{I_n}\cdot\mc N(u_{n})$,
we have
\begin{align*}
 & \left|\jp{\delta_{N_{n}}(\cdot-x_{n}),I_{\xi_n}e^{iT_{1}N_{n}^{-2}\De}u_{n}(t_{n}-T_{1}N_{n}^{-2})-I_{\xi_n}e^{iT_{2}N_{n}^{-2}\De}u_{n}(t_{n}-T_{2}N_{n}^{-2})}_{L^{2}(\T^{2})}\right|\nonumber \\
  =&\left|\jp{\delta_{N_{n}}(\cdot-x_{n}),\int_{t_{n}-T_{1}N_{n}^{-2}}^{t_{n}-T_{2}N_{n}^{-2}}I_{\xi_n}e^{i(t_{n}-s)\De}f_{n}(s)ds}_{L^{2}(\T^{2})}\right|\nonumber \\
  =&\left|\int_{t_{n}-T_{1}N_{n}^{-2}}^{t_{n}-T_{2}N_{n}^{-2}}\jp{I_{-\xi_n}e^{i(s-t_{n})\De}\delta_{N_{n}}(\cdot-x_{n}),f_{n}(s)ds}_{L^{2}(\T^{2})}\right|\nonumber \\
  \les&\norm{\chi_{[-T_{1}N_{n}^{-2},-T_{2}N_{n}^{-2}]\cap[-t_{n},\infty)}I_{-\xi_n}e^{it\De}\delta_{N_{n}}}_{L_{t,x}^{4}}\norm{f_{n}}_{L_{t,x}^{4/3}([0,t_{n}]\times\T^{2})},
\end{align*}
where the restrictions $[-t_{n},\infty)$ and $[0,t_{n}]$ are obtained from \eqref{eq:t<0 zero}. Thus, it only remains to show
\begin{equation}\label{eq:scatter T^2 really last}
\limsup_{T_1,T_2\rightarrow\infty}\limsup_{n\rightarrow\infty} N_{n}^{-1}\norm{\chi_{[-T_{1}N_{n}^{-2},-T_{2}N_{n}^{-2}]\cap[-t_{n},\infty)}I_{-\xi_n}e^{it\De}\delta_{N_{n}}}_{L_{t,x}^{4}}\norm{f_{n}}_{L_{t,x}^{4/3}([0,t_{n}]\times\T^{2})}=0.
\end{equation}
Since $t_{n}\cdot\log N_{n}\le t_{n}\cdot \log R_{n}^{3}=3t_{n}\cdot\log R_{n}\rightarrow0$, by \eqref{eq:kernel 0} and the invariance of the $L_{t,x}^{4}$-norm
under the Galilean transform $I_{-\xi_n}$, we have
\begin{align}\label{eq:scatter T^2 ingred1}
&\limsup_{T\rightarrow\infty}\limsup_{n\rightarrow\infty}N_n^{-1}\norm{I_{-\xi_n}e^{it\De}\delta_{N_{n}}}_{L_{t,x}^{4}([-t_{n},-TN_n^{-2}]\times\T^{2})}
\\\nonumber
\le&\limsup_{T\rightarrow\infty}\limsup_{n\rightarrow\infty}N_n^{-1}\norm{e^{it\De}\delta_{N_{n}}}_{L_{t,x}^{4}([-t_n,-TN_n^{-2}]\times\T^{2})}=0.
\end{align}
By \eqref{eq:Z_R stri}, we also have
\begin{equation}
\norm{f_{n}}_{L_{t,x}^{4/3}([0,t_{n}]\times\T^{2})}=\norm{u_{n}}_{L_{t,x}^{4}([0,t_{n}]\times\T^{2})}^{3}\les\norm{u_{n}}_{Z_{R_{n}}}^{3}\les1.\label{eq:scatter T^2 ingred2}
\end{equation}
By \eqref{eq:scatter T^2 ingred1} and \eqref{eq:scatter T^2 ingred2},
\eqref{eq:scatter T^2 really last} holds,
finishing the proof.
\end{proof}
\subsection{\label{sec:Proof-of-Theorem}Proof of Theorem \ref{thm:GWP defocus}
and Theorem \ref{thm:GWP focus}}

In this section, we show Theorem \ref{thm:GWP defocus} and Theorem
\ref{thm:GWP focus}. Since the proofs are almost identical, for the sake of conciseness,
we prove only the focusing case, i.e.\ Theorem \ref{thm:GWP focus}.
\begin{prop}
\label{prop:main}Let $s>0$ and $M<\norm Q_{L^{2}}^{2}$. There exist $L<\infty$
and $\epsilon>0$ satisfying the following:

Let $R\gg1$ be a dyadic number. Then, for every $u_0\in H^s(\T^2)$ such that
\[
\norm{u_0}_{L^2}^2\le M
\]
and
\[
R\le 1+\norm{u_0}_{H^s}^{1/s}<2R,
\]
the solution $u\in C^0H^s\cap Y^s([0,T))$ to \eqref{eq:NLS} with $u(0)=u_0$ satisfies
\[
\norm{u}_{Z_{R}([0,\frac{\epsilon}{\log R}))}\le L.
\]
\end{prop}
Proposition \ref{prop:main} is actually stronger than Theorem \ref{thm:GWP focus}.
Once we show Proposition \ref{prop:main}, Theorem \ref{thm:GWP focus}
can be proven as follows:
\begin{proof}[Proof of Theorem \ref{thm:GWP focus} assuming Proposition \ref{prop:main}]
Firstly, we show the global well-posedness.
Let $u$ be a Duhamel solution to \eqref{eq:NLS} in the sense of LWP
on $\T^{2}$ with initial data $u(0)=u_{0}$ of mass $\norm{u_{0}}_{L^{2}}^{2}\le M$
and positive lifespan $[0,T_{\max})$. If $T_{\max}=\infty$ there is nothing to prove. A function in $V^2$ (i.e.\ of bounded $2-$variation) can be extended continuously, hence so can a function in $Y^s$. Thus, if $T_{\max} <\infty$, we must have $\norm{u}_{Y^s([0,T_{\max}))}=\infty$.

Let $L$ and $\epsilon$ be defined in Proposition \ref{prop:main}. Define $t_{n}\in[0,T_{\max})$ and $R_{n}\in2^{\N}$ recursively as
sequences of $n\in\N$ such that $t_{0}:=0$,
\[
R_{n}\le1+\norm{u(t_{n})}_{H^{s}}^{1/s}<2R_{n},
\]
\[
\norm u_{Z_{R_{n}}([t_{n},t_n^+))}=L,
\]
and
\[
t_{n+1}:=\min\{t_n+\frac{1}{\log R_n},t_n^+\}.
\]
Note that since
\[
\norm{u(t_{n+1})}_{H^{s}}\le\norm u_{C^{0}H^{s}([t_{n},t_{n+1}))}\le R_{n}^{s}\norm u_{Z_{R_{n}}([t_{n},t_{n+1}))}\le R_{n}^{s}L,
\]
we have
\[
R_{n+1}\les_{L,s} R_{n},
\]
and thus
\[
\log R_{n}-\log R_0\les_{L,s}n.
\]
By Proposition \ref{prop:main}, we have
\[
t_{n+1}-t_n>\frac{\epsilon}{\log R_n}\gtrsim_{L,s,\epsilon}\frac{1}{n+\log R_0}\gtrsim_s\frac{1}{n+\log \norm{u_0}_{H^s}},
\]
which implies
\begin{equation}\label{eq:tn telescope}
t_n\gtrsim_{L,s,\epsilon}\sum_{k=1}^n\frac{1}{k+\log\norm{u_0}_{H^s}}\gtrsim\log\left(1+\frac{n}{\log\norm{u_0}_{H^s}}\right).
\end{equation}
Since \eqref{eq:tn telescope} is divergent, $t_n$ grows higher than $T_{\max}<\infty$ within a finite $n$. This contradicts $\norm{u}_{Y^s([0,T_{\max}))}=\infty$ and thus concludes the GWP.

The proof of \eqref{eq:GWP focus L4 bound} proceeds similarly. Let $T>0$. Let $n_T$ be the first index such that $t_n\ge T$. \eqref{eq:tn telescope} yields that 
\[
\frac{n_T-1}{\log\norm{u_0}_{H^s}}\les_{L,s,\epsilon} 2^{O_{L,s,\epsilon}(T)}-1.
\]
By \eqref{eq:Z_R stri}, we have
\begin{equation}\label{eq:prec bound precursor}
\norm{u}_{L^4_{t,x}([0,T)\times\T^2)}\le\norm{u}_{L^4_{t,x}([t_0,t_{n_T})\times\T^2)}\les\left(\sum_{n\le n_T}\norm{u}_{L^4_{t,x}([t_n,t_{n+1})\times\T^2)}^4\right)^{1/4}\les n_T^{1/4}L.
\end{equation}
Since $L$ and $\epsilon$ depend only on $M$ and $s$, \eqref{eq:prec bound precursor} can be rewritten as
\begin{equation}\label{eq:prec-bound}
\norm{u}_{L^4_{t,x}([0,T)\times\T^2)}\les_{M,s}1+\left(\left(2^{O_{M,s}(T)}-1\right)\log\norm{u_0}_{H^s}\right)^{1/4}.
\end{equation}
This finishes the proof of Theorem \ref{thm:GWP focus}.
\end{proof}
\begin{proof}[Proof of Proposition \ref{prop:main}]
Assume there is no such $(L,\epsilon)$. Then, there exists a sequence $\left\{(u_{n},I_n)\right\}$ of cutoff solutions to \eqref{eq:NLS}, a sequence $\{R_n\}$ of dyadic numbers such that $R_n\rightarrow\infty$, and a sequence $\{T_n\}$ of positive numbers such that
\[
\norm{u_n(0)}_{L^2}^2\le M,
\]
\[
T_n\cdot\log R_n\rightarrow0,
\]
\begin{equation}\label{eq:Rn def}
R_{n}\le1+\norm{u_n(0)}_{H^{s}}^{1/s}<2R_{n},
\end{equation}
and
\begin{equation}\label{eq:ZR->infty}
\norm{u_n}_{Z_{R_n}([0,T_n))}\rightarrow\infty.
\end{equation}
We will show a contradiction. Choosing smaller $T_n$ if necessary, we may assume further that
\begin{equation}\label{eq:Z norm bound}
\norm{u_n}_{Z_{R_n}([0,T_n))}\le R_n^s.
\end{equation}
By Lemma \ref{lem:nu bound}, there exists $C_0=O_s(1)$ such that for any $C>0$, for every $n\gg_C1$ and $0\le t_0\le t_1<T_n$, denoting $\varphi_n(t):=\norm{u_n}_{Z_{R_n}([0,t))}$, we have
\begin{align}\label{eq:continuity argument on Z}
\varphi_n(t_1)=\norm{u_n}_{Z_{R_n}([0,t_1))}&\le\norm{u_n}_{Z_{R_n}([0,t_0))}+\norm{u_n}_{Z_{R_n}([t_0,t_1))}
\\&\nonumber\le\varphi_n(t_0)+\norm{e^{i(t-t_0)\De}u_n(t_0)}_{Z_{R_n}}+\norm{\mc I(\chi_{[t_0,t_1)}|u_n|^2u_n)}_{Z_{R_n}}
\\&\nonumber\le C_0\left(\varphi_n(t_0)+\left(C^{-s}\varphi_n(t_1)^2+\norm{u}_{L^4_{t,x}([t_0,t_1)\times\T^2))}^2\right)\varphi_n(t_1)\right).
\end{align}
Also, for each $n\in\N$, by \eqref{eq:Rn def} and the continuity of the nonlinear flow, we have
\begin{equation}\label{eq:bdd at zero}
\lim_{t\rightarrow0+}\varphi_n(t)\les\norm{u_n(0)}_{L^2}+R_n^{-s}\norm{u_n(0)}_{H^s}\les1.
\end{equation}
Assume there exists $K<\infty$ such that along a subsequence of $n$, $\norm{u_n}_{L^4_{t,x}([0,T_n)\times\T^2)}\le K$ holds. Let $J=(2C_0)^2K^4$. For each $n$, there exists a sequence of times $0=t_{n,0}<t_{n,1}<t_{n,2}<\cdots<t_{n,J}=T_n$ such that $\norm{u_n}_{L^4_{t,x}([t_{n,j-1},t_{n,j}))}^2\le \frac{1}{2C_0}$. Then, for each $j=1,\ldots,J$ and $t\in[t_{n,j-1},t_{n,j}]$, \eqref{eq:continuity argument on Z} can be rewritten as
\[
\varphi_n(t)\le C_0\varphi_n(t_{n,j-1})+C_0C^{-s}\varphi_n(t)^3+\frac{1}{2}\varphi_n(t),
\]
where $C$ can be chosen as an arbitrary number independent of $n$. Thus, by a continuity argument, one can show
\begin{equation}\label{eq:varphi bootstrap}\limsup_{n\rightarrow\infty}\varphi_n(t_{n,j})\les\limsup_{n\rightarrow\infty}\varphi_{n}(t_{n,j-1}).
\end{equation}
Starting with \eqref{eq:bdd at zero}, repeating \eqref{eq:varphi bootstrap} $J$ times shows
\[
\limsup_{n\rightarrow\infty}\norm{u_n}_{Z_{R_n}([0,T_n))}=\limsup_{n\rightarrow\infty}\varphi_n(t_{n,J})<\infty,
\]
contradicting \eqref{eq:ZR->infty}.

Therefore, $\norm{u_n}_{L^4_{t,x}([0,T_n)\times\T^2)}\rightarrow\infty$ holds. By \eqref{eq:Z_R stri}, we have
\begin{equation}\label{eq:ZR^2->infty}
\norm{u_n}_{Z_{R_n^2}([0,T_n))}\rightarrow\infty.
\end{equation}
By \eqref{eq:ZR^2->infty} and the $H^s(\T^2)$-continuity of each $u_n(t)$, for each $j\in\N$, there exists $n(j)<\infty$ such that for $n\ge n(j)$, there exist $0=T_{n,0}<T_{n,1}<T_{n,2}<\cdots <T_{n,j}<T_{n}$ such that
\begin{equation}
\norm{u_n-\Phi_{n,j-1}}_{Z_{R_n^2}(I_{n,j})}=\epsilon_0,\label{eq:Z1>1}
\end{equation}
where $\epsilon_0\le 1$ is a universal constant to be fixed shortly and we denoted
\[
\Phi_{n,j-1}:=e^{i(t-T_{n,j-1})\De}u(T_{n,j-1})
\]
and
\[
I_{n,j}=[T_{n,j-1},T_{n,j}).
\]
Denote $\mc A:=\{(n,j)\in\N^2:n\ge n(j)\}$. By \eqref{eq:Z1>1} and \eqref{eq:Rn def}, we have
\begin{align}\label{eq:Z R^2 bound}
\norm{u}_{Z_{R_n^2}(I_{n,j})}&\le \epsilon_0+\norm{\Phi_{n,j-1}}_{Z_{R_n^2}(I_{n,j})}
\\&\nonumber\les 1+\norm{u_n(T_{n,j-1})}_{L^2}+R_n^{-2s}\norm{u(T_{n,j-1})}_{H^s}
\\&\nonumber\les 1+R_n^{-s}\norm{u_n}_{Z_{R_n}}\les1,
\end{align}
in the last line of which the $L^2$-conservation of \eqref{eq:NLS} and \eqref{eq:Z norm bound} are used.

By \eqref{eq:Z1>1} and Lemma \ref{lem:nu bound}, for each $C\in2^\N$ and $n\gg_C1$, we have
\begin{align*}
\epsilon_0=\norm{u_n-\Phi_{n,j-1}}_{Z_{R_{n}^2}(I_{n,j})}
&=\norm{\mc I(\chi_{I_{n,j}}|u_n|^2u_n)}_{Z_{R_n^2}(I_{n,j})}
\\&\les\left(C^{-s}\norm{u_n}_{Z_{R_n^2}(I_{n,j})}^2+\norm{u_n}^2_{L^4_{t,x}(I_{n,j}\times\T^2)}\right)\norm{u_n}_{Z_{R_n^2}(I_{n,j})},
\end{align*}
which implies by \eqref{eq:Z R^2 bound} that
\[
\epsilon_0\les C^{-s}+\norm{u_n}_{L^4_{t,x}(I_{n,j}\times\T^2)}^2.
\]
Thus, choosing $C$ big enough, we have
\begin{align*}
\sqrt{\epsilon_0}&\les\norm{u_n}_{L^4_{t,x}(I_{n,j}\times\T^2)}
\\&\les\norm{u_n-\Phi_{n,j-1}}_{L^4_{t,x}(I_{n,j}\times\T^2)}+\norm{\Phi_{n,j-1}}_{L^4_{t,x}(I_{n,j}\times\T^2)}
\\&\les\epsilon_0+\norm{\Phi_{n,j-1}}_{L^4_{t,x}(I_{n,j}\times\T^2)},
\end{align*}
in the last line of which we used \eqref{eq:Z_R stri} and \eqref{eq:Z1>1}. Thus, choosing $\epsilon_0>0$ small enough, we have
\begin{equation}\label{eq:Phi large for GWP}
\norm{\Phi_{n,j-1}}_{L^4_{t,x}(I_{n,j}\times\T^2)}\gtrsim_{\epsilon_0}1.
\end{equation}
For $(n,j)\in\mc A$, let $(u_{n,j},[0,T_{n,j-1}))$, $u_{n,j}\in C^{0}H^{s}\cap Y^{s}$ be the cutoff
solution
\[
u_{n,j}:=\begin{cases}
e^{it\De}u_{n}(0) & ,\qquad t<0\\
u_{n}(t) & ,\qquad t\in[0,T_{n,j-1})\\
e^{i(t-T_{n,j-1})\De}u_{n}(T_{n,j-1})=\Phi_{n,j-1} & ,\qquad t\ge T_{n,j-1}
\end{cases}.
\]
By Lemma \ref{lem:inverse Z_R stri} and \eqref{eq:Phi large for GWP}, there exists a frame $\left\{ \mc O_{n,j}\right\}_{(n,j)\in\mc A} =\left\{ (N_{n,j},t_{n,j},x_{n,j},\xi_{n,j})\right\}_{(n,j)\in\mc A}$
such that
\[
\{ \iota_{\mc O_{n,j}}\left(\Phi_{n,j-1}\right)(0)\} _{(n,j)\in\mc A}\text{ and }\{\chi_{\tilde{I_{n,j}}}\}_{(n,j)\in\mc A}
\]
are both weakly nonzero, where we denoted by $\tilde{I_{n,j}}$ the time interval mapped from $I_{n,j}$ by $\mc O_{n,j}$.
Passing to a subsequence of $n$, by Lemma \ref{lem:weak lim is sol},
we have the following weak convergence for each $j\in\N$:
\begin{equation}
\iota_{\mc O_{n,j}}u_{n,j}\weak v_{j},\label{eq:->v}
\end{equation}
where $v_{j}\in C^{0}L^{2}\cap L_{t,x}^{4}(\R\times\R^{2})$ is a
cutoff solution to NLS on $\R^{2}$ and for each $T\in\R$,
\begin{equation}
\left(\iota_{\mc O_{n,j}}u_{n,j}\right)(T)\weak v_{j}(T).\label{eq:->v(T)}
\end{equation}
By Proposition \ref{prop:focusing R^2}, we have an a priori bound
\begin{equation}
\sup_{j\in\N}\norm{v_{j}}_{C^{0}L^{2}\cap L_{t,x}^{4}}\les_M1.\label{eq:vj bound}
\end{equation}
By (\ref{eq:vj bound}) and Lemma \ref{lem:weak lim is sol}, there
exists a cutoff solution $v_{*}\in C^{0}L^{2}\cap L_{t,x}^{4}(\R\times\R^{2})$
to NLS on $\R^{2}$ that scatters and a subsequence $\left\{ j_{k}\right\} $
such that for $T\in\R$, 
\[
v_{j_{k}}(T)\weak v_{*}(T)
\]
holds. Since $\{ \iota_{\mc O_{n,j}}u_{n,j}(0)\} _{(n,j)\in\mc A}$
is weakly nonzero, so is $\{ v_{j}(0)\} _{j\in\N}$, thus
$v_{*}(0)\neq0$.

We evaluate the following (distributional) limit in two ways:
\begin{equation}
\lim_{k\rightarrow\infty}\lim_{n\rightarrow\infty}\lim_{T\rightarrow\infty}e^{iT\De}\left(\iota_{\mc O_{n,j_{k}}}u_{n,j_{k}}\right)(-T).\label{eq:limlimlim}
\end{equation}
Since $\limsup_{n\rightarrow\infty}\text{\ensuremath{\norm{u_{n,j}}}}_{Z_{R_{n}}}$
is finite for each $j\in\N$ and $\left\{ v_{j}\right\} $ is bounded
in $C^{0}L^{2}\cap L_{t,x}^{4}(\R\times\R^{2})$, by Lemma \ref{lem:scatter T^d}
and Lemma \ref{lem:scatter R^d}, (\ref{eq:limlimlim}) equals
\begin{equation}
\lim_{k\rightarrow\infty}\lim_{T\rightarrow\infty}e^{iT\De}v_{j_{k}}(-T)=\lim_{T\rightarrow\infty}e^{iT\De}v_{*}(-T),\label{eq:limlimlim 1}
\end{equation}
which exists and is nonzero since $v_{*}\neq0$ scatters.

Since $u_{n, j}=e^{it\De}u_{n}(0)$ holds for $t<0$ and $\iota_{\mc O_{n,j_{k}}}\left(e^{it\De}u_{n}(0)\right)$
is a free evolution, (\ref{eq:limlimlim}) can also be rewritten as
\begin{equation}
\lim_{k\rightarrow\infty}\lim_{n\rightarrow\infty}\lim_{T\rightarrow\infty}e^{iT\De}\left(\iota_{\mc O_{n,j_{k}}}\left(e^{it\De}u_{n, j_{k}}(0)\right)\right)(-T)=\lim_{k\rightarrow\infty}\lim_{n\rightarrow\infty}\iota_{\mc O_{n,j_{k}}}\left(e^{it\De}u_{n}(0)\right)(0).\label{eq:limlimlim 2}
\end{equation}
Thus, we have
\begin{equation}
\lim_{k\rightarrow\infty}\lim_{n\rightarrow\infty}\iota_{\mc O_{n,j_{k}}}\left(e^{it\De}u_{n}(0)\right)(0)\neq0.\label{eq:!=00003D0}
\end{equation}
By \eqref{eq:!=00003D0} and the weak nonzeroness of $\{ \chi_{\tilde{I_{n,j_{k}}}}\} _{(n,j_k)\in\mc A}\subset\{ \chi_{\tilde{I_{n,j}}}\} _{(n,j)\in\mc A}$,
the family of cutoff free evolutions $\{ \iota_{\mc O_{n,j_{k}}}(\chi_{I_{n,j_{k}}}e^{it\De}u_{n}(0))\}_{(n,j_k)\in\mc A}=\{ \chi_{\tilde{I_{n,j_{k}}}}\iota_{\mc O_{n,j_{k}}}(e^{it\De}u_{n}(0))\}_{(n,j_k)\in\mc A}$
is also weakly nonzero. Thus, we have the critical norm bound
\begin{equation}
\liminf_{k\rightarrow\infty}\liminf_{n\rightarrow\infty}\norm{e^{it\De}u_{n}(0)}_{L_{t,x}^{4}(I_{n,j_{k}}\times\T^{2})}>0,\label{eq:lim>0}
\end{equation}
By (\ref{eq:lim>0}) and the Fatou's Lemma, we have
\begin{align*}
 & \liminf_{n\rightarrow\infty}\norm{e^{it\De}u_{n}(0)}_{L_{t,x}^{4}([0,T_{n})\times\T^{2})}\\
  \ge&\liminf_{n\rightarrow\infty}\norm{\norm{e^{it\De}u_{n}(0)}_{L_{t,x}^{4}(I_{n,j_{k}}\times\T^{2})}}_{\ell_{k}^{4}}\\
  \ge&\norm{\liminf_{n\rightarrow\infty}\norm{e^{it\De}u_{n}(0)}_{L_{t,x}^{4}(I_{n,j_{k}}\times\T^{2})}}_{\ell_{k}^{4}}=\infty,
\end{align*}
which contradicts Lemma \ref{lem:Z_R L4 stri}. Thus, the assumption
of this proposition cannot hold and we finish the proof.
\end{proof}
\section{Proof of Theorem \ref{thm:L4 counterexample lem} and its consequences}\label{sec:proof-approx}

In this section, we provide proofs of Theorem \ref{thm:L4 counterexample lem} and its Corollaries \ref{cor:unif-cont} and \ref{cor:sharp-bound}.

\begin{proof}[Proof of Theorem \ref{thm:L4 counterexample lem}]
In this proof, every comparability depends in default on $\la$ and $T$. For dyadic numbers $N\gg1$, denote
\[
U^N:=e^{\mp 3it \la^2\ln N}e^{it\De}u^N_0
\]
and
\[
\mc E^N:=u^N-U^N,
\]
where $u^N$ is the solution provided by \cite[Theorem 1.4]{herr2024strichartz}. Note that $u_0^N\in H^s(\T^2)$ for every $s>0$ and in \cite[Theorem 1.4]{herr2024strichartz} the smallness threshold is independent of $s$, therefore $u^N\in C^\infty_{t,x,\loc}(\R\times\T^2)$ by applying the Sobolev embedding and using that $u^N$ is a solution.

To prove Theorem \ref{thm:L4 counterexample lem}, it suffices to show
\begin{equation}\label{eq:long time bound}
\norm{\mc E^N}_{C^0 L^2 \cap L^4_{t,x}([0,\frac{T}{\log N})\times\T^2)}=o_N(1).
\end{equation}
In the rest of this proof, we prove \eqref{eq:long time bound}. We denote $R:=N^{13}$ and work with the norm $Z_R=Z_R^1$.
For each $\xi\in N^{10}\Z^2$, we have
\begin{align*}
\norm{\chi_{[0,\frac{T}{\log N})}e^{it|\xi|^2}\widehat{U^N(t)}(\xi)}_{V^2_t}
&\les \norm{e^{\mp 3it \la^2\ln N}\widehat{u^N_0}(\xi)}_{W^{1,1}_t([0,\frac{T}{\log N}))}\\
& \les N^{-1}e^{-|\xi/N^{11}|^2}
\end{align*}
because of $W^{1,1}\hook V^2$,
which implies
\begin{align}\label{eq:U bound}
\nonumber \norm{U^N}_{Z_R([0,\frac{T}{\log N}))}&\les\norm{U^N}_{Y^0([0,\frac{T}{\log N}))}+R^{-1}\norm{U^N}_{Y^1([0,\frac{T}{\log N}))}
\\&\nonumber
\les\left(\sum_{\xi\in N^{10}\Z^2}N^{-2}e^{-2|\xi/N^{11}|^2}\right)^{1/2}+R^{-1}\left(\sum_{\xi\in N^{10}\Z^2}N^{-2}|\xi|^2e^{-2|\xi/N^{11}|^2}\right)^{1/2}
\les 1.\nonumber
\end{align}
Following the proof of Lemma \ref{lem:nu bound}, for $C>0$ and any interval $I\subset [0,\frac{T}{\log N})$, under the assumption $R=N^{13}\gg_C1$, we also have
\begin{equation}\label{eq:Z1 tri}
\norm{u_1\overline{u_2}u_3}_{Z_R(I)}\les\sum_{\{j,k,l\}=\{1,2,3\}}\left(C^{-1}\norm{u_j}_{Z_R}\norm{u_k}_{Z_R}+\norm{u_j}_{L^4_{t,x}}\norm{u_k}_{L^4_{t,x}}\right)\norm{u_l}_{Z_R}.
\end{equation}
By \eqref{eq:Z1 tri} and \eqref{eq:Z_R stri}, we have
\begin{equation}\label{eq:UUE}
\norm{\mc I(|U^N|^2\mc E^N)}_{Z_R(I)}\les\left(C^{-1}+\norm{U^N}_{L^4_{t,x}(I\times\T^2)}\right)\norm{\mc E^N}_{Z_R(I)}
\end{equation}
and the same estimate holds for $(U^N)^2\overline{\mc E^N}$.

By Proposition \ref{prop:L4 Stri}, we have
\[
\norm{U^N}_{L^4_{t,x}([0,\frac{T}{\log N})\times \T^2)}\les 1,
\]
thus $[0,\frac{T}{\log N})$ can be partitioned into a finite number of intervals $I_1=[t_0,t_1),\ldots,I_J=[t_{J-1},t_J)$, where $t_0=0$ and $J=O_T(1)$, such that
\[
\sup_{j\le J}\norm{U^N}_{L^4_{t,x}(I_j\times \T^2)}\ll1.
\]
Then, $\mc E^N(t_0)=0$ holds and for $j=1,\ldots,J$, we have
\begin{align}
\nonumber
\norm{\mc E^N}_{Z_R(I_j)}={}&\norm{\pm\mc I(\chi_{I_j}|u^N|^2u^N)+e^{i(t-t_{j-1})\De}u^N(t_{j-1})-U^N}_{Z_R(I_j)}
\\ \nonumber
\les{}&
\norm{\pm\mc I(\chi_{I_j}|U^N|^2U^N)+e^{i(t-t_{j-1})\De}U^N(t_{j-1})-U^N}_{Z_R(I_j)}
\\&\label{eq:u-U1}
{}+2\norm{\mc I(\chi_{I_j}|U^N|^2\mc E^N)}_{Z_R(I_j)}+\norm{\mc I(\chi_{I_j}(U^N)^2\overline{\mc E^N})}_{Z_R(I_j)}
\\&\label{eq:u-U1b}
{}+\norm{U^N}_{Z_R(I_j)}\norm{\mc E^N}_{Z_R(I_j)}^2+\norm{\mc E^N}_{Z_R(I_j)}^3
\\&\label{eq:u-U2}
{}+\norm{e^{i(t-t_{j-1})\Delta}(u^N(t_{j-1})-U^N(t_{j-1}))}_{Z_R(I_j)}.
\end{align}
By \eqref{eq:UUE} with $C\gg1$, the quantities \eqref{eq:u-U1} and \eqref{eq:u-U1b} can further be bounded by
\[
\le\frac{1}{2}\norm{\mc E^N}_{Z_R(I_j)}+O(1)\cdot\norm{\mc E^N}_{Z_R(I_j)}^3
\]
and \eqref{eq:u-U2} has the trivial bound
\[
\les\norm{\mc E^N(t_{j-1})}_{L^2}+R^{-1}\norm{\mc E^N(t_{j-1})}_{H^1}\les\norm{\mc E^N}_{Z_R(I_{j-1})}.
\]
Thus, we have the bootstrapping bound
\begin{equation}\label{eq:u-U iter}
\begin{split}
&\norm{\mc E^N}_{Z_R(I_j)}\\
\les{}&\norm{\pm\mc I(\chi_{I_j}|U^N|^2U^N)+e^{i(t-t_{j-1})\De}U^N(t_{j-1})-U^N}_{Z_R(I_j)}+\norm{\mc E^N}_{Z_R(I_{j-1})}
\end{split}
\end{equation}
provided that the right-hand side is small enough. Here, the last term $\norm{\mc E^N}_{Z_R(I_{j-1})}$ is void if $j=1$. 
Iterating \eqref{eq:u-U iter} on $j=1,\ldots,J$ yields our goal \eqref{eq:long time bound} (since $J=O_T(1)$) once we show
\begin{equation}\label{eq:counter claim}
\sup_{j\le J}\norm{\mathcal{V}^N_j}_{Z_R(I_j)}=o_N(1),
\end{equation}
where
\[
\mathcal{V}^N_j:=\pm\mc I(\chi_{I_j}|U^N|^2U^N)+e^{i(t-t_{j-1})\De}U^N(t_{j-1})-U^N.
\]
For $Q=(\xi,\xi_1,\xi_2,\xi_3)$ we denote $\widehat{\phi^N}(Q\setminus\{\xi\})=\widehat{\phi^N}(\xi_1)\overline{\widehat{\phi^N}(\xi_2)}\widehat{\phi^N}(\xi_3)$ and define $f,g:I_j\times\Z^2\rightarrow\C$ to be the functions
\[
f(t,\xi):=\int_{t_{j-1}}^t e^{\mp 3is \la^2\ln N}\frac{\la^3}{N^3}\sum_{\substack{Q=(\xi,\xi_1,\xi_2,\xi_3)\in\mc Q(N^{10}\Z^2)\\ \tau(Q)\neq 0}}e^{is\tau(Q)}\widehat{\phi^N}(Q\setminus\{\xi\})ds
\]
and
\begin{align*}
g(t,\xi)&:=\int_{t_{j-1}}^t e^{\mp 3is \la^2\ln N}\Big(\frac{\la^3}{N^3}\sum_{Q=(\xi,\xi_1,\xi_2,\xi_3)\in\mc Q^0(N^{10}\Z^2)}\widehat{\phi^N}(Q\setminus\{\xi\})
-3\la^3\ln N\cdot N^{-1}\widehat{\phi^N}(\xi)\Big)ds.
\end{align*}
Then, for $t\in I_j$ and $\xi\in N^{10}\Z^2$, we have
\begin{align*}
&i\d_t\left(e^{it|\xi|^2}\widehat{\mathcal{V}^N_j}(\xi)\right)
\\
=&\pm e^{it|\xi|^2}\sum_{Q=(\xi,\xi_1,\xi_2,\xi_3)\in\mc Q(N^{10}\Z^2)}\widehat{U^N(t)}(Q\setminus\{\xi\})
\mp 3\la^2\ln N e^{\mp 3it \la^2\ln N}\widehat{u^N(0)}(\xi)
\\
=&\pm e^{\mp 3it \la^2\ln N}\left(\frac{\la^3}{N^3}\sum_{Q=(\xi,\xi_1,\xi_2,\xi_3)\in\mc Q(N^{10}\Z^2)}e^{it\tau(Q)}\widehat{\phi^N}(Q\setminus\{\xi\})
-3\la^2\ln N\cdot \frac{\la}{N}\widehat{\phi^N}(\xi)\right)
\\
=&\pm\left(\d_tf(t,\xi)+\d_tg(t,\xi)\right).
\end{align*}
Since $\mathcal{V}^N_j(t_{j-1})=0$ holds, our goal \eqref{eq:counter claim} can be rephrased as
\begin{equation}\label{eq:counter Z_R}
\norm{\F^{-1}f}_{Z_R}+\norm{\F^{-1}g}_{Z_R}=o_N(1).
\end{equation}
We actually show only the $Y^0$-part of \eqref{eq:counter Z_R} (which is the essential part)
\begin{equation}\label{eq:counter f+g claim}
\norm{f}_{\ell^2_{\xi}V^2_s(I_j)}+\norm{g}_{\ell^2_{\xi}V^2_s(I_j)}=o_N(1).
\end{equation}
One can check that the estimates to be provided also yield the $Y^1$-norm bound in \eqref{eq:counter Z_R}.

We first show that the non-rectangular-resonant part $f$ is negligibly small. For $L\in 2^\N$, the number of $Q\in\mc Q(N^{10}\Z^2)$ such that $Q\ni \xi$ and $\max_{\xi'\in Q\setminus\{\xi\}}|\xi'|\sim N^{10}L$ is $O(L^4)$. One can also check $\max_{\xi'\in Q\setminus\{\xi\}}|\xi'|\ge|\xi|/10$ always holds.
Thus, we have
\begin{align}\label{eq:f W11}
\norm{f(t,\xi)}_{W^{1,1}_t(I_j)}&\les\frac{\la^3}{N^3}\cdot\sum_{L\in2^\N}\sum_{\substack{Q=(\xi,\xi_1,\xi_2,\xi_3)\in\mc Q(N^{10}\Z^2)\\ N^{10}L<\max_{\xi'\in Q\setminus\{\xi\}}|\xi'|\le 2N^{10}L}}\widehat{\phi^N}(Q\setminus\{\xi\})
\\&\nonumber
\les\frac{\la^3}{N^3}\sum_{L\in2^\N}L^4\cdot e^{-\frac12(N^{10}L/N^{11})^2}\cdot e^{-\frac{1}{2}|\xi/10N^{11}|^2}
\\&\nonumber
\les\la^3N\cdot e^{-\frac{1}{200}|\xi/N^{11}|^2}.
\end{align}
We then measure the $L^\infty$-norm of $f(\cdot,\xi)$. For any parallelogram $Q\in\mc Q(N^{10}\Z^2)$, since all vertices of $Q$ lie in $N^{10}\Z^2$, $N^{20}\mid\tau(Q)$ holds. Thus, the oscillation estimate
\[
\left|\int_{t_{j-1}}^t e^{\mp3is\la^2\ln N+is\tau(Q)}ds\right|\les\frac{1}{|\tau(Q)|-|3\la^2\ln N|}\les\frac{1}{N^{20}}
\]
leads to an additional decay by $N^{-20}$ to \eqref{eq:f W11}, i.e.,
\begin{equation}\label{Eq:f L infty}
\norm{f(t,\xi)}_{L^\infty_t(I_j)}\les\la^3 N^{-19}\cdot e^{-\frac{1}{200}|\xi/N^{11}|^2}.
\end{equation}
Interpolating between \eqref{eq:f W11} and \eqref{Eq:f L infty}, we have
\[
\norm{f(t,\xi)}_{V^2_t(I_j)}\les\la^3 N^{-9}\cdot e^{-\frac{1}{200}|\xi/N^{11}|^2}.
\]
Thus, noting that $f$ is supported on $I_j\times N^{10}\Z^2$, we have
\begin{equation}\label{eq:fs small Y}
\norm{f}_{\ell^2_\xi V^2_t(I_j)}\les N^{-9}\cdot\sum_{\xi\in N^{10}\Z^2} e^{-\frac{1}{200}|\xi/N^{11}|^2}\les N^{-7}=o_N(1),
\end{equation}
providing the first half of \eqref{eq:counter f+g claim}.

Now we estimate $g$, which is indeed the main part. Since $W^{1,1}\hook V^2$, for each $\xi\in N^{10}\Z^2$, we have
\[
\norm{g(t,\xi)}_{V^2_t}\les\int_{I_j}\left|\frac{\la^3}{N^3}\sum_{\substack{Q=(\xi,\xi_1,\xi_2,\xi_3)\in\mc Q^0(N^{10}\Z^2)}}\widehat{\phi^N}(Q\setminus\{\xi\})
-3\la^3\ln N\cdot N^{-1}\widehat{\phi^N}(\xi)\right|ds,
\]
whose integrand is independent of $s$ and thus can be bounded by
\[
\les\frac{1}{\log N}\left|\frac{1}{N^3}\sum_{\substack{Q=(\xi,\xi_1,\xi_2,\xi_3)\in\mc Q^0(N^{10}\Z^2)}}\widehat{\phi^N}(Q\setminus\{\xi\})
-3\ln N\cdot N^{-1}\widehat{\phi^N}(\xi)\right|,
\]
where we used $|I_j|\le\frac{T}{\log N}\les\frac{1}{\log N}$.
Thus, we only need to show
\[
\sum_{\xi\in N^{10}\Z^2}\left|\frac{1}{N^3\ln N}\sum_{\substack{Q=(\xi,\xi_1,\xi_2,\xi_3)\in\mc Q^0(N^{10}\Z^2)}}\widehat{\phi^N}(Q\setminus\{\xi\})
-3 N^{-1}\widehat{\phi^N}(\xi)\right|^2=o_N(1).
\]
Rescaling the coordinates mapping $N^{10}\Z^2$ to $\Z^2$, this can be rewritten as
\begin{equation}\label{eq:counter final claim}
\sum_{\xi\in \Z^2}\left|\frac{1}{N^3\ln N}\sum_{\substack{Q=(\xi,\xi_1,\xi_2,\xi_3)\in\mc Q^0(\Z^2)}}\mathfrak{g}(Q\setminus\{\xi\})
-3N^{-1}\mathfrak{g}(\xi)\right|^2=o_N(1),
\end{equation}
where $\mathfrak{g}(\xi):=e^{-|\xi/N|^2}$.
For $\eta\in \Z^2_\irr$, let $\mc Q_\eta^0$ be the set
\[
\mc Q_\eta^0:=\{(x_1,x_2,x_3,x_4)\in\mc Q^0:0\neq x_1-x_2\parallel\eta\text{ or }0\neq x_1-x_4\perp\eta\}.
\]
Then, $\mc Q^0$ is the union of $\mc Q_\eta^0$ for $\eta\in\Z^2_\irr$ plus the singleton case $\{(\xi_0,\xi_0,\xi_0,\xi_0):\xi_0\in\Z^2\}$, counting each element twice (concerning $\pm\eta$).

We have
\begin{align*}
\sum_{\substack{Q=(\xi,\xi_1,\xi_2,\xi_3)\in\mc Q^0_\eta(\Z^2)}}\mathfrak{g}(Q\setminus\{\xi\})
&
=\sum_{(m,n)\neq(0,0)}e^{(-|\xi+m\eta|^2-|\xi+n\eta^\perp|^2-|\xi+m\eta+n\eta^\perp|^2)/N^2}
\\&
=\sum_{(m,n)\neq(0,0)}e^{(-|\xi|^2-2|\xi+m\eta+n\eta^\perp|^2)/N^2}.
\end{align*}
In the case $|\eta|\le N$, since the Gaussian function $e^{-2|x|^2}$ is Lipschitz on $\R^2$, we obtain
\begin{align*}
\sum_{(m,n)\neq(0,0)}e^{(-|\xi|^2-2|\xi+m\eta+n\eta^\perp|^2)/N^2}&=\frac{N^2}{|\eta|^2}\int_{\R^2}e^{-2|x|^2}dx\cdot e^{-|\xi/N|^2}+O\left(\frac{N}{|\eta|}\right)\cdot e^{-|\xi/N|^2}
\\&
=\frac{\pi}{2}\cdot\frac{N^2}{|\eta|^2}\mathfrak{g}(\xi)+O\left(\frac{N}{|\eta|}\right)\cdot e^{-|\xi/N|^2}.
\end{align*}
In the case $|\eta|>N$ we have
\begin{align*}
&\sum_{(m,n)\neq(0,0)}e^{(-|\xi|^2-2|\xi+m\eta+n\eta^\perp|^2)/N^2}
\\
\le&\sup_{(m,n)\neq(0,0)}e^{(-|\xi|^2-|\xi+m\eta+n\eta^\perp|^2)/N^2}\cdot \sum_{(m,n)\neq(0,0)}e^{(-|\xi+m\eta+n\eta^\perp|^2)/N^2}
\\
\le& e^{-|\xi/N|^2/2-|\eta/N|^2/10}\cdot O(1),
\end{align*}
in the last line of which we used $|\xi|^2+|\xi+m\eta+n\eta^\perp|^2\ge|\xi|^2/2+|m\eta+n\eta^\perp|^2/10$ and $|m\eta+n\eta^\perp|\ge|\eta|$.

Thus, for $\xi\in\Z^2$, we conclude
\[2\cdot\sum_{\substack{Q=(\xi,\xi_1,\xi_2,\xi_3)\in\mc Q^0\\ Q\neq(\xi,\xi,\xi,\xi)}}\mathfrak{g}(Q\setminus\{\xi\})
=\sum_{\substack{\eta\in\Z^2_\irr\\ |\eta|\le N}}\frac{\pi}{2}\cdot\frac{N^2}{|\eta|^2}\mathfrak{g}(\xi)+e^{-|\xi/N|^2/2}\cdot O(N^2).
\]
Then, with the density $d_\irr$ of the coprime integers, see Remark \ref{rem:dens}, we compute
\begin{align*}
\lim_{N\rightarrow\infty}\frac{1}{\ln N}\sum_{\substack{\eta\in\Z^2_\irr\\|\eta|\le N}}\frac{1}{|\eta|^2}
&
=d_{\irr}\cdot\lim_{N\rightarrow\infty}\frac{1}{\ln N}\int_{1\le |\eta|\le N}\frac{d\eta}{|\eta|^2}
=2\pi d_{\irr},
\end{align*}
Remark \ref{rem:dens} below implies that  $d_\irr \frac{\pi^2}{2}=3$, therefore we have
\begin{equation}\label{eq:g approx}
\sum_{\substack{Q=(\xi,\xi_1,\xi_2,\xi_3)\in\mc Q^0\\ Q\neq(\xi,\xi,\xi,\xi)}}\mathfrak{g}(Q\setminus\{\xi\})=3 N^2\ln N\cdot\mathfrak{g}(\xi)+e^{-|\xi/N|^2/2}\cdot O(N^2).
\end{equation}
For the fully degenerate case $Q=(\xi,\xi,\xi,\xi)$, we have $\mathfrak{g}(Q\setminus\{\xi\})=\mathfrak{g}^3(\xi)=e^{-3|\xi/N|^2}$, thus we may drop the condition $Q\neq(\xi,\xi,\xi,\xi)$ in \eqref{eq:g approx}.

Since this holds for every $\xi$, \eqref{eq:counter final claim} immediately follows, finishing the proof.
\end{proof}

\begin{rem}\label{rem:dens}
\begin{enumerate}
\item 
We have used the asymptotic density of coprime integer points, i.e.
\[
d_{_\irr}:=\lim_{N\rightarrow\infty}\frac{\#\{\eta\in\Z^2_\irr:|\eta|\le N\}}{\pi N^2}.
\]
It is well-known that $d_{\irr}=\frac{1}{\zeta(2)}=\frac{6}{\pi^2}$. If the density is computed with respect to large squares this is a classical result of Mertens \cite{Mertens}, see also \cite[Thm. 3.9]{apostol}. For the case of large discs, which is considered here, this fact can be found in \cite{DicksPorti} or \cite[Prop. 6]{Baake}.
\item 
 In the proof of Theorem \ref{thm:L4 counterexample lem} above  the phase correction factor $3$ is a result of a subtle computation. More precisely, $$\pi d_{\irr}\int_{\R^2}e^{-2|x|^2}dx=3.$$
\end{enumerate}
\end{rem}

\begin{lem}\label{lem:free evol sharp 1/log}
Let $N\in\N$ and $T\ge1$. Let $\phi^N:=\F^{-1}(\chi_{N^{10}\Z^2}\cdot e^{-|\xi/N^{11}|^2})$. We have
$$ \|\phi^N\|_{L^2(\T^2)} \sim N$$
and
\begin{equation}\label{eq:free evol sharp 1/log}
\norm{e^{it\De}\phi^N}_{L^4_{t,x}([0,\frac{T}{\log N})\times\T^2)}\gtrsim N T^{1/4}.
\end{equation}
\end{lem}
\begin{proof}
Let $g:\R/2\pi\Z\rightarrow[0,\infty)$ be the function
\[
g(t):=\int_{\T^2}\left|e^{it\De}\phi^N\right|^4dx.
\]
Since $\supp(\widehat{\phi^N})\subset N^{10}\Z^2$, we have
\[
\supp (\widehat{g})\subset N^{20}\Z.
\]
Thus, $g$ is $2\pi N^{-20}$-periodic. Let $T'<\frac{T}{\log N}$ be the largest multiple of $2\pi N^{-20}$ less than $\frac{T}{\log N}$, so that $g$ is periodic on $[0,T']\subset[0,\frac{T}{\log N}]$. We have
\begin{equation}\label{eq:int g(t)}
\norm{e^{it\De}\phi^N}_{L^4_{t,x}([0,\frac{T}{\log N})\times\T^2)}^4=\int_0^{\frac{T}{\log N}}g(t)dt\ge\int_0^{T'}g(t)dt=T'\widehat g(0)
\end{equation}
Since $T\ge 1$, $T'\sim \frac{T}{\log N}$ holds. Thus, we can continue the estimate as
\[
\gtrsim\frac{T}{\log N}\sum_{Q\in\mc Q^0}\widehat{\phi^N}(Q).
\]
Here, since $\widehat{\phi^N}(\xi)\gtrsim1$ holds for $\xi\in[N^{11}]^2$, we have
\begin{equation}\label{eq:sum Q over N11}
\sum_{Q\in\mc Q^0}\widehat{\phi^N}(Q)\gtrsim\#\mc Q^0([N^{11}]^2\cap N^{10}\Z^2)\gtrsim\#\mc Q^0([N]^2).
\end{equation}
It is known that $\#\mc Q^0([N]^2)\gtrsim N^4\log N$; see, e.g., \cite[(2.2)]{kishimoto2014remark}. We also compute that $\|\phi^N\|_{L^2}\sim N$. Plugging \eqref{eq:sum Q over N11} into \eqref{eq:int g(t)} yields \eqref{eq:free evol sharp 1/log}, finishing the proof.
\end{proof}

The following proposition implies Corollary \ref{cor:unif-cont} and contains more details.
\begin{prop}\label{prop:unif counterexample}
Let $\epsilon>0$. There exist sequences $\{u_{0n}\}$ in $C^\infty(\T^2)$ and $\{t_n\},\{\la_n\}$ of positive reals satisfying the following:
\[
\lim_{n\rightarrow\infty}\la_n=0,
\]
\[
\lim_{n\rightarrow\infty}t_n=0,
\]
\[
\lim_{n\rightarrow\infty}\norm{u_{0n}}_{L^2(\T^2)}=\epsilon,
\]
and
\begin{equation}\label{eq:unif counterexample main}
\lim_{n\rightarrow\infty}\norm{u_n(t_n)-\tilde{u_n}(t_n)}_{L^2(\T^2)}=2\epsilon,
\end{equation}
where $u_n$ and $\tilde{u_n}$ are the solutions to \eqref{eq:NLS}, either defocusing or focusing, with the initial data $u_{0n}$ and $\tilde{u_{0n}}=(1+\la_n)u_{0n}$, respectively.
In particular, the $L^2(\T^2)$-flow map is not uniformly continuous on any neighborhood of $0$.
\end{prop}
\begin{proof}
Let $\phi^N,N\in2^\N$ be defined in Theorem \ref{thm:L4 counterexample lem}. Let $\la$ be the normalizing constant
\[
\la=\epsilon\cdot\lim_{N\rightarrow\infty}N\norm{\phi^N}_{L^2(\T^2)}^{-1}.
\]
Let $u^N_0\in C^\infty(\T^2)$ be as in Theorem \ref{thm:L4 counterexample lem}. We have
\begin{equation}\label{eq:la is normalizer}
\lim_{N\rightarrow\infty}\norm{u^N_0}_{L^2(\T^2)}=\lim_{N\rightarrow\infty}\la N^{-1}\norm{\phi^N}_{L^2(\T^2)}=\epsilon.
\end{equation}
Let $\{N_n\}$ be a sequence of dyadic numbers to be fixed later such that $\ln N_n\ge n^2$.
We provide explicitly the parameters $\la_n,t_n,u_n,\tilde{u_n}$.
For $n\in\N$, we set
\[
\la_n:=1/n,
\]
\[
t_n:=\frac{\pi}{((\la(1+\la_n))^2-\la^2)\ln N_n},
\]
\[
u_{0n}:=u^{N_n}_0\text{, and }\tilde{u_{0n}}:=(1+\la_n)u^{N_n}_0.
\]
Then, all conditions in Proposition \ref{prop:unif counterexample} but \eqref{eq:unif counterexample main} are immediate. By Theorem \ref{thm:L4 counterexample lem}, we have
\begin{equation}\label{eq:u^N_n}
\lim_{n\rightarrow\infty}\norm{u_n(t_n)-e^{\mp3it_n\la^2\ln N_n}e^{it_n\De}u_{0n}}_{L^2(\T^2)}=0
\end{equation}
and, choosing a rapidly increasing sequence $\{N_n\}$,
\begin{equation}\label{eq:tilde u^N_n}
\lim_{n\rightarrow\infty}\norm{\tilde{u_n}(t_n)-e^{\mp3it_n(\la(1+\la_n))^2\ln N_n}e^{it_n\De}\tilde{u_{0n}}}_{L^2(\T^2)}=0.
\end{equation}
By a triangle inequality on \eqref{eq:u^N_n} and \eqref{eq:tilde u^N_n}, we have
\begin{align}\label{eq:u-tilde u^N_n}
&\limsup_{n\rightarrow\infty}\norm{\tilde{u_n}(t_n)+u_n(t_n)}_{L^2(\T^2)}
\\ \nonumber
\le&\limsup_{n\rightarrow\infty}\norm{e^{\mp3it_n(\la(1+\la_n))^2\ln N_n}e^{it_n\De}\tilde{u_{0n}}+e^{\mp3it_n\la^2\ln N_n}e^{it_n\De}u_{0n}}_{L^2(\T^2)}
\\ \nonumber
\les&\limsup_{n\rightarrow\infty}\left|e^{\mp3it_n(\la(1+\la_n))^2\ln N_n}+e^{\mp3it_n\la^2\ln N_n}\right|\le0.
\end{align}
By \eqref{eq:la is normalizer}, \eqref{eq:u-tilde u^N_n}, and the $L^2$-conservation of \eqref{eq:NLS}, we conclude \eqref{eq:unif counterexample main}.
\end{proof}

\begin{proof}[Proof of Corollary \ref{cor:sharp-bound}]
By Theorem \ref{thm:L4 counterexample lem}, we have
\begin{equation}\label{eq:L4 approx for unbdd}
\limsup_{N\rightarrow\infty}\norm{u^N-e^{\mp 3it\la^2\ln N}e^{it\De}u^N_0}_{L^4_{t,x}([0,\frac{T}{\log N})\times\T^2)}=0.
\end{equation}
By Lemma \ref{lem:free evol sharp 1/log}, we have
\begin{align}\label{eq:L4 large for unbdd}
&\limsup_{N\rightarrow\infty}\norm{e^{\mp 3it\la^2\ln N}e^{it\De}u^N_0}_{L^4_{t,x}([0,\frac{T}{\log N})\times\T^2)}
\\=&\limsup_{N\rightarrow\infty}\norm{\la N^{-1}e^{it\De}\phi^N}_{L^4_{t,x}([0,\frac{T}{\log N})\times\T^2)}\gtrsim \la T^{1/4}.\nonumber
\end{align}
Applying the triangle inequality to \eqref{eq:L4 large for unbdd} and \eqref{eq:L4 approx for unbdd}, we finish the proof.
\end{proof}

\section*{Acknowledgements}
The authors would like to thank the anonymous referees for their careful reports which helped to improve the paper.

\section*{Declaration}
Funded by the Deutsche Forschungsgemeinschaft (DFG, German Research Foundation) -- IRTG 2235 -- Project-ID 282638148.
The second author is partially supported by National Research Foundation of Korea, RS-2019-NR040050.
 
\bibliographystyle{abbrv}
\bibliography{citationforTd}
\end{document}